\documentclass[twoside, 12pt]{report}
\usepackage[utf8]{inputenc}
\usepackage{graphicx}
\usepackage[a4paper, total={6.2in, 8.7in}]{geometry}
\usepackage{biblatex}

\usepackage{currvita}
\usepackage{aligned-overset}
\usepackage{mathrsfs, graphics, xfrac}
\usepackage{amssymb}

\usepackage{esint}

\usepackage{mathtools}
\usepackage{stmaryrd}

\usepackage{amsmath}
\usepackage{amsthm}

\usepackage[notref,notcite,final]{showkeys}
\usepackage{hyperref}\hypersetup{colorlinks=true, citecolor=blue, linkcolor=black}
\usepackage{graphicx, epstopdf}\graphicspath{{./immaginiprova/}}
\usepackage{bbm}
\usepackage{caption}
\usepackage{stmaryrd}
\usepackage{tikz}
\usetikzlibrary{decorations.markings}
\usetikzlibrary{arrows}
\usepackage{longtable}
\usepackage{afterpage}
\usepackage{epstopdf}

\numberwithin{equation}{section}

\newtheorem{theorem}{Theorem}[section]

\newtheorem{definition}[theorem]{Definition}
\let\olddefinition\definition
\renewcommand{\definition}{\olddefinition\normalfont}

\theoremstyle{definition}
\newtheorem{example}[theorem]{Example}

\theoremstyle{plain}
\newtheorem{lemma}[theorem]{Lemma}

\newtheorem*{plateau}{Plateau's problem for normal currents}
\newtheorem{open problem}{Open problem}
\newtheorem{claim}[theorem]{Claim}

\newtheorem{corollary}[theorem]{Corollary}
\newtheorem{proposition}[theorem]{Proposition}

\newtheoremstyle{mytheorem}
{}
{}
{\it}
{}
{\bf}
{.}
{ }
{\thmnumber{#2.~}\thmname{#1}\thmnote{~\rm#3}}

\newtheoremstyle{myremark}
{}
{}
{}
{}
{\it}
{.}
{ }
{\thmnumber{#2.~}\thmname{#1}\thmnote{~\rm#3}}

\newtheoremstyle{myparagraph}
{}
{}
{\rm}
{\parindent}
{\bf}
{.}
{ }
{\thmnumber{#2.~}\thmname{#1}\thmnote{#3}}

\newtheoremstyle{named}{}{}{\itshape}{}{\bfseries}{.}{.5em}{\thmnote{#3 }#1}
\theoremstyle{named}
\newtheorem*{namedproblem}{Problem}

\newtheoremstyle{named}{}{}{\itshape}{}{\bfseries}{.}{.5em}{\thmnote{#3 }#1}
\theoremstyle{named}

\theoremstyle{myremark}
\newtheorem{remark}[theorem]{Remark}

\theoremstyle{myparagraph}

\newtheorem*{parag*}{}

\newcommand\myfontsize{\fontsize{16pt}{20pt}\selectfont}
\newcommand\mysecondfontsize{\fontsize{12pt}{20pt}\selectfont}

\newenvironment{itemizeb}
{\begin{itemize}\itemsep=2pt}{\end{itemize}}

\usepackage{tocloft}
\usepackage{bbm}

\usepackage{setspace}

\usepackage{fancyhdr}
\let\headruleORIG\headrule
\renewcommand{\headrule}{\color{black} \headruleORIG}

\usepackage{colortbl}
\arrayrulecolor{black}

\fancypagestyle{plain}{
\fancyhf{}
  \fancyhead{}
    \fancyfoot[C]{\thepage}
  
}
\usepackage[notref,notcite,final]{showkeys}

\newcommand\res{\mathop{\hbox{\vrule height 7pt width .3pt depth 0pt
\vrule height .3pt width 5pt depth 0pt}}\nolimits}

\newcommand{\Flat}{\mathbb{F}}
\newcommand{\Mass}{\mathbb{M}}
\newcommand{\TP}{\textbf{TP}}
\newcommand{\OTP}{\textbf{OTP}}
\newcommand{\AMC}{\textbf{AMC}}

\newcommand{\BR}{\textbf{BR}}

\newcommand{\MM}{\mathbb{M}^\alpha}

\newcommand{\R}{\mathbb{R}}
\newcommand{\I}{\mathcal{I}}

\newcommand{\Ex}{\mathbf{E}}
\newcommand{\N}{\mathbb{N}}
\newcommand{\Z}{\mathbb{Z}}

\newcommand{\Po}{\mathbb{P}}

\newcommand{\Haus}{\mathcal{H}}

\newcommand{\eps}{\varepsilon}

\newcommand{\Lip}{\mathrm{Lip}}

\newcommand{\dist}{\mathrm{dist}}
\newcommand{\supp}{\mathrm{supp}}

\newcommand{\dV}{d_V\kern-1pt}

\usepackage{breqn}

\usepackage{pdfpages}

\begin{document}

\thispagestyle{empty}
\newgeometry{left=1.65cm, right=1.25cm, bottom=0.1cm}
\begin{center}
 \vspace{2cm}

\includegraphics[scale=1]{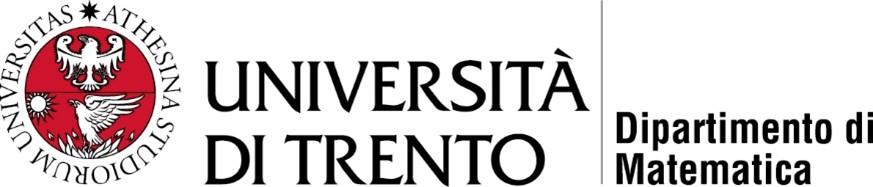}

\par
\vspace{2cm}

\Large{DEPARTMENT OF MATHEMATICS} \

\vspace{0.5cm}

MASTER DEGREE IN MATHEMATICS \

\par
\vspace{3.25cm}

\textbf{\myfontsize Well-posedness properties of geometric variational problems:}

\vspace{0.15cm}

\textbf{\myfontsize existence, regularity and uniqueness results}

\vspace{3.25cm}

\Large

\begin{minipage}{3in}
\textbf{Supervisor} \

Prof.\hspace{1mm}Andrea\hspace{1mm}Marchese
\end{minipage}
\hfill
\begin{minipage}{2.33in}
\textbf{Candidate} \

Gianmarco\hspace{1mm}Caldini
\end{minipage}

%

\vspace{2cm}

\raggedright

\center
\vspace{1cm}
\mysecondfontsize
ACADEMIC YEAR: 2021/2022

\vspace{1cm}

DATE OF DISCUSSION: OCTOBER 21, 2022
\end{center}

\pagenumbering{Roman}
\setcounter{page}{1} 
\chapter*{}
\thispagestyle{empty}
\newpage

\restoregeometry

\chapter*{Abstract}
This thesis is devoted to the study of well-posedness properties of some \textit{geometric variational problems}: existence, regularity and uniqueness of solutions. We study two specific problems arising in the context of geometric calculus of variations and sharing strong analogies: the \emph{Plateau's problem} and the \emph{optimal branched transport problem}. The first part of the thesis discusses the existence theory. Both problems are formulated in the language of Federer and Fleming's \emph{theory of currents}. After an exposition of the main results, we will present the core ideas of the (interior) regularity theory for \emph{area-minimizing currents} and for \emph{optimal transport paths}. The last part of the thesis contains two original results: the \textit{generic uniqueness} of solutions both for the Plateau's problem (in any dimension and codimension) and for the optimal branched transport problem. \newpage




\leavevmode\thispagestyle{empty}\newpage
\thispagestyle{empty}

\newpage
\tableofcontents

\clearpage{\thispagestyle{empty}\cleardoublepage}

\chapter*{Introduction} 
\addcontentsline{toc}{chapter}{Introduction}

\section*{Motivation and comments}

In the last four or five decades there has been growing interest in building robust theories to address \textit{geometric variational problems}. The archetypal geometric variational problem is the celebrated \textit{Plateau's problem}, named in honor of the Belgian physicist Joseph Plateau, who extensively studied the structure of soap films. In fact, the Plateau's problem had been formulated much earlier in the second half of the eighteenth century by Lagrange, who asked the following question: 
\begin{center}
{\it``Given a closed curve in $\R^3$, can one always find a surface of minimal area among all surfaces that bound the curve?"}
\end{center}
Many versions of this problem have been developed through the years, addressing the Plateau's problem as a \textit{boundary value problem for area-minimizing surfaces}. The first successful solution to the Plateau's problem in a concrete case was delivered by Schwarz in 1865. Nevertheless, only in 1930 a general existence theory was finally achieved by Rad\'o [\ref{Rado30}] and Douglas [\ref{Douglas31}].

The extension of the Plateau's problem to any dimension and codimension took many more years and the joint effort of some of the most brilliant mathematicians of the twentieth century, including Whitney, Reifenberg, Almgren, De Giorgi and Bombieri, just to mention a few. The ultimate solution to the Plateau's problem can be considered the work by Federer and Fleming [\ref{FF}], developing a measure-theoretic notion of generalized surfaces called \textit{integral currents}. The ``generalized Plateau's problem" can be stated as follows\footnote{For the precise definitions we refer to Chapter 1.}:
\begin{center}
{\it ``Given a $k$-dimensional integral current $S$ without boundary in $\R^d$, find a $(k+1)$-dimensional integral current $T_1$ such that $\partial T_1=S$ and $\mathbb{M}(T_1)$ is minimized among all $(k+1)$-dimensional integral currents $T$ with the same boundary $\partial T = S$."}
\end{center}
As it is often the case in the calculus of variations, in order to gain ``enough compactness" to solve a variational problem one has to enlarge the class of competitors, giving up some a priori regularity assumptions. Following this fundamental idea, many other notions of generalized surfaces have been developed: besides Federer and Fleming's \textit{theory of integral currents}, it is worth mentioning Caccioppoli and De Giorgi's \textit{finite perimeter sets}, see [\ref{Degiorgiperimetri}], and Almgren and Allard's \textit{rectifiable varifolds}, see [\ref{AAvarifolds}].

Remarkably, it turns out that area-minimizing surfaces may exhibit singularities. Hence, a whole theory has been developed to understand how regular the aforementioned area-minimizing generalized surfaces are. In cases where singularities appear, one would like to estimate their ``size" and structure and, possibly, to develop a full classification. There is a huge difference between the regularity theory for generalized area-minimizing surfaces depending on the codimension of the ambient space: indeed, the regularity theory deeply differs between the codimension one, see [\ref{Almgrencod1}, \ref{Degiorgifrontiere}, \ref{Degiorgibernstein}, \ref{Federercod1}, \ref{Flemingcod1}, \ref{Simonscod1}], and the codimension higher-than-one cases, see [\ref{Big}, \ref{Almgren2000}, \ref{Delellisnote}, \ref{DLSQ}, \ref{DLS1}, \ref{DLSsns}, \ref{DLS2}, \ref{DLS3}].

Furthermore, the very innocent question of ``how many minimal surfaces can be spanned by a given closed curve" turns out to be one of the most challenging related to the Plateau's problem. Indeed, the answer to this question is still not known in full generality. The first partial answers go back to the first decades of the twentieth century, due to the works by Rad\'{o}, Courant, Tromba, Nitsche, Tomi and many others. Uniqueness for a particular boundary curve in $\R^3$ is known only in very restrictive cases and many examples have been provided of boundaries admitting even infinitely many minimizers, see [\ref{Courant}] and [\ref{Morganinfinite}]. As a result, different approaches have been developed to study uniqueness questions. Arguably, the most fruitful was by means of Baire categories, as shown by B\"ohme and Tromba [\ref{BTromba}] and, more recently, by Morgan [\ref{Morganinventiones}, \ref{Morganindiana}, \ref{MorganARMA}]. Their main results establish that under some (rather restrictive) conditions, the set of curves which bound a unique minimizer is \textit{topologically large}. 

Several variants of the Plateau's problem have been developed in the last decades, each aiming at the minimization of different notions of \textit{energy}. One example comes from \textit{optimal transport}: the problem to find the best way to carry a given source onto a given target. Such problem witnessed an impressive progression in the last thirty years, developing deep connections with many fields of mathematics and serving as a model for biological and human-designed systems. Extensions of the original formulation of the optimal transport problem, due to Monge [\ref{Monge}], have been studied for transportation systems that privilege group flows rather than spread-out processes, leading to optimal transport networks with peculiar ramified structures: this class of problems is nowadays known as \textit{optimal branched transport}. Starting from the work by Xia [\ref{Xia2003}], it has been possible to develop the modern theory of optimal branched transport as a Plateau-type problem, see for instance [\ref{BF}, \ref{BrBuSa}, \ref{BW}, \ref{BW1}, \ref{brabutsan}, \ref{BrGaReSun}, \ref{BrPaSun}, \ref{BrSun}, \ref{CDRMjmpa}, \ref{MMST}, \ref{MMT}, \ref{Morsant}, \ref{PaoliniStepanov}]. More precisely, the optimal branched transport problem can be formulated as a boundary value problem for 1-dimensional currents minimizing a fractional power of the mass functional, called $\alpha$-\textit{mass.} Hence, it is not surprising how the theory of currents is a fruitful tool to study well-posedness properties in optimal branched transport theory.

The purpose of this thesis is to study fundamental properties of the two aforementioned examples of geometric variational problems, in particular existence, regularity and uniqueness results. We will always try to emphasize the (numerous) analogies between the theories. We conclude by mentioning that the original results presented here are part of the work done by the author during his Master studies and have been obtained in collaboration with A. Marchese and S. Steinbr\"uchel. We conclude this brief introduction by remarking that the well-posedness of geometric variational problems is still a live and flourishing research field, full of open questions that need to be investigated for years to come.

\section*{Guide to the thesis}

We briefly summarize the content of each chapter below.

\subsubsection{Chapter 1: Existence results}
The main goal of Chapter 1 is to present the existence theory for the solutions to the Plateau's problem (referred to as \textit{area-minimizing integral currents}) and to the optimal branched transport problem (referred to as \textit{optimal transport paths}). Both theories rely on the general notion of \textit{rectifiable} and \textit{normal currents} in the sense of Federer and Fleming. Hence, in Section 1.1 we will introduce the main notation and recall some preliminaries about measure theory and multilinear algebra. In Section 1.2 we will introduce the theory of currents (in the sense of de Rham), stating Federer and Fleming's celebrated \textit{closure theorem}, see [\ref{FF}], and showing the existence of solutions to the Plateau's problem for integral currents. In Section 1.3 we will then introduce the main framework of optimal branched transport theory, proving the existence of an optimal transport path with finite cost.

\subsubsection{Chapter 2: Regularity results}
The main goal of Chapter 2 is to present the regularity theory for area-minimizing integral currents and for optimal transport paths. In Section 2.1 we will investigate the interior regularity theory for area-minimizing integral currents. In Section 2.1.1 we will highlight the main ideas behind the proof of the De Giorgi-Allard $\varepsilon$\textit{-regularity theorem}, see [\ref{Degiorgifrontiere}] and [\ref{AAvarifolds}]. We are going to prove a simplified version of it in the language of Federer and Fleming's theory of currents that will highlight the main ideas of the theory such as the \textit{excess decay} and the \textit{harmonic approximation}. In doing so, we aim at a more accessible introduction to this theory, avoiding on purpose some technical details. In Section 2.1.2, we will survey the main difficulties to extend De Giorgi-Allard's regularity theory to any codimension. We will present Almgren's theory of {\em Dir-minimizing $Q$-valued functions} and describe the main issues in passing from this (linear) setting to the nonlinear version of Almgren's partial regularity theorem. At the end we collect some of the most interesting open problems in the field. In this presentation we will mostly follow [\ref{Delellisnote}, \ref{DLSQ}]. Finally, in Section 2.2, we will present the main result in the regularity of optimal branched transport, which is due to Xia [\ref{Xiaregularity}]. We will observe that a technical passage, which is only partially justified in the current literature, can be obtained as a consequence of the recent \textit{stability property} for optimal transport paths, see [\ref{CDRMcpam}].

\subsubsection{Chapter 3: Uniqueness results}

The main goal of Chapter 3 is to present the uniqueness theory for the Plateau's problem and for the optimal branched transport problem. After a brief discussion about the main uniqueness and nonuniqueness theorems for solutions of these two geometric variational problems, we will pass to the most original contributions of this thesis: in Section 3.1 we exploit Almgren's regularity theory in higher codimension to prove that, \textit{generically} (in the sense of Baire categories), every integral $(m-1)$-current without boundary spans a unique minimizer in $\mathbb{R}^{m+n}$. In Section 3.2 we prove the generic uniqueness of minimizers of the optimal branched transport problem.

\subsubsection{Acknowledgments}
I would like to spend few words to thank Prof. Andrea Marchese, who kindly led me during the study of all these topics. I would like to thank Andrea and Simone for having introduced me to my first steps into mathematical research. I would like to thank all the participants to the reading seminars ``An introduction to the regularity theory for generalized minimal surfaces" for all the valuable discussions and suggestions.

\cleardoublepage
\chapter{Existence results} 
\pagenumbering{arabic}

The main goal of Chapter 1 is to present the existence theory for the solutions to the Plateau's problem and to the optimal branched transport problem, both relying on Federer and Fleming's \textit{theory of currents}.

We do not aim to be exhaustive and most of the proofs will be omitted: we will focus on definitions and results that will be relevant for the sequel. In Section 1.1 we introduce the main notation and recall some preliminaries about measure theory and multilinear algebra. In Section 1.2 we introduce the theory of currents (in the sense of de Rham), stating Federer and Fleming's celebrated \textit{closure theorem}, see [\ref{FF}], which leads to the solution of the Plateau's problem for integral currents. In Section 1.3 we will introduce the main framework of optimal branched transport theory, proving the existence of an optimal transport path with finite cost. For a complete treatise on the subjects, we refer the reader to [\ref{GMS}], [\ref{Simonbook}] or [\ref{Federerbook}] for the theory of currents and to [\ref{BCM}] for the theory of optimal branched transport.

\section{Preliminaries}
\subsubsection{Preliminaries in measure theory}
We denote by $\mathcal{B}(\R^d)$ the Borel $\sigma$-algebra of $\R^d$, that is the smallest $\sigma$-algebra containing all open sets of $\R^d$. We will denote by $\mu$ a positive \textit{Borel} measure, which is a measure such that all Borel sets are measurable. 
\begin{definition}
A measure $\mu$ is called \textit{Borel regular} if it is Borel and if for every $\mu$-measurable set $A$ there exists $B$ Borel set such that $B \supset A$ and $\mu(A)=\mu(B)$. The measure $\mu$ is said to be a \textit{Radon} measure if it is Borel-regular and $\mu(K) < \infty$ for every compact subset $K$ of $\R^d$.
\end{definition}

Given $A \subset \R^d$, the \textit{restriction} of $\mu$ on $A$ is the measure $$(\mu \res A)(B) := \mu(A \cap B) \quad \text{for every Borel set } B. $$

We denote by $L^1(X,\mu)$ the space of all (equivalence classes of) functions $f: X \rightarrow \R$ which are $\mu$-integrable.

\begin{definition}
An extended real valued set function $\nu: \mathcal{B}(\R^d) \rightarrow \overline{\R}$ is a \textit{signed measure} if $\nu$ assumes at most one of the values $+\infty$, $-\infty$, $\nu(\emptyset)=0$ and if 

$$\nu \left(\cup_{i=1}^{\infty} A_i \right)= \sum_{i=1}^{\infty}\nu(A_i)$$ for each sequence of disjoint sets $(A_i)_i \in \mathcal{B}(\R^d)^{\N}$, where the series either converges absolutely or diverges to $+\infty$ or $-\infty$.\footnote{From now on we will adopt the slight abuse of notation writing $(A_i)_i \in \mathcal{B}(\R^d)$ for sequences with values in $\mathcal{B}(\R^d).$}
\end{definition}

Given a convex compact set $K \subset \R^d$, we denote by $\mathcal{M}(K)$ the space of signed Radon measures on $K$ and by $\mathcal{M}_{+}(K)$ the subspace of positive measures.

\begin{definition}
If $\mu \in \mathcal{M}(K)$ we define its \textit{total variation measure} $||\mu||: \mathcal{B}(\R^d) \rightarrow [0, \infty]$ as follows:

$$||\mu||(A):=\sup \left\{\sum_{i=1}^{\infty}\left|\mu\left(A_i\right)\right|:\left(A_i\right)_i \in \mathcal{B}(\R^d), \text { $A_i$ pairwise disjoint, } A=\cup_{i=1}^{\infty} A_i\right\}$$
\end{definition}
If $\mu$ is a real measure, that is $\mu$ takes values in $\R$, we define its \textit{positive} and \textit{negative parts} respectively as $$\mu^+:=\frac{||\mu|| + \mu}{2} \text{ and } \mu^-:=\frac{||\mu|| - \mu}{2}.$$ This gives $\mu=\mu^{+}-\mu^{-}$ and $||\mu||=\mu^{+}+\mu^{-}$. The pair $(\mu^+,\mu^-)$ is usually called \textit{Jordan decomposition} of $\mu$.
The \textit{mass} of $\mu$ is the quantity $\mathbb{M}(\mu) := ||\mu||(K).$

\begin{definition}
Let $\mu \in \mathcal{M}_{+}(K)$, the \textit{support} of $\mu$ is defined as 
$$\text{supp}(\mu):=\{x\in K : \mu(U) >0 \text{ for every neighbourhood } U \text{ of } x\}.$$ If $\mu \in \mathcal{M}(K)$ we call the \textit{support of} $\mu$ the support of its total variation measure $||\mu||$. We say that a measure $\mu$ is \textit{finite atomic} if its support is a finite set.
\end{definition}

\begin{theorem}[Lusin's Approximation Theorem] 
Let $K$ be a locally compact and separable metric space and $\mu$ a Borel measure on $K$. Let $f: K \rightarrow \mathbb{R}$ be a $\mu$-measurable function vanishing outside of a set with finite measure. Then for any $\varepsilon>0$ there exists a continuous function $g: K \rightarrow \mathbb{R}$ such that $$
\mu(\{x \in K: g(x) \neq f(x)\})<\varepsilon.
$$\end{theorem}

Let $\mathcal{C}_c^0\left(\mathbb{R}^d , \mathbb{R}^k\right)$ be the space of continuous functions from $\mathbb{R}^d$ to $\mathbb{R}^k$ with compact support. We endow the space $\mathcal{C}_c^0\left(\mathbb{R}^d , \mathbb{R}^k\right)$ with the topology of uniform convergence on compact sets, that can be described as follows: 

\begin{definition}\label{uniformconvg}
A sequence $\left(\varphi_h\right)_h \in \mathcal{C}_c^0\left(\mathbb{R}^d , \mathbb{R}^k\right)$ converges to $\varphi$ in $\mathcal{C}_c^0$ if there exists a compact open set $A$ such that $\operatorname{supp}\left(\varphi_h\right) \subset A$ for all $h$, and $\varphi_h \rightarrow f$ uniformly on $A$.
\end{definition}

The \textit{total variation} of a linear functional $L: \mathcal{C}_c^0\left(\mathbb{R}^d , \mathbb{R}^k\right) \rightarrow \mathbb{R}$ is the quantity $$|L|(A):=\sup \left\{L(\varphi): \varphi \in \mathcal{C}_c^0\left(\mathbb{R}^d , \mathbb{R}^k\right),\|\varphi\|_{\infty} \leqslant 1, \operatorname{supp}(\varphi) \subset A\right\}.$$ A linear functional $L$ is \textit{bounded} if $|L|(A)< \infty.$

There is a close link between bounded linear functionals on $\mathcal{C}_c^0\left(\mathbb{R}^d , \mathbb{R}^k\right)$ and Radon measures.

\begin{theorem}[Riesz representation theorem]\label{t:Riesz}
Let $L$ be a bounded linear functional on $\mathcal{C}_c^0\left(\mathbb{R}^d , \mathbb{R}^k\right)$. Then there exist a Radon measure $\mu$ and a Borel function $f: \mathbb{R}^d \rightarrow \mathbb{R}^k$ such that $|f|=1$ $\mu$-almost everywhere and $$L(\varphi)= \int \varphi \cdot f d\mu, \text{  for all } \varphi \in \mathcal{C}_c^0\left(\mathbb{R}^d , \mathbb{R}^k\right),$$
where $\varphi \cdot f$ denotes $\sum_{j=1}^{k}\varphi_j \, f_j$. Moreover, $\mu(U)=|L|(U)$ for every open set $U$.
\end{theorem}
In what follows, we shall refer to $f\mu$ as a vector-valued Radon measure and we denote the space of Radon measures on $\R^d$ with values in $\R^k$ by $\mathcal{M}(\R^d,\R^k)$. 
Thanks to the identification of $\mathcal{M}(\R^d,\R^k)$ with the dual space of a separable Banach space, it is natural to endow it with the weak$^{*}$-topology. In particular we say that a sequence $(\mu_h)_h \in \mathcal{M}(\R^d,\R^k)$ converges weakly$^{*}$ to $\mu$, and we write $\mu_h \stackrel{\ast}{\rightharpoonup} \mu$, if $$
\lim _{h \rightarrow \infty} \int \varphi \cdot d \mu_h=\int \varphi \cdot d \mu, \quad \forall \varphi \in \mathcal{C}_c^0\left(\mathbb{R}^d, \mathbb{R}^k\right).
$$ 

By classical functional analysis, the weak$^{*}$-topology on a space with separable predual enjoys sequential compactness:

\begin{theorem}[Sequential compactness for measures]
Let $\left(\mu_h\right)_h \in \mathcal{M}(\R^d,\R^k)$. Assume that $$
\sup _h\left|\mu_h\right|(A)<+\infty \quad \forall A \subset\joinrel\subset \mathbb{R}^d \text { open. }
$$ Then there exists a subsequence $\left(\mu_{h_j}\right)$ and a Radon measure $\mu$, such that $\mu_{h_j} \stackrel{\ast}{\rightharpoonup} \mu$.
\end{theorem}

We now recall the notions of \textit{Hausdorff measure} and the one of \textit{rectifiable set}. Hausdorff measures are among the most important measures. They allow us to define a notion of dimension of sets in $\R^d$ and provide us with $k-$dimensional measures in $\R^d$ for any $k$, $0 \le k \le d$ (and also in any metric space).

Let $k$ be a nonnegative real number. We denote by $\omega_k$ the volume of the unit ball in $\mathbb{R}^k$ for $k=1,2,3, \ldots$, we set $\omega_0:=1$ and we let $\omega_k$ any convenient fixed constant for nonintegers $k$. Since the measure of the unit $d$-ball is given (for $d=1,2, \ldots $) by
$$\omega_d=\frac{\pi^{d / 2}}{\Gamma(1+d / 2)},$$
where $\Gamma(t)$ is the Euler's gamma function $\Gamma(t):=\int_0^{+\infty} s^{t-1} e^{-s} d s$, we can take as $\omega_k$ the following definition:

$$\omega_k:=\frac{\pi^{k / 2}}{\Gamma(1+k / 2)}.$$

Hence, for $k \ge 0, 0<\delta \leq \infty$ and $E \subset \mathbb{R}^d$, we define the Hausdorff $k$-dimensional $\delta$-premeasure as $$\mathcal{H}_\delta^k(E):=\frac{\omega_k}{2^k} \inf \left\{\sum_i\left(\operatorname{diam} F_i\right)^k: \operatorname{diam}\left(F_i\right)<\delta, E \subset \bigcup_i F_i\right\}.$$ Since $\mathcal{H}_\delta^k$ is a decreasing function of $\delta$ we can define the Hausdorff $k$-dimensional measure in $\mathbb{R}^d$ as $$\mathcal{H}^k(E):=\sup _{\delta>0} \mathcal{H}_\delta^k(E)=\lim _{\delta \rightarrow 0^{+}} \mathcal{H}_\delta^k(E).$$

Recall that on $\R^d$ we have $\mathcal{H}^d = \mathcal{L}^d$ as measures, where $\mathcal{L}^d$ is the $d-$dimensional Lebesgue measure, see [\ref{KP}].

Rectifiable sets are a family of sets that generalize the notion of $\mathcal{C^1}$ surface.

\begin{definition}\label{rectifiable}
A set $E \subset \mathbb{R}^d$ is called $\mathcal{H}^k$-\textit{countably} $k$-\textit{rectifiable} (or simply $k$-\textit{rectifiable}) if it is $\mathcal{H}^k-$measurable and $E \subset \bigcup_{j=0}^{\infty} E_j$, where
\begin{enumerate}
\item $\mathcal{H}^k\left(E_0\right)=0$;
\item for $j \geq 1$ we have that $E_j=F_j\left(\mathbb{R}^k\right)$, where $F_j: \mathbb{R}^k \rightarrow \mathbb{R}^d$ is a Lipschitz function.
\end{enumerate}
\end{definition}

\begin{remark}
The second condition in Definition \ref{rectifiable} can be equivalently replaced with one of the following conditions:
\begin{enumerate}
\item[(2.1)]  $E_i \subseteq f\left(A_i\right)$, where $A_i \subseteq \mathbb{R}^k$ is an open set and $f_i \in \mathcal{C}^1(A_i, \mathbb{R}^d)$;
\item[(2.2)]  $E_i \subseteq \Sigma_i$, where each $\Sigma_i$ is a $k$-dimensional submanifold of class $\mathcal{C}^1$ in $\mathbb{R}^d$.
\end{enumerate}
\end{remark}

In the following we will always assume that a $k-$rectifiable set $E$ has locally finite $\mathcal{H}^k$-measure.

Rectifiable sets have nice tangential properties and they are considered a natural generalization of $\mathcal{C^1}$ surfaces since for rectifiable sets a well-defined notion of tangent space can be given for almost all points.

We denote by $\operatorname{Gr}(k, d)$ the Grassmannian manifold of the (unoriented) $k-$dimensional linear subspaces in $\R^d$, that is $$\operatorname{Gr}(k, d):=\{\text {vector subspaces } W \subset \R^d \mid \operatorname{dim} W=k\}.$$ 
Take $V \in \operatorname{Gr}(k, d)$ and $E \subseteq \mathbb{R}^d$ Borel and $\mathcal{H}^k$-locally finite. For all $x \in E$ and for all $r>0$, we denote
$$
E_{x, r}:=\frac{1}{r}(E-x)$$ the set obtained applying to $E$ the homothety that maps $B(x,r)$, which is the ball of center $x$ and radius $r$, into $B(0,1)$.

\begin{definition}\label{approximatetg}
Given $x \in E$, we say that $V$ is an \textit{approximate tangent space} to $E$ at $x$ if $\mathcal{H}^k \res E_{x, r} \stackrel{\ast}{\rightharpoonup} \mathcal{H}^d \res V$ locally in the sense of measures as $r \rightarrow 0^+$, that is
$$
\lim _{r \rightarrow 0^+} \int_{E_{x, r}} g(y) d \mathcal{H}^k(y)=\int_V g(y) d \mathcal{H}^k(y) \quad \forall g \in \mathcal{C}^0_c(\mathbb{R}^d).
$$
\end{definition}

\begin{remark}\label{r:unique}
When the approximate tangent plane exists at some point $x \in E$ it is unique (since it is a limit) and will be denoted by $T_xE$. 
\end{remark}

\begin{remark}
If the set $E$ is a submanifold of class $\mathcal{C}^1$, then the tangent space and the approximate tangent space coincide at every point. Hence there is no ambiguity between the classical definition of $T_x E$ and Definition \ref{approximatetg}.
\end{remark}

It is also remarkable the fact that the existence of the approximate tangent space $T_xE$ at $\mathcal{H}^k$-almost all points characterizes $k$-rectifiable sets:

\begin{theorem}{\normalfont [\ref{Simonbook}, Theorem 11.6]}
Suppose that $E$ is $\mathcal{H}^k$-measurable with locally finite $\mathcal{H}^k$-measure. Then $E$ is $k$-rectifiable if and only if $E$ admits the approximate tangent space $T_xE$ for $\mathcal{H}^k$-a.e. $x \in E.$
\end{theorem}

\subsubsection{Multilinear algebra and Stokes' theorem}

Now we briefly introduce the main notions in multilinear algebra that allow us to define the notion of \textit{current} in the sense of de Rham. We highlight that we do not aim at full generality, nor at the best algebraic way to describe the exterior algebra of differential forms; instead, we are going to limit ourselves to the minimum instrumental requirements to be able to describe Federer and Fleming theory of normal and integral currents, see [\ref{FF}], keeping always in mind our goal: the solution of the Plateau's problem.

In the following let $V$ be a finite dimensional vector space over $\R$ and $V^*$ its dual. Denote $\mathcal{S}_n$ the group of permutations of $\{1, \ldots, n\}$; given $\sigma \in \mathcal{S}_n$, we denote $\operatorname{sgn}(\sigma)$ the sign of the permutation $\sigma$.

\begin{definition}
A $k$-linear alternating form (or $k$-\textit{covector}) on $V$ is a multilinear function $$\alpha: \underbrace{V \times \ldots \times V}_{k \text { times }} \rightarrow \mathbb{R}$$ with the following property: for all $v_1, \ldots, v_k \in V$, for all permutation $\sigma$, there holds
$$
\alpha\left(v_{\sigma(1)}, \ldots, v_{\sigma(k)}\right)=\operatorname{sgn}(\sigma) \alpha\left(v_1, \ldots, v_k\right) .
$$
The space of $k$-covectors on $V$ is denoted by $\Lambda^k(V)$. 
\end{definition}

\begin{remark}
If $k=0$, we set $\Lambda^0(V):=\mathbb{R}$ identified with the constant functions and $\Lambda^1(V)\simeq V^*$. One can immediately see that $\Lambda^k(V)$ is a vecotr space. If $k>\operatorname{dim}(V)$, then $\Lambda^k(V)=\{0\}$ and if $k=\operatorname{dim}(V)$, then $\operatorname{dim}\left(\Lambda^k(V)\right)=1$.
\end{remark}

Considering the graded vector space structure naturally induced for every $k \in \N $ by the subspaces $\Lambda^k(V)$, it is possible to promote it into a graded algebra defining an internal product, called $\textit{wedge product}$, defined as follows:

\begin{definition}
Let $\alpha \in \Lambda^h(V)$ and $\beta \in \Lambda^k(V)$ be covectors. We define $\alpha \wedge \beta$ to be the element of $\Lambda^{h+k}$ given by
$$
\alpha \wedge \beta\left(v_1, \ldots, v_{h+k}\right):=\frac{1}{\left(h+k\right) !} \sum_{\sigma \in \mathcal{S}_{h+k}} \operatorname{sgn}(\sigma) \alpha\left(v_{\sigma(1)}, \ldots, v_{\sigma(h)}\right) \beta\left(v_{\sigma(h+1)}, \ldots, v_{\sigma(h+k)}\right)
$$
\end{definition}

If $V = \mathbb{R}^d$ let $e_1, \ldots, e_d$ be the standard basis with dual basis $d x_1, \ldots, d x_d$ so that $d x_i\left(e_j\right)=\delta_{i j}$. We shall use the standard notations for ordered multi-indices: for every positive integer $k \le d$, we denote $I(k, d):=\left\{ \left(i_1, \ldots, i_k\right) \mid 1 \leq i_1<\ldots<i_k \leq d\right\}$. For every index $I \in I(k, d)$ we define $d x_I:=d x_{i_1} \wedge \ldots \wedge d x_{i_k}$. Analogously we will do for $e_I:=e_{i_1} \wedge \ldots \wedge e_{i_k}$.

\begin{proposition}
The collection $\{d x_I\}_{I \in I(k,d)}$ is a basis for $\Lambda^k(\R^d)$. In particular, for all $\alpha \in \Lambda^k(\R^d)$ we can write $$\alpha = \sum_{I \in I(k,d)}\alpha_I dx_I$$ where $\alpha_I :=\alpha (e_{i_1}, \ldots, e_{i_k})$ for all $I=(i_1, \ldots, i_k)$. Moreover, dim $\Lambda^k(\R^d) =$$ d \choose k$.
\end{proposition}

The dual space of $\Lambda^k(\R^d)$ is called the space of $k$-\textit{vectors} and it is denoted by $\Lambda_k(\R^d)$. Before defining it properly we define a subclass of $k$-vectors that has a geometric meaning: \textit{simple} $k$-vectors.

\begin{definition}
Let $k$ be a positive integer. Given $\left(v_1, \ldots, v_k\right)$, $\left(\tilde{v}_1, \ldots, \tilde{v}_k\right) \in (\R^d)^k$, we define the equivalence relation $\left(v_1, \ldots, v_k\right) \sim\left(\tilde{v}_1, \ldots, \tilde{v}_k\right)$ if
$$
\alpha\left(v_1, \ldots, v_k\right)=\alpha\left(\tilde{v}_1, \ldots, \tilde{v}_k\right) \quad \forall \alpha \in \Lambda^k(\R^d).
$$
The elements of the quotient set with respect to $\sim$, denoted $[v_1,\dots,v_k]$, are called \textit{simple} $k$-vectors.
\end{definition}

\begin{remark}
$(\R^d)^k/\sim$ is not in general a vector space.
\end{remark}

The following proposition shows the geometric meaning of simple $k$-vectors. Two elements in the same equivalence class span the same $k$-dimensional space.

\begin{proposition} Let $\left(v_1, \ldots, v_k\right)$ and $\left(\tilde{v}_1, \ldots, \tilde{v}_k\right) \in (\R^d)^k$ then: \begin{enumerate}
\item $\left(v_1, \ldots, v_k\right) \sim(0, \ldots, 0)$ if and only if $v_1, \ldots, v_k$ are linearly dependent.
\item $\left(v_1, \ldots, v_k\right) \sim\left(\tilde{v}_1, \ldots, \tilde{v}_k\right) \nsim(0, \ldots, 0)$ if and only if $$\operatorname{Span}\left\{v_1, \ldots, v_k\right\}=\operatorname{Span}\left\{\tilde{v}_1, \ldots, \tilde{v}_k\right\}$$ and the change of basis matrix $M$ (\textit{i.e.} for all $i$, $\tilde{v}_i=$ $\sum_j M_{i, j} \,v_j$) has $\operatorname{det}(M)=1$.
\end{enumerate}
\end{proposition}

Given $v_1, \ldots, v_k \in \R^d$, let $R\left(v_1, \ldots, v_k\right):=\left\{\sum_{j=1}^k \lambda_j v_j \mid \lambda_j \in[0,1]\right\}$ be the rectangle generated by $v_1, \ldots, v_k$. Note that if $\left(v_1, \ldots, v_k\right) \sim\left(\tilde{v}_1, \ldots, \tilde{v}_k\right) \nsim(0, \ldots, 0)$ then by the area formula $\mathcal{H}^k\left(R\left(\tilde{v}_1, \ldots, \tilde{v}_k\right)\right) = \mathcal{H}^k\left(R\left(v_1, \ldots, v_k\right)\right)$ since $\operatorname{det}(M)=1$. Hence we can define the following map that is called the ``norm'' of the simple $k$-vector\footnote{Which is actually not a norm, since it is not defined on a vector space.} $[v_1,\dots,v_k]$:

\begin{equation}\label{simplenorm}
\big|\left[v_1, \ldots, v_k\right]\big|:=\mathcal{H}^k\left(R\left(v_1, \ldots, v_k\right)\right).
\end{equation}

Recall that an orientation of a vector space is an equivalence class of bases with respect to the following equivalence relation: let $\left(v_1, \ldots, v_n\right), \left(\tilde{v}_1, \ldots, \tilde{v}_n\right) \in (\R^d)^k$, then $\left(v_1, \ldots, v_n\right) \approx\left(\tilde{v}_1, \ldots, \tilde{v}_n\right)$ if and only if the change of basis matrix has positive determinant. Hence we can state the following proposition telling us that simple unitary (in norm) $k$-vectors represent oriented $k$-planes:

\begin{proposition}
Consider the map $\psi$ which associates to a simple unitary $k$-vector $\left[v_1, \ldots, v_k\right]$ the $k$-dimensional subspace $\operatorname{Span}\left\{v_1, \ldots, v_k\right\}$ oriented by $\left(v_1, \ldots, v_k\right)$. This map is well-defined and it is a bijection.
\end{proposition}

We define now a vector space that includes the space of simple $k$-vectors. The space of $k$-vectors on $\R^d$ is defined as $\Lambda_k(\R^d):=\Lambda^k\left((\R^d)^*\right)$. We define the duality pairing of $\Lambda^k(\R^d)$ and $\Lambda_k(\R^d)$ as the bilinear form $\langle\cdot, \cdot\rangle: \Lambda^k(\R^d) \times \Lambda_k(\R^d) \rightarrow \mathbb{R}$ by setting for all $I,J \in I(k,d)$
$$
\langle dx_I, e_J\rangle:=\delta_{I,J} 
$$
In particular, $\langle\cdot, \cdot\rangle$ gives an isomorphism between $\Lambda_k(\R^d)$ and $\left(\Lambda^k(\R^d)\right)^*$. Moreover by bilinearity of $\langle\cdot, \cdot\rangle$ we can also write for all $\alpha \in \Lambda^k(\R^d)$ and for all $(v_1, \ldots, v_k) \in (\R^d)^k$ that $\langle\alpha, v_1 \wedge \cdots \wedge v_k\rangle=\alpha\left(v_1, \ldots, v_k\right)$.

\begin{remark}
Simple $k$-vectors can be embedded in $\Lambda_k(\R^d)$ by identifying $\left[v_1, \ldots, v_k\right]$ with $v_1 \wedge \cdots \wedge v_k$. Indeed it easy to see that $\left(v_1, \ldots, v_k\right) \sim\left(\tilde{v}_1, \ldots, \tilde{v}_k\right)$ if and only if $v_1 \wedge \cdots \wedge v_k=\tilde{v}_1 \wedge \cdots \wedge \tilde{v}_k$.
\end{remark}

\begin{remark}
There are $k$-vectors that are not simple: $\tau= e_1 \wedge e_2 + e_3 \wedge e_4 \in \Lambda_2(\R^4)$ is not simple since $\tau \wedge \tau \neq 0$.
\end{remark}

There are two natural choices of norms on $\Lambda_k(\R^d)$. The first is the Euclidean norm $|\cdot|$ defined for $w=\sum_{I\in I_{k, d}} w_{I} e_I \in  \Lambda_k(\R^d)$ as
$$
|w|:=\sqrt{\sum_{I \in I_{k, d}} w_{I}^2}.
$$

\begin{remark}
The Euclidean norm of a simple $k$-vector agrees with the ``norm'' defined in \eqref{simplenorm}. More precisely, given $v_1 \wedge \cdots \wedge v_k \in \Lambda_k(\R^d)$  we have that $$\left|v_1 \wedge \cdots \wedge v_k\right|=\mathcal{H}^k\left(R\left(v_1, \ldots, v_k\right)\right).$$
\end{remark}

The second norm is the so-called \textit{mass norm}, defined to be the largest norm $\|\cdot\|$ on $\Lambda_k(\R^d)$ such that $$
\|v_1 \wedge \cdots \wedge v_k\|=\left|v_1 \wedge \cdots \wedge v_k\right| \quad \forall v_1, \ldots, v_k \in \R^d.
$$

Mass norm and Euclidean norm agree on simple $k$-vectors and one can see that the unit ball with respect to $\|\cdot\|$ is the convex envelope of the set of simple $k$-vectors with unit ``norm". Once a norm is defined on $\Lambda_k(\R^d)$, we get a dual norm on $\Lambda^k(\R^d)$, defined as follows:

\begin{definition}
The dual norm $\|\cdot\|^*$ induced by the mass norm is called \textit{comass norm}, that is
$$
\|\alpha\|^*=\sup \{\langle \alpha, w \rangle\mid \|w\| \leq 1\} \quad \forall \alpha \in \Lambda^k(\R^d).
$$
\end{definition}

Now we can introduce differential forms and define a norm on the space of differential forms, making it into a well-behaved topological space.

Let $\Omega \subset \R^d$ be an open set. A $k$-form $\omega$ on $\Omega$ is a section of the $k$-alternating tensor bundle of $\Omega$. Since we are in $\R^d$, we can particularize the above definition saying that a \textit{differential k-form} $\omega$ on $\mathbb{R}^d$ is a $k$-covector field, that is a map
$$
\omega: \mathbb{R}^d \rightarrow \Lambda^k(\mathbb{R}^d),
$$ that, in other words, means $\omega(x)$ is just is a $k$-alternating multilinear map from $\mathbb{R}^d$ to $\mathbb{R}$, for all $x \in \Omega$. We call $k$ the \textit{degree} of the form.

Fixed the standard basis of $\Lambda^k(\mathbb{R}^d)$, we can write $\omega$ in local coordinates as
$$
\omega(x)=\sum_{I \in  I(k, d)} \omega_{I}(x) d x_{I},
$$
where the coefficients $\omega_{I}(x)$ are real-valued functions on $\R^d$: depending on their regularity we say that the $k$-differential form has that regularity; if we say that the form is smooth we mean that $\omega_I(x)$ are $\mathcal{C}^{\infty}$ functions. The support of $\omega$ is defined as the closure of the set $\left\{x \in \mathbb{R}^d: \omega(x) \neq 0\right\}$ and we will denote it as $\operatorname{supp}(\omega)$; we say that the form has \textit{compact support} if $\operatorname{supp}(\omega)$ is compact.

\begin{definition}
The \textit{exterior derivative} of a smooth differential $k$-form $\omega$ is the differential $(k+1)$-form:
$$
d \omega(x)=\sum_{I \in I(k, d)} d \omega_{I} \wedge d x_{I},
$$
where
$$
d\omega_I(x)=\sum_{j=1}^d \frac{\partial w_I}{\partial x_j}(x) d x_j.
$$
\end{definition}

A $k$-form $\omega$ is said to be \textit{closed} if $d \omega=0$, while if there exists a $(k-1)$-form $\eta$ such that $d \eta=\omega$, then $\omega$ is said to be \textit{exact}. One of the most important properties of the exterior derivative $d$ is that $d \circ d = 0$ (giving rise to the beautiful and elegant theory of de Rham cohomology).

Towards the statement of Stokes' theorem and exploiting the geometric meaning of simple $k$-vectors, we define a notion of \textit{orientation} of a $k$-submanifold\footnote{It can be proven to be equivalent to the classical definition of orientation given through transition maps which preserve the orientation of the tangent space.}.

\begin{definition}\label{d:orientation}
Let $\Sigma$ be smooth $k$-submanifold in $\R^d$. An \textit{orientation} of $\Sigma$ is a continuous map $\tau_\Sigma: \Sigma \rightarrow \Lambda_k(\mathbb{R}^d)$ such that $\tau_\Sigma(x)$ is unitary and spans $T_x \Sigma$ for every $x \in \Sigma$. 
\end{definition}

An orientation of $\Sigma$ induces a canonical orientation of $\partial \Sigma$ the boundary of $\Sigma$, namely the one such that
\begin{equation}\label{boundaryorient}
\tau_\Sigma(x)=\nu(x) \wedge \tau_{\partial \Sigma}(x) \text { for every } x \in \partial \Sigma,
\end{equation}
where $\nu$ is the outer normal to $\partial \Sigma$.

Assume $\Sigma$ is a $k$-surface oriented by $\tau_\Sigma$. The integral of a differential $k$-form $\omega$ on it can be defined as
$$
\int_\Sigma \omega:=\int_\Sigma\left\langle\omega(x) , \tau_\Sigma(x)\right\rangle d \mathcal{H}^k(x),
$$ where for every $x \in \R^d$, $\langle\omega(x) , \tau_\Sigma(x)\rangle$ is the duality pairing of a $k$-covector $\omega(x)$ acting on a simple $k$-vector $\tau_\Sigma(x)$.

For its relevance in the sequel, we introduce the notion of \textit{pull-back} of a differential $k$-form on $\mathbb{R}^{m}$ according to a smooth map $f: \mathbb{R}^d \rightarrow \mathbb{R}^{m}$.

\begin{definition}
For any differential $k$-form $\omega$ on $\mathbb{R}^{m}$, we define its \textit{pull-back} $f^* \omega$ on $\mathbb{R}^d$ by setting for all $p \in \mathbb{R}^d$ and for all $v_1 \wedge \ldots \wedge v_k \in \Lambda_k(\mathbb{R}^d)$
$$
\left\langle f^* \omega(p) , v_1 \wedge \ldots \wedge v_k\right\rangle:=\left\langle\omega(f(p)) , df(p) v_1 \wedge \ldots \wedge df(p) v_k \right\rangle.
$$
This map is extended to all $k$-vectors by linearity. In the case $k=0$, the formula reduces to the composition of functions $f^* \omega = \omega \circ f$.
\end{definition}

Merging together two differential operators (the exterior derivative $d$ and the integral $\int$) with a topological one (the boundary $\partial$), Stokes' theorem can be considered one of the most important (and elegant) results in the theory of integration. 

\begin{theorem}[Stokes' theorem]\label{t:stokes}
Let $\Sigma$ be a compact oriented submanifold of dimension $k$ of class $\mathcal{C}^1$ with $\partial \Sigma$ of class $\mathcal{C}^1$ and let $\omega$ be a $(k-1)$-form of class $\mathcal{C}^1$. Then $$\int_\Sigma d \omega=\int_{\partial \Sigma} \omega.$$
\end{theorem}
See for instance [\ref{KP}] for the proof.

\begin{corollary}
Let $\Sigma$ be a $k$-dimensional oriented submanifold with $\partial \Sigma=\emptyset$ and let $\omega$ be a $(k-1)$-form of class at least $\mathcal{C}^1$. Then $\int_\Sigma d \omega=0$.
\end{corollary}

For the next section, it is important to consider the space $\mathcal{D}^k(U)$ of smooth compactly supported differential $k$-forms on $U$ where $U \subset \R^d$ is an open set\footnote{In the literature there are many notations to refer to the space of smooth compactly supported differential $k$-forms, including the most common $\Omega^k_c(U)$. We will stick to the notation mostly used in the context of geometric measure theory to highlight the duality with currents.}. Since the space is sequential, we can characterize its topology by means of converging sequences. There are more abstract ways to define the following topology on $\mathcal{D}^k(U)$, but we consider useful to provide an explicit description of converging sequences to highlight the similarities with the one on the set $\mathcal{C}^0_c$ in Definition \ref{uniformconvg}.

\begin{definition}\label{convforms}
A sequence $(\omega^{n})_{n \in \mathbb{N}} \in \mathcal{D}^k(U)$, that in local coordinates takes the form
$$
\omega^{n}(x)=\sum_{I \in I(k, d)} \omega_I^{n}(x) d x_I,$$ is said to \textit{converge} to $\omega \in \mathcal{D}^k(U)$ as $n \rightarrow \infty$ if there exists a compact set $K \subset U$ such that
\begin{enumerate}
\item $\operatorname{supp} (\omega_{I}^{n} )\subset K$ for any $I \in I(k, d)$ and for any $n \in \N$,
\item for every choice of the multi-index $\alpha$ we have $D^{\alpha} \omega_{I}^{n} \rightarrow D^{\alpha} \omega_I$ uniformly in $K$ for every $I \in I(k, d)$.
\end{enumerate}
\end{definition}

\begin{remark}
Note that in local coordinates we are just equipping $\mathcal{D}^k(U)$ with the topology on the predual of the space of distributions $\mathcal{D}^0(U)$: there is indeed the following identification $$\mathcal{D}^k(U) \simeq \left(\mathcal{D}^0(U)\right)^{|I(k, d)|}.$$ The topology in Definition \ref{convforms} turns $\mathcal{D}^k(U)$ into a locally convex metrizable and separable topological vector space: nice properties to be able to deal with compactness in its dual space.
\end{remark}

\section{The Federer-Fleming theory of integral currents}

The first appearance of the notion of \textit{current}\footnote{The choice of the term ``current" is motivated by the fact that in a 3-dimensional space ``1-dimensional currents" can be interpreted as electrical currents and indeed, in [\ref{deRham}] and [\ref{deRham2}], de Rham thought of them as cables carrying an electrical current of unit intensity.}, in a less general and less precise form, was at the beginning of the thirties by Georges de Rham in [\ref{deRham}] and [\ref{deRham2}]. It was only after Schwartz's introduction of the concept of distribution in 1945, see [\ref{Schwartz}], that de Rham reframed its definition from the one dealing with homologies on forms to the cleaner one that we are going to define now.

In this section, nevertheless, we are more interested in the notion of current seen as a measure-theoretic generalization of an oriented surface. From this point of view, its theory was brought to fruition by Federer and Fleming in the late fifties to prove the existence of an area-minimizing surface spanning a given contour: the Plateau's problem.

We will now focus on the introduction of the space of currents as the dual space of smooth and compactly supported differential forms and we will define a generalized version of the Plateau's problem in the language of currents. Then we will introduce all the needed machinery to solve the generalized Plateau's problem and to describe further applications in the sequel of this thesis.

\subsection{Currents}

Let $U \subset \R^d$ be an open set and $0 \le k \le d$.
\begin{definition}\label{d:current}
A $k$-\textit{dimensional current} in $U$ is a continuous linear functional on $\mathcal{D}^k(U)$, endowed with the topology in Definition \ref{convforms}. 
\end{definition}
The space of $k$-dimensional currents in $U$ is denoted by $\mathcal{D}_k(U)$. We will often use the notation $\langle T, \omega \rangle$ to emphasize the duality pairing with a form $\omega$, but later we are going to simply write the action of $T$ as $T(\omega)$.

\begin{definition}\label{d:boundary}
Given $T \in \mathcal{D}_k\left(U\right)$, we define the \textit{boundary} of $T$ as the $(k-1)$-current defined as
$$
\langle\partial T, \omega\rangle:=\langle T, d \omega\rangle \hspace{0.3cm} \text{for all }\,\, \omega \in \mathcal{D}^{k-1}(U).
$$
\end{definition}

\begin{remark}
The functional $\partial T$ is well-defined, linear and continuous. Note that $\partial$ on $k$-currents is just the adjoint operator of $d$ on smooth, compactly supported $k$-forms. The counterpart of the fact that $d \circ d =0$ is that $\partial (\partial T)=0$ for all $T \in \mathcal{D}_k(U)$.
\end{remark}

\begin{definition}\label{d:mass}
Given $T \in \mathcal{D}_k\left(U\right)$, we define the \textit{mass} of $T$ as
$$
\mathbb{M}(T):=\sup \left\{\langle T, \omega \rangle |\, \omega \in \mathcal{D}^k\left(U\right), \|\omega(x)\|^* \leq 1 \,\,\, \forall x \in U\right\},
$$
where $\|\omega(x)\|^*$ is the comass norm\footnote{Some authors prefer to take the euclidean norm $|\cdot|$ instead; clearly the value of the mass changes and in general $\mathbb{M}(T)$ is strictly greater than the corresponding value replacing the comass norm with $|\cdot|$, but most of the theory remains consistent by equivalence of the norms.} of $\omega(x)$.
\end{definition}

\begin{remark}
Definitions \ref{d:current}, \ref{d:boundary} and \ref{d:mass} can be considered as generalized concepts for the notions of manifold, boundary of a manifold and volume of a manifold respectively. Indeed, let $\Sigma$ be a smooth, oriented, $k$-dimensional submanifold in $\mathbb{R}^d$. In the spirit of Poincaré duality, we define the following linear functional $T_{\Sigma}(\cdot)$ on $\mathcal{D}^k\left(U\right)$:
$$
\omega \mapsto \int_{\Sigma} \omega.
$$ Such a current is often denoted by $\llbracket \Sigma \rrbracket$. Note that $T_{\Sigma}$ is uniquely determined by $\Sigma$ in the sense that $\Sigma \neq \Sigma^{\prime}$ (as oriented submanifolds) implies $T_{\Sigma} \neq T_{\Sigma^{\prime}}.$ We can rewrite Stokes' theorem in the following way $$\partial T_{\Sigma}=T_{\partial \Sigma}$$ and it is possible to prove that $$\mathcal{H}^k(\Sigma)=\sup \left\{T_{\Sigma}(\omega)\, |\, \omega \in \mathcal{D}^k\left(U\right), \|\omega(x)\|^* \leq 1 \,\,\, \forall x \in U\right\},$$ justifying the terminology of ``generalized surfaces''.
\end{remark}

Since we are interested in a variational problem, as a dual space, $\mathcal{D}_k(U)$ is naturally equipped with the weak$^*$-topology. 

\begin{definition}
Given a sequence of $k$-currents $(T_n)_n$ and a $k$-current $T$, we say that $(T_n)_n$ \textit{converges to} $T$ \textit{in the sense of currents} if the sequence converges to $T$ with respect to the weak-* topology, that is
$$
\lim _{n \rightarrow \infty}\langle T_n, \omega \rangle=\langle T, \omega \rangle \quad \text{for all} \,\,\,  \omega \in \mathcal{D}^k\left(U\right).
$$
\end{definition}

\begin{remark}\label{r:lsc}
The boundary operator is continuous and the mass is a lower semicontinuous functional, both with respect to the weak$^*$-convergence of currents. More formally, if $(T_n)_n$ is a sequence of $k$-currents converging to a $k$-currents $T$, then we have:
\begin{enumerate}
\item $(\partial T_n)_n$ converges to $\partial T$ in the sense of currents,
\item $\mathbb{M}(T) \leq \liminf _{n \rightarrow \infty} \mathbb{M}(T_n).$
\end{enumerate}
\end{remark}

As for distributions, we have a notion of support of a current.

\begin{definition}
The \textit{support} of a $k$-current $T$ in $\mathcal{D}_k(U)$, with $U \subset \mathbb{R}^d$ open, is the set
$$\operatorname{supp}(T):=\mathbb{R}^d \setminus \bigcup\left\{V \subset U, V \text{ open} :\, \omega \in \mathcal{D}^k(U), \,\operatorname{supp}(\omega) \subset V \Rightarrow T(\omega)=0\right\}.$$
\end{definition}

\begin{remark}
The notion of current is so general that it is possible to define for every $k=1,2, \dots, d$ a $k$-current $T \in \mathcal{D}_k(\R^d)$ whose support is a singleton: calling them $k$-\textit{dimensional} objects does not encode the existence of a surrounding $k$-dimensional geometry.
\end{remark}

Important subclasses of currents need to be defined to obtain a reasonable geometric object in the solution of the generalized Plateau's problem.

\begin{definition}
We say that a $k$-current $T$ has \textit{finite mass} if $\mathbb{M}(T)<\infty.$
\end{definition}

\begin{remark}
Note that $0$-currents with finite mass can be identified with signed measures. The definition of current with finite mass is not trivial remarking the fact that the topology on $\mathcal{D}^k(\R^d)$ in Definition \ref{convforms} to which currents are dual is finer than the topology induced by the norm $$\|\omega(x)\|_{\infty}:=\sup \left\{\|\omega(x)\|^*: x \in \mathbb{R}^d\right\}.$$ The usual examples are derivatives of the Dirac delta at a point $x_0 \in U$: pick the $0$-current $T$ on $\R$ such that, fixing $x_0 \in U \subset \R$ $$T(\varphi) = \varphi'(x_0) \hspace{0.3cm} \text{for all } \, \varphi \in \mathcal{C}^{\infty}_c(U).$$
\end{remark}

Given $T \in \mathcal{D}_k(\mathbb{R}^d)$ with finite mass we obtain
$$
|\langle T, \omega \rangle| \leq \mathbb{M}(T)\|\omega\|_{\infty} \quad \forall \omega \in \mathcal{D}^k(\mathbb{R}^d).
$$ Hence, $T$ can be extended by density to a linear continuous functional defined on the space of continuous and infinitesimal at infinity $k$-forms $\omega \in \mathcal{C}_0\left(\mathbb{R}^d, \Lambda^k(\mathbb{R}^d)\right)$. By the Riesz representation theorem\footnote{We are actually invoking a slightly more general version than the one stated in Theorem \ref{t:Riesz}.}, $T$ can be represented by integration with respect to a measure with values in $\Lambda^k(\mathbb{R}^d)^*=\Lambda_k(\mathbb{R}^d)$. Thus, there exists a positive and finite measure $\mu$ on $\mathbb{R}^d$ and a vector field $\tau: \mathbb{R}^d \rightarrow \Lambda_k(\mathbb{R}^d)$ in $L^1(\mathbb{R}^d, \mu)$, unitary (in mass norm) $\mu$-a.e. such that \begin{equation}\label{e:finitemass}\langle T, \omega \rangle=\int_{\mathbb{R}^d}\langle \omega(x), \tau(x) \rangle d\mu(x) \quad \text{ for all } \omega \in \mathcal{D}^k(\R^d).\end{equation} We denote by $T= \tau \cdot \mu$ a $k$-current whose action on a form is as in \eqref{e:finitemass}. We will often call $\mu$ the \textit{total variation measure} associated with $T$ and denote it by $\|T\|$. It can be checked that $\mathbb{M}(T) = \mu(\R^d)$, that is the mass of $T$ equals the mass of the measure $\|T\|$. In particular, if $T \in \mathcal{D}_k(U)$ and with finite mass, if $A \subset U$ is any Borel set we can define the \textit{restriction} of $T$ on $A$, denoted with $T \res A$, by
$$
T \res A(\omega):=\int_A \langle \omega(x), \tau(x) \rangle d\mu(x).
$$

\begin{definition}
A $k$-current $T \in \mathcal{D}_k(U)$ is called \textit{normal} if both $T$ and $\partial T$ have finite mass, that is $$\mathbb{M}(T)<\infty \hspace{0.3cm} \text{and} \hspace{0.3cm} \mathbb{M}(\partial T)<\infty.$$ The space of normal $k$-currents is denoted by $\mathbb{N}_k(U)$.
\end{definition}

Normal currents have good compactness properties: 
\begin{proposition}\label{p:compactnessnormal}
Let $(T_n)_n$ be a sequence of normal $k$-currents s.t.
$$
\sup _n \left(\mathbb{M}(T_n)+ \mathbb{M}(\partial T_n)\right)<\infty .
$$
Then, up to subsequences, $(T_n)_n$ converges in the sense of currents to a $k$-current $T$. Moreover, we have that
$$
\mathbb{M}(T) \leq \liminf _{n \rightarrow\infty} \mathbb{M}(T_n) \hspace{0.3cm} \text{and} \hspace{0.3cm}  \mathbb{M}(\partial T) \leq \liminf _{n \rightarrow\infty} \mathbb{M}(\partial T_n) .
$$
In particular, $T$ is a normal $k$-current.
\end{proposition}

At this point one could formulate a very general version of the Plateau's problem.

\begin{plateau}
Given $S \in  \mathbb{N}_{k-1}(U)$ such that $S=\partial T$ for some $T\in  \mathbb{N}_k(U)$, does there exist $T'\in  \mathbb{N}_k(U)$ such that $\partial T'=S$ and $\mathbb{M}(T')$ is minimized?
\end{plateau}

\begin{remark}
By Proposition \ref{p:compactnessnormal} it is possible to employ the direct methods in the calculus of variations: the functional $\mathbb{M}(\cdot)$ is lower semicontinuous by \ref{r:lsc} with respect to convergence in the sense of currents and by sequential compactness in \ref{p:compactnessnormal} we solve the very general Plateau's problem.
\end{remark}

An interesting result for normal currents states that they cannot concentrate on small closed sets; more precisely we have the following theorem:

\begin{theorem}{\normalfont [\ref{GMS}, Section 2.3, Theorem 2]}
Let $T \in \mathbb{N}_{k}(U)$ with $U \subset \mathbb{R}^d$ open. For any $I= (i_1,\dots,i_k) \in I(k, d)$ let $\pi_I$ denote the orthogonal projection
$$
\pi_I:\left(x_1, \ldots, x_d\right) \rightarrow \left(x_{i_1}, \ldots, x_{i_k}\right) .
$$
Then, for any closed set $C \subset U$ such that $\mathcal{H}^k\left(\pi_I(C)\right)=0$ for all $I \in I(k, d)$, we have
$$
\|T\|(C)=0 \hspace{0.3cm} \text{that is} \hspace{0.3cm} T \res C=0.$$
\end{theorem}

Nevertheless, the Plateau's problem formulation for normal currents is still not satisfactory since the space $\mathbb{N}_k(U)$ contains again elements that do not have the geometric meaning of being $k$-dimensional (for any reasonable notion of dimension). Indeed we can construct the following example.

\begin{example}\label{e:geodesics}{\normalfont [\ref{Delellisnote}, Example 2.8]}
Consider the south and north poles $S$ and $N$ in the sphere $\mathbb{S}^2 \subset \mathbb{R}^3$ and let $Z$ be the 0-dimensional current $\llbracket N \rrbracket-\llbracket S \rrbracket$. For any meridian $\gamma$ joining $S$ to $N$ the corresponding current $\llbracket \gamma \rrbracket$ is a minimizer of the mass among all currents $T$ with $\partial T=Z$ and $\operatorname{supp}(T) \subset \mathbb{S}^2$. However the same holds if we parametrize the meridians as a one-parameter family $\left\{\gamma_t\right\}_{t \in \mathbb{S}^1}$, where $t$ is the intersection of $\gamma_t$ with the equator $\left\{x_3=0\right\} \cap \mathbb{S}^2$. If $\mu$ is a probability measure on $\mathbb{S}^1$, then the current $$T_1(\omega):=\int_{\mathbb{S}^1} \langle \llbracket \gamma_t \rrbracket ,\omega \rangle d \mu(t)$$ is also mass-minimizing among all currents $T$ with $\operatorname{supp}(T) \subset \mathbb{S}^2$ and $\partial T=Z$.
\end{example} 

One could argue that Example \ref{e:geodesics} may be considered not so troublesome by noticing that at least, among all minimizers, it is somehow possible to find ``classical minimizers''. In fact, deeper issues arise as the following theorem shows (see [\ref{Whitelavrentiev}] for a short proof):

\begin{theorem}[Lavrentiev gap]
For every smooth closed embedded curve $\gamma$ in $\mathbb{R}^4$ define
$$
\begin{aligned}
M(\gamma)&:=\inf \left\{\operatorname{Area}(\Sigma): \Sigma \text { is immersed, oriented and } \partial \Sigma=\gamma\right\}, \\
m(\gamma)&:=\min \{\mathbb{M}(T): \partial T=\llbracket \gamma \rrbracket\}.
\end{aligned}
$$
Then there are $\gamma$'s for which $m(\gamma) < M(\gamma)$.
\end{theorem}

As a result, we need to add more structure to enclose some reasonable geometric meaning in the notion of current.

\subsection{Existence for the generalized Plateau's problem}

To rule out minimizers as in Example \ref{e:geodesics} we need to add more structure to the previous notions of currents. We will be mainly interested in what is known as \textit{integer multiplicity rectifiable currents} which can be thought as the natural extension of currents of the type $\llbracket \Sigma \rrbracket$ in which we allow $\Sigma$ to be a rectifiable set and we also allow it to carry an ``integer multiplicity''.

Given a $k$-rectifiable set $E$ in $\R^d$, one can define an \textit{orientation} of $E$ as a Borel function $\tau_E : E \rightarrow \Lambda_k(\R^d)$ such that $\tau_E(x)$ is a simple unit $k$-vector spanning $T_xE$ for $\mathcal{H}^k$-a.e. $x \in E$.

\begin{definition}
Let $U$ be an open set in $\R^d$. We call a $k$-current $T$ \textit{rectifiable} if $T$ admits the following representation
$$
\langle T , \omega\rangle=\int_E\left\langle\omega(x) , \tau_E(x)\right\rangle \theta(x) d \mathcal{H}^k(x),
$$ where $E$ is a $k$-rectifiable set, $\tau_E$ is an orientation of $E$, and $\theta$ is a real-valued function $\theta \in L^1(U, \mathcal{H}^k \res E)$. The function $\theta$ is often called \textit{multiplicity} of the current $T$. We will often use the notation $T=\llbracket E, \tau_E, \theta \rrbracket$ and the set of rectifiable $k$-currents in $U$ is denoted by $\mathscr{R}_k(U)$.
\end{definition}

\begin{remark}
Given $T=\llbracket E, \tau_E, \theta \rrbracket$ a $k$-rectifiable current, then $E, \tau_E, \theta$ in the representation of $T$ are not uniquely determined: one could write equivalently $\llbracket E, -\tau_E, -\theta \rrbracket$ instead of $\llbracket E, \tau_E, \theta \rrbracket$. If we require in addition that $\theta > 0$ for $\mathcal{H}^k$-a.e. $x \in E$, then $E, \tau_E, m$ are uniquely determined (up to $\mathcal{H}^k$-null sets).
\end{remark}

In particular we have that \begin{equation}\label{e:massrectif}
\mathbb{M}(T)=\int_E|\theta(x)| d \mathcal{H}^k(x)= \|\theta\|_{L^1(U, \mathcal{H}^k\res E)}.\end{equation}

\begin{definition}
A rectifiable current $T=\llbracket E, \tau_E, \theta \rrbracket$ whose multiplicity $\theta$ takes values in $\mathbb{Z}$ is called an \textit{integer multiplicity rectifiable current}. The set of integer multiplicity rectifiable $k$-currents in $U$ is denoted by $\mathcal{R}_k(U)$.
\end{definition}

\begin{remark}
Any $T \in \mathcal{R}_0(U)$ can be written as finite sum of weighted Dirac masses. More formally, let $x_i \in \R^d$, $\theta_i \in \mathbb{Z}$ and $\delta_{x_i}$ the Dirac mass at point $x_i$ for $i=1, \dots, N$. Then we can write $$T=\sum_{i=1}^N \theta_i \delta_{x_i}.$$ Indeed, a function $\theta \in L^1(U,\mathcal{H}^0\res E)$ with values in $\mathbb{Z}$ is a function that attains a finite number of values in a finite number of points, vanishing elsewhere.
\end{remark}

\begin{remark}
Unlike the space of rectifiable currents, the set of integer multiplicity rectifiable currents is not a real vector space since in general $\lambda T \in \mathcal{R}_k(U)$ only if $T \in \mathcal{R}_k(U)$ and $\lambda$ is an integer. As a result, there is no hope to invoke any simple functional-analytic principle to obtain good compactness properties in $\mathcal{R}_k(U)$. Still, $\mathcal{R}_k(U)$ maintains the algebraic structure of abelian group with respect to the sum.
\end{remark}

\begin{definition}
If both $T$ and $\partial T$ are integer multiplicity rectifiable currents, then $T$ is called an \textit{integral current}. The corresponding space is denoted by $\mathcal{I}_k(U)$.
\end{definition}

The following theorem by Federer and Fleming can be considered as one of the cornerstones in the theory of geometric variational problems, see [\ref{FF}]. 

It is worth pointing out that given a bounded sequence of integral currents $(T_n)_n$, the existence of a converging subsequence and a limit current $T$ follows from Proposition \ref{p:compactnessnormal}. The fact that the limit current $T$ is not just normal but integral deeply relies on the geometry of integral currents: hence the name \textit{closure} theorem, telling us that bounded sets (with respect to the mass and the mass of the boundary) in the space $\mathcal{I}_k(U)$ are (sequentially) weak$^*$-closed in $\mathbb{N}_k(U)$.

\begin{theorem}[Closure theorem]\label{t:closure}
Let $(T_n)_n$ be a sequence of integral $k$-currents in $U \subset \mathbb{R}^d$ such that $$\operatorname{sup}_n\big(\mathbb{M}(T_n)+\mathbb{M}(\partial T_n) \big)<\infty.$$ Then there exist $T \in \mathcal{I}_k(U)$ and a subsequence $(T_{n_j})_j$ such that $T_{n_j}$ converges in the sense of currents to $T$.
\end{theorem}

The closure theorem was initially proved for codimension one currents, i.e. $(d-1)$-dimensional currents in $\mathbb{R}^d$ by De Giorgi, see [\ref{Degiorgiperimetri}] and [\ref{Degiorginuovi}], in the language of \textit{sets of finite perimeter} few years before the appearence of Federer and Fleming foundational article [\ref{FF}]. It is worth mentioning this since the ideas introduced by De Giorgi in [\ref{Degiorgiperimetri}] and [\ref{Degiorginuovi}] still play an important role in the general case. 

Federer and Fleming's proof of closure theorem is based on a rectifiability argument that goes under the name of \textit{structure theorem} for sets of finite Hausdorff measure, the proof of which is rather demanding. Several years later, two different proofs of the closure theorem were proposed by Solomon [\ref{Solomon}] and by Almgren [\ref{Almgrenclosure}], without employing the machinery of the structure theorem: their proofs relied on various facts about multivalued functions, a tool that will be introduced in Chapter 2 and that still requires some preliminary work. Moreover, another different proof of the closure theorem without the structure theorem nor multivalued functions and that develops in the same spirit as De Giorgi's codimension 1 proof is due to White [\ref{Whiteclosure}].

As a consequence of the closure theorem, it is important to state another fundamental theorem, telling us that $$\mathcal{R}_k(U) \cap \mathbb{N}_k(U)=\mathcal{I}_k(U).$$

\begin{theorem}[Boundary rectifiability theorem]
Let $T$ be an integer multiplicity rectifiable current with $\mathbb{M}(\partial T)<\infty$. Then $T$ is an integral\,\footnote{Arguably the terminology comes from merging the two words \textit{integer} and \textit{normal}.}current.
\end{theorem}

As a corollary of Federer and Fleming closure theorem we can prove existence of a solution of the generalized Plateau's problem in the class of integral currents\footnote{It is fair to remark that in the case of hypersurfaces Hardt and Pitts solved the Plateau's problem for integral currents without employing the closure theorem, using instead basic facts about $BV$ functions in $\R^d$ and normal currents, see [\ref{HardtPitts}].}; it will be the main object of study in Chapter $2$ to investigate how \textit{regular} a solution is.

\begin{theorem}[Generalized Plateau's solution]\label{t:plateausolution}
Given $S \in \mathcal{I}_k(\R^d)$ with $\partial S=0$, there exists a $(k+1)$-dimensional integral current $T_1$ such that $\partial T_1=S$ and $\mathbb{M}(T_1)$ is minimized among all integral currents $T \in \mathcal{I}_{k+1}(\R^d)$ satisfying $\partial T = S$. We call such a solution an \emph{area-minimizing integral current}.
\end{theorem}

\begin{proof}
Let $$m:= \operatorname{inf}\{\mathbb{M}(T) \,|\,  T \in \mathcal{I}_{k+1}(\R^d) : \partial T= S \}$$ and let $(T_n)_n$ be a minimizing sequence. Note that $m$ is finite, due to the \textit{cone construction}, see [\ref{Simonbook}, equation 26.26]. Since $\mathbb{M}(T_n)$ is bounded and $\mathbb{M}(\partial T_n)$ is constant, we can apply Theorem \ref{t:closure} to the sequence $(T_n)_n$ getting a subsequence converging in the sense of currents to $T_1 \in \mathcal{I}_{k+1}(\R^d)$. 

By Remark \ref{r:lsc} we get $$\partial T_1=S \,\,\text{ and }\,\, \mathbb{M}(T_1) \le m,$$ concluding the proof.
\end{proof}

We end this subsection with the following important compactness theorem of area-minimizing currents:
\begin{theorem}\label{t:compactnessforam}{\normalfont [\ref{Simonbook}, Theorem 34.5]}
Suppose $(T_j)_j$ is a sequence of area-minimizing $k$-currents in $U\subset \R^d$ with respect to their own boundaries $ (\partial T_j)_j$ such that $$\sup_{j}\big(\mathbb{M}(T_j) + \mathbb{M}(\partial T_j)\big) < \infty,$$ and suppose that $T_j $ converges to $T$ in the sense of currents for some $T \in \mathcal{D}_k(U)$. Then $T$ is area-minimizing in $U$.
\end{theorem}

\subsubsection{Further topics on currents}

As it always happens in mathematics, different constructions may lead to the definition of the same object. This is the case, for instance, for the definition of Sobolev spaces, that can be equivalently defined as a particularization of the abstract point of view of the theory of distributions or by a completion procedure with respect to a suitable norm. Analogous is the case for the theory of currents.

\begin{definition}
A $k$-dimensional \textit{polyhedral} current (or $k$-polyhedral chain) is a $k$-current $P$ of the form  
\begin{equation}\label{ee:poly}
P:=\sum_{i=1}^N\theta_i\llbracket \sigma_i\rrbracket,
\end{equation}
where $\theta_i\in\R\setminus\{0\}$, $\sigma_i$ are nontrivial $k$-dimensional simplexes in $\R^d$ with disjoint relative interiors and oriented by constant $k$-vectors $\tau_i$ such that $\llbracket \sigma_i \rrbracket=\llbracket\sigma_i,\tau_i, 1 \rrbracket$ is the multiplicity-one rectifiable current naturally associated to the simplex $\sigma_i$. 
The vector space of polyhedral $k$-currents with support in $U \subset \R^d$ is denoted by $\mathbb{P}_k(U)$. A polyhedral current with integer coefficients $\theta_i$ is called \textit{integer polyhedral}.
\end{definition}

We now introduce a very useful norm on the space of currents: the \textit{flat norm}.

\begin{definition}\label{d:flatnorm1}
If $T$ is a $k$-current then its \textit{flat norm} is $$\Flat(T):=\inf \left\{\mathbb{M}(R)+\mathbb{M}(S): T=R+\partial S\right\}$$ where $R \in \mathcal{D}_k(U)$ and $S \in \mathcal{D}_{k+1}(U)$.
\end{definition}

\begin{remark}
It is immediate to see that the flat norm induces a coarser topology than the one induced by the mass norm. It is important to recall that if $(T_n)_n$ is a sequence of $k$-currents such that $\Flat(T_n - T) \rightarrow 0$ for some $T \in \mathcal{D}_k(U),$ then $T \stackrel{\ast}{\rightharpoonup} T$. The converse implication is also true, provided there exists $K$ compact set such that $\operatorname{supp}(T_n) \subset K$ for all $n$ and $T_n,T \in \mathbb{N}_k(U)$ with a uniform bound $\sup_n(\mathbb{M}(T_n) + \mathbb{M}(\partial T_n)) < \infty.$ Hence, the flat norm metrizes the dual topology of (normal, with uniform bound on their masses and the ones of their boundaries) currents: this may be regarded as the analogous of the classical fact that $L^1$ convergence and distributional convergence are equivalent for equibounded sequences of Sobolev $W^{1,1}$ functions.
\end{remark}

Denote by $\text{int}(K)$ the interior of the set $K$. The main motivation for introducing polyhedral chains and the flat norm is the following theorem, the proof of which relies on deeper machinery.


\begin{theorem}[Polyhedral approximation]\label{t:polyapprox}{\normalfont [\ref{Federerbook}, 4.2.21]}
If $T \in \mathcal{I}_k(\mathbb{R}^d),$ $\varepsilon>0$, $K \subset \mathbb{R}^d$ is a compact set such that $\text{supp}(T) \subset \text{int}\, K$, then there exists $P \in \mathbb{P}_k(\mathbb{R}^d)$, with $\text{supp}(P) \subset K$, such that
$$
\mathbb{F}_K(T-P)<\varepsilon, \quad \mathbb{M}(P) \leq \mathbb{M}(T)+\varepsilon, \quad \mathbb{M}(\partial P) \leq \mathbb{M}(\partial T)+\varepsilon. $$ Moreover, $P$ can be taken with integer coefficients. \end{theorem}

It is useful to work with currents with support in a compact set. Let $K$ be a compact set such that $K \subset U \subset \R^d$. It is customary to introduce the notations such as $\mathcal{D}_k(K)$ for the class of $k$-currents with support in $K$, \textit{i.e.} $$\mathcal{D}_k(K):=\{T \in \mathcal{D}_k(U) \, | \, \operatorname{supp}(T) \subset K\}$$ and $\mathcal{D}_{k,cpt}(U)$ as the class of all such compactly supported currents, \textit{i.e.} $$\mathcal{D}_{k,\text{cpt}}(U):=\cup \{T \in \mathcal{D}_k(K) \, | \, K \subset U, K \text{ compact} \}.$$

Analogously we do for all previously defined classes of currents and we will denote by $\Flat_K$ the flat norm of a current $T \in \mathcal{D}_k(K)$.

We briefly mention the notion of \textit{flat chains}, which originated in the work of Whitney [\ref{Whitneybook}].

\begin{definition}
Let $U \subset \mathbb{R}^d$ be an open set and let $K \subset U$ be a compact set. The class of \textit{integral flat chains} supported in $K$ is defined by
$$
\mathscr{F}_{k}(K):=\left\{R+\partial S \mid R \in \mathcal{R}_{k}(K), S \in \mathcal{R}_{k+1}(K)\right\}.
$$
\end{definition}

\begin{proposition}
We have the following results: \begin{enumerate}
\item $\mathscr{F}_{k}(K)$ is a complete metric space with respect to the metric induced by $\Flat_K$;
\item The set of $k$-integral currents with support in $K$ is $\mathbb{F}_K$-dense in $\mathscr{F}_{k}(K)$\footnote{In fact, the same result holds more in general whenever the ambient space is a closed convex subset $C$ of a Banach space, working with the general theory of currents in metric spaces introduced by Ambrosio and Kirchheim in [\ref{AKmetric}], inspired by one of the last ideas by De Giorgi [\ref{Degiorgiplateau}].}.
\end{enumerate}
\end{proposition}

Consequently $\mathscr{F}_{k}(K)$ can be regarded as the $\mathbb{F}_K$-completion of the space of integral currents with compact supports in $K$. Moreover one can note that the mass $\mathbb{M}$ is lower semicontinuous with respect to the $\Flat_K$-convergence in $\mathscr{F}_{k}(K)$. We define now the space of (real) $k$-\textit{flat chain}, denoted as $\mathbf{F}_k$.

\begin{definition}
We say that $T$ belongs to $\mathbf{F}_{k}(K)$ if there exists a sequence of normal currents $T_n \in \mathbb{N}_{k}(K)$ such that $\Flat_K\left(T-T_n\right) \rightarrow 0$. In other words we can write:
$$\mathbf{F}_{k}(K) := \overline{\mathbb{N}_{k}(K)}^{\Flat_K}.$$
\end{definition}

\begin{remark}
Note that by \ref{t:polyapprox} we get that:  $$\mathbf{F}_{k}(K) = \overline{\mathbb{N}_{k}(K)}^{\Flat_K} = \overline{\mathbb{P}_{k}(K)}^{\Flat_K}.$$ 

\end{remark}

In the spirit of operations that are usually done with manifolds, we conclude this section defining some operations on currents that will be useful in the sequel.

%

\begin{definition}
If $f: U \subset \mathbb{R}^d \rightarrow V \subset \mathbb{R}^{m}$ is a smooth, proper\footnote{We define a function $f:X\rightarrow Y$ between two topological spaces \textit{proper} if the preimage of every compact set in $Y$ is compact in $X$.} map, then it is possible to define the \textit{push-forward} (or \textit{image}) of a $k$-current $T$ on $U \subset \mathbb{R}^d$ as the $k$-current $f_{*} T$ on $\mathbb{R}^{m}$ defined by
$$
\left\langle f_{*} T , \omega\right\rangle:=\left\langle T , f^{*} \omega\right\rangle,
$$
for any $\omega \in \mathcal{D}^k\left(V\right)$. Note that the boundary operator commutes with the push-forward operation, that is $\partial\left(f_{*} T\right)=$ $f_{*}(\partial T)$ and that $\operatorname{supp}(f_{*}T) \subset f (\operatorname{supp}(T))$.

\end{definition}

In general there is no natural definition for the notion of intersection of two currents, since even for the intersection theory for smooth manifolds some ``safety" conditions are required. However, it is possible to define the intersection of a normal $k$-current $T$ and a level set $f^{-1}(y)$ of a smooth map $f: \mathbb{R}^d \rightarrow \mathbb{R}^m$ (with $k \leq m \leq d$ ) for almost every $y$, resulting in a normal current $T_y$ with the expected dimension $m-k$. This operation is called \textit{slicing} and to define it properly we need to recall two important results.

\begin{theorem}[Sard theorem]
Let $f: \mathbb{R}^d \rightarrow \mathbb{R}^m$ be of class $\mathcal{C}^k$ for some $k \geq \max \{d-m, 1\}$. Denote by
$$
C_f:=\left\{x \in \mathbb{R}^d: \operatorname{rank}(D f(x))<m\right\}
$$
the set of critical points of $f$. Then $\mathcal{L}^m(f(C_f))=0$, that is the set of critical values of $f$ is of null $\mathcal{L}^m$-measure.
\end{theorem}

\begin{corollary}\label{c:sard}
Let $0<m \leq k \leq d$. Let $M$ be a smooth $k$-surface in $\mathbb{R}^d$ and $f: \mathbb{R}^d \rightarrow \mathbb{R}^m$ be smooth. Denote $M_y:=M \cap f^{-1}(y)$. Then for $\mathcal{L}^m$-a.e. $y$, $M_y$ is a smooth surface of dimension $k-m$ (or it is empty).
\end{corollary}

In order to extend Corollary \ref{c:sard} where $M$ is replaced by a rectifiable set $E$ and $f$ is Lipschitz we need to recall the following version of the \textit{coarea formula}.

\begin{theorem}[Coarea formula]\label{t:coarea}
Let $E \subset \mathbb{R}^d$ be $k$-rectifiable and $f : \mathbb{R}^d \rightarrow \mathbb{R}^m$ a Lipschitz function. For $y \in \mathbb{R}^m$, denote $E_y:=E \cap f^{-1}(y)$. For $\mathcal{H}^k$-a.e. $x \in E$ denote $D_\tau f(x)$ the tangential gradient of $f$ at $x$ and denote the tangential Jacobian as
$$
J_\tau f(x):=\left|D_\tau f_1(x) \wedge \cdots \wedge D_\tau f_m(x)\right|
$$
Then, for every Borel function $g: E \rightarrow[0,+\infty]$, we have 
\begin{equation}
\int_{\mathbb{R}^m}\left(\int_{E_y} g(x) \,d \mathcal{H}^{k-m}(x)\right) d \mathcal{L}^m(y)=\int_{E} g(x)\, J_\tau f(x) \, d \mathcal{H}^k(x). \end{equation}
\end{theorem}

By Sard theorem and coarea formula we have the following proposition that allows us to define the notion of \textit{slice of a rectifiable current}:

\begin{proposition}[Slicing of rectifiable currents]\label{p:slice}
Let $T=\llbracket E, \tau, \theta \rrbracket$ be a rectifiable $k$-current in $\mathbb{R}^d$, with $\mathbb{M}(T)<\infty$. Let $f: \mathbb{R}^d \rightarrow \mathbb{R}^m$ be a Lipschitz function, with $0<m \leq k \leq d$. Denote
$$
\tilde{E}:=\left\{x \in E: D_\tau f(x) \,\text{ is defined and has rank \textit{m}} \,\right\}.
$$
For $y \in \mathbb{R}^m$, denote $E_y:=E \cap f^{-1}(y)$. Then the following hold true:
\begin{enumerate}
\item $\mathcal{H}^{k-m}(E_y \backslash \tilde{E})=0$, for $\mathcal{L}^m$-a.e. $y$;
\item $E_y$ is $(k-m)$-rectifiable for $\mathcal{L}^m$-a.e. $y$;
\item Denoting $\eta(x):=D_\tau f_1(x) \wedge \cdots \wedge D_\tau f_m(x),$
we have that
$$
T_xE=T_xE_y \oplus \operatorname{Span}\{\eta(x)\},
$$
for $\mathscr{L}^m$-a.e. $y$ and for $\mathcal{H}^{k-m}$-a.e. $x \in E_y$. Hence, for $\mathcal{L}^m$-a.e. $y$, we can define the orientation on $E_y$ as the $(k-m)$-vector $\tilde{\tau}$ such that
$$
\frac{\eta(x)}{|\eta(x)|} \wedge \tilde{\tau}(x)=\tau(x), \quad \text { for } \mathcal{H}^{k-m}\text{-a.e.} \,y \in E_y.
$$
\end{enumerate}
\end{proposition}

\begin{definition}\label{d:slice}
Under the assumptions of Proposition \ref{p:slice}, the $(k-m)$-rectifiable current $T_y:=\llbracket E_y, \tilde{\tau}, \theta \res E_y\rrbracket$ is well-defined for $\mathcal{L}^m$-a.e. $y$ and it is defined as the \textit{slice} of $T$ at $y$ according to $f$. Sometimes the notation $\langle T,f,y\rangle$ is used to highlight $f$.
\end{definition}

By the coarea formula with $g=|\theta|$ we get:

\begin{corollary}\label{c:simon1} Under the assumptions of Proposition \ref{p:slice}, let $T_y:=\llbracket E_y, \tilde{\tau}, \theta \res E_y\rrbracket$. Then
$$\int_{\mathbb{R}^m} \mathbb{M}(T_y) d \mathcal{L}^m(y)=\int_{ E}|\theta(x)| J_\tau f(x) d \mathcal{H}^k(x) \leq [\operatorname{Lip}(f)]^m \mathbb{M}(T) .$$
\end{corollary}

\begin{proposition}\label{p:simonslice}
Let $T \in \mathbb{N}_k(\R^d)$ and rectifiable. Let $f: \mathbb{R}^d \rightarrow \mathbb{R}$ be a Lipschitz and $\mathcal{C}^1$ function. Let $T_y$ be as in Definition \ref{d:slice}, then for $\mathcal{L}^1$-a.e. $y \in \mathbb{R}$ we have:
$$
T_y=\partial(T\res\{f \leq y\})-\partial T\res\{f \leq y\}.
$$
\end{proposition}

As anticipated at the beginning and motivated by the above characterization for codimension-one slices of rectifiable and normal currents, it is possible to define the \textit{slice} of a normal current as well.

\begin{definition}
Let $T \in \mathbb{N}_k(\mathbb{R}^d)$ and $f: \mathbb{R}^d \rightarrow \mathbb{R}$ be a Lipschitz and $\mathcal{C}^1$ function. For every $y \in \mathbb{R}$ we define the \textit{slice of a normal current} as
$$T_y:=\partial(T\res\{f \leq y\})-(\partial T)\res\{f \leq y\}.$$
\end{definition}

It is a bit more complicated to deal with normal currents slices of codimension $m>1$.

\begin{definition}
Let $T \in \mathbb{N}_k(\mathbb{R}^d)$ and $f: \mathbb{R}^d \rightarrow \mathbb{R}^m$ be a Lipschitz and $\mathcal{C}^1$ function. Denote $f_1, \ldots, f_m$ the components of $f$. For every $y \in \mathbb{R}^m, y=\left(y_1, \ldots, y_m\right)$ we define recursively
$$
\begin{aligned}
T_{y_1}&:=\partial(T\res\{f_1 \leq y_1\})-(\partial T)\res\{f_1 \leq y_1\}, \\
T_{y_1, y_2}&:=\partial(T_{y_1}\res\{f_2 \leq y_2\})-(\partial T_{y_1})\res\{f_2 \leq y_2\}, \\
&\cdots \\
\langle T,f,y\rangle = T_y&:=\partial(T_{y_1, \ldots, y_{m-1}}\res\{f_m \leq y_m\})-(\partial T_{y_1, \ldots, y_{m-1}})\res\{f_m \leq y_m\}.
\end{aligned}
$$

\end{definition}

\begin{remark}
When $m>1$ we have to ensure that after every iteration the slices are still normal currents for $T_y$ to be well-defined.
\end{remark}

We conclude with two useful properties of slicing for normal currents and a cornerstone theorem in the theory of currents proved by White in [\ref{Whiterectifiabilityflatchains}].

\begin{proposition}
Let $T \in \mathbb{N}_k(\mathbb{R}^d)$ and $f: \mathbb{R}^d \rightarrow \mathbb{R}$ be a Lipschitz and $\mathcal{C}^1$ function. Then for every $y \in \mathbb{R}$ we have

$$\partial(T_y)=-(\partial T)_y$$
and
$$\int_{\mathbb{R}} \mathbb{M}(T_y) d y \leq \operatorname{Lip}(f) \, \mathbb{M}(T).$$

\end{proposition}

\begin{theorem}[Rectifiability of flat chains with finite mass]\label{t:white99}
Let $T \in \mathbb{N}_k(\mathbb{R}^d)$. For every multi-index $I=\left(i_1, \ldots, i_k\right) \in I(k, d)$, we set
$$
\pi_I:(x_1, \ldots, x_d) \mapsto (x_{i_1}, \ldots, x_{i_k}) \in \mathbb{R}^k.
$$
$T$ is integer rectifiable if and only if, for every $I$, $\mathcal{L}^k$-almost every slice $T_y$ of $T$ according to $\pi_I$ is integer rectifiable.
\end{theorem}

\section{Optimal branched transport theory}

Optimal transport aims to find the best way to carry a given source into a given target. Such topic witnessed an impressive progression in the last thirty years, developing deep connections with many fields of mathematics and serving as a model for many biological and human-designed systems. 

The classical Monge's problem dates back in its original discrete formulation to 1781, see [\ref{Monge}], and can be stated in the following general modern form:

\begin{namedproblem}[Monge's]
Given two probability measures $\mu$ and $\nu$, defined on the measurable spaces $X$ and $Y$, find a measurable map $T: X \rightarrow Y$ such that
$$
T_{*} \mu=\nu,
$$
i.e.
$$
\nu(A)=\mu\left(T^{-1}(A)\right) \quad \text{ for all } A \subset Y \text { measurable, }
$$
and in such a way that $T$ minimizes the transportation cost, that is
$$
\int_X c(x, T(x)) d \mu(x)=\min _{S_* \mu=\nu}\left\{\int_X c(x, S(x)) d \mu(x)\right\},
$$
where $c: X \times Y \rightarrow \mathbb{R} \cup\{+\infty\}$ is a given cost function.
\end{namedproblem}

A huge literature has been developed out of this problem, with many books and surveys such as [\ref{Villani}]. Nevertheless, two main objections may be raised to the Monge's problem. In primis, the model does not take into account the trajectories of the moving particles, assuming implicitly that they are segments. In addition, the above formulation of Monge's problem does not take into account possible dynamic effects and interactions among the moving particles, focussing only on the best coupling between initial points and final distribution of mass. Nevertheless, from modeling purposes there may be the need to look at the transport as a dynamic process, allowing for both nonlinear trajectories and interactions among particles. Just to mention a few examples, systems for which such features are relevant appear in nature, \textit{e.g.} in roots systems of trees and leaf nerves, the nervous, the bronchial and the cardiovascular systems, and in human-designed structures like supply-demand distribution networks, irrigation networks and electric power supply systems.

Mainly for these reasons, extensions of the Monge's problem have been studied for transportation systems that privilege group flows rather than spread-out processes, leading to optimal transport networks with peculiar ramified structures: this class of problems is nowadays known as \textit{optimal branched transport.} In all of the many different formulations of the problem, the main feature is the fact that the cost functional is designed in order to privilege large flows and to prevent diffusion; indeed the transport actually happens on a 1-dimensional network.

The aim of this section is to introduce both the discrete formulation of optimal branched transport, dating back to an embryonal version due to Gilbert, see [\ref{Gilbert}], and the variational formulations, which aim at transporting ``diffuse" measures by means of a ramified structure. There are two main formulations of the variational optimal branched transport problem: the \textit{Eulerian} formulation and the \textit{Lagrangian} formulation, see [\ref{Xia2003}] and [\ref{MSM}] respectively. The equivalence between both model formulations was shown in [\ref{Pegon}].

We will introduce the Eulerian formulation which, from a geometric point of view, can be described by means of the theory of currents, while we will only recall some basic facts about the Lagrangian formulation. The Eulerian verison of the optimal branched transport problem can be stated as a 1-dimensional Plateau-type problem in which one aims to minimize a fractional power of the mass functional.

Eventually, we will describe the existence theory for the Eulerian formulation of the optimal branched transport problem, proving the existence of solutions with finite costs.

\subsection{The discrete model}

The \textit{discrete} model is the simplest possible model of optimal branched transport and it takes its name from the fact the initial measure $\mu_-$ and the target measure $\mu_{+}$ are finite atomic, that is $$\mu_-:= \sum_{i=1}^k a_i \delta_{x_i} \hspace{0.3cm} \text{ and } \hspace{0.3cm} \mu_+:= \sum_{j=1}^l b_j \delta_{y_j}$$ with $a_i, b_j \in \R$. 

\begin{definition}\label{d:discrete}
A \textit{transportation network} from $\mu_-$ to $\mu_{+}$ is a finite weighted oriented graph $G$ embedded in $\R^d$, consisting of a set of vertices $V(G)$, a set of oriented edges $E(G)$ and a weight function
$$
w: E(G) \rightarrow \R
$$

such that
\begin{enumerate}
\item $\left\{x_1, x_2, \ldots, x_k\right\} \cup\left\{y_1, y_2, \ldots, y_l\right\} \subset V(G)$.
\item For each initial vertex $x_i$ with $ i=1, \ldots, k$,
$$
\sum_{e \in E(G)\,:\, x_i=e_-} w(e) - \sum_{e \in E(G)\,:\, x_i=e_+} w(e)  = a_i,
$$
where $e_{-}$ and $e_{+}$ denote respectively the first and second endpoint of the oriented edge $e \in E(G)$.

\item For each target vertex $y_j$ with $j=1, \ldots, l$,
$$
\sum_{e \in E(G)\,:\, y_j=e_+} w(e) - \sum_{e \in E(G)\,:\, y_j=e_-} w(e)  = b_j.
$$

\item For any vertex $v \in V(G)\setminus \{x_1, x_2, \ldots, x_k, y_1, y_2, \ldots, y_l\}$.
$$
\sum_{e \in E(G)\,: \,v=e_+} w(e) = \sum_{e \in E(G)\,:\, v=e_-} w(e).
$$
In other words, we say that $(G,w)$ satisfies the ``Kirchoff's laws" at each of its vertices.
\end{enumerate}
\end{definition}

We define now the energy that defines the cost of the transportation.

\begin{definition}
Fix $\alpha \in [0,1)$, we define the $\alpha$-\textit{energy} of $G$ as \begin{equation}\label{e:alphaegy}
\mathbb{E}^{\alpha}(G):= \sum_{e \in E(G)}|w(e)|^{\alpha}\, \mathcal{H}^1(e).
\end{equation}
\end{definition}

We can now formulate what is known as the {\em Gilbert's problem}:

\begin{namedproblem}[Gilbert's]
Fix $\alpha \in [0,1)$ and $\mu_-$, $\mu_+$ finite atomic with the same total mass. Find a transportation network $G_1$ between $\mu_-$ and $\mu_+$ such that \begin{equation}\label{GP} \mathbb{E}^{\alpha}(G_1) \le \mathbb{E}^{\alpha}(G)\end{equation} among all transportation networks $G$ between $\mu_-$ and $\mu_+$. We call $G_1$ an \emph{optimal transportation network}.
\end{namedproblem}

\begin{figure}[h]
    \centering
    \includegraphics[width=0.625\textwidth]{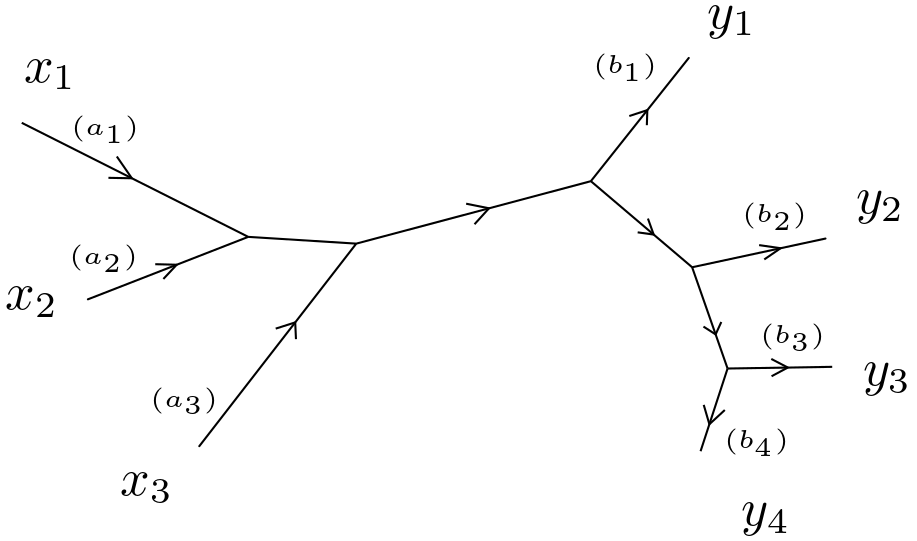}
    \caption{Representation of Definition \ref{d:discrete}}
   \end{figure}

Minimizers of the Gilbert's Problem turn out to satisfy some necessary topological condition; to state it formally, we need to recall some basic definitions from graph theory. We call a \textit{cycle} a closed chain $(e_1, \dots, e_n)$ of consecutive\footnote{Where consecutive means that $e_{i^+}=e_{{i+1}^-}$ for $i=1,\dots,n-1$.} distinct oriented edges. A \textit{loop} is a chain of consecutive distinct oriented edges which can be turned into a cycle possibly switching the orientation of some edges. An oriented weighted graph is \textit{acyclic} if it does not contain any cycle and it is a \textit{tree} if it does not contain any loop.

\begin{remark}
By definition it is immediate to check that a tree is an acyclic graph but the converse is not true.
\end{remark}

\begin{proposition}\label{p:tree}{\normalfont [\ref{BW1}, Lemma 2.6]}
Let $(G,w)$ be an optimal transportation network. Then $G$ is a tree.
\end{proposition}

The main idea to prove Proposition \ref{p:tree} is to argue by contradiction: assume there is a loop and build a competitor of the problem (\textit{i.e.} another graph $G'$ satisfying the Kirchhoff's laws) such that $\mathbb{E}^{\alpha}(G') < \mathbb{E}^{\alpha}(G)$ by exploiting concavity of the function $x\mapsto |x|^\alpha$.

\begin{remark}
If we know the basic topological property that optimal transportation networks need to be trees, then we get an a priori bound on the number of vertices a solution can have. Hence, the problem is equivalent to a finite number of finite-dimensional optimization problems, providing existence of a solution to the discrete formulation of the optimal branched transport problem.
\end{remark}

It is possible to show a rigidity property of the geometry of optimal transportation networks: whenever three edges come together at a vertex, the angle they form must satisfy some necessary conditions. More formally:

\begin{proposition}{\normalfont [\ref{BCM}, Lemma 12.1]}\label{p:12punto1}
Let $x_1, x_2, y_1$ be three distinct points in $\mathbb{R}^d$, $\mu_{-}=a_1 \delta_{x_1}+a_2 \delta_{x_2}$ and $\mu_{+}=b_1 \delta_{y_1}$ with $b_1=a_1+a_2$ and $a_1, a_2>0$.
If $x_1, x_2, y_1$ are aligned, an optimal transportation network from $\mu_{-}$ to $\mu_{+}$ is the minimal segment containing $x_1, x_2, y_1$. If $x_1, x_2, y_1$ are not aligned, an optimal transportation network lies in the triangle $x_1, x_2, y_1$. In addition, it is a graph with two or three edges.
\end{proposition}

\begin{proposition}{\normalfont [\ref{BCM}, Lemma 12.2]}\label{p:BCMangoli}
Let $(G,w)$ be an optimal transportation network between $\mu_{-}=a_1 \delta_{x_1}+a_2 \delta_{x_2}$ and $\mu_{+}=b_1 \delta_{y_1}$ with $b_1=a_1+a_2$ and $a_1, a_2>0$ and suppose it has 3 edges. With the notation of Figure \ref{f:angoli}, then the vertex $V$ must satisfy the following angle conditions: 

$$\begin{aligned} \cos \left(\theta_1\right) &=\frac{k_1^{2 \alpha}+1-k_2^{2 \alpha}}{2 k_1^\alpha}, \\ \cos \left(\theta_2\right) &=\frac{k_2^{2 \alpha}+1-k_1^{2 \alpha}}{2 k_2^\alpha}, \\ \cos \left(\theta_1+\theta_2\right) &=\frac{1-k_2^{2 \alpha}-k_1^{2 \alpha}}{2 k_1^\alpha k_2^\alpha}, \end{aligned}$$

where $k_1=\frac{a_1}{a_1+a_2}, k_2=\frac{a_2}{a_1+a_2}$.
\end{proposition}

\begin{figure}[h]
    \centering
    \includegraphics[width=0.323\textwidth]{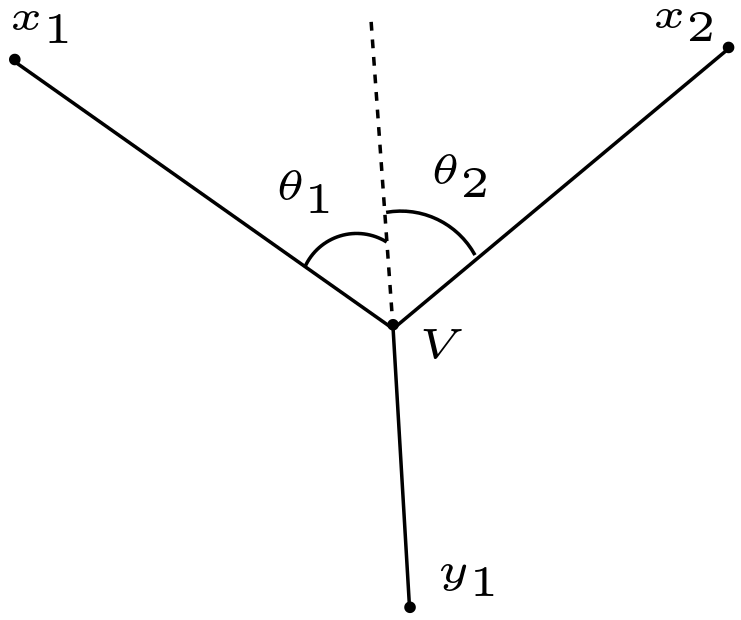}
    \caption{}
    \label{f:angoli}
\end{figure}

\begin{remark}
Once $a_1, a_2$ are fixed, the exponent $\alpha$ determines the aperture of the angles made by the two joining edges. Moreover, if $a_1=a_2$, then we have $\theta_1=\theta_2=\arccos(2^{2\alpha-1}-1)/2.$
\end{remark}

\subsection{The variational formulation}
Now we turn to the Eulerian variational formulation of the optimal branched transport problem. Our aim is to generalize the main objects of the discrete case to suit a continuous framework: we aim to transport an initial measure $\mu_-$ to a target measure $\mu_+$, where now $\mu_-$ and $\mu_+$ are probability measures with compact support. Hence, we need to generalize the notion of oriented weighted graph, the Kirchhoff's laws contraint and the $\alpha$-energy of the graph: to do so we will employ the theory of $1$-dimensional currents, following Xia [\ref{Xia2003}], see also [\ref{PaoliniStepanov}].

\begin{definition}
Let $K\subset$ $\R^d$ denote a convex compact set. If $\mu_-$ and $\mu_+$ are elements of $\mathcal{M}_+(K)$ such that $\Mass(\mu_-)=\Mass(\mu_+)$, we call a \emph{transport path} a $1$-dimensional normal current $T \in \mathbb{N}_1(K)$ with $\partial T= \mu_+ - \mu_-$. We denote by $\TP(b)$ the set of \textit{transport paths} with boundary $b$, that is $$\TP(b):=\{T\in\mathbb{N}_1(K) : \partial T=b\}.$$
\end{definition}

\begin{remark}
Since 1-dimensional normal currents are in one-to-one correspondence with vector-valued Radon measures with distributional divergence which is a measure, then we can always think of a transport path $T$ as a vector-valued measure such that $\text{div}(T)= \mu_- -\mu_+$. The condition $\partial T= \mu_+ - \mu_-$ (or, equivalently, $\text{div}(T)= \mu_- - \mu_+$) is the generalization of the discrete model Kirchhoff's laws.
\end{remark}

\begin{remark}\label{r:graphsdense}
We note that finite, oriented, weighted graphs can be identified with $1$-polyhedral chains of the form \eqref{ee:poly} simply by taking $\theta_i = w(e)$ for every $e \in E(G)$, $\llbracket \sigma_i \rrbracket = \llbracket e_i \rrbracket $ and taking as orientations $\tau_i$ the one given by the tangent vectors to each point $x \in G \subset \R^d$. 
\end{remark}

By Remark \ref{r:graphsdense}, we can now define the cost functional for the variational formulation of the optimal branched transport problem as the lower semicontinuous relaxation of the $\alpha$-energy $\mathbb{E}^\alpha$ with respect to the flat norm. More formally:

\begin{definition}\label{d:lscrelaxmass}
Let $\mathbb{E}^\alpha : \mathbb{P}_1(\R^d) \rightarrow [0,\infty)$ be the $\alpha$-energy defined as in \eqref{e:alphaegy}, having identified the graph $G$ with the corresponding polyhedral 1-chain. For every $T \in \mathbf{F}_1(\mathbb{R}^d)$ we call the $\alpha$-\textit{mass} the functional $\mathbb{M}^{\alpha}: \mathbf{F}_1(\mathbb{R}^d) \rightarrow[0, \infty]$ defined as
$$
\mathbb{M}^{\alpha}(T):=\inf \left\{\liminf _{j \rightarrow \infty} \mathbb{E}^\alpha(G_j): G_j \in \mathbb{P}_1(\mathbb{R}^d) \text { with } \mathbb{F}\left(T-G_j\right) \rightarrow 0\right\}.
$$
\end{definition}

It turns out that the $\alpha$-mass admits the following representation, see [\ref{CDRMS}].

\begin{proposition}\label{p:stuvard}
Given $\alpha \in [0,1)$ and a 1-current $T \in \mathbb{N}_1(K) \cup \mathscr{R}_1(K)$, we have that
$$
\mathbb{M}^\alpha(T)= \begin{cases}\int_E|\theta|^\alpha d \mathcal{H}^1, & \text { if } T=\llbracket E, \tau, \theta \rrbracket \in \mathscr{R}_1(K) ; \\ +\infty, & \text { otherwise. }\end{cases}
$$
\end{proposition}

\begin{remark} It is very important to use the flat norm. Indeed if the lower semicontinuous relaxation had been taken with respect to the convergence in the sense of currents, then it would trivialize to the zero functional. To prove this, it sufficies to consider the current $\llbracket \ell\rrbracket$ associated to a line segment of length 1. It is easy to see that $\ell$ can be approximated in the sense of currents by a sequence of currents $(T_n)_n$, where each $T_n$ is made by $n$ segments of length $1/n^2$ and multiplicity $n$. Note that $\mathbb{E}^{\alpha}(T_n)=n^{\alpha-1}$, converging to $0$ for $n\rightarrow \infty$ since $\alpha < 1$.
\end{remark}

\begin{remark} The proof of Proposition \ref{p:stuvard} is relying mostly on three fundamental ingredients in geometric measure theory: the slicing technique for rectifiable currents, the so-called integralgeometric formula, see [\ref{KP}, Section 2.1.4] and White's rectifiability Theorem [\ref{Whiterectifiabilityflatchains}].
\end{remark}

\begin{remark}
The $\alpha$-mass $\Mass^\alpha$ is a subadditive functional, namely
$$
\mathbb{M}^\alpha(T_1+T_2) \leq \mathbb{M}^\alpha(T_1)+\mathbb{M}^\alpha(T_2) \text { for every } T_1, T_2 \in \mathscr{R}_1(\mathbb{R}^d) \cup \mathbb{N}_1(\mathbb{R}^d) .
$$
Indeed, the inequality is trivial if $T_1$ or $T_2$ is not rectifiable. On the other hand, if $T_i=\llbracket E_i, \tau_i, \theta_i\rrbracket$, $i=1,2$, the multiplicity $\theta$ of the rectifiable current $T_1+T_2$ is obtained as the sum of the multiplicities of $T_1$ and $T_2$ with possible signs, so that $|\theta| \leq |\theta_1|+|\theta_2|$. Hence we deduce that
\begin{equation}\label{e:subadditivity}
\mathbb{M}^\alpha(T_1+T_2) \leq \int_{E_1 \cup E_2}|\theta_1+\theta_2|^\alpha d \mathcal{H}^1 \leq \int_{E_1 \cup E_2}|\theta_1|^\alpha+|\theta_2^\alpha| d \mathcal{H}^1=\mathbb{M}^\alpha(T_1)+\mathbb{M}^\alpha(T_2).
\end{equation}
\end{remark}

\begin{definition}
A current $T$ with finite mass is called \emph{acyclic} if there exists no nontrivial current $S$ such that
$$
\partial S=0 \quad \text { and } \quad \mathbb{M}(T)=\mathbb{M}(T-S)+\mathbb{M}(S).
$$
\end{definition}

\subsection{Existence for the optimal branched transport problem}

We can now state the (Eulerian) \emph{optimal branched transportation problem} with boundary $b=\mu_+-\mu_-$.

\begin{namedproblem}[Optimal Branched Transport]
Find a normal current $T\in \mathbb{N}_1(K)$ which minimizes the $\alpha$-mass $\Mass^\alpha$ among all normal 1-currents $S$ with boundary $\partial S=b$. 
\end{namedproblem}

\begin{proof}
Let $(T_j)_j \in \TP(b)$ be a minimizing sequence for the $\alpha$-mass. Then, by [\ref{MW}] for each $T_j$ there exists a polyhedral chain $G_j$ such that $$\mathbb{M}^\alpha (G_j) \leq \mathbb{M}^\alpha (T_j),\,\,\, \mathbb{M}(\partial G_j) \leq \mathbb{M}(\partial T_j)= \mathbb{M}(b) \text{ and } \mathbb{F}(\partial T_j-\partial G_j)\leq 2^{-j}.$$ One can replace the $G_j$'s with acyclic polyhedral chains $G'_j$ enjoying the above properties, since this operation decreases the $\alpha$-mass, see [\ref{BW1}, Lemma 2.6]. Since, by the Kirchhoff's laws, the multiplicity $\theta$ of each $G'_j$ is bounded by $\mathbb{M}(b)$ we have $\mathbb{M}(G'_j) \leq C\, \mathbb{M}^\alpha (G_j) \leq C \, \mathbb{M}^\alpha (T_j),$ getting a uniform bound. Hence by Proposition \ref{p:compactnessnormal} and lower semicontinuity of the $\alpha$-mass we conclude the proof.
\end{proof}

For notational purposes, we denote the \textit{least transport energy} associated to $b$ as $$\mathbf{E}^{\alpha}(b):= \inf\{\MM(T): T \in \TP(b)\}.$$ We define the set of \textit{optimal transport paths} with boundary $b$ by $$\OTP(b):= \{T\in \TP(b) :\MM(T)=\mathbf{E}^{\alpha}(b)\}.$$ 

The first observation is that the existence of elements with finite $\alpha$-mass in $\TP(b)$ is not guaranteed in general. For example in [\ref{DevSol}] it is proved that if $\alpha \le 1- 1/d$ then there are boundaries $b$ such that $\OTP(b)$ degenerates to the set of all currents $T$ with boundary $\partial T = b$, since there is no 1-current $T$ with $\partial T = b$ and $\MM(T)<\infty$. On the other hand, under the assumption $\alpha > 1- 1/d$, then the existence of traffic paths with finite $\alpha$-mass is guaranteed and, moreover, there is also a quantitative upper bound on the minimal transport energy.

\begin{theorem}[Existence of transports with finite cost]
Let $K \subset \R^d$ be a convex and compact set with diameter $L$, $\alpha>1-1/d$ and $\mu_{-}, \mu_{+} \in \mathcal{M}_{+}(\mathbb{R}^d)$ be two measures with equal mass. Then there exists a normal $1$-current $T \in \mathbb{N}_1(K)$ such that $\partial T = \mu_+ - \mu_-$ and
$$
\mathbb{M}^\alpha(T) \leq C_{\alpha, d} \, L\, \| \mu_+\|^\alpha ,
$$
where $C_{\alpha, d}$ is a geometric constant depending only on the exponent $\alpha$ and the dimension $d$.
\end{theorem}

\begin{remark}
We aim to transport $\mu_-$ to $\mu_+$ and without loss of generality we assume $\mu_-:= \delta_0$, since if we can transport $\delta_0$ to $\mu_+$ with finite $\alpha$-mass $\mathbb{M}^{\alpha}$, then we can start from an arbitrary measure $\mu_-$ and transport it to $\delta_0$ again with finite $\mathbb{M}^{\alpha}(T)$. Concatenating the two transports we get that the sum of the two traffic paths will have cancellation of the boundary $\delta_0$ and, by \eqref{e:subadditivity}, the $\alpha$-mass $\mathbb{M}^{\alpha}$ of the sum of the two traffic paths will still have finite cost.
\end{remark}

\begin{proof}
Assume without loss of generality that $K :=[0,1]^d$, $\mu_-:= \delta_0$ and $\mu_+$ is a probability measure. Consider a dyadic decomposition of $K$ and up to a translation we may assume that $\mu_+(\partial Q)=0$ for any dyadic cube $Q$ of any generation, see [\ref{CDRMcpam}, Lemma 3.1].

For every $n$ let $$\mu_+^n:= \sum_{i=1}^{2^{nd}}a_i^{(n)} \delta_{x_i^{(n)}},$$ where $a_i^{(n)}:= \mu_+(Q_i^{(n)})$ with $Q_i^{(n)}$ dyadic cubes of $n$-th generation and $x_i^{(n)}$ representing the center of each $Q_i^{(n)}$. 

Let $P_n$ be the polyhedral chain transporting $\mu_+^{n-1}$ onto $\mu_+^{n}$ by connecting each $x_i^{(n-1)}$ to the centers $\{x_j^{(n)}\}_{j=1}^{2^d}$ of the dyadic cubes of the $n$-th generation contained in $Q_i^{(n-1)}$, see Figure \ref{f:dyadic}.

We have \begin{equation}\label{e:existence1}
\mathbb{M}(P_n) = \sum_{j=1}^{2^{nd}} c_d \, 2^{-n} \, a_j^{(n) }= c_d \, 2^{-n},
\end{equation}
where $c_d$ is a constant depending only on the dimension $d$. The last passage follows from the fact $\sum_{j=1}^{2^{nd}} a_j^{(n)}=1$ since $\mu_+^n$ is a probability measure. Moreover, we can write \begin{equation}\label{e:existence2}
\begin{aligned}
\mathbb{E}^\alpha(P_n) =\sum_{j=1}^{2^{nd}} c_d \, 2^{-n} \, (a_j^{(n)})^{\alpha} &= c_d \, 2^{-n}2^{nd}\Big(2^{-nd}\sum_{j=1}^{2^{nd}}(a_j^{(n)})^{\alpha}\Big) \\
&\leq c_d \, 2^{n(d-1)}\Big(2^{-nd}\underbrace{\sum_{j=1}^{2^{nd}}a_j^{(n)}}_{=1}\Big)^{\alpha} = c_d \, 2^{n(d-1-d\alpha)},
\end{aligned}
\end{equation} where $d-1-d\alpha <0$ since by assumption $\alpha > 1- 1/d$.
By \eqref{e:existence1} we have that $T_m:= \sum_{n=1}^{m}P_n$ is a Cauchy sequence in mass and, a fortiori, with respect to the flat norm. Hence, there exists a flat chain $T$ such that $T_m \stackrel{\Flat}{\rightarrow} T$ and $\mathbb{M}(T)$ is finite\footnote{Since the masses of the partial sums were equibounded and $\mathbb{M}$ is lower semicontinuous with respect to the $\Flat$-convergence.}. Moreover, we get $\mu_+^m - \mu_- = \partial T_m $ and, by continuity of the boundary operator $$\partial T_m \stackrel{\Flat}{\rightarrow} \partial T.$$ Since $\mu_+^m \stackrel{\Flat}{\rightarrow} \mu_+$, we conclude that $$\partial T = \mu_+ - \mu_-.$$ 

Finally, we show $\mathbb{M}^{\alpha}(T) <\infty$. By \eqref{e:existence2}, Definition \ref{d:lscrelaxmass} and subadditivity of $\mathbb{E}^{\alpha}$ we conclude:

$$\mathbb{M}^{\alpha}(T) \leq \liminf _{m \rightarrow \infty} \big(\mathbb{E}^\alpha(T_m)\big) \le \lim_{m \rightarrow \infty}\left(\sum_{n=1}^m \mathbb{E}^\alpha(P_n)\right)\leq C_{\alpha, d}.$$
\end{proof}

\begin{remark}
The constant $C_{\alpha, d}$ tends to $\infty$ when $\alpha \rightarrow (1- 1/d)$. Clearly, if $\alpha \le 1- 1/d$ then the series in \eqref{e:existence2} diverges, but this would not tell us that the threshold $\alpha > 1- 1/d$ is sharp. Nevertheless, it turns out that $\alpha > 1- 1/d$ is indeed a sharp bound for the existene of a transport path with finite cost and the ideas used to prove it are not far from the above argument; we refer to [\ref{DevSol}] for a detailed discussion.
\end{remark}

\begin{figure}[h]
    \centering
    \includegraphics[width=0.4\textwidth]{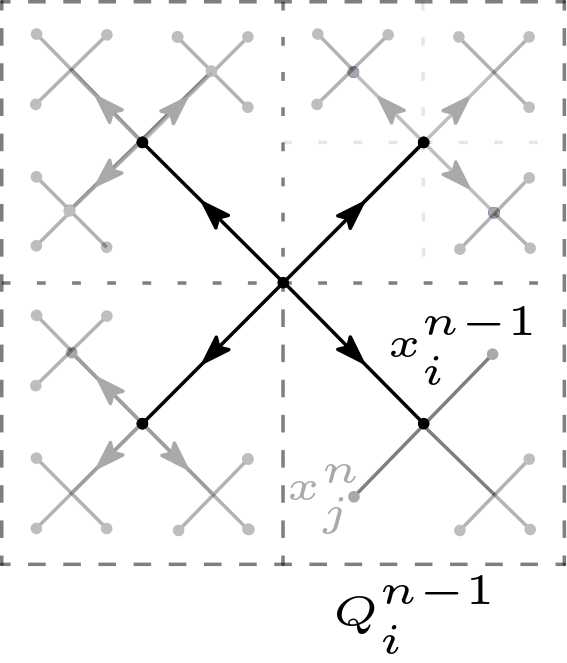}
    \caption{}
    \label{f:dyadic}
\end{figure}

\subsubsection{Useful properties}

We collect some results in optimal branched transport theory that will be useful in the sequel.

We now introduce some language from the \textit{Lagrangian formulation} of the optimal branched transport problem. In this formulation, see [\ref{MSM}], transport paths are modeled as probability measures on the space of Lipschitz curves, where each curve represents the trajectory of a single particle. The Eulerian and the Lagrangian formulations have been proved to be equivalent (see [\ref{Pegon}]) and, in particular, the link between these two formulations of the optimal branched transport problem is encoded in a deep result due to Smirnov on the structure of acyclic normal 1-currents, see [\ref{Smirnov}].

\begin{definition}
We denote by Lip the space of 1-Lipschitz curves $\gamma:[0, \infty) \rightarrow \mathbb{R}^d$. For $\gamma \in$ Lip we denote by $T_0(\gamma)$ the value
$$T_0(\gamma):=\sup \{t: \gamma \text{ is constant on } [0, t]\}$$
and by $T_{\infty}(\gamma)$ the (possibly infinite) value
$$T_{\infty}(\gamma):=\inf \{t: \gamma \text{  is constant on } [t, \infty)\}.$$ Given a Lipschitz curve with finite length $\gamma:[0, \infty) \rightarrow \mathbb{R}^d$, we denote $$\gamma(\infty):=\lim _{t \rightarrow \infty} \gamma(t).$$
\end{definition}

\begin{definition}
 We say that a curve $\gamma \in$ Lip of finite length is \textit{simple} if $\gamma(s) \neq \gamma(t)$ for every $T_0(\gamma) \leq s<$ $t \leq T_{\infty}(\gamma)$ such that $\gamma$ is nonconstant in the interval $[s, t]$.
\end{definition}

To a Lipschitz simple curve with finite length $\gamma:[0, \infty) \rightarrow \mathbb{R}^d$, we associate canonically the following rectifiable 1-current
$$
R_\gamma:=\left\llbracket\operatorname{Im}(\gamma), \frac{\gamma^{\prime}}{\left|\gamma^{\prime}\right|}, 1\right\rrbracket .
$$

\begin{remark}
It follows from \eqref{e:massrectif} that $\mathbb{M}\left(R_\gamma\right)=\mathcal{H}^1(\operatorname{Im}(\gamma))$
and it is immediate to check that $
\partial R_\gamma=\delta_{\gamma(\infty)}-\delta_{\gamma(0)}.$
Since $\gamma$ is simple, if it is also nonconstant, then $\gamma(\infty) \neq \gamma(0)$ and $\mathbb{M}\left(\partial R_\gamma\right)=2$.
\end{remark}

\begin{definition}\label{d:albmarch}
Let $I$ be a finite measure space and, for each $\lambda \in I$, let $T_\lambda$ be a 1-current such that \begin{itemize}
\item[(a)] the function $\lambda \mapsto \langle T_\lambda , \omega \rangle$ is measurable for every $\omega \in \mathcal{D}^1(\R^d),$
\item[(b)] $\int_I\mathbb{M}(T_{\lambda}) \,d\lambda < \infty.$
\end{itemize}
Then we denote by $T:=\int_I T_\lambda \,d\lambda$ the 1-current defined by $$\langle T, \omega \rangle = \int_I\langle T_\lambda , \omega \rangle \, d\lambda \, \text{ for every } \omega \in \mathcal{D}^1(\R^d).$$
\end{definition}

\begin{remark}
Note that assumption (a) and the definition of mass imply that the function $\lambda \mapsto \mathbb{M}(T_{\lambda})$ is measurable, thus the integral in assumption (b) is well-defined.
\end{remark}

\begin{definition}
Let $T \in \mathbb{N}_1(\mathbb{R}^d)$ and let $\pi \in \mathcal{M}_{+}$(Lip) be a finite positive measure supported on the set of curves with finite length such that
\begin{equation}\label{e:integration}
T=\int_{\text {Lip }} R_\gamma \, d \pi(\gamma),
\end{equation}
in the sense of Definition \ref{d:albmarch}. We say that $\pi$ is a \textit{good decomposition} of $T$ if $\pi$ is supported on nonconstant, simple curves and satisfies the equalities
$$
\begin{aligned}
&\mathbb{M}(T)=\int_{\text{Lip}} \mathbb{M}(R_\gamma) d \pi(\gamma),\\
&\mathbb{M}(\partial T)=\int_{\text{Lip}} \mathbb{M}(\partial R_\gamma) d \pi(\gamma)=2 \pi(\text{Lip}).
\end{aligned}
$$
\end{definition}

We recall a fundamental result by Smirnov [\ref{Smirnov}], which establishes that every acyclic normal $1$-current can be written as a weighted average of simple Lipschitz curves, \textit{i.e.} it admits a good decomposition.

\begin{theorem}\label{t:smirnov}
Let $T \in \mathbb{N}_1(\mathbb{R}^d)$ be an acyclic normal 1-current. Then there is a Borel finite measure $\pi$ on Lip such that $T$ can be decomposed as
$$
T=\int_{\mathrm{Lip}} R_\gamma \,d \pi(\gamma)
$$
and $\pi$ is a good decomposition of $T$.
\end{theorem}

\begin{proposition}{\normalfont [\ref{PaoliniStepanov}, Theorem 10.1]}\label{p:paolstep}
Let $\mu_{-}, \mu_{+} \in \mathcal{M}_{+}(\mathbb{R}^d)$ and $T \in \OTP(\mu_+-\mu_-)$ with finite $\alpha$-mass. Then $T$ is acyclic.
\end{proposition}

\begin{remark}
By Proposition \ref{p:paolstep} any optimal transport path admits a good decomposition.
\end{remark}

We recall some useful properties of good decompositions and we refer to [\ref{CDRMcalcvar}] for a proof.

\begin{proposition}\label{pp:gooddec}
If $T \in \mathbb{N}_1(\mathbb{R}^d)$ has a good decomposition $\pi$, the following statements hold:
\begin{enumerate}
\item The positive and the negative parts of the signed measure $\partial T$ are
\begin{equation}
\partial_{-} T=\int_{\text{Lip}} \delta_{\gamma(0)} \,d \pi(\gamma) \,\,\text { and } \,\, \partial_{+} T=\int_{\text{Lip}} \delta_{\gamma(\infty)} \,d \pi(\gamma).
\end{equation}
\item If $T=\llbracket E, \tau, \theta \rrbracket$ is rectifiable, then
\begin{equation}
|\theta(x)|=\pi\left(\{\gamma: x \in \operatorname{Im}(\gamma)\}\right) \text { for } \mathcal{H}^1 \text {-a.e. } x \in E.
\end{equation}
\item For every $\pi^{\prime}$ such that $\pi^{\prime}(A) \leq \pi(A)$ for every Borel set $A$, the representation
\begin{equation}
T^{\prime}:=\int_{\text{Lip }} R_\gamma \,d \pi^{\prime}(\gamma)
\end{equation}
is a good decomposition of $T^{\prime}$. Moreover, if $T=\llbracket E, \tau, \theta \rrbracket$ is rectifiable, then $T^{\prime}$ can be written as $T^{\prime}=\llbracket E, \theta^{\prime}, \tau \rrbracket$ with $\theta^{\prime} \leq \min \left\{\theta, \pi^{\prime}(\text{Lip})\right\}$.
\end{enumerate}
\end{proposition}

\begin{proposition}{\normalfont [\ref{BCM}, Proposition 7.4]}\label{p:singlepath}
Let $\alpha \in [0,1)$ and $T \in \OTP(b)$ such that $\mathbb{M}^{\alpha}(T) < \infty$. Then $T$ satisfies the \emph{single path property}, namely for every $x, y \in \text{supp}(T)$ $\pi$-a.e. $\gamma$ passing through $x,y$ follows the same trajectory in between (with the same orientation).
\end{proposition}

\begin{remark}
Proposition \ref{p:singlepath} is a necessary condition in the same spirit of the tree property of discrete graphs as in Proposition \ref{p:tree}.
\end{remark}

We now state one of the most important well-posedness properties in the theory of optimal branched transport, establishing that optima are \textit{stable} with respect to variations of the initial and final distributions of mass. In its full generality, the validity of such property was still an open problem in the field since few years ago, see [\ref{CDRMcpam}].

\begin{theorem}[Stability of minimizers]\label{t:stabilityCPAM}
Let $\alpha \in(0,1)$, $\mu_{-}$ and  $\mu_{+}$ be mutually singular positive measures on $\overline{B_r(0)}$ for some $r>0$, satisfying $\mu_{-}(\mathbb{R}^d)= \mu_{+}(\mathbb{R}^d)$. Let $(\mu_{-}^{n})_{n \in \mathbb{N}},(\mu_{+}^{n})_{n \in \mathbb{N}}$ be two sequences of positive measures on $\overline{B_r(0)}$ such that for every $n \in \mathbb{N}$ we have $\mu_{-}^{n}(\mathbb{R}^d)=\mu_{+}^{n}(\mathbb{R}^d)$ and
$$
\mu_{\pm}^{n} \stackrel{*}{\rightharpoonup} \mu_{\pm}.
$$
Assume there exist $T_n$ optimal transport paths with boundaries $b_n= \mu_{+}^{n} - \mu_{-}^{n}$ satisfying
$$
\sup _{n \in \mathbb{N}} \mathbb{M}^\alpha(T_n)<\infty.
$$
Then, the (nonempty) family of subsequential weak$^{*}$-limits of $T_n$ is contained in $\OTP(b)$, where $b=\mu_{+} - \mu_{-}.$
\end{theorem}

Finally we mention a slight improvement of Theorem \ref{t:stabilityCPAM}, relaxing the assumption requiring $\mu_-$ and $\mu_+$ to be mutually singular, pointing out that it is actually irrelevant in the proof given in [\ref{CDRMcpam}]. More precisely, we prove the following result which will be important in the sequel.

\begin{definition}
We denote the set of boundaries by
\[ {\mathscr{B}}_{0}(K) := \{b\in{\mathcal{D}}_{0}(K): \text{there is an } S \in {\mathcal{D}}_{1}(K) \text{ with } \partial S= b\}\,. \]

Fix an arbitrary constant $C>0$ and define 
\begin{equation}
    A_C:=\{b\in{\mathscr{B}}_{0}(K): \mathbb{M}(b)\leq C\;{\rm{and}}\;\MM(T)\leq C\mbox{ for every }T\in\OTP(b)\}.
\end{equation}
\end{definition}
We metrize $A_C$ with the \emph{flat norm} $\Flat_K$ and we observe that the set $A_C$ endowed with the induced distance is a nontrivial complete metric space.

\begin{theorem}\label{t:stability_new}
Let $b_n\in A_C$ and let $S_n \in \OTP(b_n)$. For every subsequential limit $T$ of $S_n$ we have $T\in\OTP(\partial T)$.
\end{theorem}
\begin{proof}
The subsequential convergence $\Flat(S_n-T)\to 0$ implies $\Flat(b_n-\partial T)\to 0$ and writing $b_n=\mu_+^n-\mu_-^n$ (being $\mu_+^n$ and $\mu_-^n$ respectively the positive and the negative part of the signed measure $b_n$) and $\mu_\pm:=\lim_{n\to\infty}\mu_\pm^n$, we have $\partial T=\mu_+-\mu_-$, where $\mu_+$ and $\mu_-$ are not necessarily mutually singular.

Hence, with respect to [\ref{CDRMcpam}, Theorem 1.1] we simply need to remove the assumption that $\mu_-$ and $\mu_+$ are mutually singular. In fact we observe that such assumption does not have a fundamental role in the proof already given in [\ref{CDRMcpam}] and, more precisely, we analyze all the points where such assumption is relevant.
\begin{itemize}
    \item In [\ref{CDRMcpam}, equation (4.9)] the assumption is used, but we observe that if we do not assume that $\mu_-$ and $\mu_+$ are mutually singular, [\ref{CDRMcpam}, equation (4.9)] would be replaced by 
    $$\partial T^{ij}=\int_{\Lip(Q^i,Q^j)}\delta_{\gamma(\infty)}-\delta_{\gamma(0)}\,dP(\gamma),$$
    which suffices to obtain [\ref{CDRMcpam}, equation (4.23)], which is the only point where [\ref{CDRMcpam}, equation (4.9)] is (implicitly) used.   
    \item In [\ref{CDRMcpam}, page 852, line 7], the fact that $\mu_-$ and $\mu_+$ are mutually singular is actually not necessary.
    \item The fact that $\mu_-$ and $\mu_+$ are mutually singular is necessary to obtain [\ref{CDRMcpam}, equations (4.16), (4.17)] and more precisely without such assumption the validity of those equations might fail in the cubes $\{Q^h:h=1,\dots,N\}$ but it remains true (with the same argument) in the remaining cubes $Q^i\in\Lambda(Q,k)$. However, we observe that [\ref{CDRMcpam}, equations (4.16), (4.17)] are only used to obtain [\ref{CDRMcpam}, equation (4.18)], which remains valid, precisely because it is stated only for the cubes $Q^i\in\Lambda(Q,k)\setminus \{Q^h:h=1,\dots,N\}$.
\end{itemize}
In conclusion, with the minor modifications listed above, the proof of [\ref{CDRMcpam}, Theorem 1.1] remains valid even without the assumption that $\mu_-$ and $\mu_+$ are mutually singular, thus concluding our proof. 
\end{proof}

\chapter{Regularity results} 

The main goal of Chapter 2 is to present the regularity theory for area-minimizing integral currents and for optimal transport paths. Indeed, once we have established under which conditions the Plateau's problem and the optimal branched transport problem admit a solution, the natural following question is how \textit{regular} this solution actually is. 

In Section 2.1 we will investigate the interior regularity theory for area-minimizing integral currents. Following the analogy with Sobolev functions that minimize the Dirichlet energy, one could hope that area-minimizing integral currents are, a posteriori, ``everywhere regular" manifolds. It turns out that this is not the case since area-minimizing surfaces are substantially more complicated geometric objects than functions, allowing for the presence of singularities; this fact makes the study of the regularity of area-minimizing integral currents one of the most difficult and exciting topics in the field of geometric variational problems. We will investigate the \textit{interior regularity} theory, but it is very important to mention that \textit{boundary regularity} for area-minimizing integral currents is a well-studied and live topic of research as well, witnessing major developments in the last years, see [\ref{Borda}] and full of widely open conjectures. Our main references for Section 2.1.1 and Section 2.1.2 are [\ref{Delellisnote}, \ref{DLSQ}].

Finally, in Section 2.2, we will present the main results in the regularity theory for optimal transport paths, pointing out that one of the main ingredients in the present theory, which is mostly due to Xia [\ref{Xiaregularity}], can be reread as a consequence of the recent \textit{stability} result, see Theorem \ref{t:stabilityCPAM} or [\ref{CDRMcpam}].

\section{Regularity theory for area-minimizing currents}

Let us assume that $T$ is an area-minimizing integral $m$-current in $\mathbb{R}^d$. Let $n:=d-m$ be the codimension of $T$. Let $\mathbf{B}_r(p) \subset \R^{n+m}$ denote the ball of radius $r$ centered at $p$. We say that $p \in \operatorname{supp}(T) \setminus \operatorname{supp}(\partial T)$ is an \textit{interior regular point} if there is a positive radius $r>0$, a smooth embedded submanifold $\Sigma \subset \R^{n+m}$ and a positive integer $Q$ such that $T \res \mathbf{B}_r(p)=Q \llbracket \Sigma \rrbracket$. The set of interior regular points, which is relatively open in $\operatorname{supp}(T) \setminus \operatorname{supp}(\partial T)$, is denoted by $\operatorname{Reg}(T)$. Its complement $\operatorname{supp}(T) \setminus(\operatorname{supp}(\partial T) \cup \operatorname{Reg}(T))$ is denoted by $\operatorname{Sing}(T)$ and is called the \textit{interior singular set} of $T$. 

The problem is to understand whether $\operatorname{Sing}(T)$ is empty and, in case it is not, to estimate how ``large" it can be in terms of reasonable notions of dimension, like the Hausdorff dimension. Surprisingly, the answer strongly depends on the codimension $n$. 
Indeed, if $n=1$, $\operatorname{Sing}(T)$ of an area-minimizing integral current $T$ has Hausdorff dimension at most $m-7$ or, in other words, area-minimizing integral currents are smooth (and also real analytic) submanifolds exept for a closed set of Hausdorff dimension at most $m-7$. On the other hand, if $n\ge 2$, solutions to the generalized Plateau's problem exhibit singularities already in dimension $m=2$. Indeed, in higher codimension, $\operatorname{Sing}(T)$ of an area-minimizing integral $m$-current $T$ has Hausdorff dimension at most $m-2$. More formally, we can state the following theorems.

\begin{theorem}[Regularity in codimension $n=1$] \label{t:regularity1}
Let $\Omega \subset \mathbb{R}^{m+1}$ an open set and let $T$ an area-minimizing $m$-dimensional integer rectifiable current in $\Omega$. Then
\begin{itemize}
\item[\textit{i)}] for $m \le 6$, \emph{Sing($T$)} $\cap$ $\Omega$ is empty (see \emph{[\ref{Degiorgifrontiere}, \ref{Degiorgibernstein}, \ref{Flemingcod1}]} for $m=2$, \emph{[\ref{Almgrencod1}]} for $m=3$, \emph{[\ref{Simonscod1}]} for $4 \le m \le 6$),
\item[\textit{ii)}] for $m=7$, \emph{Sing($T$)} $\cap$ $\Omega$ consists of isolated points (see \emph{[\ref{Federercod1}]}),
\item[\textit{iii)}] for $m \ge 8$, \emph{Sing($T$)} $\cap$ $\Omega$ has Hausdorff dimension not larger than $m-7$ (see \emph{[\ref{Federercod1}]}), it is $(m-7)$-rectifiable and of locally finite $\mathcal{H}^{m-7}$-measure (see \emph{[\ref{Simoncod1}, \ref{Nabervaltorta}]}),
\item[\textit{iv)}] the above results are optimal: for every $m \ge 7$ there are area-minimizing integral currents $T$ in $\mathbb{R}^{m+1}$ for which \emph{Sing($T$)} has positive $\mathcal{H}^{m-7}$-measure (see \emph{[\ref{BDGcodim1}]}).
\end{itemize}
\end{theorem}

\begin{theorem}[Regularity in codimension $n\ge2$] \label{t:regularityh}
Let $\Omega \subset \mathbb{R}^{m+n}$ an open set and let $T$ an area-minimizing $m$-dimensional integer rectifiable current in $\Omega$. Suppose that $n \ge 2$, then
\begin{itemize}
\item[\textit{i)}] for $m =1 $, \emph{Sing($T$)} $\cap$ $\Omega$ is empty,
\item[\textit{ii)}] for $m=2$, \emph{Sing($T$)} $\cap$ $\Omega$ consists of isolated points (see \emph{[\ref{Chang}, \ref{DLSS1}, \ref{DLSS2}, \ref{DLSS3}]}),
\item[\textit{iii)}] for $m \ge 3$, \emph{Sing($T$)} $\cap$ $\Omega$ has Hausdorff dimension not larger than $m-2$ (see \emph{[\ref{Big}, \ref{Almgren2000}, \ref{DLSQ}, \ref{DLSsns}, \ref{DLS1}, \ref{DLS2}, \ref{DLS3}]}),
\item[\textit{iv)}] the above results are optimal: for every $m \ge 2$ there are area-minimizing integral currents $T$ in $\mathbb{R}^{m+n}$ for which \emph{Sing($T$)} has positive $\mathcal{H}^{m-2}$-measure (see \emph{[\ref{Federer65}]}).
\end{itemize}
\end{theorem}

Not only are the theorems different in terms of results, but also they differ in terms of level of difficulty and techniques involved in the proofs. Even though regularity theory in codimension 1 can be considered as a delicate topic, it is by now well-studied and fairly understood both in terms of results and employed techniques. On the other hand, the presence in higher codimension of what are known as \textit{branching singularities} required the development of a full new theory and more sophisticated machinery that was initially contained in a 1728-page long typewritten monograph by Frederick J. Almgren [\ref{Big}]. As mentioned, Theorems \ref{t:regularity1} and \ref{t:regularityh} are sharp as shown by the following two celebrated examples: if $n=1$ cosider what is called \textit{Simons' cone} in $\mathbb{R}^8$:
$$
S=\left\{x \in \mathbb{R}^{8}: x_{1}^{2}+x_{2}^{2}+x_{3}^{2}+x_{4}^{2}=x_{5}^{2}+x_{6}^{2}+x_{7}^{2}+x_{8}^{2}\right\},
$$ which clearly has a singularity in the origin. Much more complicated was instead to show that $S$ was (locally) area-minimizing in $\mathbb{R}^8$, which was proved by Bombieri, De Giorgi and Giusti, see [\ref{BDGcodim1}].

If $n\ge2$, then it is easier to show that Theorem \ref{t:regularityh} is optimal. Consider the following holomorphic curve $$\Gamma = \{(z,w) \in \mathbb{C}^2 \simeq \mathbb{R}^{4} : z^2=w^3\}.$$
One can see that the origin is a singular point for $\Gamma$. Moreover, once the so-called \textit{Wirtinger's inequality}, see [\ref{Federerbook}, 1.8.2] has been established, it is fairly easy to prove that $\Gamma$ is area-minimizing. To do so, we introduce some basic facts about complex geometry.

Recall that \textit{holomorphic subvarieties} of $\mathbb{C}^{n}$, namely zeros of holomorphic maps $u: \mathbb{C}^{n} \rightarrow \mathbb{C}^{n-k}$ (with $k$ and $n-k$ the complex dimension and codimension of the variety respectively) can be given a natural orientation. We identify $\mathbb{C}^{n}$ with $\mathbb{R}^{2n}$ as usual: if $z_1, \ldots, z_{n}$ are complex coordinates and $x_j=\operatorname{Re} z_j$, $y_j=\operatorname{Im} z_j$, we let $x_1, y_1, \ldots, x_{n}, y_{n}$ be the standard coordinates of $\mathbb{R}^{2n}$. We remind that a holomorphic subvariety $\Gamma$ of $\mathbb{C}^{n}$ of complex dimension $k$ is an oriented real analytic submanifold of $\mathbb{R}^{2n} \setminus \operatorname{Sing}(\Gamma)$ of (real) dimension $m=2 k$, where $\operatorname{Sing}(\Gamma)$ is a holomorphic subvariety of complex dimension $k-1$. At each point $p \in \Gamma \setminus \operatorname{Sing}(\Gamma)$, the (real) tangent $2k$-dimensional plane $T_p \mathit{\Gamma}$ can be identified with a complex $k$-dimensional plane of $\mathbb{C}^n$. If $v_1, \ldots, v_k$ is a complex basis of $T_p \mathit{\Gamma}$, we can then define a canonical orientation for $T_p \mathit{\Gamma}$ using the simple $2k$-vector $\operatorname{Re} v_1 \wedge \operatorname{Im} v_1 \wedge \ldots \wedge \operatorname{Re} v_k \wedge \operatorname{Im} v_k$. From this, we can now define the current $\llbracket \Gamma \rrbracket$ by integrating forms over the oriented submanifold $\Gamma \setminus \operatorname{Sing}(\Gamma)$. We refer to [\ref{Joyce}] for an extended discussion about K\"alher manifolds and complex geometry.

The discussion can be localized to holomorphic subvarieties in open subsets $\Omega$ of $\mathbb{C}^{n}$ and note that, if $\Omega^{\prime}$ is a bounded open subset of the domain $\Omega$ where $\Gamma$ is defined, then $\llbracket \Gamma \rrbracket$ has finite mass in $\Omega^{\prime}$ and it is thus an integer rectifiable current.

\begin{proposition}[Wirtinger's inequality]\label{p:wirtinger}
Let $\Sigma$ be a K\"alher manifold\footnote{Recall a K\"ahler manifold is a symplectic manifold equipped with a compatible integrable almost-complex structure.} with K\"alher form $\omega$ and denote $v$ a unitary simple $2k$-vector. Then 
$$
\langle \underbrace{\omega \wedge \cdots \wedge \omega}_{k \text { times }}, v \rangle \leq k !
$$
\end{proposition}

\begin{remark}
In other words, Wirtinger's inequality tells that the $k$-th exterior power of the K\"alher form $\omega$ is bounded above by $k!$, when evaluated on unitary simple $2k$-vectors. Equality holds if and only if the $2k$-plane spanned by $v$ is a complex $k$-plane.
\end{remark}

\begin{theorem}\label{t:holomorphicsubvariety} Any compact portion of a holomorphic subvariety in $ \mathbb{C}^n$ is area-minimizing.
\end{theorem}

\begin{proof}
On $\mathbb{C}^{n} \simeq \mathbb{R}^{2 n}$ consider the K\"ahler form
$$
\omega:=d x_{1} \wedge d y_{1}+\cdots+d x_{n} \wedge d y_{n}.
$$

Note that $d \omega= 0$ and, by Wirtinger's inequality \ref{p:wirtinger}, the $2k$-form $$\varphi:=\frac{1}{k !}\, \omega^{k}$$ satisfies $\|\varphi(v)\|_{\infty} \leq 1$ for all $2k$-planes\footnote{We say that $\omega$ is a \textit{calibration}, that is a closed $m$-form such that $\|\omega\|_{\infty} \leq 1$.} identified by $v$, with equality if and only if the $2k$-plane is a complex $k$-plane. Let $\Sigma$ be a compact portion of a $k$-dimensional holomorphic variety and let $\Gamma$ be any other submanifold such that $\partial \Sigma =  \partial \Gamma$. By the Poincaré lemma in $\mathbb{R}^{2n}$ we have that $\omega$ is exact. Indeed we can even write explicitly $$\omega=  d (x_1dy_1 \wedge ... \wedge x_n dy_n) = : d\eta .$$

By applying twice Stokes's theorem (Theorem \ref{t:stokes}) we get
\begin{equation}
\mathbb{M}\big(\llbracket\Sigma\rrbracket\big)=\int_\Sigma 1 =\int_\Sigma \omega =  \int_{ \partial \Sigma}\eta=  \int_{\partial \Gamma} \eta = \int_{\Gamma}\omega  \leq \mathbb{M}\big(\llbracket\Gamma\rrbracket\big).
\end{equation}
\end{proof}

\begin{remark}
As a corollary we get that the holomorphic curve $$\Gamma = \{(z,w) \in \mathbb{C}^2 \simeq \mathbb{R}^{4} : z^2=w^3\}$$ is a (locally) area-minimizing current of  dimension $2$ in $\R^4$, showing that Theorem \ref{t:regularityh} is sharp. We will comment more later on the ``branch point" singularity that appears at the origin. We remark that it has been rather simple to provide examples of area-minimizing integral currents with singularities in higher codimension: this should be the first warning about the degree of difficulty of developing a full regularity theory in that setting.
\end{remark}

Before turning to the analysis of codimension 1 theory, we recall a fundamental formula that plays a crucial role in (almost) any regularity theory for ``weak objects": the so-called \textit{monotonicity formula}.

Let $T$ be an area-minimizing $m$-current in $\mathbb{R}^{m+n}$ and suppose $T$ represents (by integration) a smooth surface $\Sigma$ (that is $T=\llbracket \Sigma \rrbracket)$. Fix a point $p \in \Sigma \setminus \partial \Sigma$ and a radius $r<\operatorname{dist}(p, \partial \Sigma)$. Assume further that $\partial \mathbf{B}_r(p)$ intersects $\Sigma$ transversally. If we replace $\Sigma$ in the ball $\mathbf{B}_r(p)$ with the cone having vertex $p$ and boundary $\Sigma \cap \partial \mathbf{B}_r(p)$ we increase the volume of $\Sigma$. More formally: $$
\mathcal{H}^m(\Sigma \cap \mathbf{B}_r(p)) \leq \frac{r}{m} \mathcal{H}^{m-1}(\Sigma \cap \partial \mathbf{B}_r(p)).
$$
On the other hand coarea formula (Theorem \ref{t:coarea}) implies that
$$
\mathcal{H}^{m-1}( \Sigma \cap \partial \mathbf{B}_r(p)) \leq \frac{d}{d t} \Big|_{t=r} \mathcal{H}^m(\mathbf{B}_t(p) \cap \Sigma)
$$
and solving the resulting differential inequality we can conclude that
$$
\frac{d}{d r} \frac{\mathcal{H}^{m}(\Sigma \cap \mathbf{B}_r(p))}{r^m} \geq 0 .
$$
A more careful computation, exploiting the stationarity of $\Sigma$, provides the following much more precise formula:
\begin{equation}\label{e:monotonicity}
\frac{\mathcal{H}^{m}(\Sigma \cap \mathbf{B}_r(p))}{r^m}-\frac{\mathcal{H}^{m}(\Sigma \cap \mathbf{B}_s(p))}{s^m}=\int_{\Sigma \cap\left(\mathrm{B}_r(p) \setminus \mathrm{B}_s(p)\right)} \frac{\left|(x-p)^{\perp}\right|^2}{|x-p|^{m+2}} d \mathcal{H}^m(x),
\end{equation}
for $0<s<r$, where $(x-p)^{\perp}$ denotes the component of the vector $(x-p)$ which is orthogonal to the tangent space $T_x \Sigma$. 
The formula in \eqref{e:monotonicity} is still valid for area-minimizing integer rectifiable currents as well. Indeed, note that the right-hand side of \eqref{e:monotonicity} is well-defined for $T$ (replacing $\mathcal{H}^m$ with $\|T\|$) since at $\|T\|$-a.e. $x$ we have a well-defined tangent plane, allowing us to define $(x-p)^{\perp}$ for $\|T\|$-a.e. $x$. 

Recall the following Proposition:
\begin{proposition}
If $T$ is an integer rectifiable current, then the number
$$
\Theta(T, p):=\lim _{r \rightarrow 0} \frac{\|T\|\left(\mathbf{B}_r(p)\right)}{\omega_m r^m}
$$
exists and it is a positive integer for $\|T\|$-a.e. point $p$. 
\end{proposition} 

Given a current $T$ we will denote by $T_{p, r}$ the \textit{blow-up} of $T$, that is the result of translating $T$ so that $p$ becomes the origin and enlarging it of a factor $r^{-1}$. More formally:

\begin{definition}\label{d:miserveblowup}
Let $\iota_{p, r}$ denote the map $ x \mapsto (x-p) / r$, then we define $T_{p, r}:=\left(\iota_{p, r}\right)_{*} T$. Note that when $T=\llbracket \Sigma \rrbracket$ for some smooth surface $\Sigma$ then $T_{p, r}=\llbracket \iota_{p, r}(\Sigma) \rrbracket$.
\end{definition}

We recall now a very important definition.
\begin{definition}\label{d:cone}
An area minimizing \textit{cone} of dimension $m$ is an integer rectifiable $m$-current $S$ such that $\partial S=0$ and $S_{0, r}=S$ for every positive $r$ and $S\res \Omega$ is area-minimizing for any bounded open set $\Omega$.
In addition, if $T$ and $S$ are two currents such that, for some $p \in \operatorname{supp}(T)$ and some $r_k \rightarrow 0$, $T_{p, r_k}$ converges to $S$, we say that $S$ is \textit{tangent} to $T$ at $p$.
\end{definition}

The monotonicity formula has several important consequences that we recall here.
\begin{proposition}\label{p:subseqtgcone}
Let $T \in \mathcal{I}_m(\R^{m+n})$ be area-minimizing. Then \begin{itemize}
\item[$i)$] The density $\Theta(T, p)$ is well-defined at every $p \notin \operatorname{supp}(\partial T)$, it is at least 1 at each point $p \in \operatorname{supp}(T) \setminus \operatorname{supp}(\partial T)$ and it is upper semicontinuous.
\item[$ii)$] For every $p \notin \operatorname{supp}(\partial T)$ and every sequence $r_k \rightarrow 0$ there is a subsequence, not relabeled, and an area-minimizing cone $T_0$ such that $T_{p, r_k} \rightarrow T_0.$ Moreover, $T_0 \neq 0$ if and only if $p \in \operatorname{supp}(T)$.
\end{itemize}
\end{proposition}

Note that if $T \in \mathcal{R}_m(\R^{m+n})$ then, by Remark \ref{r:unique}, $\|T\|$-a.e. $p \in \operatorname{supp}(T) \setminus \operatorname{supp}(\partial T)$ there is a \textit{unique} tangent cone, which is an integer multiple (where the multiple is $\Theta(T, p)$) of an $m$-dimensional plane $\pi(p)$. This motivates the following definition.

\begin{definition}
A tangent cone $S$ is called \textit{flat} if it is a (nonzero integer) multiple of an $m$-dimensional plane.
\end{definition}

Note that at every $p \in \operatorname{Reg}(T)$ there is a unique tangent cone and it is flat so that, equivalently, if there is a single tangent cone at $p$ which is not flat, then $p \in \operatorname{Sing}(T)$. Hence, one should be tempted to guess that a possible characterization of a regular point is a point where at least one tangent cone is flat: this is true in the very particular case of codimension 1, but not necessarily in higher codimension as the holomorphic curve $$\Gamma=\left\{(z, w) \in \mathbb{C}^2: z^2=w^3\right\}$$ already shows. 

We conclude this part by mentioning what can be considered the most challenging open problem in the regularity theory of generalized area-minimizing surfaces.

\begin{open problem}\label{op:utgcone} It is not known whether the tangent cone to an area-minimizing current $T$ is unique at every point $p \in \operatorname{supp}(T) \setminus \operatorname{supp}(\partial T)$.\end{open problem}

\begin{remark}
Proposition \ref{p:subseqtgcone} does not imply equality for the subsequential limits. Some partial results are known, such as for $2$-dimensional currents in any codimension by White [\ref{Whitetgcone}] and in codimension 1 at any isolated singularity by Simon [\ref{Simontgcone}]. More refined analyses of the structure of the singular set have been developed relying on the uniqueness of tangent cones for $2$-dimensional currents, as an example see [\ref{DLSS1}, \ref{DLSS2}, \ref{DLSS3}]. The remaining cases are widely open.
\end{remark}

\subsection{Area-minimizing currents in codimension one}

In this section we aim to highlight the main ideas behind the proof of what is known in the literature as the De Giorgi-Allard $\varepsilon$\textit{-regularity theorem}. Indeed, the first breakthrough in regularity theory for generalized area-minimizing surfaces is due to De Giorgi, see [\ref{Degiorgifrontiere}], by means of \textit{finite perimeter sets} (or \textit{Caccioppoli sets}); in his pioneering work he realizes that the existence of a flat tangent plane at a point $p$ is enough to conclude that $p$ is a regular point \textit{in codimension} $1$. In the framework of generalized area-minimizing surfaces, the most important generalization of De Giorgi's $\varepsilon$-regularity theorem is due to Allard in [\ref{AAvarifolds}] using the concept of integer rectifiable varifolds with sufficiently integrable ``generalized mean curavature". 

We will start stating the De Giorgi-Allard $\varepsilon$-regularity theorem in all dimensions and codimensions, emphasizing why it is valid in codimension 1 only. Then, for the sake of the exposition, we will prove a simplified version of the De Giorgi-Allard $\varepsilon$-regularity theorem in the language of Federer and Fleming's theory of currents. We will not prove it in full generality, but we will start from the very strong assumption that the (support of the) area-minimizing currents is already a graph of a Lipschitz function; nevertheless, we consider more useful to avoid the technical part of the approximation of area-minimizing currents by means of Lipschitz graphs with the scope of highlighting the main ideas of an $\varepsilon$-regularity theorem such as the \textit{excess decay} and the \textit{harmonic approximation}.

\begin{remark}
When we say $\varepsilon$\textit{-regularity theorem} what we mean is a statemant the following type: \begin{center}{\it``If some suitable quantity, usually an integral quantity, is sufficiently small at a given scale, then we can conclude that our object is regular at a smaller scale".}\end{center} Hence, when we have an $\varepsilon$-regularity theorem the main question to ask is under which assumptions this suitable integral quantity is small enough.
\end{remark}

The main quantity which turns out to be fundamental in the codimension one theory is the \textit{excess} of the current $T$, which we now define. We introduce the following notations: let $(p + \pi)$ be the affine plane passing through $p$, $\pi^{\perp}$ the orthogonal complement of $\pi$ and we denote $B_r(p, \pi):= \mathbf{B}_r(p) \cap (p + \pi)$.
\begin{definition}
Given a smooth $m$-dimensional submanifold $\Sigma \subset \R^{n+m}$ and $\pi$ an $m$-plane, we call the \textit{excess} of the submanifold $\Sigma$ at point $p$ in the ball $\mathbf{B}_r(p)$ of radius $r$ with respect to $\pi$ the quantity
\begin{equation}
   \mathbf{E}(\Sigma,p,r,\pi):=  \frac{1}{r^m}\int_{\Sigma\cap \mathbf{B}_r(p)}\|T_x\Sigma - \pi\|^2 d\mathcal{H}^m(x),
\end{equation}
where we identify an $m$-plane $\pi$ with the orthogonal projection onto it, which is a linear map $\pi:\R^{n+m} \rightarrow \R^{n+m}$ and we use the Hilbert-Schmidt norm for linear maps\footnote{It is important to choose the Hilbert-Schmidt norm since, under particular assumptions that will be made clear in the sequel, it will Taylor expand to a Dirichlet energy term that will be crucial in the rest of the proof.}. The \textit{excess} of the submanifold $\Sigma$ at point $p$ in the ball $\mathbf{B}_r(p)$ of radius $r$ is 
\begin{equation}
   \mathbf{E}(\Sigma,p,r):= \min\big\{\mathbf{E}(\Sigma,p,r,\pi) \,|\, \pi \, \text{is an oriented $m$-plane} \big\}.
\end{equation}
\end{definition}

\begin{remark}
Note that $\mathbf{E}$ is scale-invariant, \textit{i.e.} invariant under dilations, thanks to the presence of the factor $r^{-m}$. Informally, $\mathbf{E}$ is an integral measure of the oscillation of the tangent plane to the submanifold.
\end{remark}

Provided we use the notion of weak tangent space, see Definition \ref{approximatetg}, we have a well-defined notion of excess for integer rectifiable $m$-currents.

\begin{definition}
The \textit{excess} of an integer rectifiable $m$-current $S$ at $p$ on the ball $\mathbf{B}_r(p)$ of radius $r$ with respect to $\pi$ is:

\begin{equation}
   \mathbf{E}(S,p,r,\pi):=  \frac{1}{r^m}\int_{\mathbf{B}_r(p)}\|T_xS - \pi\|^2 d\|S\|(x).
\end{equation}
The \textit{excess} of $S$ at point $p$ in the ball $\mathbf{B}_r(p)$ of radius $r$ is 
\begin{equation}
   \mathbf{E}(S,p,r):= \min\big\{\mathbf{E}(S,p,r,\pi) \,|\, \pi \, \text{is an oriented $m$-plane} \big\}.
\end{equation}
\end{definition}

\begin{remark}
For notation purposes, we denote by $\operatorname{gr}(f)$ the graph of a function $f$ and if $f$ is Lipschitz continuous we denote by $\text{Lip}(f)$ its Lipschitz constant. When we say that the current $T$ \textit{is} a graph of a function, what it formally means is that $T = \llbracket\operatorname{gr}(f) \rrbracket$. We also remark that, in general, $\mathbf{E}(T,p,r)$ may be achieved on a $m$-plane $\pi_0$ so that the function $f$ that parametrizes the submanifold as a graph is defined as $ f: B_r(p, \pi_0)  \rightarrow \pi_0^{\perp}$. Nevertheless, as we are going to remark later, since $\text{Lip}(f)$ is small, it is without loss of generality to assume $\pi_0$ is the (oriented) horizontal plane $\R^m \times \{0\}$ so that we can write $ f: B_r(p,\R^m) \rightarrow \R^n$. We will simply write $B_r(p)$ for $B_r(p,\R^m)$.
\end{remark}
To highlight the importance of the codimension $n$, we state now the following ``theorem".

\begin{theorem}[False De Giorgi-Allard]\label{t:falseda}
There exists $\varepsilon_0>0$ such that if $\Ex(T,p,r)<\varepsilon_0$ and if $T \in \mathcal{R}_m(\R^{n+m})$ is area-minimizing, then $T$ is a single $\mathcal{C}^{1,\alpha}$-submanifold in $\mathbf{B}_{r/2}(p)$.
\end{theorem}

\begin{remark}
The above ``theorem" is, in general, false: a counterexample is the holomorphic curve $$\Gamma = \{(z,w) \in \mathbb{C}^2 : z^2=w^3\}.$$ In Section 2.1.2 we will show that $\Ex(\llbracket \Gamma \rrbracket,p,r) \rightarrow 0$ for $r \rightarrow 0$ but it is not possible to write it as a graph of a (single-valued) function. Theorem \ref{t:falseda} turns out to be true only if we assume $n=1$.
\end{remark}

\begin{remark}
Once we have established that in $\mathbf{B}_r(p)$ the current $T$ is a single $\mathcal{C}^{1,\alpha}$-submanifold, then we can write the associated partial differential system of equations for the function $f$ that parametrizes the submanifold as a graph $f: B_{r/2}(p) \rightarrow \R^n$ and by classical regularity theory for second order elliptic systems of differential equations we conclude that the submanifold is smooth (in fact analytic, using the classical result by Morrey, see [\ref{Morreysystem}]).
\end{remark}

\subsubsection{Baby version of De Giorgi-Allard's Theorem}

\begin{theorem}[Baby De Giorgi-Allard]\label{t:babyda}
Let $T \in \mathcal{R}_m(\R^{n+m})$ area-minimizing. Assume $T$ is a graph of a Lipschitz function $f:B_{r}(p,\pi_0)  \rightarrow \pi_0^\perp$ and assume $Lip(f)$ is sufficiently small in its domain, then $f\in \mathcal{C}^{1,\alpha}(B_{r/2}(p,\pi_0), \R^n)$.
\end{theorem}

\begin{remark}
As always in the theory of partial differential equations, the proof of ${C}^{1,\alpha}$-regularity comes with some estimates. We will see that, in the end, $\|f\|_{{C}^{1,\alpha}}$ will be estimated in terms of the excess, showing also from this baby version that there is hope to relax the two very strong assumptions of being parametrized by a Lipschitz graph with small Lipschitz constant\footnote{The assumption of a sufficiently small Lipschitz constant is necessary: in general, there exist Lipschitz solutions to the minimal surface system which are not of class $\mathcal{C}^1$, see [\ref{LawsonOsserman}].}.
\end{remark}

The core of the proof of Theorem \ref{t:babyda} is a decay estimate, usually called \textit{excess decay estimate}: under the assumption that the Lipschitz graph has sufficiently small Lipschitz constant we want to show that the excess at every point $p$ in $\mathbf{B}_r(p)$ decays as \begin{equation}\label{e:exdecay}\Ex(T,p,r)\leq C\, r^{2\alpha}\end{equation} where $C$ is a constant depending only on the excess at the largest scale. 

\begin{remark}
To obtain \eqref{e:exdecay} several novel ideas were combined together: the crucial point is to note that a similar decay estimate holds for harmonic function and if the current $T$ satisfies the assumptions of Theorem \ref{t:babyda}, then $T$ is close to the graph of a harmonic function and we can transfer such harmonic decay to $T$.
\end{remark}

To prove Theorem \ref{t:babyda} we now state some preliminary geometric lemmas that will highlight the connections between the excess $\Ex(T,p,r)$ and the regularity of the Lipschitz function $f:B_{r}(p,\pi_0) \rightarrow \pi_0^\perp$.

\begin{lemma}\label{l:comparable}
On a Lipschitz graph $T$, there exist two constants $\overline{C}$ and $\tilde{C}$ such that
\begin{equation}\label{affine0}
\overline{C} \min \fint_{P_{\R^m}}|Df(x)-A|^2 dx \leq \Ex(T,p,r)
\end{equation} and

\begin{equation}\label{affine}
 \Ex(T,p,r) \leq  \tilde{C} \min \fint_{P_{\R^m}}|Df(x)-A|^2 dx 
\end{equation}

where the minimum is taken among all affine functions $A(x):\mathbb{R}^m\rightarrow \mathbb{R}^n$, $P_{\R^m}$ is the projection of $B_r(p,\pi_0)$ onto $\R^m$ and by $\fint$ we mean the average integral.
\end{lemma}

\begin{remark}
The constants $\overline{C}$ and $\tilde{C}$ bounding $\Ex(T,p,r)$ can be shown to converge to $1$ when $\text{Lip}(f) \rightarrow 0$. Moreover, note that the quantity $$\min \fint_{P_{\R^m}}|Df(x)-A|^2 dx $$ is the $L^2$-norm of some function minus a constant: hence we can write
\begin{equation}\label{e:minaverage}
\min \fint_{P_{\R^m}}|Df(x)-A|^2 dx  = \int_{P_{\R^m}}\Big|Df(x)-\fint_{P_{\R^m}}Df\Big|^2.
\end{equation}
\end{remark}

Hence, if we are able to prove that the excess decay \eqref{e:exdecay} holds true, then this would imply that $Df$ decays as well at the same rate. More formally, by \eqref{e:exdecay}, Lemma \ref{l:comparable} and \eqref{e:minaverage} we would get

\begin{equation}\label{e:morrey}
    \fint_{P_{\R^m}}\Big|Df(x)-\fint_{P_{\R^m}} Df\Big|^2\leq C\, r^{2\alpha}
\end{equation} and \eqref{e:morrey} is known as a {\em Morrey-Campanato estimate}. We recall the following result from elliptic regularity theory for partial differential equations.

\begin{theorem}[Morrey-Campanato]\label{t:morreycampanato}
Let $\Omega \subseteq \mathbb{R}^m$ be an open set. If $g\in L^1(\Omega, \R^{mn})$ and there exists a constant $C$ independent of $r$ and $x$ such that $$\fint_{B_r(x)}\Big|\,g-\fint_{B_r(x)}g\,\Big|^2\leq C\, r^{2\alpha} \hspace{0.2cm} \text{ for every } B_r(x)\subset \Omega$$ then $g \in \mathcal{C}^{0,\alpha}(\Omega,\R^{mn})$.
\end{theorem}

\begin{remark}
To be able to apply Theorem \ref{t:morreycampanato} to \eqref{e:morrey} we need a uniform control with respect to $r$ and $x$ in the excess decay proof. Nevertheless, going through the proof carefully, one realizes that the constant $C$ in \eqref{e:morrey} blows up when $x$ approaches the boundary of the domain $\partial \mathbf{B}_r(p)$. Therefore, if we restric our analysis to a smaller ball  $B_{r/2}(x, \pi_0) \subset B_r(x, \pi_0) $ then we obtain a uniform control of $C$: this is why we conclude $f\in \mathcal{C}^{1,\alpha}$ in $B_{r/2}(p, \pi_0)$.
\end{remark}

We state formally the excess decay we want to prove.

\begin{proposition}[Excess decay]\label{p:excessdecay}
There exist $\eta, \Bar{\varepsilon} <1$ such that if $\text{Lip}(f) <\Bar{\varepsilon}$, then 
\begin{equation}\label{e:edecayreal}
\Ex\left(T,p,\frac{r}{2}\right)\leq \eta \, \Ex(T,p,r).\end{equation}
\end{proposition}

\begin{remark}
Note that from Theorem \ref{p:excessdecay} and the discussion above, we can easily get \eqref{e:morrey} by iteration: apply \eqref{e:edecayreal} $k$-times\footnote{The monotonicity formula plays a fundamental role in this step.} to get
$$\Ex\left(T,p,\frac{r}{2^k}\right)\leq \eta^k \Ex(T,p,r)=2^{-k(-\log_2\eta)}\Ex(T,p,r).$$
Denoting $2\alpha:=-\log_2\eta$ we conclude getting
$$\Ex\left(T,p,\frac{r}{2^k}\right)\leq (2^{-k})^{2\alpha} \Ex(T,p,r)$$ where $\rho:=r/2^k \text{ and } (2^{-k})^{2\alpha}  = \left(\rho/r\right)^{2\alpha}$.
\end{remark}

Hence if we prove Proposition \ref{p:excessdecay} we conclude the proof of Theorem \ref{t:babyda}. 
In order to prove Proposition \ref{p:excessdecay} we will divide the proof into three main steps that we summarize here: \begin{enumerate}
\item By a Taylor expansion of the Hilbert-Schmidt norm we will state some useful estimates of the excess. Exploiting the assumption that $\text{Lip}(f)$ is small enough, we will then rephrase the problem in terms of a \textit{decay estimate} that, once proved, will imply Proposition \ref{p:excessdecay}.
\item We are going to prove two lemmas that will show De Giorgi's fundamental idea: when $Df$ is very small, the area functional is well-approximated by the Dirichlet energy, telling us that minimizers of the area functional are, in some sense, very close to harmonic functions.
\item Finally we will prove (by means of spherical harmonics expansion) that harmonic functions satisfy a stronger \textit{decay estimate} and we will transfer such decay to the Lipschitz minimizer of the area functional $f$.
\end{enumerate}

\begin{remark}\label{r:wlog}
Recall that the excess was defined as $$\Ex(\text{gr}(f),p,r) = \min_{\pi} \frac{1}{r^m}\int_{\text{gr}(f) \cap \mathbf{B}_r(p)}\|T_x\text{gr}(f) - \pi\|^2 \,d\mathcal{H}^m(x).$$

Note that, up to rotating the coordinate system and increasing Lip$(f)$ by a controlled small amount, we can assume without loss of generality that the plane $\pi_0$ which minimizes the excess is the (oriented) horizontal plane $\R^m \times \{0\}$; indeed, under the condition that $T$ is a graph of a Lipschitz function with small $\text{Lip}(f)$, then the ``tilt" of $\pi_0$ to the horizontal plane is controlled.

Moreover, without loss of generality we can assume $r=1$. Indeed, we can apply a homothety mapping $x \mapsto x/r$, that is sending $\mathbf{B}_r(p)$ into $\mathbf{B}_1(p)$: the homothetic submanifold is still area-minimizing among the class of new competitors and the excess is scale-invariant under homotheties.
\end{remark}

By Remark \ref{r:wlog} we can write $$\|T_x\text{gr}(f) - \pi_0\|^2= \|T_x\text{gr}(f) - \R^{m}\times \{0\}\|^2.$$ Now note that since we used the Hilbert-Schmidt norm for linear functions, by a simple Taylor expansion we get \begin{equation}\label{e:hsnormtaylor}
\begin{aligned}
\|T_x\text{gr}(f) - \R^{m}\times \{0\}\|^2 &= |Df|^2 + \mathcal{O}(|Df|^4) \\
&= |Df|^2 + \mathcal{O}\big((\text{Lip}(f))^2|Df|^2\big)\\ 
&\sim  |Df|^2 (1 + C\,(\text{Lip}(f))^2).
\end{aligned}
\end{equation} Hence, denoting $\overline{\varepsilon}:=\text{Lip}(f)$, by \eqref{e:hsnormtaylor} we can write

\begin{equation}\label{e:stime1}
\frac{1}{2}\int_{B_{(1-c\sqrt{\Bar{\varepsilon}})}}|Df|^2\leq \Ex\big(\text{gr}(f),p,1\big)\, [1+C\Bar{\varepsilon}^2].
\end{equation}
where $C$ is a geometric constant, depending only on the dimension. In a similar way\footnote{That is, by Taylor expansion and by observing that the affine plane minimizing $\Ex\big(\text{gr}(f),p,1/2\big)$ is not far from the average.}, the following inequality can be proved:
\begin{equation}\label{e:stime2}
\Ex\big(\text{gr}(f),p,1/2\big)\leq\frac{1}{2}\Big( \int_{B_{1/2}}\Big|Df-\fint_{B_{1/2}}Df\Big|^2\Big) \, [1+C\Bar{\varepsilon}^2].
\end{equation}

\begin{remark}
In the left-hand side of \eqref{e:stime1} we integrate on the ball of radius $1-c\sqrt{\Bar{\varepsilon}}$. Roughly speaking, this comes from a fine comparison between the intersection of $\text{gr}(f)$ with a cylinder and with the sphere. Since $f$ is not in general a constant function, we need to control by a factor that takes into account the fact that $f$ may change, but this change is controlled by the Lipschitz constant $\Bar{\varepsilon}$. On the other hand, we do not write $1-c\sqrt{\Bar{\varepsilon}}$ on the right-hand side of \eqref{e:stime2}, since in this case we look at the excess in the $1/2$-ball which is fully inside the cylinder of radius $1/2$.
\end{remark}

Summarizing, with $\pi_0 = \R^{m} \times \{0\}$, $r=1$ and for $\eta < 1$ then \eqref{e:edecayreal} becomes:
\begin{equation}\label{e:stime3}
\Ex\left(T,p,\frac{1}{2}\right)\leq \eta \, \Ex(T,p,1).\end{equation} Putting together \eqref{e:stime3} with \eqref{e:stime1} and \eqref{e:stime2} we realize that if we are able to prove that, for $\Tilde{C}<1$, the following is true
\begin{equation}\label{e:cbarra}
\Bigg(\frac{1}{2}\int_{B_{1/2}}|Df-\fint_{B_{1/2}}Df|^2\Bigg)\leq \Tilde{C}\Bigg(\frac{1}{2}\int_{B_{(1-C\sqrt{\Bar{\varepsilon}})}}|Df|^2\Bigg),
\end{equation} then we prove \eqref{e:edecayreal} as well since $\overline{\varepsilon}$ can be chosen sufficiently small.

We can state the following Lemma, that will allow us to conclude the proof of Theorem \ref{t:babyda}.

\begin{lemma}\label{l:lastlemma}
For every $\sigma>0$ there exists $\varepsilon>0$ such that if $\Bar{\varepsilon}<\varepsilon$ then
\begin{equation}
    \int_{B_{1/2}}\Big| Df-\fint Df\Big|^2\leq \bigg(\frac{1}{4}+\sigma \bigg)\int_{B_{(1-c\sqrt{\Bar{\varepsilon}})}}|Df|^2.
\end{equation}
\end{lemma}

\begin{remark}
Note that Lemma \ref{l:lastlemma} remarks that the constant $\Tilde{C}<1$ in \eqref{e:cbarra} can be achieved to be as close as we want to $1/4$. This is important since we are going to prove a similar stronger decay for harmonic functions.
\end{remark}

\begin{proof}

We argue by contradiction: we want to show that there exists $\Bar{\sigma}>0$ such that there exists a sequence of Lipschitz functions $(f_k)_k$ with
\begin{itemize}
    \item[$i)$] $f_k(0)=0$ for all $k$ (up to a translation);
    \item[$ii)$] $\text{Lip}(f_k)=:\Bar{\varepsilon}_k\rightarrow 0$ for $k\rightarrow \infty$;
    \item[$iii)$] $\text{gr}(f_k)$ is an area-minimizing\footnote{One will realize that \textit{stationary} would have been enough.} submanifold for all $k$ such that\end{itemize}
  \begin{equation}\label{contra}
        \int_{B_{1/2}}\Big|Df_k-\fint Df\Big|^2 \geq \bigg(\frac{1}{4}+\Bar{\sigma}\bigg)\int_{B_{(1-C\sqrt{\Bar{\varepsilon}_k})}}|Df_k|^2.
    \end{equation}

Now rescale $(f_k)_k$ in such a way each Dirichlet energy is unitary. To do so, define

$$g_k:=\frac{f_k}{\Big(\int_{B_{(1-C\sqrt{\Bar{\varepsilon}_k})}}|Df_k|^2\Big)^{1/2}}$$ and notice $g_k$ has a uniform bound in $W^{1,2}(B_1, \R^n)$. Hence, by Rellich's theorem \begin{equation}\label{e:rellichtog}
g_k \rightarrow g \,\text{ in } \,L^2(B_1, \R^n), \text{ for some }  g\in W^{1,2}(B_1, \R^n).
\end{equation}

We now state two key steps in the proof.
\begin{claim}\label{claim1}
$g$ is harmonic.
\end{claim}

\begin{claim}\label{claim2}
$g_k$ converges strongly in $W^{1,2}(B, \R^n)$ for every ball $B \subset\joinrel\subset B_1$. 
\end{claim}

As an immediate corollary of Claim \ref{claim2}, we have $$ \int_{B}|Dg_k|^2\rightarrow \int_{B} |Dg|^2,$$ as long as $B$ is compactly contained in $B_1$. Assume Claim \ref{claim1} and Claim \ref{claim2} to be true and note \eqref{contra} holds for $(g_k)_k$ as well since \eqref{contra} is invariant under multiplications of $f_k$ by constants. By Claim \ref{claim2} we can write

  \begin{equation}
        \int_{B_{1/2}}\Big|Dg_k-\fint Dg_k\Big|^2\rightarrow \int_{B_{1/2}}\Big|Dg-\fint Dg\Big|^2.
    \end{equation}
 Moreover we know that, up to a nonrelabeled subsequence,
\begin{equation}\label{e:indicator}
        Dg_k \, \mathbbm{1}_{B_{(1-C\sqrt{\Bar{\varepsilon}_k})}} \rightharpoonup Dg \quad \text{ in } L^2(B_1, \R^{mn}),
    \end{equation} where $\mathbbm{1}$ denotes the indicator function. By lower semicontinuity of the norm with respect to weak convergence we can write
 \begin{equation} \label{16}
    \int_{B_{1/2}}\Big|Dg-\fint Dg\Big|^2\geq \bigg(\frac{1}{4}+\Bar{\sigma}\bigg)\int_{B_1}|Dg|^2.
\end{equation} By Claim \ref{claim1} and the mean value property of harmonic functions we can rewrite \eqref{16} as
\begin{equation} \label{17}
    \int_{B_{1/2}}|Dg-Dg(0)|^2\geq \bigg(\frac{1}{4}+\Bar{\sigma}\bigg)\int_{B_1}|Dg|^2.
\end{equation}

In particular, \eqref{17} shows that the harmonic map $g$ is $\neq 0$, since\footnote{This is an important remark since the harmonic function $g\equiv 0$ would satisfy \eqref{16}, not allowing us to conclude the contradiction argument.} $$\int_{B_{1/2}}|Dg- Dg(0)|^2\ge \bigg(\frac{1}{4}+\Bar{\sigma}\bigg) >0.$$

At this point, we show that harmonic functions admits the following decay estimate:
\begin{claim}\label{claim3}
\begin{equation}
    \int_{B_{1/2}}|Dg-Dg(0)|^2\leq \frac{1}{2^m}\frac{1}{4}\int_{B_1}|Dg|^2.
\end{equation}
\end{claim}

Once Claim \ref{claim3} is proved, we contradict \eqref{17}, concluding the proof of Theorem \ref{t:babyda}. Hence, all we need to prove is Claim \ref{claim1}, Claim \ref{claim2} and Claim \ref{claim3}.

\vspace{0.5cm}
{\it Proof of Claim \ref{claim1}.}

By assumption $f_k$ is a minimizer of area functional for all $k$. Fix a test function $\varphi \in \mathcal{C}_c^\infty(B_1, \R^n)$ so that we can write

\begin{equation}\label{e:test}
    \frac{d}{d\varepsilon}\mathbb{M}(\text{gr}(f_k+\varepsilon \varphi))|_{\varepsilon=0}=0.
\end{equation}
The mass of $\text{gr}(f_k+\varepsilon \varphi)$ is given (in general codimension $n$) by:
\begin{equation}
    \mathbb{M}(\text{gr}(f_k+\varepsilon \varphi))=\int\sqrt{1+|D(f_k+\varepsilon \varphi)|^2+\sum_{\text{k-minors}}(\operatorname{det}(M))^2},
\end{equation} where the sum is over all $k \times k$ minors $M$ of $D(f_k+\varepsilon \varphi)$. Hence we have
\begin{equation}\label{e:taylorlong}
\begin{aligned}
    0=\frac{d}{d\varepsilon}\mathbb{M}(\text{gr}(f_k+\varepsilon \varphi))|_{\varepsilon=0}&= \int\frac{\langle Df_k,D\varphi\rangle+\mathcal{O}(|Df_k|^3\cdot|D\varphi|)}{\sqrt{1+|D(f_k+\varepsilon \varphi)|^2+\sum_{\text{k-minors}}(\operatorname{det}(M))^2}} \\
&\stackrel{(a)}=\int \langle Df_k, D\varphi\rangle+\mathcal{O}(|Df_k|^3\cdot|D\varphi|)\\
&\stackrel{(b)}=\int\langle Df_k, D\varphi\rangle+\mathcal{O}(|Df_k|\cdot |D\varphi|\cdot \hspace{0.04cm} \Bar{\varepsilon}_k^2)\\
&\stackrel{(c)}=\int\langle Dg_k, D\varphi\rangle+\mathcal{O}\left(\int|Dg_k|\cdot \hspace{0.04cm} \Bar{\varepsilon}_k^2 \cdot \|D\varphi\|_{\mathcal{C}^0}\right),
\end{aligned}
\end{equation} where in (a) we Taylor expanded the denominator, in (b) we noted that $|Df_k|$ can be controlled since $|Df_k| \hspace{0.04cm} \leq \Bar{\varepsilon}_k$ and in (c) we divided by a constant passing to $g_k$.

Now, since $k\rightarrow \infty$ we know $g_k \rightharpoonup g$ in $W^{1,2}(B_1, \R^n)$, $\mathcal{O}$ converges to 0\footnote{Since we can easily uniformly control $|Dg_k|$ by H\"older's inequality, by the fact its Dirichlet energy is equal to $1$ and by the fact $\Bar{\varepsilon} \rightarrow 0$ as $k\rightarrow \infty$.} and $D\varphi$ is a fixed test function, then we get
\begin{equation}
    0=\int\langle Dg,D\varphi\rangle \quad \text{ for all }\varphi \in \mathcal{C}_c^\infty(B_1, \R^n).
\end{equation}
We conclude that $g$ is harmonic.

\vspace{0.5cm}
{\it Proof of Claim \ref{claim2}.}

Note that we can test equality \eqref{e:test} with $(\varphi f_k)$ instead of $\varphi$, since $(\varphi f_k)$ is still a (Lipschitz) compactly supported perturbation. Hence, \eqref{e:taylorlong} is analogous and we reach

\begin{equation}\label{e:this}
    0=\int\langle Dg_k,D(\varphi g_k)\rangle + \mathcal{O}\left(\int|Dg_k|\cdot \hspace{0.04cm} \Bar{\varepsilon}_k^2 \cdot \|D(\varphi g_k)\|_{\mathcal{C}^0}\right).
\end{equation} Expanding \eqref{e:this}:
\begin{equation}\label{e:this2}
    \int\varphi|Dg_k|^2=-\int g_k\langle Dg_k,D\varphi\rangle + \ \text{vanishing terms as}\  k\rightarrow \infty.
\end{equation} Now let $\Bar{\varepsilon}_k\stackrel{k \rightarrow \infty}\rightarrow 0$ and note that by \eqref{e:rellichtog} and \eqref{e:indicator} we get
 \begin{equation}\label{e:that}
g_k Dg_k \rightharpoonup gDg \,\text{ in } L^2.\end{equation} So by \eqref{e:this2} and \eqref{e:that} we can write 
\begin{equation}\begin{aligned}
    \lim_{k\rightarrow \infty}\int \varphi|Dg_k|^2=-\int g\langle Dg, D\varphi\rangle &= -\left(\int\langle Dg,D(g\varphi)\rangle-\int |Dg|^2\varphi\right)\\
& = \int |Dg|^2\varphi,
\end{aligned}
\end{equation} where the last passage follows from Claim \ref{claim1}. Hence we get
\begin{equation}
    \lim_{k\rightarrow \infty}\int |Dg_k|^2 \varphi=\int|Dg|^2\varphi,
\end{equation}
concluding strong $L^2$-convergence of $Dg_k$ in a compactly contained ball $B\subset\joinrel\subset B_1$.
Now we prove Claim \ref{claim3}, concluding the proof of Theorem \ref{t:babyda}.

\vspace{0.5cm}
{\it Proof of Claim \ref{claim3}.}

Recall that we want to prove that if $g$ is harmonic then
\begin{equation} \begin{aligned}
    \int_{B_{1/2}}|Dg(x)-Dg(0)|^2 &\leq \frac{1}{4\cdot2^m}\int_{B_1}|Dg(x)-Dg(0)|^2  \\
&\leq \frac{1}{4\cdot2^m}\int_{B_1}|Dg(x)|^2.
\end{aligned}
\end{equation} Note that without loss of generality we can assume $Dg(0)=0$\footnote{Since $g$ is harmonic, then $Dg(x)- Dg(0)$ is harmonic as well.} Moreover, it is enough to prove\footnote{Since we want to prove Claim \ref{claim3} for any arbitrary harmonic function and if $g$ is harmonic then $Dg$ is also harmonic.}
\begin{equation}\label{e:afterwlogs}
\int_{B_{1/2}(x)}|g(x)|^{2} \leq \left(\frac{1}{2}\right)^{m+2} \int_{B_{1}(x)}|g(x)|^{2}.\end{equation} Harmonic functions are real analytic, so we can write $$g(x)= \sum_{i=1}^{\infty}P_i\,,$$ where $P_i$ are (vectors of) homogeneous polynomials of degree $i$. If $g$ is harmonic, then $P_i$ are harmonic polynomials for every $i$. Since harmonic polynomials of different degrees are $L^2$-orthogonal when restricted to the unit sphere (see [\ref{SteinWeiss}, Ch.5 Section 2]), for any fixed $r$ we can write the following identity:
\begin{equation}
\int_{B_r} |g(x)|^2 dx = \sum_{i=1}^{\infty}\int_{B_r} |P_i(x)|^2 = \sum_{i=1}^{\infty} c_i \,r^{m +2i}, \end{equation} where $c_i$ are constants given by the integration of $|P_i(x)|^2$ on the ball of radius $r$. In particular, for $r=1$ and $r=1/2$ we have
\begin{equation}\label{e:finalone1}
\int_{B_{1/2}} |g(x)|^2 = \sum_{i=1}^{\infty} c_i \left(\frac{1}{2}\right)^{m +2i}= \frac{1}{2^{m}}\sum_{i=1}^{\infty} c_i \left(\frac{1}{2}\right)^{2i}\end{equation} and \begin{equation}\label{e:finalone2}
\int_{B_1} |g(x)|^2 = \sum_{i=1}^{\infty} c_i .\end{equation} From \eqref{e:finalone1} and \eqref{e:finalone2}, we can easily show that \begin{equation} \int_{B_{1/2}(x)}|g(x)|^{2} \leq \left(\frac{1}{2}\right)^{m+2} \int_{B_{1}(x)}|g(x)|^{2},\end{equation}
concluding the proof of Theorem \ref{t:babyda}.
\end{proof}

We end this section with few remarks and two corollaries of Theorem \ref{t:babyda}.

\begin{remark}
Note that the proof of Theorem \ref{t:babyda} works \textit{in every codimension}, since we never used the assumption $n=1$. Hence, everytime it is possible to ``well"-approximate an area-minimizing current with a graph of a Lipschitz function with small Lipschitz constant, then some perturbation of the proof of Theorem \ref{t:babyda} can be still applied. Unfortunately, it is not always the case that an integer rectifiable area-minimizing current with sufficiently small excess is ``close" to the graph of a (\textit{single-value}) Lipschitz function; this is true in codimension $n=1$ only. We will investigate better the case $n\ge2$ in the next section.
\end{remark}

\begin{remark}
We highlight the importance of the quadratic decay of the excess in \eqref{e:exdecay}. In codimension higher than 1, De Giorgi's variational idea will still play a very important role, together with an almost quadratic decay of the excess.
\end{remark}

\begin{remark}
Arguably, the most important insight of the proof of Theorem \ref{t:babyda} is the harmonicity of the limit $g$. Indeed, $(g_k)_k$ are minimizers of the area functional with $\text{Lip}(g_k)\rightarrow 0$. We remark once again that the fundamental idea relies on the observation that if one computes the area functional on a graph (assume for simplicity that $n=1$) 
$$F(u)=\int\sqrt{1+|Du|^2}$$ then the Taylor expansion of the integrand is the following: $$1+\frac{|Du|^2}{2}+ \text{ higher order terms.}$$
Roughly speaking, the area functional is ``very close" to be the Dirichlet energy when the gradients are very small. Hence, an area-minimizing current $T$ is ``very close" to the graph of a harmonic function.
\end{remark}

As a consequence of De Giorgi-Allard $\varepsilon$-regularity theorem it is possible to prove the following corollary.

\begin{corollary}\label{c:4punto4}
If $T$ is an area-minimizing current of dimension $m$ in $\R^{m+1}$,\footnote{Or, more generally, in a $\mathcal{C}^2$-submanifold $\Sigma$ of dimension $m+1$.} then any point $p$, at which there is a flat tangent cone, is a regular point. In particular, we conclude that $\|T\|(\emph{Sing}(T))=0.$
\end{corollary}

Even if in higher codimension things change dramatically, it is still possible to prove the following result, which was the best higher codimension regularity theorem available before Almgren's theorem (Theorem \ref{t:regularityh}).

\begin{corollary}{\normalfont [\ref{Simonbook}, Theorem 36.2]}\label{c:4punto5}
If $T$ is an integer rectifiable area-minimizing current of dimension $m$ in $\R^{m+n}$ and $n\ge1$, then $\operatorname{Reg}(T)$ is dense in $\operatorname{supp}(T) \setminus \operatorname{supp}(\partial T)$\footnote{This statement has been recently extended to any Hilbert space, see [\ref{ADLS}], using the tool of metric currents, see [\ref{AKmetric}].}.
\end{corollary}

\subsection{Area-minimizing currents in higher codimension}

If $n\ge2$, we have already seen that it is somehow easy to show Theorem \ref{t:regularityh} is optimal. Recall we considered the following holomorphic curve \begin{equation}\label{e:holobadguy} \Gamma = \{(z,w) \in \mathbb{C}^2 : z^2=w^3\}.\end{equation} One can see that the origin belongs to $\operatorname{Sing}(\llbracket \Gamma \rrbracket)$ and we proved in Theorem \ref{t:holomorphicsubvariety} that the current $ \llbracket \Gamma \rrbracket$ is (locally) area-minimizing. If we identify $\mathbb{C}$ with $\mathbb{R}^{2}$, it is simple to see that $\Gamma$ is an immersed (real) 2-dimensional submanifold, globally parametrized, for instance, by the map:
$$
u:\{(\rho, \theta): \rho>0, \theta \in[0,2 \pi)\} \longrightarrow \mathbb{R}^{4}
$$
given by
\begin{equation}\label{e:holobadguyparam}
(\rho, \theta) \mapsto\left(\rho \cos \theta, \rho \sin \theta, \rho^{\frac{3}{2}} \cos (3 \theta), \rho^{\frac{3}{2}} \sin (3 \theta)\right) .
\end{equation}
Nevertheless, $\Gamma$ is not an embedded submanifold in a neighborhood of the origin, because it is not a graph over the plane \begin{equation}\label{e:planebadguy}
\pi:=\{(z,w) \in \mathbb{C}^2: z=0\} \subset \mathbb{R}^{4}, \end{equation}
no matter how small is the neighborhood $U$ of the origin that we choose. 
The unique flat tangent cone at $0$ is given by $2 \llbracket \pi \rrbracket$ for $\pi$ as in \eqref{e:planebadguy} and the density $\Theta(\llbracket \Gamma \rrbracket,0) =2$.

The origin is a typical example of \textit{branch point} (sometimes called \textit{ramification point}). Moreover, it is important to notice that if we evaluate the excess $\Ex(\llbracket \Gamma \rrbracket, 0,r)$, which was the main parameter to detect regularity in a point in codimension 1, one can show that $$\Ex\left(\llbracket \Gamma \rrbracket, 0,r\right)\rightarrow 0 \text{  for  }  r\rightarrow 0. $$ We conclude that, even if the excess at a point can be made arbitrarily small (hence smaller than the $\varepsilon$-threshold that would have ensured the point $0$ to be a regular point in the codimension 1 $\varepsilon$-regularity theorem), this does not guarantee anymore that $\llbracket \Gamma \rrbracket$ is well-approximated by a single graph of a Lipschitz function. Hence Corollary \ref{c:4punto4} is false for 2-dimensional area minimizing currents in $\mathbb{R}^4$: $\llbracket \Gamma \rrbracket$ is singular at the origin in spite of the existence of a flat tangent cone. However, in examples like \eqref{e:holobadguy}, the current turns out to be a ``multivalued" graph, where the number of values is in fact determined by the multiplicity $Q=\Theta(\llbracket \Gamma \rrbracket, 0)=2$.

The main goal is to write a non-parametric problem for objects like the complex curve in \eqref{e:holobadguy}. The starting point of Almgren's Big Regularity Paper, see [\ref{Almgren2000}] and [\ref{DLSQ}, \ref{DLS1}, \ref{DLSsns}, \ref{DLS2}, \ref{DLS3}] for a shorter and improved version, is indeed to replace harmonic (single-valued) functions with \textit{multiple-valued} functions minimizing a \textit{suitable} notion of ``Dirichlet energy", developing a whole new theory and a first-order calculus for these peculiar maps.

\subsubsection{The space of $Q$-points $\mathcal{A}_Q(\R^n)$}

Consider again the current $\llbracket \Gamma \rrbracket$ in \eqref{e:holobadguy}. The support of such current, namely the complex curve $\Gamma$, can be viewed as the graph of a function which associates to any $w \in \mathbb{C}$ two points in the $z$-plane:
\begin{equation}\label{e:primaqvalori}
w \mapsto\left\{z_1(w), z_2(w)\right\} \quad \text { with } z_i(w)^2=w^3 \text { for } i=1,2.
\end{equation}
The map in \eqref{e:primaqvalori} is an example of a $2$\textit{-valued} function. From now on, let $Q \geqslant 1$ be a fixed positive integer. Roughly speaking, $Q$-valued functions can be considered as mappings taking their values in the \textit{unordered} sets of $Q$-points of $\mathbb{R}^n$, taking into account the fact that we may have a multiplicity. More precisely, we have the following definition.

\begin{definition}\label{d:qpoints}
Denote by $\llbracket P_{i} \rrbracket$ the Dirac delta in $P_{i} \in \mathbb{R}^{n}$ and define the space of $Q$\textit{-points} as $$\mathcal{A}_{Q}\left(\mathbb{R}^{n}\right):=\left\{\sum_{i=1}^{Q} \llbracket P_{i} \rrbracket: P_{i} \in \mathbb{R}^{n} \text { for every } i=1, \ldots, Q\right\} .$$
\end{definition}

\begin{remark} Observe that the notation $\llbracket P_{i} \rrbracket$ to denote the Dirac delta $\delta_{P_{i}}$ is consistent with the notion of current associated to a submanifold. Indeed, if $P \in \mathbb{R}^n$ then the action of the $0$-current associated to $P$ is precisely given by
$$
\llbracket P \rrbracket(f)=f(P) \,\text { for every } f \in \mathcal{C}_{c}^{\infty}(\mathbb{R}^n).
$$
\end{remark}

\begin{remark}
In other words, Definition \ref{d:qpoints} identifies the space of $Q$ unordered points in $\mathbb{R}^n$ with the set of positive atomic measures of mass $Q$ and note that the points $P_{i} \in \mathbb{R}^{n}$ are not necessarily different (for example, $Q \llbracket P \rrbracket \in \mathcal{A}_Q\left(\mathbb{R}^n\right)$). Moreover, we remark that the absence of the order for points in $\mathcal{A}_{Q}\left(\mathbb{R}^{n}\right)$ is fundamental: $\mathcal{A}_2(\mathbb{R}^n)$ cannot be identified with $\mathbb{R}^n \times \mathbb{R}^n$.
\end{remark}

We will sometimes use the notations $\mathcal{A}_Q$ and $\sum_i \llbracket P_i \rrbracket$ when $n$ and $Q$ are clear from the context.
\begin{remark}
Note that $\mathcal{A}_Q\left(\mathbb{R}^n\right)$ is just the quotient of $\left(\mathbb{R}^n\right)^Q$ via the action of the group of permutations of $Q$ indexes $\mathcal{S}_Q$. In other words, defining the equivalence relation
$$
\left(P_1, \ldots, P_Q\right) \sim\left(P_{\sigma(1)}, \ldots, P_{\sigma(Q)}\right) \quad \text{ for all } \sigma \in \mathcal{S}_Q,
$$
then $$\mathcal{A}_Q \simeq\left(\mathbb{R}^n\right)^Q / \sim.$$ It follows then that the space of $Q$-points, though it is not a linear space\footnote{Unless the trivial case $Q=1$.}, inherits many properties from the Euclidean space. \end{remark}

\begin{remark}
One of the major novelties of De Lellis and Spadaro's works with respect to Almgren's theory is that they avoid lots of combinatorial arguments, just considering $\mathcal{A}_Q$ as an abstract metric space. For this reason De Lellis and Spadaro's metric approach to $Q$-valued functions is sometimes named \textit{intrinsic theory}, as opposed to Almgren's \textit{extrinsic} one. Indeed, in [\ref{Almgren2000}], Almgren developed the theory of $Q$-valued function mostly using two maps: the first one is a bi-Lipschitz embedding $\xi$ of $\mathcal{A}_Q\left(\mathbb{R}^n\right)$ into $\mathbb{R}^{N(Q, n)}$, where $N(Q, n)$ is a sufficiently large integer. By means of the map $\xi$ one can define a Sobolev theory for $Q$-valued functions as classical $\mathbb{R}^N$-valued Sobolev maps taking values in $\xi(\mathcal{A}_Q).$ The second map $\rho$ is a Lipschitz retract of $\mathbb{R}^{N(Q, n)}$ onto $\xi(\mathcal{A}_Q)$, which is useful in various approximation arguments.
\end{remark}

\begin{definition} For every $T_{1}, T_{2} \in \mathcal{A}_{Q}\left(\mathbb{R}^{n}\right)$, with $T_{1}=\sum_{i} \llbracket P_{i} \rrbracket$ and $T_{2}=\sum_{i} \llbracket S_{i} \rrbracket$, we
define
\begin{equation}
\mathcal{G}\left(T_{1}, T_{2}\right):=\min _{\sigma \in \mathcal{S}_{Q}} \sqrt{\sum_{i}\left|P_{i}-S_{\sigma(i)}\right|^{2}}.
\end{equation}
\end{definition}

\begin{remark}
One can realize that $\mathcal{G}$ coincides with the $L^2$-Wasserstein distance on the space of positive measures with finite second moment (see [\ref{Villani}]). It is immediate to see that $(\mathcal{A}_{Q}\left(\mathbb{R}^{n}\right), \mathcal{G})$ is a complete, locally compact and separable metric space.
\end{remark}

\begin{definition}
Let $\Omega \subset \mathbb{R}^m$ open, bounded with smooth boundary. A $Q$\textit{-valued function} is a map $$f: \Omega \rightarrow\left(\mathcal{A}_{Q}\left(\mathbb{R}^{n}\right), \mathcal{G}\right).$$
\end{definition}

\begin{remark}
We say that a $Q$\textit{-valued function} $f$ is \textit{continuous} (\textit{Lipschitz, Hölder} and \textit{measurable} respectively) is it is so as function between metric spaces. Similarly, $u \in L^{p}\left(\Omega, \mathcal{A}_{Q}\right), 1\le p \le \infty,$ if $x \mapsto \mathcal{G}(u(x), Q \llbracket 0 \rrbracket) \in L^{p}(\Omega)$\footnote{Since $\Omega$ is bounded, this is equivalent to ask that $\|\mathcal{G}(u, T)\|_{L^p}$ is finite for every $T \in \mathcal{A}_Q$.}.
\end{remark}

Any measurable $Q$-valued function admits the following representation in \textit{measurable selections}. 

\begin{proposition}\label{p:measurableselections}
Let $B \subset \mathbb{R}^{m}$ be a measurable set and let $f: B \rightarrow \mathcal{A}_{Q}$ be a measurable function. Then, there exist $f_{1}, \ldots, f_{Q}$ measurable $\R^n$-valued functions such that
$$
f(x)=\sum_{i=1}^{Q} \llbracket f_{i}(x) \rrbracket \quad \text { for a.e. } x \in B.
$$ We call such a representation \emph{measurable selection} (or \emph{measurable superposition}).
\end{proposition}

\begin{remark}
The proof of Proposition \ref{p:measurableselections} is done by induction on the number of values of $Q$, making use of the following observation: if $\bigcup_{i \in \mathbb{N}} B_i$ is a covering of $B$ by measurable sets, then it is sufficient to find a measurable selection of $f_{|_{B_i \cap B}}$ for every $i$.
\end{remark}

\begin{remark}
Roughly speaking, by Proposition \ref{p:measurableselections} we can see every measurable $Q$-valued function as a ``sum" (very far from being unique) of $Q$ measurable\footnote{In general, even if $u$ is regular, one cannot expect (globally) more than measurability of the selections.} functions, called \textit{selections} (or \textit{superpositions}). Hence, $\llbracket f_{i}(x) \rrbracket$ are Dirac deltas at points $f_i(x)$ and they have to be though as just labels to name each single value.
\end{remark}

Since the final goal is to analyze minimizers of the area functional, we need to introduce a notion of derivative for such functions and to develop a first-order calculus. 

\subsubsection{Sobolev $Q$-valued functions}

We introduce the Sobolev spaces of functions taking values in the metric space of $Q$-points. The approach is based on the pioneering work by Ambrosio in [\ref{Ambrosio90}], see also [\ref{Resh97}, \ref{Resh04}, \ref{Resh06}], by looking at composition with Lipschitz functions.

\begin{definition}
A measurable function $f: \Omega \rightarrow \mathcal{A}_{Q}$ is in the Sobolev class $W^{1, p}$ $(1 \leq p \leq \infty)$ if there exist $m$ positive functions $\varphi_j \in L^{p}(\Omega)$ such that for every $T \in \mathcal{A}_{Q}$ we have:
\begin{itemize}
\item[\textit{i)}] $x \mapsto \mathcal{G}(f(x), T) \in W^{1, p}(\Omega)$,
\item[\textit{ii)}] $|\partial_j \mathcal{G}(f(x), T)| \leq \varphi_j(x)$ a.e. in $\Omega$ for all $j=1,\dots,m$.
\end{itemize}
\end{definition}

\begin{remark}
One can show $f \in W^{1, p}(\Omega,\mathcal{A}_Q)$ if and only if there exists $\psi \in L^p(\Omega)$ such that, for every $F: \mathcal{A}_Q \rightarrow \mathbb{R}$ Lipschitz,
$$
F \circ f \in W^{1, p}(\Omega) \text { and }|D(F \circ f)| \leq \operatorname{Lip}(F) \,\,\psi \text { a.e. in } \Omega.$$ One implication of the proof is trivial, the converse follows by fixing $\{T_i\}_{i \in \N}$ a countable dense subset, by McShane's extension theorem (see, for instance [\ref{ATilli}, Theorem 3.1.2]) and by choosing $$\psi:=\left(\sum_j \left(\sup _{i \in \mathbb{N}}\left|\partial_j \mathcal{G}\left(f, T_i\right)\right|\right)^2\right)^{1 / 2}.$$
\end{remark}

A first step in the analysis of the area functional is to study its linearized version. Hence, we need to introduce a notion of ``modulus" of the gradient of a Sobolev function in order to define a \textit{suitable} notion of Dirichlet energy in this framework. What we mean by \textit{suitable} is that it appears as the first nontrivial term in the Taylor expansion of the mass of (the current associated to) a multivalued graph.

\begin{definition}\label{d:metricmodulus}
Let $f: \Omega \rightarrow \mathcal{A}_{Q}$. Fix a countable dense subset $\left\{T_{i}\right\}_{i \in \mathbb{N}}$ of $\mathcal{A}_{Q}$. For every $j=1, \ldots, m$, we define
$$
\left|\partial_{j} f\right|:=\sup _{i \in \mathbb{N}}\left|\partial_{j} \mathcal{G}\left(f, T_{i}\right)\right| \quad \text { and } \quad|Df|^{2}:=\sum_{j=1}^{m}\left|\partial_{j} f\right|^{2}.
$$
\end{definition}

\begin{remark}
Definition \ref{d:metricmodulus} is well-posed since it does not depend on the choice of the countable set. Note that in Definition \ref{d:metricmodulus} $|Df|^{2}$ is just a positive quantity depending only on the metric structure of $\mathcal{A}_{Q}$.
\end{remark}

\begin{remark}\label{r:coordinates}
For functions on a general Riemannian manifold $(M,g)$, we choose an orthonormal frame $X_1, \ldots X_m$ and set $|D f|^2:=\sum\left|\partial_{X_i} f\right|^2$. This definition is independent of the choice of coordinates (respectively, of the frames), see [\ref{DLSQ}, Proposition 2.17].
\end{remark}

All the metric notions we have introduced so far would still be well-defined for maps with values in any complete separable metric space. In the special case of $Q$-valued maps, it is also possible to give a notion of pointwise derivative.

\begin{definition}\label{d:pointwisediff}
Let $f: \Omega \rightarrow \mathcal{A}_{Q}$ and $x_{0} \in \Omega$. We say that $f$ is \textit{differentiable at} $x_{0}$ if there exist $Q$ matrices $L_{1}, \dots, L_{Q}$ and a selection $f_{1}, \dots, f_{Q}$ such that: \begin{itemize}
\item[\textit{(i)}] $\lim\limits_{x\rightarrow x_0}(x-x_0)^{-1}\mathcal{G}\left(f(x), T_{x_{0}} f\right)=0,$ where
$$
T_{x_{0}} f (x):=\sum_{i=1}^Q \llbracket L_{i} \cdot\left(x-x_{0}\right)+f_{i}\left(x_{0}\right) \rrbracket ,
$$
\item[\textit{(ii)}] $f_{i}\left(x_{0}\right)=f_{j}\left(x_{0}\right)$ implies $L_{i}=L_{j}$ for all $i,j \in \{1,\dots,Q\}$\footnote{This second condition is sometimes called ``no-crossing condition".}. \end{itemize}
The $Q$-valued map $T_{x_{0}} f$ is called the \textit{first-order approximation} of $f$ at $x_0$.
\end{definition}

\begin{remark}
It is useful to fix the notation $D f_i$ for $L_i$ in Definition \ref{d:pointwisediff}. Note that by $(ii)$ this is unambiguous: namely, if $g_1, \ldots, g_Q$ is a different selection for $f$, $x_0$ a point of differentiability and $\sigma$ a permutation such that $g_i(x_0)=f_{\sigma(i)}(x_0)$ for all $i \in\{1, \ldots, Q\}$, then $D g_i(x_0)=D f_{\sigma(i)}(x_0)$. When the $f_i$'s are a differentiable functions and $f$ is differentiable, then the $D f_i$'s coincide with the classical differentials.
\end{remark}

\begin{remark}
Note that according to Definition \ref{d:pointwisediff}, the origin in \eqref{e:holobadguy} is a point of differentiability for the multivalued function having $\Gamma$ as a graph.
\end{remark}

The pointwise differentiability property would be an empty definition unless there exist some functions that satisfy it. Hence we have the following theorems.

\begin{theorem}[Generalized Rademacher]\label{t:qrademacher}
Let $f: \Omega \rightarrow \mathcal{A}_{Q}$ be a Lipschitz function. Then, $f$ is differentiable almost everywhere in $\Omega$ with respect to the $m$-dimensional Lebesgue measure.
\end{theorem} We summarize below the main steps of the proof.

\begin{proof}
The main idea to treat multiple-valued functions is to distinguish among their multiplicities and apply an induction argument on $Q$.
\begin{itemize}

\item For $Q=2$, consider $\tilde{\Omega}:=\left\{x \in \Omega: f_{1}(x)=f_{2}(x)\right\}$, the set of points where $f$ takes a single value with multiplicity $2$.
\item It is easy to show that in $\Omega \setminus \tilde{\Omega}$ one can apply the classical Rademacher's theorem.
\item Then one considers all (``well-behaved") points $x \in \tilde{\Omega}$ such that $\tilde{\Omega}$ has density 1 and (a Lipschitz extension $g$ of) $f_1$ is differentiable: on all these points one proves that $f$ is differentiable in terms of $Q$-valued map with first-order approximation given by $$T_{x}f(y)=2\llbracket Dg(x)(y-x) +g(x)\rrbracket.$$
\item By the Lebesgue differentiation theorem and a projection argument, one concludes the case $Q=2$. For further $Q$'s one proceeds by induction. \qedhere
\end{itemize} 
\end{proof}
It is now rather simple to prove that every Sobolev $Q$-valued function is in fact ``approximately" pointwise differentiable at almost every point. Indeed, as for the classical theory, one can prove that a Lusin-type approximation holds for $Q$-valued Sobolev functions.

\begin{definition}
A $Q$-valued function $f$ is \textit{approximately differentiable} in $x_0$ if there exists a measurable subset $\tilde{\Omega} \subset \Omega$ containing $x_0$ such that $\tilde{\Omega}$ has density 1 at $x_0$ and $f_{|_{\tilde{\Omega}}}$ is differentiable at $x_0$.
\end{definition}

\begin{proposition}[Lipschitz Approximation]{\normalfont [\ref{DLSQ}, Proposition 4.4]}
There exists a constant $C=C(m, \Omega, Q)$ with the following property. For every $f \in W^{1, p}(\Omega, \mathcal{A}_Q)$ and every $\lambda>0$, there exists a $Q$-function $f_\lambda$ such that $\operatorname{Lip}(f_\lambda) \leq C \lambda$ and
$$
\mathcal{L}^m\left(\left\{x \in \Omega: f(x) \neq f_\lambda(x)\right\}\right) \leq \frac{C}{\lambda^p}\int_{\Omega}|Df|^p,$$ where $|Df|$ is defined as in Definition \ref{d:metricmodulus}. \end{proposition}

\begin{remark}
In the classical theory, one of the most efficient ways to approximate Sobolev functions by smooth functions is via regularization by convolution. Nevertheless, this method cannot be applied in the metric framework since $\mathcal{A}_Q$ is not a linear space, hence it is not possible to integrate a kernel against a function.
What still works in a much more general setting is another method usually known as ``truncation along the maximal function of the gradient", see [\ref{HKST}]. Coupling this procedure with a generalized Lipschitz extension theorem and Theorem \ref{t:qrademacher}, one concludes that any $f \in W^{1, p}(\Omega, \mathcal{A}_Q)$ is approximately differentiable at almost every point. 
\end{remark}

At this point one could show that the ``modulus of the gradient", as defined in Definition \ref{d:metricmodulus}, and the pointwise differential are linked.

\begin{proposition}\label{p:samederivative}{\normalfont [\ref{DLSQ}, Proposition 2.17]}
For every $f \in W^{1,2}\left(\Omega, \mathcal{A}_{Q}\right)$ and every $j=1, \ldots, m$, we have
$$
|D f|^{2}=\sum_{i}\left|D f_{i}\right|^{2} \quad \text { a.e., }
$$
where $|L_i|$ denotes the Hilbert-Schmidt norm of the matrix $L_i$.
\end{proposition}

\begin{remark}
Proposition \ref{p:samederivative} justifies the metric definition $|Df|^2$ in \ref{d:metricmodulus}. Indeed, when the $Q$-valued function $f$ is the superposition of $Q$ smooth functions $f_1,\dots, f_Q$, then the first-order expansion of the area functional is given by $|Df_i|^2$, for each $i$.
\end{remark}

The usual notion of trace at the boundary can be easily generalized in this setting.

\begin{definition}\label{d:qtrace}
Let $\Omega \subset \mathbb{R}^m$ be a bounded open set with Lipschitz boundary and $f \in W^{1, p}\left(\Omega, \mathcal{A}_Q\right)$. A function $g$ belonging to $L^p\left(\partial \Omega, \mathcal{A}_Q\right)$ is said to be the \emph{trace of $f$ at} $\partial \Omega$ (and we denote it by $f_{|_{\partial \Omega}}$) if, for every $T \in \mathcal{A}_Q$, the trace of the real-valued Sobolev function $\mathcal{G}(f, T)$ coincides with $\mathcal{G}(g, T)$.
\end{definition}

A Morrey-Campanato estimate in the spirit of Theorem \ref{t:morreycampanato} holds for $Q$-valued functions.

\begin{theorem}[Morrey-Campanato]\label{t:qmorreycampanato}
Let $f \in W^{1,2}(B_1, \mathcal{A}_Q)$ and $\alpha \in(0,1]$ be such that $$\int_{B_r(y)}|D f|^2 \leq A r^{m-2+2 \alpha} \, \text{ for every }\, y \in B_1 \, \text{ and }\, a.e.\, r \in( 0,1-|y|).$$
Then, for every $0<\delta<1$, there exists a constant $C=C(m, n, Q, \delta)$ such that
$$
[f]_{\mathcal{C}^{0, \alpha}(\overline{B_\delta})}:=\sup _{x, y \in \overline{B_\delta}} \frac{\mathcal{G}(f(x), f(y))}{|x-y|^\alpha} \leq C \sqrt{A}.
$$
\end{theorem}

\begin{remark}
Many results of the classical theory of Sobolev spaces can be generalized to the Sobolev class $W^{1, p}(\Omega,\mathcal{A}_Q)$, such as chain rules, existence and uniqueness of the trace for Sobolev $Q$-functions, weak convergence, Sobolev embeddings, Poincaré inequality and so on; we refer to [\ref{DLSQ}] for more details.
\end{remark}

\subsubsection{Dir-minimizing $Q$-valued functions}

We aim at finding solutions of minimization problems framed in the context of $Q$-valued functions. In principle we should look at the minimization of the area functional, which is a delicate problem because of its nonlinear nature. As a starting point, we begin the investigation with the linear problem given by the minimization of the Dirichlet energy, which now can be defined thanks to the first-order calculus developed so far.

\begin{definition}
The generalized \textit{Dirichlet energy} of $f \in W^{1,2}(\Omega, \mathcal{A}_Q)$ is given by
$$
\operatorname{Dir}(f, \Omega):=\int_{\Omega}|D f|^2=\sum_i \int_{\Omega}\left|D f_i\right|^2 .
$$
We say that a function $f \in W^{1,2}(\Omega, \mathcal{A}_Q)$ is \textit{Dir-minimizing} if $$\operatorname{Dir}(f, \Omega) \leq \operatorname{Dir}(g, \Omega),$$ for all $g \in W^{1,2}(\Omega, \mathcal{A}_Q)$ with $f_{|_{\partial \Omega}}=g_{|_{\partial \Omega}}$ (in the sense of Definition \ref{d:qtrace}).
\end{definition}

Now we describe three fundamental theorems in the theory of Dir-minimizing $Q$-valued functions. The first theorem provides \textit{existence} of Dir-minimizing functions, while the second and the third theorems deal with \textit{regularity} results: (interior) H\"older regularity of Dir-minimizers and an estimate of the singular set. Indeed, as already mentioned, the first step of Almgren's theory of partial regularity for area-minimizing currents in higher codimension is to develop a theory concerning existence and regularity for the first nonconstant term in the area functional.

\begin{theorem}[Existence of Dir-minimizing functions]
Let $g \in W^{1,2}(\Omega, \mathcal{A}_{Q}).$ Then, there exists a Dir-minimizing function $f \in W^{1,2}(\Omega, \mathcal{A}_{Q})$ such that $f_{|_{\partial \Omega}}=g_{|_{\partial \Omega}}.$
\end{theorem}

The proof of the existence theorem for Dir-minimizing functions follows by a straightforward application of the direct methods in the calculus of variations. Indeed, $Q$-valued functions and the generalized Dirichlet energy satisfy the same results as in the classical setting, namely weak sequential compactness (see [\ref{DLSQ}, Proposition 2.11]), continuity of the trace under weak convergence (see [\ref{DLSQ}, Proposition 2.10]) and weak sequential lower semicontinuity of the Dirichlet energy (see [\ref{DLSQ}, Section 2.3.2]). Once these three properties have been proved, then we can easily conclude as follows.

\begin{proof}
Let $(f_k)_k$ in $W^{1,2}(\Omega, \mathcal{A}_Q)$ be a minimizing sequence of functions, that is for every $k$ we have ${f_{k}}_{|_{\partial \Omega}}=g_{|_{\partial \Omega}}$ and
$$\lim _{k \rightarrow \infty} \operatorname{Dir}(f_k, \Omega)=\inf \{\operatorname{Dir}(h, \Omega): h \in W^{1,2}(\Omega, \mathcal{A}_Q)  \text{ with } h_{|_{\partial \Omega}}=g_{|_{\partial \Omega}}\}.$$
Then, by weak sequential compactness, there exists a subsequence $(f_{k_j})_j$ which is $L^2$-converging to some function $f$, that is:
$$
\lim _{j\rightarrow \infty}\left\|\mathcal{G}(f_{k_j}, f)\right\|_{L^2(\Omega)}=0.
$$ By continuity of the trace under weak convergence we get $f_{|_{\partial \Omega}}=g_{|_{\partial \Omega}}$ and by the lower semicontinuity of the Dirichlet energy we conclude that
$$
\begin{aligned}
\operatorname{Dir}(f, \Omega) & \leq \lim _{j \rightarrow \infty} \operatorname{Dir}(f_{k_j}, \Omega) \\
&=\inf \left\{\operatorname{Dir}(h, \Omega): h \in W^{1,2}\left(\Omega, \mathcal{A}_Q\right) \text { with } h_{|_{\partial \Omega}}=g_{|_{\partial \Omega}}\right\},
\end{aligned}
$$
hence proving that $f$ is a minimizer.
\end{proof}

A fundamental result in the theory of higher codimension area-minimizing currents is the H\"older continuity of Dir-minimizing $Q$-valued functions in the interior of $\Omega$.

\begin{theorem}[H\"older-regularity of Dir-minimizing functions]\label{t:qholderregularity}
There exists a positive constant $\alpha=\alpha(m, Q)>0$ such that if $f \in W^{1,2}(\Omega, \mathcal{A}_Q)$ is Dir-minimizing, then $f \in \mathcal{C}^{0, \alpha}(\Omega^{\prime})$ for every $\Omega^{\prime} \subset\joinrel\subset \Omega \subset \mathbb{R}^m$. For two-dimensional domains, we have the explicit constant $\alpha(2, Q)=1 / Q$.
\end{theorem}

Note that, although Dir-minimizing functions act as of classic harmonic functions in the setting of $Q$-valued maps, they are not in general analytic. One can show that interior H\"older continuity for Dir-minimizing functions is sharp. Indeed, not only are holomorphic varieties examples of singular (locally) area-minimizing integral currents (as shown in Theorem \ref{t:holomorphicsubvariety}), but also they provide examples of Dir-minimizing $Q$-valued functions. More formally one could prove the following theorem, see [\ref{Almgren2000}, Theorem 2.20] and [\ref{Spadarocomplex}].

\begin{theorem}\label{t:spadarocomplex}
Let $\mathcal{V} \subseteq B_2 \times \mathbb{R}^{2 m} \subseteq \mathbb{R}^{2 n+2 m} \simeq \mathbb{C}^{n+m}$ be an irreducible holomorphic variety with the property that $\pi_{*} \llbracket \mathcal{V} \rrbracket=Q \llbracket B_2 \rrbracket$, where $\pi$ is the orthogonal projection\footnote{This condition is sometimes referred to as $\mathcal{V}$ to be a $Q: 1$-cover of the ball $B_2 \subseteq \mathbb{C}^n$ under the orthogonal projection $\pi$ onto $B_2$.}. Then, there exists a Dir-minimizing $Q$-valued function $f \in W^{1,2}\left(B_1, \mathcal{A}_Q(\mathbb{R}^{2 m})\right)$ such that $\operatorname{graph}(f)=\mathcal{V} \cap\left(B_1 \times \mathbb{C}^m\right)$.
\end{theorem}

\begin{remark}In particular, a counterexample to higher-than-H\"older regularity comes considering the holomorphic curve $\Gamma$ as in \eqref{e:holobadguy}. We could consider a different parametrization (that would still be a Dir-minimizing $Q$-valued function since its graph is again $\Gamma$) than the one in \eqref{e:holobadguyparam}, just by inverting the role of $z$ and $w$.\end{remark}

The proof of Theorem \ref{t:qholderregularity} is divided into two main steps. The core of the proof relies on the first step which, in some sense, can be considered as a ``geometric differential inequality". Its proof strongly differs from the $m=2$ to the $m\ge3$ case: indeed, while in the planar case it is somehow not difficult to decompose a Sobolev $Q$-valed functions into Sobolev selections, this cannot be said for the general case. Hence, more sophisticated analytic results need to be employed, such as a maximum principle for Dir-minimizing functions (see [\ref{DLSQ}, Proposition 3.5]) and a decomposition theorem for Dir-minimizers (see  [\ref{DLSQ}, Proposition 3.6]). The first step can be stated as follows.

\begin{proposition}\label{p:1perholder}
Let $f \in W^{1,2}(B_r, \mathcal{A}_Q)$ be Dir-minimizing and suppose that
$$
g=f_{|_{\partial B_r}} \in W^{1,2}(\partial B_r, \mathcal{A}_Q).
$$
Then, we have that
\begin{equation}\label{e:diffineq}
\int_{B_r}|Df|^2 \leq C(m) r \int_{\partial B_r}|Dg|^2,
\end{equation}
where $C(2)=Q$ and $C(m)<(m-2)^{-1}$.
\end{proposition}

\begin{remark}
In some sense, Proposition \ref{p:1perholder} ensures a geometric control of the trace by means of the energy: roughly speaking it tells that every time we have a Dir-minimizing function $f$ and we select a ``slice" $\partial B_r$ where $f_{|_{\partial B_r}}$ is still Sobolev, then the ``differential inequality"\footnote{We call \eqref{e:diffineq} ``differential inequality" since the quantity on the right-hand side of \eqref{e:diffineq} resembles very much the derivative of the left-hand side with respect to $r$.} estimate \eqref{e:diffineq} holds.
\end{remark}

The second step of the proof is, instead, a standard application of the Morrey-Campanato estimate of Theorem \ref{t:qmorreycampanato}. Assuming Proposition \ref{p:1perholder} (see [\ref{DLSQ}, Sections 3.3.2, 3.3.3]), we prove Theorem \ref{t:qholderregularity}.

\begin{proof}
Set \begin{equation}\label{e:gammaperq}
\gamma(m):= \begin{cases}2 Q^{-1} & \text { for } m=2 \\ C(m)^{-1}-m+2 & \text { for } m>2,\end{cases}
\end{equation} where $C(m)$ is the constant in \eqref{e:diffineq}\footnote{Note that $\gamma(m)>0$ since $C(m)^{-1} > m-2$.} and define $h(r):=\int_{B_r}|D f|^2$. Note that $h$ is absolutely continuous so that, by Remark \ref{r:coordinates}, we can write
\begin{equation}\label{e:tangentialder}
h^{\prime}(r)=\int_{\partial B_r}|D f|^2 \geq \int_{\partial B_r}|\partial_\tau f|^2 \quad \text { for a.e. } r,
\end{equation}
where $|\partial_\tau f|^2=|D f|^2-\sum_{i=1}^Q\left|\partial_\nu f_i\right|^2$. Here $\partial_\tau$ and $\partial_\nu$ denote the tangential and the normal derivatives respectively. Hence, by \eqref{e:tangentialder} and Proposition \ref{p:1perholder} we get $$
 h(r) \leq \frac{1}{m-2+\gamma} \,r h^{\prime}(r) \iff \frac{m-2+\gamma}{r} \le \frac{h^{\prime}(r)}{h(r)} .
$$
Integrating this differential inequality between $r$ and 1, we easily obtain
$$
h(r) \leq r^{m-2+\gamma} h(1),
$$ that is \begin{equation}\label{e:morreyfinal} \int_{B_r}|D f|^2\le r^{m-2+\gamma} \int_{B_1}|D f|^2. \end{equation}

By \eqref{e:morreyfinal} and the Morrey-Campanato estimate in Theorem \ref{t:qmorreycampanato}, we conclude the Hölder continuity of $f$ with exponent $\alpha=\gamma/2$.
\end{proof}

\begin{remark}
Note that if $m=2$, by \eqref{e:gammaperq} we have $\alpha=1/Q$. An intuition is that the higher the number of $Q$-points, the ``less continuity" we are guaranteed. This intuition is confirmed since the higher the value of $Q$, the smaller the exponent $\gamma$ in \eqref{e:morreyfinal}, resulting in a worse estimate.
\end{remark}

\begin{remark}
Theorem \ref{t:qholderregularity} is about interior regularity and tells nothing about what happens at the boundary of the original domain $\Omega$. We mention that Hirsch [\ref{Hirsh16}] extends the H\"older regularity for Dirichlet minimizing $Q$-valued functions up to the boundary assuming $\mathcal{C}^1$-regularity of the bounded domain $\Omega$ and $\mathcal{C}^{0, \alpha}$-regularity of the boundary datum $f_{|_{\partial \Omega}}$ with $\alpha>\frac{1}{2}$.
\end{remark}

We now turn to the second theorem about the regularity properties of Dir-minimizing $Q$-valued functions, concerning the analysis of the singularities. In particular, we introduce the following natural definition of regular and singular points.

\begin{definition}
A $Q$-valued function $f$ is \emph{regular at a point} $x \in \Omega$ if there exist $r>0$ and $Q$ analytic functions $f_{i}: B_{r}(x) \rightarrow \mathbb{R}^{n}$ such that
$$
f(y)=\sum_{i} \llbracket f_{i}(y) \rrbracket \quad \text { for every } y \in B_{r}(x)
$$
and either $f_{i}(y) \neq f_{j}(y)$ for every $y \in B_{r}(x)$ or $f_{i} \equiv f_{j}.$ The set of regular points of $f$ is denoted reg$(f)$. The \textit{singular set} $\Sigma_{f}$ of $f$ is the complement of the set of regular points.
\end{definition}

\begin{figure}[h]
    \centering
    \includegraphics[width=0.55\textwidth]{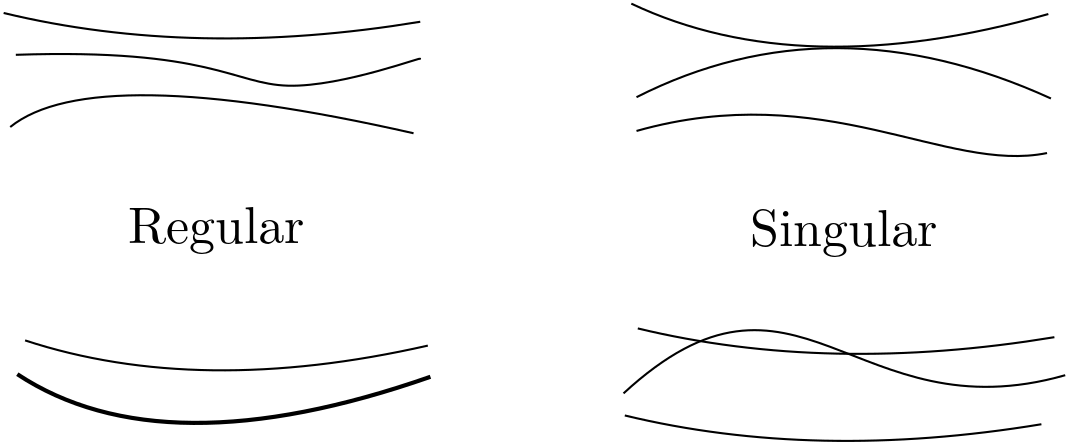}
    \caption{}
    \label{f:regsing}
\end{figure}

\begin{remark}
Note that reg$(f)$ is open (by definition) and the singular set $\Sigma_{f}$ is relatively closed in $\Omega$. Intersections of different selections or branch points are ruled out from the definition of regular point, see Figure \ref{f:regsing}. The rationale is to define a notion of regular point that conveys the same geometric meaning of an embedded submanifold.
\end{remark}

The following analogous result of partial regularity for area-minimizing currents holds in the setting of Dir-minimizing functions.

\begin{theorem}[Estimate of the singular set]\label{t:estimatesingset}
Let $f$ be a Dir-minimizing function. The singular set $\Sigma_{f}$ of $f$ has Hausdorff dimension at most $m-2$ and it is at most countable if $m=2$.
\end{theorem}

\begin{remark}
Note that the estimate in terms of Hausdorff dimension is the same as in the general case of Theorem \ref{t:regularityh}. The only difference is that Theorem \ref{t:estimatesingset} provides an estimate on the singular set \textit{in the domain} $\Omega$ of the Dir-minimizing function $f$. Nevertheless, in case the multiple-valued graph of $f$ is ``quite flat", it is reasonable to imagine the same estimate holds true for the singular set of the current.
\end{remark}

The analysis of the singularities in higher codimension area-minimizing currents depends deeply on a new monotonicity formula discovered by Almgren, that recently witnessed many applications in very different fields, see for instance [\ref{Signorini}, \ref{Signorini2}]. In almost all geometric problems, the starting point is usually a monotonicity estimate, which is an \textit{a priori} estimate on a monotone quantity related to the problem of interest, in the same spirit of \eqref{e:monotonicity}. A new monotonic quantity has been discovered\footnote{In the preface of [\ref{Big}], Jean Taylor recalls that Almgren started conceiving the idea of using this particular function in a particular jogging track in Oxford.} by Almgren for higher codimension area-minimizing currents and it is now named \textit{Almgren's frequency function}.

\begin{definition}\label{d:almgrenfrequency}
Let $f$ be a Dir-minimizing function, $x \in \Omega$ and $0<r<\operatorname{dist}(x, \partial \Omega)$. We define the functions \begin{equation}\label{e:dhifrequency}
D_{x, f}(r):=\int_{B_{r}(x)}|D f|^{2}, \quad H_{x, f}(r):=\int_{\partial B_{r}}|f|^{2} \quad \text { and } \quad I_{x, f}(r):=\frac{r D_{x, f}(r)}{H_{x, f}(r)}.
\end{equation}
$I_{x, f}$ is called \textit{Almgren's frequency function.}
\end{definition}

\begin{remark}
Note that, by Theorem \ref{t:qholderregularity}, $|f|^2$ is a continuous function. Hence, $H_{x, f}(r)$ is a well-defined quantity for every $r$. Moreover, if $H_{x, f}(r)=0$, then by minimality $f_{|_{B_r(x)}} \equiv 0$. So, apart for this case, $I_{x, f}(r)$ is always well-defined.
\end{remark}

\begin{remark}
When $x$ and $f$ are clear from the context, it is customary to use the shorthand notations $D(r), H(r)$ and $I(r)$. The reason why $I$ is called \textit{frequency} function can be explained by looking at its value on the planar harmonic functions $f_k(r, \theta)=r^k \cos (k \theta)$: one can easily compute that $I_{0, f_k}(r) \equiv k$, which is the corresponding frequency of the angular parameter. Notations $D_{x, f}$ and $H_{x, f}$ in \eqref{e:dhifrequency} refers to ``Dirichlet" energy and ``Height" of the function $f$. In a sense, Almgren's frequency function aims at capturing ``the degree of disorder" of a harmonic function, since it measures its energy in terms of its norm on the boundary.
\end{remark}

The most important estimate in the analysis of singular points is the following monotonicity theorem.

\begin{theorem}[Monotonicity]\label{t:monotonefreq}{\normalfont [\ref{DLSQ}, Section 3.4.1]}
Let $f$ be Dir-minimizing and $x \in \Omega.$ Either there exists $\rho$ such that $f_{|_{B_{\rho}(x)}} \equiv 0$ or $I_{x, f}(r)$ is an absolutely continuous nondecreasing positive function on $\left(0, \operatorname{dist}(x,\partial \Omega)\right)$. In particular, in the latter case the following limit exists
$$
I_{x, f}(0):=\lim _{r \rightarrow 0} I_{x, f}(r) >0.
$$
\end{theorem}

\begin{remark}
Broadly speaking, Theorem \ref{t:monotonefreq} tells that the average (up to constants) of the energy of a Dir-minimizing function decreases on small scales (when $r \rightarrow 0$) with respect to its zero-degree norm.
The proof is not as enlightening as its consequences, and it is rather elementary. The main idea is based on some simple derivation arguments coupled with the tools of \textit{inner} and \textit{outer variation} of a Dir-minimizer\footnote{\textit{Inner} and \textit{outer variations} are two natural notions of variation that can be used to perturb Dir-minimizing $Q$-valued functions, see [\ref{DLSQ}, Section 3.1].}.
\end{remark}

It is worth mentioning the following very important corollary, see [\ref{DLSQ}, Section 3.4.2] for a straightforward proof.

\begin{corollary}
Let $f$ be Dir-minimizing in $B_{\varrho}$. Then, $I_{0, f}(r) \equiv \alpha$ if and only if $f$ is $\alpha$-homogeneous, that is
\begin{equation}\label{e:homogeneityfreq}
f(y)=|y|^\alpha f\left(\frac{y \varrho}{|y|}\right) .
\end{equation}
\end{corollary}

\begin{remark}
In \eqref{e:homogeneityfreq} the following convention has been adopted: if $\varphi$ is a scalar function and $f=\sum_i \llbracket f_i \rrbracket$ a $Q$-valued function, by $\varphi f$ is meant the function $\sum_i \llbracket \varphi f_i \rrbracket$.
\end{remark}

\subsubsection{Blow-up analysis}

The main strategy to analyze singular points is to perfororm a \textit{blow-up analysis}, by expanding homothetically a ball around the singular point, see Figure \ref{f:blowup}. The aim is to exploit possible symmetries of the limiting object, reducing the complexity of the problem.

\begin{figure}[h]
    \centering
    \includegraphics[width=0.45\textwidth]{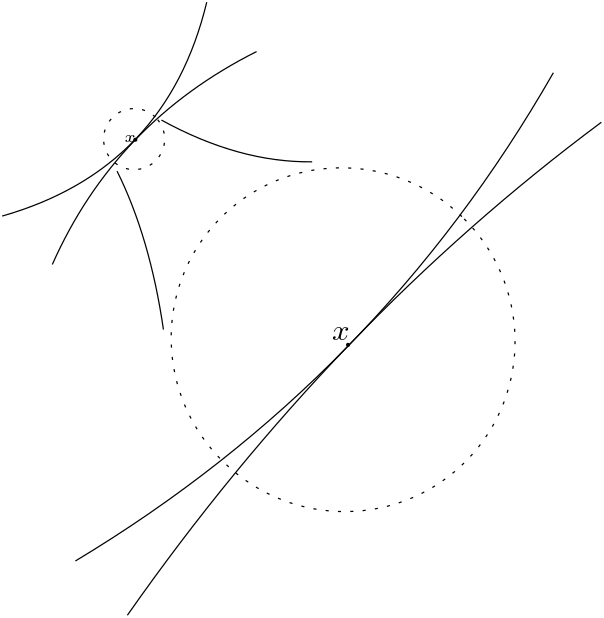}
    \caption{}
    \label{f:blowup}
\end{figure}

In order to perform an effective blow-up analysis, one needs to guarantee that the singular point is preserved in the blow-up limit, which is not granted in general. One of the possible obstructions to persistence in the limit of a singularity is that the first singular expansion of the object (either of the current\footnote{In the more general case of area-minimizing currents an ``almost monotonicity" formula holds for the frequency function, still assuring its boundedness and hence nontrivial blow-ups.}  or the Dir-minimizer) around its regular part may occur with an infinite order of contact, resulting in a trivial blow-up. The main strategy is in analogy with unique continuation theory, where a harmonic function is proved to be analytic by detecting the lowest frequency in its Taylor expansion. Almgren's frequency function is indeed an integral a priori estimate to detect the lowest frequency of an object and its main use is to ensure the persistence of the singularity in the blow-up limit.

The presence of branch points requires a different blow-up procedure than the one that is used in the codimension one case, since an ``homogeneous" blow-up around a branch point would end up to be a flat plane, losing all the information about the singularity. One of Almgren's main ideas was to rescale differently the ``horizontal directions" (namely those of flat tangent cone at the point) and the ``vertical directions" (which are the orthogonal complements of the horizontal ones). This different type of blow-up is sometimes called ``inhomogeneous" blow-up (or ``anisotropic" blow-up). 

Since we will look at point of maximal multiplicity, we assume that $f$ is a nontrivial\footnote{That is, $\int_{B_\rho(y)}|Df|^2>0$ for every $\varrho$.} Dir-minimizing $Q$-valued function such that $f(y)=Q \llbracket 0 \rrbracket$. The main idea of inhomogeneous blow-ups is to rescale ``according to the energy" in the following way:
\begin{equation}\label{e:iblowup}
f_{y, \varrho}(x):=\frac{\varrho^{\frac{m-2}{2}} f(\varrho x+y)}{\sqrt{\int_{B_{\varrho}(y)}|Df|^2}} .
\end{equation}

\begin{remark}
From \eqref{e:iblowup} it is easy to see that $$\int_{B_1}|Df_{y, \varrho}|^2=1,$$ showing that the energy is preserved in the limit. To simplify the notation, we will not display the subscript $y$ in $f_{y, \rho}$ when $y$ is the origin.
\end{remark} To better understand this point, let us consider again the current $\llbracket \Gamma \rrbracket$ of example \eqref{e:holobadguy}. Recall that its support can be seen as the graph of a function which associates to any $w \in \mathbb{C}$ two points in the $z$-plane:
$$
w \mapsto\left\{z_1(w), z_2(w)\right\} \quad \text { with } z_i(w)^2=w^3 \text { for } i=1,2 .
$$
The right rescaling according to Almgren's inhomogeneous blow-up is the one producing in the limit a multiple-valued harmonic function preserving the Dirichlet energy. In the case of $\Gamma$, the functions $z_1$ and $z_2$ are already harmonic functions (at least away from the origin), since they are two determinations of the square root of $w^3$. Hence, the correct blow-up is the one which ``keeps $\Gamma$ fixed". Hence, for every $\lambda>0$, we consider $f_\lambda: \mathbb{C}^2 \rightarrow \mathbb{C}^2$ given by
$$
f_\lambda(w, z)=(\lambda^2 w, \lambda^3 z),
$$
and note that $(f_\lambda)_{*} \llbracket \Gamma \rrbracket=\llbracket \Gamma \rrbracket$ for every $\lambda>0$. 

An important result is the convergence of blow-ups $f_{\rho}$ of Dir-minimizing functions to some limiting Dir-minimizing functions with some extra symmetries, which are called \textit{tangent functions}.  

\begin{proposition}\label{p:buhomogenbu}
Let $f \in W^{1,2}(B_1, \mathcal{A}_Q)$ be Dir-minimizing, with $f(0)=Q \llbracket 0 \rrbracket$ and $\int_{B_{\varrho}(y)}|Df|^2>0$ for every $\varrho \leq 1$. Then every sequence $f_{\varrho_k}$ with $\rho_k \rightarrow 0$ has a subsequence converging locally uniformly to a function $g: \mathbb{R}^m \rightarrow \mathcal{A}_Q(\mathbb{R}^n)$ satisfying:
\begin{itemize}
\item[$i)$] $\int_{B_1}|Dg|^2=1$ and $g_{|_{\Omega}}$ is Dir-minimizing for any bounded $\Omega$,
\item[$ii)$] $g(x)=|x|^\alpha g\left(\frac{x}{|x|}\right)$, where $\alpha=I_{0, f}(0)>0$ is the frequency of $f$ at 0 (that is, $g$ is $\alpha$-homogeneous).
\end{itemize}
\end{proposition}

The proof of Proposition \ref{p:buhomogenbu} relies on the following preliminary compactness lemma (we refer to [\ref{DLSQ}, Propositions 3.19 and 3.20] for the proofs of Proposition \ref{p:buhomogenbu} and Lemma \ref{l:bucompactness} respectively).

\begin{lemma}\label{l:bucompactness}
Let $f_k \in W^{1,2}\left(B_1, \mathcal{A}_Q\right)$ be Dir-minimizing $Q$-valued functions such that $\sup _k \int_{B_1}|Df_k|^2<\infty$ and $f_k \rightarrow f$ uniformly. Then, for every $r<1,$ $f_{|_{B_r}}$ is Dir-minimizing and $$\lim_{k\rightarrow \infty}\int_{B_r}|Df_k|^2 = \int_{B_r}|Df|^2.$$
\end{lemma}

Now we can pass to the core of the estimate on the Hausdorff dimension of the singular set of a Dir-minimizing $Q$-valed function. The key lemma is a control on the set of singular points of highest multiplicity $Q$. More formally, Theorem \ref{t:estimatesingset} is an easy consequence of the following Lemma.

\begin{lemma}\label{l:bufinallemma}
Let $\Omega$ be connected and $f \in W^{1,2}(\Omega, \mathcal{A}_{Q}(\mathbb{R}^{n}))$ be Dir-minimizing. Then, either $f=Q \llbracket \zeta \rrbracket$ with $\zeta: \Omega \rightarrow \mathbb{R}^{n}$ harmonic in $\Omega$, or the set
$$
\Sigma_{Q, f}=\left\{x \in \Omega: f(x)=Q \llbracket y \rrbracket,\, y \in \mathbb{R}^{n}\right\}
$$
(relatively closed in $\Omega$) has Hausdorff dimension at most $m-2$ and it is locally finite for $m=2$.
\end{lemma}

\begin{remark}
Note that Lemma \ref{l:bufinallemma} is the analogue of Theorem \ref{t:estimatesingset} in the case one considers just points of multiplicity $Q$.
\end{remark} Assuming Lemma \ref{l:bufinallemma}, it is not difficult to conclude Theorem \ref{t:estimatesingset} by an induction argument on $Q$. 

\begin{proof}
For $Q=1$ there is nothing to prove, since Dir-minimizing 1-valued functions are classic harmonic functions. Suppose now the theorem to be true for every $Q^*$-valued functions, with $Q^*<Q$. 

If $f=Q \llbracket \zeta \rrbracket$ with $\zeta$ harmonic, then $\Sigma_f=\emptyset$ and the theorem is proved. If not, we first consider $\Sigma_{Q, f} \subset \Sigma_f$ that, by Lemma \ref{l:bufinallemma}, is a closed subset of $\Omega$ with Hausdorff dimension at most $m-2$, at most countable if $m=2$. 

Then, we also consider the open set $$\Omega^{\prime}:=\Omega \backslash \Sigma_{Q, f}.$$ Since $f$ is continuous, we can find countable open balls $B_k$ such that $\Omega^{\prime}=\cup_k B_k$ and $f_{|_{B_k}}$ can be decomposed as the sum of two multiple-valued Dir-minimizing functions $$f_{|_{B_k}}=\llbracket f_{k, Q_1} \rrbracket+\llbracket f_{k, Q_2} \rrbracket$$ with $Q_1<Q$, $Q_2<Q$ and
\begin{equation}\label{e:bulastcond}\operatorname{supp}\left(f_{k, Q_1}(x)\right) \cap \operatorname{supp}\left(f_{k, Q_2}(x)\right)=\emptyset \quad \text{ for every } x \in B_k.\end{equation} From \eqref{e:bulastcond} we have that $\Sigma_f \cap B_k=\Sigma_{f_k, Q_1} \cup \Sigma_{f_{k, Q_2}}$. Note further that $f_{k, Q_1}$ and $f_{k, Q_2}$ are Dir-minimizing and, by inductive hypothesis, $\Sigma_{f_{k, Q_1}}$ and $\Sigma_{f_{k, Q_2}}$ are closed subsets of $B_k$ with Hausdorff dimension at most $m-2$. We conclude that
$$
\Sigma_f=\Sigma_{Q, f} \cup \bigcup_{k \in \mathbb{N}}\left(\Sigma_{f_{k, Q_1}} \cup \Sigma_{f_{k, Q_2}}\right)
$$
has Hausdorff dimension at most $m-2$ and it is at most countable if $m=2$.
\end{proof}

\begin{remark}
As a result, by the simple induction argument shown above, we managed to reduce the whole proof of Theorem \ref{t:estimatesingset} to its analogous verison for highest multiplicity points only.
\end{remark}

\subsubsection{The need of centering}

In this final part we are going to prove Lemma \ref{l:bufinallemma} in the simplified planar case (for the general case $m\ge3$ we refer to [\ref{DLSQ}, Section 3.6.2], even if the main ideas of blow-up and dimension reduction arguments are the same as in the planar case). As we have already mentioned, one of the main issues in the blow-up analysis of the singularities is the persistence of such singularities in the limiting function. Even if Almgren's frequency function guarantees nontrivial blow-ups, this still does not ensure the singularity does not vanish in the limiting object. Consider the following complex variety:
\begin{equation}\label{e:holoverybadguy}
\Xi:=\left\{(z, w) \in \mathbb{C}^2:\left(z-w^2\right)^2=w^{2021}\right\} \subset \mathbb{C}^2.
\end{equation}
It is simple to see that $\Xi$ is the graph of the 2-valued function $f$ given by
$$
w \stackrel{f}{\mapsto} \sum_{\eta \,:\, \eta^2=w} \llbracket w^2+\eta^{2021} \rrbracket \in \mathcal{A}_2(\mathbb{C}) \simeq \mathcal{A}_2\left(\mathbb{R}^2\right) .
$$
By Theorem \ref{t:spadarocomplex} the function $f$ is Dir-minimizing in any compact set of $\mathbb{R}^2$ and by direct computation it is easy to verify that the rescaled functions $f_{\rho}$ in \eqref{e:iblowup} converge uniformly to a regular 2-valued function. This shows regularity of the limiting blow-up even if the origin was a singular point for $f$, hence excluding the possibility to estimate the size of the singular set of $f$ by means of its tangent functions. In other words, $\Xi$ is just an ``almost indistinguishable perturbation" of the (smooth) current $2\llbracket \{z=w^2\} \rrbracket$, see Figure \ref{f:centering}.

\begin{figure}[h]
    \centering
    \includegraphics[width=0.65\textwidth]{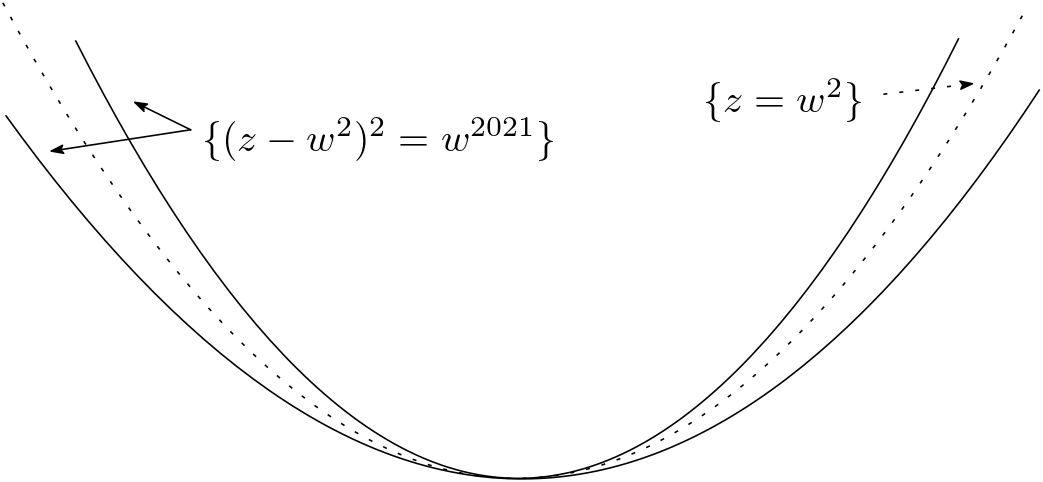}
    \caption{}
    \label{f:centering}
\end{figure}

The solution to such a problem is to perform a sort of ``change of coordinates", averaging out the regular first expansion of the blow-up, on top of which the singular branching behavior happens. In particular, in the previous example the regular part was exactly the smooth complex curve $\left\{z=w^2\right\}$, while the singular branching is due to the determinations of the square root of $z^{2021}$. 

Hence, it becomes clear why one should look for parametrizations of $\Xi$ as a multiple-valued function defined on $\left\{z=w^2\right\}$, so that the singular map to be considered reduces to
$$w \mapsto \sum_{\eta \,:\, \eta^2=w} \llbracket \eta^{2021}\rrbracket \in \mathcal{A}_2\left(\mathbb{R}^2\right) . $$ The blow-up of such map is the map itself, and the singular point 0 persists in the limit.

The regular surface $\left\{z=w^2\right\}$ is called \textit{center manifold} by Almgren, because it behaves like (and in this case it is exactly) the average of the sheets of the current in a suitable system of coordinates. In general, the determination of the center manifold is not as immediate as in this linearized problem of Dir-minimizing functions. For the nonlinear case of area-minimizing currents the construction of the center manifold and the parametrization of the current on its normal bundle actually represent one of the main difficulties in the analysis of singularities\footnote{In Almgren's theory for the nonlinear problem the construction of the center manifold takes almost three quarters of his Big Regularity Paper [\ref{Almgren2000}]. De Lellis and Spadaro [\ref{DLSQ}, \ref{DLSsns}, \ref{DLS1}, \ref{DLS2}, \ref{DLS3}] managed to reduce the whole argument, but still takes a fair half of the whole work.}.

We now give just a glimpse of how to deal with the case of Dir-minimizing functions, stating two lemmas (see [\ref{DLSQ}, Section 3.6.1] for a proof) that will allow to conclude the proof of (the planar case of) Lemma \ref{l:bufinallemma} (and so, Theorem \ref{t:estimatesingset}). 

In order to state the first lemma, we introduce the function $\eta: \mathcal{A}_Q\left(\mathbb{R}^n\right) \rightarrow \mathbb{R}^n$ mapping each measure $T=\sum_i \llbracket P_i \rrbracket$ to its center of mass,
$$
\eta(T):=\frac{\sum_i P_i}{Q} .
$$

\begin{lemma}\label{l:averagio}
Let $f: \Omega \rightarrow \mathcal{A}_Q\left(\mathbb{R}^n\right)$ be Dir-minimizing. Then, \begin{itemize}
\item[$i)$] the function $\eta \circ f: \Omega \rightarrow \mathbb{R}^n$ is harmonic,
\item[$ii)$] for every $\zeta: \Omega \rightarrow \mathbb{R}^n$ harmonic, $g:=\sum_i \llbracket f_i+\zeta \rrbracket$ is Dir-minimizing.
\end{itemize}
\end{lemma}

\begin{remark}
Note that, in particular, we have that $$g(x)=\sum_i \llbracket f_i(x) - \eta \circ f (x)\rrbracket$$ is still Dir-minimizing. The proof is basically an integration by parts computation.
\end{remark}

\begin{remark}
Note that $$\Sigma_g=\Sigma_f$$ but now $\Sigma_{Q,g}=\{x : g(x)=Q\llbracket 0\rrbracket \}$, telling us that we managed to reduce Lemma \ref{l:bufinallemma} to the case where all points of multiplicity $Q$ are of the form $Q \llbracket 0 \rrbracket$. In this situation, the collapse to one single sheet in the blow-up analysis means that this collapse should happen at 0, which is forbidden by the estimates on the frequency function. Hence, this ``subtracting the average"-procedure tackles the problem of very high-order perturbations.
\end{remark}
The second lemma we need is a technical lemma on cylindrical blow-ups of homogeneous functions.

\begin{lemma}\label{l:cylindricalbu}
Let $g: B_1 \rightarrow \mathcal{A}_Q\left(\mathbb{R}^n\right)$ be an $\alpha$-homogeneous Dir-minimizing function with $\int_{B_1}|Dg|^2>0$ and set $\beta:=I_{z, g}(0)$. Suppose also that $g(z)=Q \llbracket 0 \rrbracket$ for $z=e_1 / 2$. Then, the tangent functions $h$ to $g$ at $z$ are $\beta$-homogeneous with $\operatorname{Dir}\left(h, B_1\right)=1$ and satisfy:\begin{itemize}
\item[$i)$] $h\left(s e_1\right)=Q \llbracket 0 \rrbracket$ for every $s \in \mathbb{R}$,
\item[$ii)$] $h\left(x_1, x_2, \ldots, x_m\right)=\hat{h}\left(x_2, \ldots, x_m\right)$, where $\hat{h}: \mathbb{R}^{m-1} \rightarrow \mathcal{A}_Q\left(\mathbb{R}^n\right)$ is Dir-minimizing on any bounded open subset of $\mathbb{R}^{m-1}$.
\end{itemize}
\end{lemma}

We can finally prove Lemma \ref{l:bufinallemma}, concluding the proof of Theorem \ref{t:estimatesingset} in the case $m=2$.

\begin{proof}
Note that, as we remarked above, by Lemma \ref{l:averagio} it is sufficient to consider a Dir-minimizing function $f$ such that $\eta \circ f \equiv 0$ so that,
$$
\Sigma_{Q, f}=\{x: f(x)=Q \llbracket 0 \rrbracket\} .
$$
We prove that $\Sigma_{Q, f}$ consists of isolated points, except for the case where all sheets collapse. Without loss of generality, let $0 \in \Sigma_{Q, f}$ and assume that $f \neq Q \llbracket 0 \rrbracket$ in a neighborhood of 0. 

Suppose by contradiction that there exist a sequence $x_k \rightarrow 0$ such that $f(x_k)=Q \llbracket 0 \rrbracket$. By Proposition \ref{p:buhomogenbu}, the blow-ups $f_{\left|x_k\right|}$ converge uniformly, up to a subsequence, to some homogeneous Dir-minimizing function $g$, with $$\int_{B_1}|Dg|^2 =1 \text{ and } \eta \circ g \equiv 0.$$ Since $f(x_k)$ are $Q$-multiplicity points, we also deduce that there exists $w$ in the unit circle $S^1$ such that $g(w)=Q \llbracket 0 \rrbracket$. Up to a rotation, we can assume that $w=e_1$. Considering the blow-up of $g$ in the point $e_1 / 2$, by Lemma \ref{l:cylindricalbu}, we get a new tangent function $h$ with the property that $h(0, x_2)=\hat{h}(x_2)$ for some function $\hat{h}: \mathbb{R} \rightarrow \mathcal{A}_Q$ which is Dir-minimizing on every interval, $\eta \circ \hat{h} \equiv 0$ and $\hat{h}(0)=Q \llbracket 0 \rrbracket$. Moreover, since $$\int_{B_1}|Dh|^2=1,$$ then \begin{equation}\label{e:htilde}\int_{I}|D\hat{h}|^2>0 \text{ for } I=[-1,1],\end{equation} which is a contradiction. Indeed, by means of a comparison argument, one can prove that every Dir-minimizing 1-dimensional function $\hat{h}$ is an affine function of the form $$\hat{h}(x)=\sum_i \llbracket L_i(x) \rrbracket$$ with the property that either $L_i(x) \neq L_j(x)$ for every $x$ or $L_i(x)=L_j(x)$ for every $x$. Since $\hat{h}(0)=Q \llbracket 0 \rrbracket$, we would conclude that $\hat{h}=Q \llbracket L \rrbracket$ for some linear $L$. On the other hand, by $\eta \circ \hat{h} \equiv 0$ we would conclude $L=0$, contradicting \eqref{e:htilde}.

We conclude that if $x \in \Sigma_{Q, f}$ then either $x$ is isolated, or $U \subset \Sigma_{Q, f}$ for some neighborhood $U$ of $x$. Since by assumption $\Omega$ is connected we obtain that either $\Sigma_{Q, f}$ consists of isolated points, or $\Sigma_{Q, f}=\Omega$, concluding the proof.
\end{proof}

\begin{remark}
The proof of Lemma \ref{l:bufinallemma} is a very standard \textit{motiv} in geometric measure theory: one fixes a singular point, by contradiction the singular set is assumed to be ``too large" (like in this case where we assumed accumulation of other singular points) and then by blow-up procedure one derives a contradiction. 
\end{remark}

In [\ref{DLSQ}, Chapter 5], the authors improved Theorem \ref{t:estimatesingset} proving the following refinement in the planar case. One of the most important tools is the uniqueness of the tangent function to a Dir-minimizer. Thanks to this, the authors succeeded in developing a better description of the behavior of a Dir-minimizing function around singular points.

\begin{theorem} Let $f$ be Dir-minimizing and $m=2$. Then, the singular set $\Sigma_f$ of $f$ consists of isolated points. \end{theorem}

About more recent studies of fine properties of the singular set of Dir-minimizing function, it is worth mentioning the result in [\ref{DLMarchese}]. The main result proved by the authors is that if $f$ is a Dir-minimizing function, then $\Sigma_f$ is countably $(m-2)$-rectifiable (and hence $\mathcal{H}^{m-2}$ $\sigma$-finite).

\subsubsection{Hints to the nonlinear case} 

So far, we showed the analogous of Almgren's partial regularity theorem, Theorem \ref{t:regularityh}, in the very special case of parametrizations minimizing the (generalized) Dirichlet energy, which correspond to the linearized version of the whole problem. Unfortunately (or, depending on the point of view, luckily), the full proof of Almgren's partial regularity theorem for area-minimizing currents in higher codimension is way longer. 

In fact, the analytic and geometric issues of the linear case can be simplistically summarized as follows:\begin{itemize}
\item[$a)$] The problem of dealing with multiple-valued functions, 
\item[$b)$] The problem of getting trivial blow-ups due to a possible infinite order of contact,
\item[$c)$] The need of a centering to ensure persistence of the singularity in the limit. 
\end{itemize}

In the linear case, everything is very clean and the theory is fairly easy to understand. Nevertheless, many technical and convoluted difficulties need to be tackled in exporting all the aforementioned techniques to the general nonlinear problem, see [\ref{DLS1}, \ref{DLS2}, \ref{DLS3}]. 

The first major difficulty is that it is not known any a priori estimate telling whether the area-minimizing current is a graph of a suitable function in higher codimension. Hence, a whole theory of approximation needs to be developed to pass from the graphical case to the general one. 

\begin{remark} At this point, differently from the linear case, the average of the different sheets of a current does not solve in general any given partial differential equation, thus not allowing any simple translation or reparametrization argument. \end{remark}

As in the linear case, the proof of Theorem \ref{t:regularityh} is done by contradiction, where the contradiction assumption is the following: there exist numbers $m \geq 2,n \geq 1, \alpha>0$ and\footnote{Note that the hypothesis $m \geq 2$ is justified because, for $m=1$, an area-minimizing current is locally the union of finitely many nonintersecting open segments.} an area-minimizing $m$-dimensional integer rectifiable current $T$ in $\mathbb{R}^{m+n}$ such that
$$
\mathcal{H}^{m-2+\alpha}(\operatorname{Sing}(T))>0 .
$$
The aim of the proof is now to show that there exist suitable points of Sing$(T)$ where we can perform the blow-up analysis in the same spirit of the linear case. 

This process consists of several different steps: we list here the most important ones in the proof of Theorem \ref{t:regularityh}, following the neat description given by Spadaro in [\ref{Spadaronote}].

\begin{enumerate}
\item Find a point $x_0 \in \operatorname{Sing}(T)$ and a sequence of radii $(r_k)_k$ with $r_k \rightarrow 0$ such that: \begin{itemize}
\item The currents $T_{x_0, r_k}=(\iota_{x_0, r_k})_{*} T$ converge to a flat tangent cone,
\item $\mathcal{H}^{m-2+\alpha}\left(\operatorname{Sing}(T_{x_0, r_k}) \cap B_1\right)>\eta>0$ for some $\eta>0$ and for every $k \in \mathbb{N}$.
\end{itemize}
Note that both conclusions hold for suitable subsequences, which in principle may not coincide. What we need to prove is that we can select a point and a subsequence satisfying both.
\item Construction of the center manifold $\mathcal{M}$ and of a normal Lipschitz approximation $F:\mathcal{M} \rightarrow \mathcal{A}_Q(U)$, see [\ref{DLS2}, Definition 2.3], where $U$ is a (kind of) tubular neighborhood of $\mathcal{M}.$

\item The center manifold that one constructs in step $2$ can only be used in general for a finite number of radii $r_k$ of step $1$. The reason is that, in general, its degree of approximation of the average of the area-minimizing currents $T$ is under control only up to a certain distance from the singular point under consideration. This leads to the definition of the sets where the approximation works, called \textit{intervals of flattening}, and to the construction of an entire \textit{sequence of center manifolds} which will be used in the blow-up analysis.

\item Next, one has to deal with the problem of the infinite order of contact and this is done in two substeps. In the first substep, an \textit{almost monotonicity formula} is derived for a slight variation of Almgren's frequency function, deducing that the order of contact remains finite within each center manifold of the sequence described in step $3$.
In the second substep, one needs to compare different center manifolds and to show that the order of contact still remains finite. This is done by exploiting a deep consequence of the construction in step $3$, which is called \textit{splitting before tilting phenomenon}\footnote{Which is, roughly speaking, a multivalued version of what is known as \textit{tilt-lemma}, that is an estimate of the $L^2$-deviation from a tangent plane by means of the excess: the analogous of a reverse Poincaré inequality for elliptic partial differential equations.}, where the terminology was borrowed by [\ref{Rivieresplitting}].
\item With all this hard analysis at disposal, one can finally pass to the limit and conclude the convergence of the rescaling of the \textit{normal part} of $F$ to the graph of a Dir-minimizing $Q$-valued function $u$.
\item In conclusion, one can use a delicate capacitary argument leading to the persistence of the singularities to show that the function $u$ in step $5$ must have a singular set with positive $\mathcal{H}^{m-2+\alpha}$-measure, thus contradicting the partial regularity estimate for $Q$-valued Dir-minimizing functions in Theorem \ref{t:estimatesingset}.
\end{enumerate}

\subsubsection{Future research directions}

\begin{remark}
One could hope to be able to further investigate if Almgren's partial regularity theorem (Theorem \ref{t:regularityh}) can be improved (for $m \ge 3$) to derive fine properties of the interior singular set of any area-minimizing integral $m$-dimensional current in $\mathbb{R}^{m+n}$, above all if it is $m-2$-countably rectifiable and with locally finite $\mathcal{H}^{m-2}$-measure. Moreover, in the process of digging deeper in the above analysis of the singular set, this study would possibly shed light on at least some partial cases of Open problem \ref{op:utgcone}. \end{remark}

At the moment, the following questions are still open:

\begin{open problem}\label{op:rectifiability}
Consider an area-minimizing integral current $T$ of dimension $m$ in $\mathbb{R}^{m+n}$. Is \emph{Sing($T$)} $m-2$-rectifiable?
\end{open problem}

The fine structure of the singular set (hence, a fortiori, Open problem \ref{op:rectifiability}) shares many deep connections with the uniqueness of the tangent cone. In particular, Open problem \ref{op:rectifiability} seems to be relying on the following simpler tangent cone uniqueness question:
\begin{open problem}\label{op:rectifiability2}
Consider an area-minimizing integral current $T$ of dimension $m$ in $\mathbb{R}^{m+n}$ and let $p \in \emph{Sing($T$)}$ be a point where one tangent cone is flat. Is the latter the unique tangent cone to $T$ at $p$?
\end{open problem}

By the results obtained in the forthcoming work [\ref{DSkoro}]\footnote{Where the authors study, among other things, how to subdivide singularities based on the value of the frequency function.}, the authors suggest that a positive answer to the previous question, together with the additional information of a polynomial convergence rate, would imply $m-2$-rectifiability of $\operatorname{Sing}(T)$.

About the possibilities to go beyond Almgren's theory it is worth mentioning the case of currents equipped with special calibrated structures. Specific instances are complex integral currents of arbitrary dimension and codimension in $\mathbb{C}^d$, where the regularity analysis of [\ref{harvey}] fully characterized them as integrations over a pure $k$-dimensional algebraic variety (known as ``holomorphic $k$-chains"). Another more recent example is the one of special Lagrangian $3$-d cones in $\mathbb{R}^6$, proved in [\ref{bellettini}] to be smooth except for a finite number of half-lines, emanating from the vertex of the cone; a priori, Almgren's result ensures that area-minimizing integer rectifiable $3$-currents are smooth outside of a set of Hausdorff dimension $1$ that, in the case of a cone, roughly translates into having radial lines of singularities, possibly accumulating onto each other. Strongly relying on the special Lagrangian calibrated geometry, Bellettini and Rivière established in [\ref{bellettini}] that there can only be a finite number of such lines. 

\begin{open problem}\label{op:calibrated}
Is it possible to derive refined properties of \emph{Sing($T$)} in Almgren's theorem (Theorem \ref{t:regularityh}) when special structures are assumed on calibrated currents?
\end{open problem}

\section{Regularity theory for optimal transport paths}

Once an existence theory has been developed for the optimal branched transport problem, the natural following question is to ask whether such optimal transport paths with finite costs enjoy finer regularity properties. In general, one cannot hope for smooth minimizers, as it happens (for some dimensions) in the regularity theory for area-minimizing currents, because of the intrinsic nature of the networks. Nevertheless, an interior regularity theory has been developed starting from the work by Xia [\ref{Xiaregularity}], where an optimal transport path of finite cost is proved to be made by a finite union of line segments near each interior point of the path.

The main strategy to prove such a statement deeply relies on ideas borrowed from the theory of generalized area-minimizing surfaces. The key step is a blow-up procedure, studying tangent cones of minimizers at an arbitrary point; the main tools to perform such a blow-up analysis are the monotonicity formulae, as already widely discussed in Section 2.1. In analogy with area-minimizing currents, in optimal branched transport one should expect a monotonicity formula given by a suitable ratio involving the specific ``cost" of this Plateau-type problem: the $\alpha$-mass. Indeed, this is the case as shown in the following proposition.

\begin{proposition}\label{p:alphamonotonicity}{\normalfont [\ref{Xiaregularity}, Corollary 3.1]}
Let $T=\llbracket E,\tau,\theta \rrbracket$ be a transport path such that $T \in \OTP(\partial T)$. Then for any $x \in \operatorname{supp}(T) \setminus \operatorname{supp}(\partial T)$ and any $0<\rho < \operatorname{dist}(x,\operatorname{supp}(\partial T))$, the quantity \begin{equation}\label{e:alphamonotonicity} \frac{\int_{B_{\rho}(x)}|\theta|^{\alpha}d\mathcal{H}^1 \res E}{\rho} \end{equation}
is a nondecreasing function of $\rho$.
\end{proposition}

\begin{remark}
As it always happens for monotonicity formulae, once the monotone ratio has been ``discovered", then the proof relies on some straightforward (and usually not so enlightening) computations.
\end{remark}

As soon as a monotonicity formula for the $\alpha$-mass is available, one can start the usual blow-up procedure studying the existence of tangent cones. Given the blow-up sequence $T_{0,r}$ as in Definition \ref{d:miserveblowup}, we need to prove that for $r_k \rightarrow 0$ as $k\rightarrow \infty$ we have the following uniform bound: \begin{equation}\label{e:uniformboundbranched}
\sup _k \left(\mathbb{M}(T_{0,r_k})+ \mathbb{M}(\partial T_{0,r_k})\right)<\infty .
\end{equation} Then, by Proposition \ref{p:compactnessnormal}, $T_{0,r}$ converges (up to subsequences) in the sense of currents to the tangent cone $T$.
\begin{remark}
Note the analogy with (for instance) Lemma \ref{l:bucompactness}, where one would like to prove uniform bounds to gain a suitable notion of (sequential) compactness and deduce convergence of the blow-ups.
\end{remark} Notice that by Proposition \ref{p:alphamonotonicity} we have $$\sup_k \mathbb{M}^{\alpha}(T_{0,r_k}) <\infty.$$ Moreover, since for every $k$ we have $T_{0,r_k} \in \OTP(\partial T_{0,r_k})$, we conclude that $$\sup_k \mathbb{M}(T_{0,r_k}) \leq \sup_k \mathbb{M}^{\alpha}(T_{0,r_k}) <\infty.$$

One can also prove $\sup_k \mathbb{M}^{\alpha}(\partial T_{0,r_k}) <\infty.$ The main idea of this proof is done by slicing, deriving a formula of the following type for the $\alpha$-mass of a current $S$ (see, [\ref{CDRMPP}, Proposition 2.9]): $$
\int \mathbb{M}^{\alpha}(S_y) \leq C \, \mathbb{M}^{\alpha}(S).
$$

\begin{remark}
In the same spirit of point $i)$ of Proposition \ref{p:buhomogenbu}, the fact that tangent cones are proved to be minimizers is of fundamental importance. Hence, also in the optimal branched transport problem one would like to conclude that the limiting object is a minimizer for the $\alpha$-mass, which is a priori not granted at all. 
Moreover, in analogy with point $ii)$ of Proposition \ref{p:buhomogenbu}, one wants to understand if the limiting object has some extra symmetries or finer properties: we would like to say that such a cone cannot have infinitely many segments spreading out of the object.
\end{remark}
More formally, we would like to prove the following lemma.

\begin{lemma}\label{p:incriminata}
Let $T=\llbracket E,\tau,\theta \rrbracket$ be a transport path such that $T \in \OTP(\partial T)$. Then for any $x \in \operatorname{supp}(T) \setminus \operatorname{supp}(\partial T)$ there exists a tangent cone $C_p$ of $T$ at $p$. Moreover, $C_p$ is again a minimizer for the $\alpha$-mass.
\end{lemma}

\begin{remark}
To prove Lemma \ref{p:incriminata}, which is a key step in the regularity theory for optimal branched transport, the author in [\ref{Xiaregularity}, Proposition 3.3] introduces the so-called ``Whitney flat norm". Roughly speaking, the Whitney flat norm mimicks the definition of the flat norm $\mathbb{F}$ as in Definition \ref{d:flatnorm1}, substituting the mass with the $\alpha$-mass. In the proof [\ref{Xiaregularity}, page 290, line 28] the author hints at the possibility to generalize the compactness theorem for area-minimizing currents [\ref{Simonbook}, Theorem 34.5] to the case of equibounded $\alpha$-masses and $\alpha$-masses of the boundaries. This argument seems rather convoluted and would deserve some care in the delicate passages. We remark here that, in fact, Lemma \ref{p:incriminata} is a straightforward application of stability of minimizers, which is valid for any $\alpha \in (0,1)$, see Theorem \ref{t:stabilityCPAM} and [\ref{CDRMcpam}]. \end{remark}

More in general, we can conclude the following interior regularity theorem for tangent cones as a consequence of the stability property of minimizers. This is the fundamental step in order to prove Theorem \ref{t:ultimissimoregbranch}.

\begin{theorem}\label{t:bttgcones}
Let $T=\llbracket E,\tau,\theta \rrbracket$ be a transport path such that $T \in \OTP(\partial T)$. Then for any $p \in \operatorname{supp}(T) \setminus \operatorname{supp}(\partial T)$ there exists a tangent cone $C_p$ of $T$ at $p$ which is a minimizer for the $\alpha$-mass. Moreover we have that \begin{equation}\label{e:formadelcono}C_p \res B_1 = \sum_{i=1}^{k} m_i \llbracket e_i \rrbracket,\end{equation} where \begin{itemize}
\item[$a)$] $k \le C=C(\alpha,d)$,
\item[$b)$]  $e_i$ are segments of the form $\overline{p_i,0}$ for some $p_i$ in the unit $d-1$-sphere $S^{d-1}$.
\item[$c)$]  $m_i$ are real-valued multiplicities satisfying \begin{equation}\label{e:balancingconditions}\sum_{i=1}^{k}m_i=0 \, \text{ and } \sum_{i=1}^{k}\frac{m_i}{m_i^{1-\alpha}}p_i=0.\end{equation}
\end{itemize}
\end{theorem}

\begin{remark}
Theorem \ref{t:bttgcones} ensures that the tangent cone is a finite (real) polyhedral current where the number $k$ of (nonoverlapping) segments is uniformly bounded above by a geometric constant depending only on $\alpha$ and the dimension $d$. Moreover, note that since the point $p$ has been chosen in $\text{supp}(T) \setminus \text{supp}(\partial T)$, then $\sum_{i=1}^{k}m_i=0$ follows immediately. Property $\sum_{i=1}^{k}\frac{m_i}{m_i^{1-\alpha}}p_i=0$ is a ``balancing equation" coming, instead, from a first order condition of the minimal cone $C_p$, relating the $m_i$'s and the vectors $p_i$'s.
\end{remark}

\begin{proof}
The existence of a tangent cone $C_p$ comes from Lemma \ref{p:incriminata}. $C_p$ is optimal by Theorem \ref{t:stabilityCPAM} (or, more generally, Theorem \ref{t:stability_new}) and so, by a standard computation in analogy with Proposition \ref{p:BCMangoli} one can show that the minimal angle between two vectors $e_i$ and $e_j$ is a strictly positive quantity depending only on $\alpha$ (and the dimension $d$). This uniform lower bound forces the number of segments $e_i$'s to be finite. Hence, $C_p$ must be of the form \eqref{e:formadelcono}. By optimality of $C_p$, the balancing conditions \eqref{e:balancingconditions} follow immediately, concluding the proof.
\end{proof}

\begin{remark}
More precisely, to be able to conclude, one would need to rule out the possibility that the vertices are not accumulating in the interior, which would require some other results. For the sake of the exposition, we assume this technical passage to be true. 
\end{remark}

Finally we can state the following interior regularity theorem for optimal transport paths as a corollary of Theorem \ref{t:bttgcones}.

\begin{theorem}[Interior regularity]\label{t:ultimissimoregbranch}
Suppose $T \in \OTP(\partial T)$ is an optimal transport path such that $\mathbb{M}^\alpha(T)<\infty$. For any point $p \in \operatorname{supp}(T) \setminus \operatorname{supp}(\partial T)$ there is an open neighborhood $U$ of $p$ such that $T \res U$ is a polyhedral current.
\end{theorem}

\chapter{Uniqueness results} 

The main goal of Chapter 3 is to present the uniqueness theory for the Plateau's problem and for the optimal branched transport problem. After a brief discussion about the main uniqueness and nonuniqueness theorems for solutions of these two geometric variational problems, we will pass to the most original contributions of this thesis: in Section 3.1 we exploit the regularity theory in higher codimension of Theorem \ref{t:regularityh} to prove that, \textit{generically} (in the sense of Baire categories), every integral $(m-1)$-current without boundary spans a unique minimizer in $\mathbb{R}^{m+n}$. In Section 3.2 we prove the generic uniqueness of minimizers of the optimal branched transport problem.

\section{Generic uniqueness of minimal surfaces}

Arguably, the very innocent question of ``how many minimal surfaces can be spanned by a given closed Jordan curve" turns out to be one of the most challenging questions that can be raised in connection with the Plateau's problem. Indeed, the answer to this question is still not known in full generality and goes back at least to the first decades of the twentieth century, to the works by Rad\'{o}, Courant, Tromba, Nitsche, Tomi and many others, see [\ref{Minimalsurfaces}] for a beautiful survey. Indeed, the problem of uniqueness still deserves some attention even considering the more modest question of asking how many minimal surfaces \textit{of the type of the disk}\footnote{For a precise definition we refer the reader to [\ref{Minimalsurfaces}, Section 4.2, Definition 1].} can be spanned in a given closed Jordan curve $\Gamma$. 
Many examples of minimal surfaces have been developed in the literature that warn us not to expect uniqueness even for disk-type solutions of the Plateau's problem. Hence, we may ask whether additional geometric conditions for $\Gamma$ are known to ensure this uniqueness. We mention here two pioneering uniqueness results. 

\begin{theorem}[Rad\'{o}'s Theorem]\label{t3:rado}{\normalfont [\ref{Rado16}]}
If $\mathit{\Gamma}$ has a one-to-one parallel projection onto a planar convex curve $\gamma$, then $\mathit{\Gamma}$ bounds at most one disk-type minimal surface.
\end{theorem}

\begin{theorem}[Nitsche's Theorem]\label{t3:nitsche}{\normalfont [\ref{Nitsche26}]}
If $\mathit{\Gamma}$ is a real analytic, regular Jordan curve with a total curvature less than or equal to $4 \pi$, then $\mathit{\Gamma}$ bounds only one disk-type minimal surface. Moreover, the solution can be continued analytically across $\mathit{\Gamma}$ as a minimal surface\hspace{0.04cm}\footnote{The unique solution in Nitsche's theorem is not just immersed: it is also embedded.}.
\end{theorem}

To understand why the uniqueness question in the Plateau's problem is so complicated one should consider the following observation. Remarkably enough, Courant outlines in his famous book ``\textit{Dirichlet principle, conformal mappings and minimal surfaces}", see [\ref{Courant}], an argument showing that there might exist a rectifiable Jordan curve $\Gamma$ which bounds \textit{uncountably many} minimal surfaces (of the type of the disk). In fact, Courant did not prove a fundamental step of his construction called \textit{strong bridge theorem}, which roughly tells about the possibility to connect two Jordan curves by a \textit{bridge} consisting of two arcs, giving rise to many disk-type minimal surfaces, see Figures \ref{f:monster1} and \ref{f:monster2}. The validity of his example strictly depends on the validity of this result, which was rigorously proved by White in fairly general and strong versions, see [\ref{Whitebridge1}, \ref{Whitebridge2}]. More recently, Morgan [\ref{Morganinfinite}] gave an example of a smooth curve in $\R^4$ that bounds a whole continuum of (unoriented) area-minimizing analytic manifolds. The author proved the theorem relying on a generalization of the theory of currents called theory of currents modulo $p$, see [\ref{Federerbook}] for a more detailed discussion. Hence, from these examples it is clear that uniqueness is highly unexpected in Plateau-type problems.

\begin{figure}[h]
    \centering
    \includegraphics[width=0.9\textwidth]{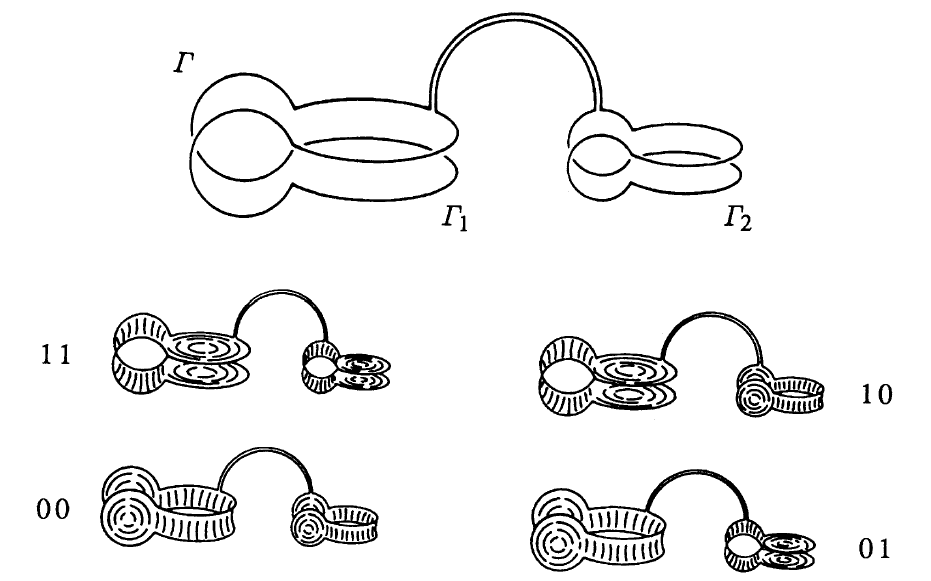}
    \caption{Application of Courant's bridge theorem. (Adapted from [\ref{Minimalsurfaces}]).}
    \label{f:monster1}
\end{figure}

\begin{figure}[h]
    \centering
    \includegraphics[width=0.9\textwidth]{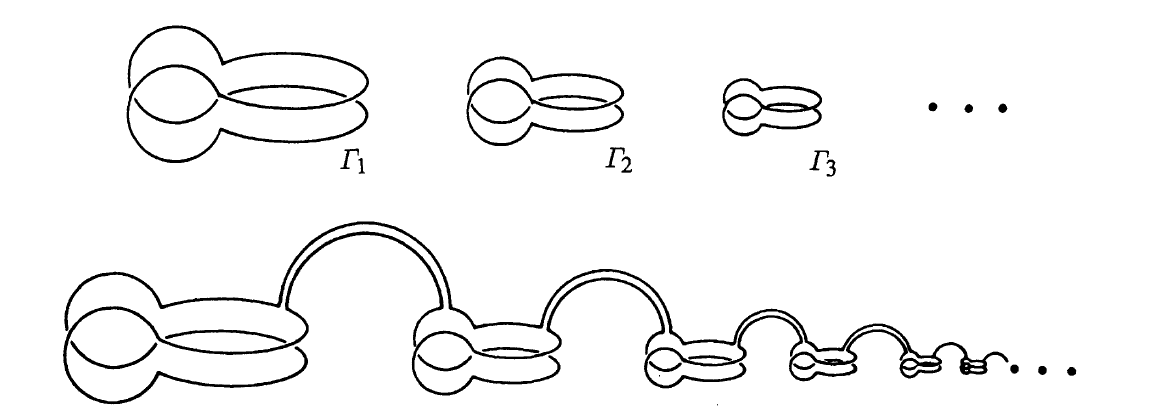}
    \caption{Construction of a curve $\Gamma$ bounding uncountably minimal surfaces. (Adapted from [\ref{Minimalsurfaces}]).}
    \label{f:monster2}
\end{figure}

Consequently, other approaches have been developed to study uniqueness questions. Arguably, the most fruitful was by means of Baire categories. Indeed we state here one of the first satisfactory partial answers to the finitness question by B\"ohme and Tromba, see [\ref{BTromba}].

\begin{theorem}{\normalfont [\ref{BTromba}]}
Generically\hspace{0.045cm}\footnote{We remark that by \textit{generically} the authors mean that there exists an open dense subset of boundaries for which there exists only a finite number of disk-type solutions of the Plateau's problem.}, the number of disk-type solutions of the Plateau's problem is finite.
\end{theorem}

 \begin{remark}
Despite the generic finiteness result by B\"ohme and Tromba, the question whether a (reasonably smooth) curve $\Gamma$ bounds only finitely many disk-type solutions of the Plateau's problem is widely open.
\end{remark}

More recently, Morgan [\ref{Morganinventiones}] proved in the context of geometric measure theory that almost every curve (with respect to a suitable \textit{geometrically meaningful} measure) in $\R^3$ bounds a unique area-minimizing surface. More formally, Morgan considered the space of parametrizations $$\mathscr{C}:= \{f : f \in \mathcal{C}^{2,\alpha}(S^1,\R^3)\},$$ where $S^{1}$ is the unit circle and $\alpha < 1$ equipped with the $\mathcal{C}^2$-norm. A measure on the space of curves $\mathscr{C}$ is constructed by defining a countably infinite product of standardized Gaussian measures in $\R^3$ measuring the Fourier coefficients of the embeddings in $\mathscr{C}$. Almost every such formal series converges uniformly to an element of $\mathscr{C}$ and the resulting measure $\mu$ has the fundamental geometric property that open balls in $\mathscr{C}$ with respect to the $\mathcal{C}^2$-norm are measurable and with positive measure\footnote{Due to this property Morgan's result holds in the sense of Baire's categories as well.}. Morgan's result can be stated as follows:

\begin{theorem}[Morgan's generic uniqueness]{\normalfont [\ref{Morganinventiones}]}\label{t3:morgan}
$\mu$-almost every $b \in \mathscr{C}$ bounds a unique area-minimizing surface.
\end{theorem}

We highlight two key passages of the proof of Theorem \ref{t3:morgan}. The first result is a \textit{unique continuation theorem} assuring that two area-minimizing surfaces with the same boundary which have the same tangent space along a stretch of positive length of the boundary are indeed the same surface. By the interior regularity (in fact, analyticity) of the surfaces, it is enough to prove this step only locally at a boundary point where the two surfaces can be parametrized as graphs of functions $f,g:\R^2 \rightarrow \R$ satisfying the minimal surface equation; in this setting the problem becomes a straightforward application of the theory of elliptic partial differential equations. 

The second fundamental ingredient is a \textit{uniform boundary regularity theorem} for integral currents that allows the author to conclude that if the tangents to two area-minimizing surfaces with the same boundary are ``close together", then the two surfaces are ``close together". The author deeply relies on a result by Allard telling that if an oriented submanifold $\Gamma$ is contained in the boundary of a uniformly convex set, then every boundary point $p \in \Gamma$ is \textit{regular} and has density $1/2$. By \textit{regular} we mean that the support of the current is a regular submanifold with boundary in a neighborhood of the boundary point\footnote{In fact, [\ref{Morganinventiones}, \ref{Morganindiana}] and their later generalization to any dimension and codimension [\ref{MorganARMA}] all rely on a more general theorem developed by Allard, see [\ref{Allardboundary}], telling that (among other things) if $p \in \Gamma$ is a point where the density $\Theta(T, p)$ is $\frac{1}{2}$ (in any dimension and codimension), then $p$ is a regular boundary point.}. 

The main reason why Morgan's result (and his later generalizations, see [\ref{Morganindiana}, \ref{MorganARMA}]) can be improved is that he strongly relies on Allard's assumptions in [\ref{Allardboundary}]. Indeed, what was missing was a boundary regularity theory for higher codimension integral currents without any assumption on the multiplicity of the boundary. Indeed, before the fundamental work by De Lellis, De Philippis, Hirsch and Massaccesi [\ref{Borda}] one could not even exclude in general that the set of regular boundary points was empty, see [\ref{Almgren2000}, Section 5.23] and [\ref{Borda}, Corollary 1.10]. In [\ref{Borda}] the authors prove the first general boundary regularity theorem without any restrictions on the codimension and on the geometry of $\Gamma$, showing that the set of regular boundary points is dense. After [\ref{Borda}], it is reasonable to consider further studies than the ones initiated by Morgan in [\ref{Morganinventiones}, \ref{Morganindiana}, \ref{MorganARMA}].

\subsection{Generic uniqueness of higher codimension area-minimizing currents}

In the remaining part we will prove an original result which is based on a joint work with A. Marchese, see [\ref{caldiniinpreparation}]. We prove that \textit{generically} (in the sense of Baire categories) and with respect to the flat norm, every integral $(m-1)$-current without boundary spans a unique minimizer in $\mathbb{R}^{m+n}.$ The techiniques we adopt have strong analogies with the (more complicated) proof of the generic uniqueness for the optimal branched transport problem we will present in the next section.

Given $K\subset \R^{m+n}$ a compact set we denote 
$$\mathcal{E}(B):= \operatorname{inf}\{\mathbb{M}(T) \,|\,  T \in \mathcal{I}_m(K) : \partial T= B \}.$$
Moreover, we denote the set of solutions to the generalized Plateau's problem with boundary $B$ as
$$\mathbf{AMC}(B):=\{T \in \mathcal{I}_{m}(K): \partial T= B \text{ and } \mathbb{M}(T)=\mathcal{E}(B)\}.$$
We denote the set of $(m-1)$-boundaries by $$\mathcal{B}_{m-1}(K):=\{B \in \mathcal{D}_{m-1}(K) : \partial B=0 \}.$$
Fix an arbitrary $C>0$ and define \begin{equation} \mathcal{A}_C:=\{B \in \mathcal{B}_{m-1}(K) : \mathbb{M}(B)\leq C \text{ and } \mathbb{M}(T) \leq C \text{ for every } T \in \mathbf{AMC}(B)\}.\end{equation}

\begin{lemma}\label{l3:closed}
The space $(\mathcal{A}_C, d_{\,\flat})$ is a nontrivial complete metric space, where $d_{\,\flat}$ is the distance induced by the flat norm $\mathbb{F}_{K}$.
\end{lemma}

We prove that $\mathcal{A}_C$ is $\Flat_K$-closed, then completeness follows from $\mathbb{F}_K$-compactness of integral $(m-1)$-currents (without boundary) with support in $K$ and mass bounded by $C$, see Theorem \ref{t:closure}.

\begin{proof}
Let $(B_j)_{j\in\N}$ be a sequence of elements of $\mathcal{A}_C$ and let $B$ be such that for $j\to\infty$ we have $\mathbb{F}_K(B_j-B)\to 0$. By the lower semicontinuity of the mass (with respect to the flat convergence), we have $\Mass(B)\leq C$. For any $j\in\N$, let $T_j \in \AMC(B_j)$. Note we also have $\mathbb{M}(T_j)\leq C$. By Theorem \ref{t:closure}, there exists $T\in{\mathcal{I}}_{m}(K)$ such that, up to (nonrelabeled) subsequences, $\mathbb{F}_K(T_j-T)\to 0$. By the continuity of the boundary operator we have $\partial T=B$ and by the lower semicontinuity of the mass, we have $\mathbb{M}(T)\leq C$ and hence $B\in \mathcal{A}_C$. 
\end{proof}

The main result we prove can be stated as follows.

\begin{theorem}[Generic uniqueness]\label{t:genericga}
The set of boundaries $B \in \mathcal{A}_C$, for which $\mathbf{AMC}(B)$ is a singleton, is residual\,\footnote{Recall that a set is \textit{residual} if it contains a countable intersection of open dense subsets.}.
\end{theorem}
Consider the following subset of $\mathcal{A}_C$, which denotes the set of boundaries admitting nonunique minimizers:
$$\mathcal{N}\mathcal{U}_C:=\{B\in \mathcal{A}_C: \text{ there exist } T^1, T^2 \in\AMC(B) \text{ such that } T^1\neq T^2\}.$$

\begin{lemma}
Assume that the set $\mathcal{A}_C\setminus \mathcal{N}\mathcal{U}_C$ is $\mathbb{F}_K$-dense in $\mathcal{A}_C$. Then it is residual.
\end{lemma}

\begin{proof}
For $m \in \N \setminus \{0\}$, consider the sets 
$$\mathcal{N}\mathcal{U}_C^m:=\{B\in \mathcal{A}_C: \text{ there exist } \{T^1,T^2\}\subset\AMC(B) \text{ with } \mathbb{F}_K(T^2-T^1)\geq m^{-1}\}.$$
Since $\mathcal{N}\mathcal{U}_C^m\subset \mathcal{N}\mathcal{U}_C$, then $(\mathcal{A}_C\setminus \mathcal{N}\mathcal{U}_C^m)\supset(\mathcal{A}_C\setminus \mathcal{N}\mathcal{U}_C)$ and hence, by assumption, $\mathcal{A}_C\setminus \mathcal{N}\mathcal{U}_C^m$ is $\mathbb{F}_K$-dense in $\mathcal{A}_C$ for every $m$. Therefore $\mathcal{N}\mathcal{U}_C^m$ has empty interior in $\mathcal{A}_C$ for every $m$. By proving that $\mathcal{N}\mathcal{U}_C^m$ is closed for every $m$ we conclude. 

Consider a sequence $(B_j)_{j\in\N}$ of elements of $\mathcal{N}\mathcal{U}_C^m$ and let $B\in \mathcal{A}_C$ be such that $\mathbb{F}_K(B_j-B)\to 0$. For every $j\in\N$, take 
$$\{T^1_j,T^2_j\}\subset\AMC(B_j)\quad {\mbox{ with }}\quad \mathbb{F}_K(T^2_j-T^1_j)\geq 1/m.$$ 
As in the proof of Lemma \ref{l3:closed}, we deduce that there exist $T^1,T^2\in{\mathcal{I}}_{m}(K),$ such that $\partial T^1=\partial T^2=B$ and, up to (nonrelabeled) subsequences, $\mathbb{F}_K(T^1_j-T^1)\to 0$, $\mathbb{F}_K(T^2_j-T^2)\to 0$ as $j\to\infty$. Clearly $\mathbb{F}_K(T^2-T^1)\geq 1/m$. By Theorem \ref{t:closure} and Theorem \ref{t:compactnessforam}, we have $\{T_1, T_2\}\subset\AMC(B)$, hence $B\in \mathcal{N}\mathcal{U}_C^m$.
\end{proof}

To prove Theorem \ref{t:genericga} we are left to prove that the set of boundaries $B \in \mathcal{A}_C$ for which $\AMC(B)$ is a singleton is dense in the metric space $(\mathcal{A}_C,d_{\,\flat}).$

The main idea is to consider the following argument: we first reduce to a polyhedral boundary $P$ by suitably applying Theorem \ref{t:polyapprox}, see Lemma \ref{l:moltisperem}. Then we fix an area-minimizing current $S$ with $\partial S=P$ and, by Theorem \ref{t:regularityh}, we can find $x_0 \in \text{Reg}(S)$. Starting from such interior regular point, a simple perturbation argument proves the density of boundaries with unique minimizer.

\begin{lemma}\label{l:moltisperem}
For any $B \in \mathcal{A}_C$ and $\varepsilon >0$ there exist a $\delta >0$ and a polyhedral boundary $P \in \mathcal{A}_{C-\delta}\cap \mathbb{P}_{m-1}(K)$ such that $$\mathbb{F}_K(B-P)\leq \varepsilon.$$
\end{lemma}

\begin{proof}
Without loss of generality and up to rescaling, we can assume $C=1$ and write $\mathcal{A}$ instead of $\mathcal{A}_C$. Consider $B \in \mathcal{A}$ and $T \in \AMC(B)$. Now define the following rescalings $$T_{\varepsilon}:=(1-\varepsilon/2)T$$ with $B_{\varepsilon}:=\partial T_{\varepsilon}=(1-\varepsilon/2)B$ and note $T_{\varepsilon} \in \AMC(B_{\varepsilon})$ with \begin{equation}\label{e:sperem}\mathbb{M}(T_{\varepsilon}) \leq 1-\varepsilon/2  \, \text{ and } \, \mathbb{M}(B_{\varepsilon}) \leq 1-\varepsilon/2.\end{equation} Since we have $\mathbb{M}(B-B_{\varepsilon})\leq  \varepsilon/2$ then we also have \begin{equation}\label{e:sperem2}\mathbb{F}_{K}(B-B_{\varepsilon})\leq  \varepsilon/2.\end{equation}
By Theorem \ref{t:polyapprox} there exists a polyhedral $m$-current $T_{P} \in \mathbb{P}_m(K)$ such that, denoting $P:= \partial T_{P}$, we have \begin{equation}\label{e:sperem3}\mathbb{F}_K(T_{P}-T_{\varepsilon})\leq \varepsilon/4, \quad \mathbb{M}(T_{P})\leq  \mathbb{M}(T_{\varepsilon}) +\varepsilon/4  \, \text{ and }\, \mathbb{M}(P) \leq \mathbb{M}(B_{\varepsilon})+\varepsilon/4.\end{equation} In particular we get \begin{equation}\label{e:sperem4}\mathbb{F}_{K}(P-B_{\varepsilon})\leq \varepsilon/4.\end{equation}
By \eqref{e:sperem3} and \eqref{e:sperem} we get \begin{equation}\label{e:sperem5}\mathbb{M}(T_P) < 1.\end{equation}
Combining \eqref{e:sperem2} and \eqref{e:sperem4} we get $\mathbb{F}_{K}(P-B)\leq\varepsilon$ and by \eqref{e:sperem3} and \eqref{e:sperem} we conclude that \begin{equation}\label{e:sperem6}\mathbb{M}(P)\leq 1- \varepsilon/4 <1.\end{equation} Hence by \eqref{e:sperem5} and \eqref{e:sperem6} we have $P \in \mathcal{A}_{1-\delta}\cap \mathbb{P}_{m-1}(K)$ for some $\delta >0$.
\end{proof}

Now consider an area-minimizing $m$-current $S$ such that $\partial S=P$, which exists by Theorem \ref{t:plateausolution}, and fix $x_0 \in \text{Reg}(S)$. 

\begin{remark}Note that by Theorem \ref{t:regularityh} and by the fact $\partial S \in \mathbb{P}_{m-1}(K)$ we have $\text{Reg}(S) \neq \emptyset$. Indeed $P$ is of the form \eqref{ee:poly}, hence $\text{supp}(P)$ cannot be dense in $\text{supp}(S)$ since it is made by a finite number of $(m-1)$-dimensional simplexes. \end{remark}

Recall the definition of interior regular point: we say that $x_0 \in \operatorname{supp}(S) \setminus \operatorname{supp}(\partial S)$ is an interior regular point if there is a positive radius $\overline{r}>0$ and a ball $\mathbf{B}_{\overline{r}}(x_0) \subset \R^{n+m}$, a smooth embedded submanifold $\Sigma \subset \R^{n+m}$ and a positive integer $Q$ such that $T \res \mathbf{B}_{\overline{r}}(x_0)=Q \llbracket \Sigma \rrbracket$. 

\begin{proof}[Proof (Theorem \ref{t:genericga})]

For some radius $r$ such that $r < \overline{r}$, define $$S':= S - S \res \mathbf{B}_{r}(x_0)  \, \text{ and }\, B':= \partial S'.$$ Note further that by Lemma \ref{l:moltisperem}, for $r$ small enough, $B' \in \mathcal{A}_C$. Note that $S' \in \AMC(B')$, since otherwise we would find $\tilde S$ with $\partial\tilde S=B'$ and $\mathbb{M}(\tilde S)<\mathbb{M}(S')$. This would lead to the contradiction $\partial (\tilde S + S \res \mathbf{B}_{r}(x_0))=P$ but $\mathbb{M}(\tilde S + S \res \mathbf{B}_{r}(x_0))< \mathbb{M}(S).$ Hence, if we are able to prove that $\AMC(B') = \{S'\}$, then we prove $\mathcal{A}_C\setminus \mathcal{N}\mathcal{U}_C$ is $\mathbb{F}_K$-dense in $\mathcal{A}_C$. 

Suppose there exists $S'' \in \AMC(B')$ such that $S'\neq S''$. Denote $\hat S:=S''+ S \res \mathbf{B}_{r}(x_0)$. By the minimality of $S$ one immediately sees that $\text{supp}(\hat S)\supset \text{supp} (S)\cap\mathbf{B}_{r}(x_0)$. By Theorem \ref{t:regularityh} there exists $x_1 \in \partial\mathbf{B}_{r}(x_0) \cap \text{Reg}(S) \cap \text{Reg}(\hat S)$. For a sufficiently small radius $\rho$ such that $\rho<\text{dist}(x_1,\partial \mathbf{B}_{\overline{r}})$, we can write $$S\res \mathbf{B}_{\rho}(x_1)= Q_1\llbracket \Sigma_1 \rrbracket \res \mathbf{B}_{\rho}(x_1) \, \text{ and } \, \hat S\res \mathbf{B}_{\rho}(x_1)= Q_2\llbracket \Sigma_2 \rrbracket \res \mathbf{B}_{\rho}(x_1).$$ By the same argument, the tangents to $\Sigma_1$ and $\Sigma_2$ coincide on $\text{supp} (S)\cap\mathbf{B}_{r}(x_0)$ locally around $x_1$. By a unique continuation argument for the minimal surface system, see [\ref{MorganARMA}, Lemma 7.2], the two submanifolds $\Sigma_1, \Sigma_2$ must coincide locally around $x_1$. By Theorem \ref{t:regularityh} this equivalence can be extended to the full regular parts of $S$ and $\hat S$. This easily implies the contradiction $\hat S= S$.
\end{proof}

The initial approximation argument by means of polyhedral currents is motivated by the following remark. 
\begin{remark}
Given an integral $(m-1)$-current without boundary $T \in \mathcal{I}_{m-1}(\R^{m+n})$ and an area-minimizing $m$-current $S$ such that $T= \partial S$, it is not possible to conclude Reg$(T) \neq \emptyset$. In other words, it is possible that the interior regularity theorems may be empty theorems, as the following example shows.
\end{remark}

\begin{example}\label{e:bordobrutto}
In $\R^2$, consider a sequence of positive real numbers $(r_j)_j$ such that $\sum_jr_j < \infty$ and a sequence $(q_j)_j$ that is dense in $\R^2$. Consider the balls $B_{r_j}(q_j)$ and the $2$-current defined as $$T:= \sum_{j \in \N}\llbracket B_{r_j}(q_j)\rrbracket.$$ Note that $$\partial T= \partial \left(\sum_{j \in \N}\llbracket B_{r_j}(q_j)\rrbracket\right) =\sum_{j \in \N}\partial\llbracket B_{r_j}(q_j)\rrbracket$$ since the intersection between two circumferences with different centers has $\mathcal{H}^1$-measure equal to zero. Moreover it is easy to check that $$\mathbb{M}(T)= \pi \sum_j r_j^2 <\infty \,\text{ and }\, \mathbb{M}(\partial T)= 2\pi \sum_j r_j<\infty.$$ Fix $x \in \R^2$, $\varepsilon >0$ arbitrary and consider the ball $B_{\varepsilon}(x)$. Then there exists $q_{j_0} \in B_{\varepsilon}(x)$ such that $r_{j_0} < \text{dist}(q_{j_0}, \partial B_{\varepsilon}(x)).$ Hence, $\text{supp}(\partial T) \cap B_{\varepsilon}(x) \neq \emptyset$, showing that $\text{supp}(\partial T)=\R^2$. Since $\partial T$ is integral, by Theorem \ref{t:plateausolution} there exists an area-minimizing current $T_1$ such that $\partial T_1 = \partial T$ but now $$\text{Reg}(T_1)\subset \text{supp}(T_1) \setminus \text{supp}(\partial T)= \emptyset.$$ Note further that the interior regularity theorem for area-minimizing currents is still valid, but it trivializes to an empty theorem since $\text{Reg}(T_1)=\emptyset$. The same argument can be easily generalized in any dimension. \end{example}

\section{Generic uniqueness of optimal transport paths} 

This section is based on a joint work with A. Marchese and S. Steinbr\"uchel, see [\ref{caldini}]. 

It is well-known that there are boundaries $b$ such that $\OTP(b)$ contains more than one element of finite $\alpha$-mass; for instance one can exhibit a nonsymmetric minimizer $T$ for which $\partial T$ is symmetric, so that the network $T'$ symmetric to $T$ is a different minimizer (see Figure \ref{f:1}).

\begin{figure}[ht]
	\centering 
    \includegraphics[width=0.5\textwidth]{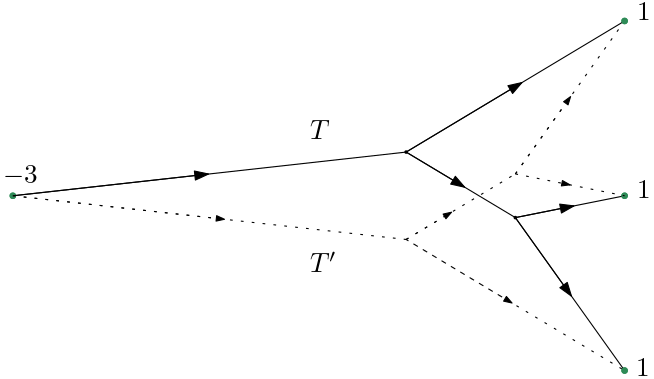}
    \caption{The boundary $\partial T$ is symmetric (with respect to the horizontal axis) and $T\in\OTP(\partial T)$ is not symmetric, hence the symmetric copy $T'$ is a different minimizer.}\label{f:1}
\end{figure}

We prove that for the generic boundary, in the sense of Baire categories, there exists a unique minimizer of the associated optimal branched transport problem. Slightly improving upon the main result of [\ref{CDRMcpam}], see Theorem \ref{t:stability_new}, we advance on the study of the well-posedness properties of the optimal branched transport problem, as we establish the first result on the generic uniqueness of minimizers and in full generality, namely in every dimension $d$ and for every exponent $\alpha\in(0,1)$. Prior to our work, we are aware of only one elementary result on the uniqueness of minimizing networks. It appeared in the original paper by Gilbert [\ref{Gilbert}] and says that there exists at most one discrete \emph{minimum cost communication network} with a given \textit{Steiner topology}.

Recall we denoted the set of boundaries by
\[ {\mathscr{B}}_{0}(K) := \{b\in{\mathcal{D}}_{0}(K): \text{there is an } S \in {\mathcal{D}}_{1}(K) \text{ with } \partial S= b\}\,, \]
we fixed an arbitrary constant $C>0$ and we defined 
\begin{equation}\label{def_AC}
    A_C:=\{b\in{\mathscr{B}}_{0}(K): \mathbb{M}(b)\leq C\;{\rm{and}}\;\MM(T)\leq C\mbox{ for every }T\in\OTP(b)\}.
\end{equation}
We metrize $A_C$ with the \emph{flat norm} $\Flat_K$ and we observe that the set $A_C$ endowed with the induced distance is a nontrivial complete metric space, see Lemma \ref{l:closed}.
Our main result is the following.

\begin{theorem}[Generic uniqueness of optimal transport paths]\label{t:main}
    The set of boundaries $b\in A_C$, for which $\OTP(b)$ is a singleton, is residual.
\end{theorem}

Several variants and generalizations of the optimal branched transport problem were proposed and studied by many authors in recent years, see for instance [\ref{BF}, \ref{BrBuSa}, \ref{BW}, \ref{BW1}, \ref{brabutsan}, \ref{BrGaReSun}, \ref{BrPaSun}, \ref{BrSun}, \ref{CDRMjmpa}, \ref{MMST}, \ref{MMT}, \ref{PaoliniStepanov}]. For the sake of simplicity, we prove the generic uniqueness of minimizers only for the Eulerian formulation introduced in [\ref{Xia2003}].

\subsubsection{Strategy of the proof}
Using Theorem \ref{t:stability_new}, we show that in order to prove Theorem \ref{t:main}, it suffices to prove the \emph{density} of the set of boundaries $b\in A_C$ for which $\OTP(b)$ is a singleton, see Lemma \ref{l:residual}. A similar reduction principle is used in [\ref{Morganinventiones}, \ref{Morganindiana}, \ref{MorganARMA}] to prove that the generic (higher dimensional) boundary spans a unique solution to the Plateau's problem. 

The proof of the density result is based on the following perturbation argument. Firstly, we prove that we can reduce to a finite atomic boundary $b$ whose multiplicities are integer multiples of a fixed positive number, exploiting the fact that such boundaries are dense in $A_C$, see Lemma \ref{l:integral_bdry}. For these boundaries, we prove that the solutions to the optimal branched transport problem are multiples of polyhedral integral currents, see Lemma \ref{l:integral_minimizers}. Then we improve the uniqueness result of [\ref{Gilbert}] to suit the discrete optimal branched transport problem, obtaining as a byproduct that for every finite atomic boundary $b$ as above the set $\OTP(b)$ is finite, see Lemma \ref{l:finiteminimizers}. We deduce the existence of a set of points $\{p_1,\dots,p_h\}$ in the regular part of the support of a fixed transport path $T\in\OTP(b)$ with the property that $T$ is the only element in $\OTP(b)$ whose support contains $\{p_1,\dots,p_h\}$, see Lemma \ref{l:magic_points}. 

Next, we aim to ``perturb" the boundary $b$ close to the points $p_1,\dots,p_h$ in order to obtain boundaries with unique minimizers, keeping in mind the fact that the perturbed boundaries should not escape from the set $A_C$. More in detail, we define a sequence $(b_n)_{n\geq 1} \in A_C$ of boundaries for the optimal branched transport problem with the property that $\Flat_K(b_n-b)\to 0$ as $n\to\infty$. Moreover, each $b_n$ has points of its support (with small multiplicity) in proximity of $p_1,\dots,p_h$, so that every minimizing transport path $S_n$ with boundary $\partial S_n=b_n$ is forced to have such close-by points in its support. Exploiting again the stability property of Theorem \ref{t:stability_new}, we deduce that for every choice of $S_n\in\OTP(b_n)$ there exists $S\in\OTP(b)$ such that, up to subsequences, it holds $\Flat_K(S_n-S)\to 0$ and we can infer the Hausdorff convergence of the supports of the $S_n$'s to the union of the support of $S$ and the points $p_1,\dots,p_h$, see Lemma \ref{l:conv_hauss}. Notice that at this stage we cannot deduce from Lemma \ref{l:magic_points} that $S=T$, since the portion of $S_n$ which is in proximity of some of the $p_i$'s might vanish in the limit. In order to exclude this possibility, we perform a fine analysis of the structure of the network $S_n$ around the points $p_1,\dots, p_h$, see Section \ref{Step2}: this allows us to exclude all possible local topologies except for two, see \eqref{e:THE_ONE}, proving that $p_1,\dots,p_h$ are contained in the support of $S$ (so that in particular $S=T$ by Lemma \ref{l:magic_points}) and that $\OTP(b_n)=\{S_n\}$, for $n$ sufficiently large, see Lemma \ref{l:main}, which concludes the proof of Theorem \ref{t:main}.

\begin{remark}
It is much easier to prove just density in $\bigcup_{C>0}A_C$ of the boundaries $b$ for which $\OTP(b)$ is a singleton. Indeed, it is significantly simpler to perform the strategy outlined above if one is allowed to choose $b_n$ simply satisfying $\Flat_K(b_n-b)\to 0$ and $\MM(S_n)\leq C$, but possibly with $\Mass(b_n)>C$: for instance it suffices to choose the perturbation $b_n$ as in \eqref{e:bienne} with $k=1$, in which case it is easy to prove that $\OTP(b_n)$ is a singleton. Obviously such type of perturbation is not admissible in order to prove the residuality result of Theorem \ref{t:main}, since such boundaries $b_n$ do not belong to $A_C$. One of the challenges in our proof is therefore to find suitable perturbations $b_n$ of $b$ which are internal to the set $A_C$ and such that for the boundary $b_n$ there exists a unique minimizer of the optimal branched transport problem, for $n$ sufficiently large.
\end{remark}

\subsubsection{Preliminaries}\label{s:preliminaries}
Through the section $K\subset$ $\R^d$ denotes a convex compact set. Let $A_C$ be the set of boundaries defined in \eqref{def_AC}. Due to the Baire category theorem, the next lemma ensures that a residual subset of $A_C$ is dense. 
\begin{lemma}\label{l:closed}
The set $A_C$ is $\mathbb{F}_K$-closed. In particular $(A_C,d_{\flat})$ is a complete metric space.
\end{lemma}
\begin{proof}
The second part of the statement follows from the first part and from the $\mathbb{F}_K$-compactness of 0-currents with support in $K$ and mass bounded by $C$, see [\ref{Federerbook}, 4.2.17].

In order to prove that $A_C$ is $\Flat_K$ closed, let $(b_j)_{j\in\N}$ be a sequence of elements of $A_C$ and let $b$ be such that $\mathbb{F}_K(b_j-b)\to 0$ as $j\to\infty$. We want to prove that $b\in A_C$. By the lower semicontinuity of the mass (with respect to the flat convergence), we have $\Mass(b)\leq C$. For any $j\in\N$, let $T_j \in \OTP(b_j)$. By [\ref{CDRMcalcvar}, Proposition 3.6], we have $\mathbb{M}(T_j)\leq C^{1-\alpha}\mathbb{M}^\alpha(T_j)\leq C^{2-\alpha}$.
By the compactness theorem for normal currents, there exists $T\in{\mathbb{N}}_{1}(K)$ such that, up to (nonrelabeled) subsequences $\mathbb{F}_K(T_j-T)\to 0$. By the continuity of the boundary operator we have $\partial T=b$ and by the lower semicontinuity of the $\alpha$-mass, see [\ref{CDRMS}], we have $\mathbb{M}^\alpha(T)\leq C$ and hence $b\in A_C$. 
\end{proof}

Consider the following subset of $A_C$, which represents the set of boundaries admitting nonunique minimizers:
$$NU_C:=\{b\in A_C: \exists\,\, T^1, T^2 \in\OTP(b) \text{ such that } T^1\neq T^2\}.$$
Notice that since $b\in A_C$ then $\MM(T^1)=\MM(T^2)\leq C$. We have the following.

\begin{lemma}\label{l:residual}
Assume that the set $A_C\setminus NU_C$ is $\mathbb{F}_K$-dense in $A_C$. Then it is residual.
\end{lemma}
\begin{proof}
For $m \in \N \setminus \{0\}$, consider the sets 
$$NU_C^m:=\{b\in A_C: \exists\, \{T^1,T^2\}\subset\OTP(b)\;{\rm with}\;\mathbb{F}_K(T^2-T^1)\geq m^{-1}\}.$$
Since $NU_C^m\subset NU_C$, then $(A_C\setminus NU_C^m)\supset(A_C\setminus NU_C)$ and hence, by assumption, $A_C\setminus NU_C^m$ is $\mathbb{F}_K$-dense in $A_C$ for every $m$. Therefore $NU_C^m$ has empty interior in $A_C$ for every $m$.

To conclude, it is sufficient to prove that $NU_C^m$ is closed for every $m$. Consider a sequence $(b_j)_{j\in\N}$ of elements of $NU_C^m$ and let $b\in A_C$ be such that $\mathbb{F}_K(b_j-b)\to 0$. We need to prove that $b\in NU_C^m$. For every $j\in\N$, take 
$$\{T^1_j,T^2_j\}\subset\OTP(b_j)\quad {\mbox{ with }}\quad \mathbb{F}_K(T^2_j-T^1_j)\geq m^{-1}.$$ 
As in the proof of Lemma \ref{l:closed}, we deduce that there exist $T^1,T^2\in \mathbb{N}_{1}(K),$ such that $\partial T^1=\partial T^2=b$ and, up to (nonrelabeled) subsequences, $\mathbb{F}_K(T^1_j-T^1)\to 0$ and $\mathbb{F}_K(T^2_j-T^2)\to 0$ as $j\to\infty$. Clearly $\mathbb{F}_K(T^2-T^1)\geq m^{-1}$. By Theorem \ref{t:stability_new}, we have $\{T_1, T_2\}\subset\OTP(b)$, hence $b\in NU_C^m$.
\end{proof}

\subsubsection{Preliminary reductions}

\begin{lemma}\label{l:integral_bdry}
    For any $b \in A_C$ and $\eps >0$, there exist $\delta>0$ and a boundary $b'' \in A_{C-\delta}$ with
    \[ \Flat_K(b-b'') < \eps  \qquad \text{ and } \qquad
    b'' = \eta b_I\]
    for some $\eta >0$ and $b_I\in\I_0(K)$.
\end{lemma}
\begin{proof}
Without loss of generality and up to rescaling, we can assume $C=1$ and write $A$ instead of $A_C$. Let $b \in A$ and $T\in\OTP(b)$ and define $T_{\varepsilon}:= (1-\varepsilon/4) T$. Then $b_{\varepsilon}:=\partial T_{\varepsilon}=(1-\varepsilon/4)b$ and $T_{\varepsilon}\in\OTP(b_{\varepsilon})$ satisfy 
\begin{equation}\label{e:estTepsbeps}
\mathbb{M}^{\alpha} (T_{\varepsilon})\le (1 - \varepsilon/4)^{\alpha}\quad\mbox{ and }\quad \Mass(b_\varepsilon)\leq 1 - \varepsilon/4.  
\end{equation}
Since we also have $\mathbb{M}(b-b_{\varepsilon}) \le \varepsilon/4$, we deduce that 
\begin{equation}\label{e:estflatbeps}
\mathbb{F}_K(b-b_{\varepsilon}) \le \varepsilon/4.    
\end{equation}
Now apply a polyhedral approximation theorem by combinig [\ref{MW}, Lemma 9] and [\ref{Chambolle}, Theorem 1.2] to obtain, possibly after rescaling, a current $T'_\eps \in\Po_1(K)$, such that, denoting $b'_\eps:=\partial T'_\eps$, we have 
\begin{equation}\label{e:estTprimo}
\mathbb{M}^{\alpha}(T'_\eps)\leq \mathbb{M}^{\alpha}(T_{\varepsilon}),\quad\mathbb{M}(b'_\eps) \leq \mathbb{M}(b_{\varepsilon})\quad\mbox{and}\quad \Flat_K(T'_\eps-T_\eps)\leq \varepsilon/4,    
\end{equation}
and in particular $\Flat_K(b'_\eps-b_\varepsilon)\leq \varepsilon/4$. We can write 
$$T'_\eps = \sum_{i=1}^{N} \theta_i' \llbracket\sigma_i\rrbracket$$ as in \eqref{ee:poly}. Up to changing the orientation of $\llbracket\sigma_i\rrbracket$, we may assume $\theta'_i>0$ for every $i$. Fix $\eta := \varepsilon/(16N)$ and denote 
$\theta_i'' := \eta  \left\lfloor \frac{\theta_i'}{\eta} \right\rfloor$ (where $\lfloor x\rfloor$ is the largest integer smaller than or equal to $x$) so that
\begin{equation}\label{e:appx}
     0\leq\theta_i'-\theta_i''< \frac{\varepsilon}{16N} \qquad\mbox{for every $i \in \{1,\dots,N\}$}.
      \end{equation} Define $$T'':= \sum_{i=1}^{N}\theta_i'' \llbracket\sigma_i\rrbracket$$ and denote $b''=\partial T''$. Observe that by \eqref{e:appx} and \eqref{e:estTprimo} we have 
\begin{equation}\label{e:alphamasstsecondo}
    \MM(T'')\leq \MM(T'_\eps)\leq \MM(T_\varepsilon)< 1.
\end{equation} 
      
For every $i\in \{1,\dots,N\}$, we denote by $x_i$ and $y_i$ respectively the first and second endpoint of the oriented segment $\sigma_i$, so that we can write 
$$b'_\eps= \sum_{i=1}^{N} \theta_i' (\delta_{y_i} - \delta_{x_i})$$ 
which we can rewrite as 
$$b'_\eps= \sum_{j=1}^{M}\beta_j' \delta_{z_j},$$
where all points $z_j$ are distinct and
$$\beta_j':= \left( \sum_{\{i: y_i = z_j \}} \theta_i' - \sum_{\{i: x_i = z_j\}} \theta_i'\right).$$
Analogously, we define $$\beta_j'':= \left( \sum_{\{i: y_i = z_j \}} \theta_i'' - \sum_{\{i: x_i = z_j\}} \theta_i''\right),$$
so that we can write
$$b''= \sum_{j=1}^{M}\beta_j'' \delta_{z_j}.$$
Thus we obtain
\begin{equation}\label{e:massb''b'}
\begin{split}
\Mass(b''-b'_\eps ) = \sum_{j=1}^{M}|\beta_j'' - \beta_j'|
\leq \Bigg|\sum_{\{i: x_i = z_j \}} (\theta_i''- \theta_i')\Bigg| + \Bigg|\sum_{\{i: y_i = z_j\}} (\theta_i''- \theta_i') \Bigg| \stackrel{(\ref{e:appx})}< \frac{\eps}{16} + \frac{\eps}{16} = \frac{\eps}{8}
\end{split}
\end{equation} 
and by \eqref{e:estTepsbeps} and \eqref{e:estTprimo} we deduce
\begin{equation}\label{e:massbsecondo}
    \Mass(b'') \leq \Mass(b'_\eps) +\Mass(b''-b'_\eps) <1. 
\end{equation}
Combining \eqref{e:massb''b'} with \eqref{e:estflatbeps} and \eqref{e:estTprimo}, we get
$\mathbb{F}_K(b-b'') < \varepsilon$. The conclusion follows denoting $b_I:=\eta^{-1}b''$ and observing that $b_I\in\I_0(K)$ (as the $\theta''_i$ are multiples of $\eta$) and that by \eqref{e:alphamasstsecondo} and \eqref{e:massbsecondo} we have $b''\in A_{1-\delta}$ for some $\delta>0$.
\end{proof}

\begin{lemma}
\label{l:integral_minimizers}
If $b\in\I_0(K)$ and $T\in\mathbb{N}_1(K)$ is in $\OTP(b)$ then $T\in \Po_1(K)\cap \I_1(K)$.
\end{lemma}
\begin{proof}
Combining the \emph{good decomposition} properties of optimal transport paths, see Proposition \ref{pp:gooddec} and their \emph{single path property}, see Proposition \ref{p:singlepath} with the assumption $\partial T\in\Po_0(K)$, we deduce that there are finitely many Lipschitz simple paths $\gamma_1,\dots,\gamma_N$ of finite length such that $T$ can be written as a $T=\sum_{i=1}^N a_i\llbracket\gamma_i\rrbracket$, where $a_i>0$ for every $i$ and $\llbracket\gamma_i\rrbracket\in\I_1(K)$ is the current $({\rm{Im}}(\gamma_i), \gamma_i'/|\gamma_i'|, 1)$. Moreover, again by Proposition \ref{p:singlepath}, one can assume that ${\rm{Im}}(\gamma_i)\cap {\rm{Im}}(\gamma_j)$ is connected for every $i,j$, which in turn implies that $T\in\Po_1(K)$. Hence we can write $$T:=\sum_{\ell=1}^N\theta_\ell\llbracket\sigma_\ell\rrbracket,$$
where $\sigma_\ell$ are non-overlapping oriented segments and $\theta_\ell\in\R$. We want to prove that $\theta_\ell\in\Z$, for all $\ell$.

Denote $$\mathscr{I}:=\{\ell\in\{1,\dots,N\}:\theta_\ell\in\R\setminus\Z\}$$
and let $\hat T:=\sum_{\ell\in\mathscr{I}}\theta_\ell\llbracket\sigma_\ell\rrbracket$. Assume by contradiction that $\hat T\neq 0$.
Note that $T-\hat T\in\I_1(K)$ and therefore, since $b\in\I_0(K)$, we have $\partial\hat T=b-\partial(T-\hat T)\in\I_0(K)$. Hence, for every point $x$ in the support of $\partial\hat T$ there are at least two distinct segments $\sigma_{\ell_1}$ and $\sigma_{\ell_2}$ having $x$ as an endpoint. This implies that the support of $\hat T$, and in particular also the support of $T$, contains a loop, which contradicts [\ref{BCM}, Proposition 7.8]\footnote{This proposition can be considered as the continuous analogue of Proposition \ref{p:tree}.}.
\end{proof}
\subsection{Finiteness of the set of minimizers for integral boundaries}
\begin{definition}\label{d:topo}
Let $b\in\I_0(K)$ and let $T, T'\in\Po_1(K)$ with $\partial T=\partial T'=b$. We say that $T$ and $T'$ have the same \emph{topology} if there exist two ordered sets, each made of distinct points, $\{x_1,\dots,x_M\}$ and $\{x'_1,\dots,x'_M\}$ with the following properties:
\begin{itemizeb}
\item[(i)] for every $p\in\supp(b)$ there exists $i$ such that $x_i=p=x'_i$;
\item[(ii)] denoting $\sigma_{ij}$ the segment with first endpoint $x_i$ and second endpoint $x_j$ and $\sigma'_{ij}$ the segment with first endpoint $x'_i$ and second endpoint $x'_j$, $T$ and $T'$ can be written respectively as 
\begin{equation}\label{e:def_topo}
T=\sum_{i<j}a_{ij}\llbracket\sigma_{ij}\rrbracket,\quad T'=\sum_{i<j}a'_{ij}\llbracket\sigma'_{ij}\rrbracket, \quad \mbox{ for some $a_{ij}, a'_{ij}\in\R$.}    
\end{equation}
\item[(iii)] the representations in \eqref{e:def_topo}, restricted to the nonzero addenda, are of the same type as \eqref{ee:poly}. In particular, if $a_{ij}$ and $a_{kl}$ (respectively $a'_{ij}$ and $a'_{kl}$) are nonzero, then $\sigma_{ij}$ and $\sigma_{kl}$ (respectively $\sigma'_{ij}$ and $\sigma'_{kl}$) have disjoint interiors. Moreover, the number of nonzero addenda in the representation of $T$ (respectively $T'$) given in \eqref{e:def_topo} coincides with the smallest number $N$ for which $T$ (respectively $T'$) can be written as in \eqref{ee:poly}.
\item [(iv)]$a_{ij}=0$ if and only if $a'_{ij}=0$. In particular, the number $N$
of the previous point is the same for $T$ and $T'$.
\end{itemizeb}
One can check that the above conditions define an equivalence relation on the set of polyhedral currents. We call the \emph{topology} of a polyhedral current $T$ the corresponding equivalence class. Notice that the number $M$ depends only on the equivalence class and for every $T$ the (unordered) set $\{x_1,\dots,x_M\}$ is uniquely determined, by property (iii). The set $\{x_1,\dots,x_M\}\setminus \supp(b)$ is called the set of \emph{branch points} of $T$ and denoted by $\BR(T)$. By Lemma \ref{l:integral_minimizers}, for every $T\in\OTP(b)$ the topology of $T$ and the set $\BR(T)$ are well-defined. 
\end{definition}

\begin{lemma}\label{l:numberBranch}
	Let $0\neq b\in\I_0(K)$ and $T \in \OTP(b)$. Then $ \Haus^0(\BR(T))\leq \Haus^0(\supp(b))-2$.
\end{lemma}
\begin{proof}
	Suppose without loss of generality that $\Haus^0(\BR(T))>0$. Assume by contradiction that the lemma is false and let $n$ be the minimal number such that there exist $b\in\I_0(K)$ and $T \in \OTP(b)$ such that $$\Haus^0(\BR(T))+2> n=\Haus^0(\supp(b)).$$ 
	Notice that $n>2$. 
	Fix $p \in \BR(T)$ and let $\varepsilon>0$ be such that $$(\overline B_{\varepsilon}(p)\setminus\{p\})\cap(\supp(b)\cup\BR(T))=\emptyset.$$ Denote by $T_1, \dots, T_m$ the restriction of $T$ to the connected components of $\supp(T)\setminus B_\varepsilon(p)$. We notice that $m \geq 3$. Indeed, if $m=1$ we would have the contradiction $p\in\supp(b)$ and if $m=2$, writing $T$ as in \eqref{e:def_topo}, the only two segments with nonzero coefficient having $p$ as an endpoint cannot be collinear by property (iii): this contradicts the the fact that $T\in\OTP(b)$. Observe that for every $i$ we have that $\supp(\partial T_i)\setminus \supp(b)$ consists of exactly one point $p_i$, so that
	\begin{equation}\label{b}
 n = \sum_{i=1}^m \left( \Haus^0(\supp(\partial T_i)) - 1 \right).
 \end{equation}
By minimality of $n$ and the fact that $m \geq 3$, we have \begin{equation}\label{i_s} \Haus^0(\BR(T_i)) \leq \Haus^0(\supp(\partial T_i)) -2 \quad \text{ for all $i \in \{1,\dots ,m\}$}. \end{equation} Since $m\geq 3$, the combination of \eqref{b} and \eqref{i_s} leads to a contradiction.
\end{proof}

\begin{lemma}\label{l:topo_and_support}
Let $b\in\I_0(K)$ and let $T,T'\in\Po_1(K)$ with $\partial T = b =\partial T'$ have the same topology. Assume moreover that $\supp(T)$ and $\supp(T')$ do not contain loops. Write $T$ and $T'$ as in \eqref{e:def_topo} with properties (i)-(iv) and with the same orientation on each segment. Then $a_{ij}=a'_{ij}$ for every $i,j$.
\end{lemma}

\begin{proof}
By contradiction, let $T,T'$ be nonzero currents with the same topology, $\partial T = b =\partial T'$, and minimizing the quantity $M$ in Definition \ref{d:topo} among all pairs for which the lemma is false. We claim that there exists a point $p\in\supp(b)$ and (up to reordering) indexes $i,j\in\{1,\dots,M\}$ such that
\begin{itemizeb}
\item [(a)] $a_{lj}=0=a'_{lj}$ for every $l\neq i$;
\item [(b)] $x_j=p=x'_j$ and $a_{ij}\neq a'_{ij}$, with $a_{ij},a'_{ij}\in\R\setminus\{0\}$. 
\end{itemizeb}
The validity of (a) follows from the absence of loops. On the other hand, if a point $p$ as in (a) violated (b), one could restrict the currents $T$ and $T'$ respectively to the complementary of $\sigma_{ij}$ and $\sigma'_{ij}$, thus contradicting the minimality of $M$.
The validity of (a) and (b) is a contradiction because the multiplicities $a_{ij}$ and $a'_{ij}$ correspond to the multiplicity of $p$ as point in the support of $b$. 
\end{proof}

\begin{lemma}\label{l:supports_determine_current}
Let $b\in\I_0(K)$ and $S,T\in\OTP(b)$ with $\supp(S)=\supp(T)$. Then $S=T$. 
\end{lemma}
\begin{proof}
Assume by contradiction $S\neq T$. By Lemma \ref{l:integral_minimizers}, $S-T\in\Po_1(K)\cap\I_1(K)$ is a nontrivial current with $\partial(S-T)=0$ and by assumption $\supp(S-T)\subset \supp(S)$. As in the proof of Lemma \ref{l:integral_minimizers} we deduce that $\supp(S-T)$ contains a loop. In particular, so does $\supp(S)$, which contradicts [\ref{BCM}, Proposition 7.8].
\end{proof}

\begin{lemma}\label{l:finiteminimizers}
    Let $b\in\I_0(K)$ be a boundary. Then $\OTP(b)$ is finite.
\end{lemma}
\begin{proof}
By Lemma \ref{l:numberBranch} the range of the integer $M$ of Definition \ref{d:topo} among all $T\in\OTP(b)$ is finite. In turn this implies that the set of possible topologies of currents $T\in\OTP(b)$ is finite. Indeed the topology of a polyhedral current $T$ as in Definition \ref{d:topo}, up to choosing the order of the points $\{x_1,\dots, x_M\}$, is uniquely determined by the $M\times M$ matrix $A:=(|{\rm{sign}}(a_{ij})|)_{ij}$. Hence it is sufficient to prove that if $T$ and $T'$ are in $\OTP(b)$ and have the same topology, then $T=T'$, and by Lemma \ref{l:supports_determine_current} it suffices to prove that $\supp(T)=\supp(T')$.

By [\ref{BCM}, Proposition 7.8] the support of $T$ and $T'$ does not contain loops, hence we can apply Lemma \ref{l:topo_and_support} and we can assume that $T$ and $T'$ can be written as in \eqref{e:def_topo} with $a_{ij}=a'_{ij}$ for every $i,j=1,\dots, M$. This means that the set of competitors for the optimal branched transport problem with boundary $b$ and a given topology can be reduced to a family of polyhedral currents $T\in\Po_1(K)$ whose only unknown is the position of the points $\{x_1,\dots,x_M\}\setminus \supp(b)$. Accordingly, we denote $n:=\Haus^0(\supp(b))$ and we order the points $\{x_1,\dots, x_M\}$ in such a way that $\BR(T)=\{x_1,\dots,x_{M-n}\}$.
The $\alpha$-mass of such $T$ is computed as 
$$\MM(T)=\sum_{i<j}|a_{ij}|^\alpha\Haus^1(\sigma_{ij})$$
and by the previous discussion, since the vector $(x_{M-n+1},\dots, x_M)$ is fixed, this is a functional of the vector $(x_1,\dots,x_{M-n})$ only, which can be written as
\begin{equation}\label{e:convex}
    \MM(T)=F(x_1,\dots,x_{M-n}):=\sum_{i<j}|a_{ij}|^\alpha|x_j-x_i|=C+\sum_{i=1}^{M-n}\sum_{j=i+1}^M |a_{ij}|^\alpha|x_j-x_i|,
\end{equation}
where $C= \sum_{i=M-n+1}^{M}\sum_{i<j}|a_{ij}|^\alpha|x_j-x_i|$.
One can immediately see that $F$ is convex, being a sum of convex functions. Moreover each term $|a_{ij}|^\alpha|x_j-x_i|$ in \eqref{e:convex}, as a function of the variable $x_j$, is strictly convex on a segment $[s,t]$ whenever $x_i$, $s$ and $t$ are not collinear. 

Assume by contradiction that $T\neq T'\in\OTP(b)$ have the same topology and consider the corresponding sets $$\BR(T)=\{x_1,\dots,x_{M-n}\},\quad \BR(T')=\{x'_1,\dots,x'_{M-n}\}.$$
By Lemma \ref{l:topo_and_support} there exists $j\in\{1,\dots,M-n\}$ such that $x_j\neq x'_j$. As in the proof of Lemma \ref{l:numberBranch} we infer that $x_j$ is an endpoint of at least three segments in the support of $T$ which are not collinear. We deduce by the discussion after \eqref{e:convex} that the function $F$ is strictly convex in the $j$-th variable. Since $F(x_1,\dots,x_{M-n})=F(x'_1,\dots,x'_{M-n})$ we deduce that there exists a point $(y_1,\dots,y_{M-n})$ with \begin{equation}\label{e:bettery}
    F(y_1,\dots,y_{M-n})<F(x_1,\dots,x_{M-n}).
\end{equation}
Denote 
$$
z_{i}:=
\begin{cases}
y_i \quad \mbox{if $i\leq M-n$}\\
x_i \quad \mbox{otherwise}
\end{cases}
$$
and let $S$ be the current
\begin{equation}\label{e:S}
    S:=\sum_{i<j}a_{ij}\llbracket\tilde\sigma_{ij}\rrbracket,
\end{equation}
where $\tilde\sigma_{ij}$ is the segment with first endpoint $z_i$ and second endpoint $z_j$. Notice that in principle it might happen that $S$ does not have the same topology as $T$ and $T'$, since \eqref{e:S} might fail to have property (iii) of Definition \ref{d:topo}. However we have $\partial S=b$ and by \eqref{e:bettery} 
$$\MM(S)\leq F(y_1,\dots,y_{M-n})<F(x_1,\dots,x_{M-n})=\MM(T),$$
which contradicts the assumption $T\in\OTP(b)$. 
\end{proof}

\begin{lemma}\label{l:magic_points}
    For every boundary $b\in\I_{0}(K)$ and $T \in \OTP(b)$, there is a set of distinct points $\{p_1, \dots, p_h\} \subset \supp(T)\setminus(\BR(T)\cup\supp (b))$ such that 
    \[ \big\{ S \in \OTP(b): \{ p_1, \dots, p_h \} \subset \supp(S)\big\} = \{T\}\,. \]
    Moreover, the $p_i$'s can be chosen so that if $p_i\in\supp(T)\,\cap\,\supp(S)$ for some $ S\in\OTP(b)$, then there exists $\rho>0$ such that $\supp(T)\cap B_\rho(p_i)=\supp(S)\cap B_\rho(p_i)$.
\end{lemma}

\begin{proof}
By Lemma \ref{l:finiteminimizers} we have that $\OTP(b)$ consists of finitely many polyhedral currents $ T^1,\dots, T^h$ and, by Lemma \ref{l:supports_determine_current}, the symmetric difference $\supp(T^i)\triangle \,\supp(T^j)$ is a relatively open set of positive length for every $i\neq j$. Up to reordering, we assume $T^1=T$ and for every $i \in \{2, \dots, h\}$ we consider the set $U_i:= \supp(T) \setminus (\supp(T^i) \cup \BR(T))$. We observe, recalling that $\BR(T)$ is finite by Lemma \ref{l:numberBranch}, that each $U_i$ is relatively open with positive length. 
Define the subset
\begin{align*}
    V_i := U_i \cap  \Big( \bigcup_{j \neq i} \BR(T^j) \cup \left\{p \in U_i: \text{$\supp(T^j)$ intersects $U_i$ transversally at $p$} \right\} \Big)
\end{align*}
and observe that $V_i$ is finite since every $T^j$ is polyhedral. Then choose $p_i \in U_i \setminus V_i$. Clearly $p_i\in \supp(T)\setminus(\BR(T)\cup \supp(b))$ and $p_i\not\in\supp(T^i)$; moreover if $p_i \in \supp(T^j)$ then locally $\supp(T^j)$ agrees with $\supp(T)$.
\end{proof}

\subsection{Perturbation argument}
\subsubsection{Construction of the perturbed boundaries}
Let us fix a boundary $b\in\I_0(K)$, an integer polyhedral current $T\in\OTP(b)$, see Lemma \ref{l:integral_minimizers}, and points $\{p_1,\dots,p_h\}$ as in Lemma \ref{l:magic_points}. For a fixed $k\in\N\setminus\{0\}$ and for $n=1,2,\dots$ we denote 
\begin{equation}\label{e:bienne}
\begin{split}
    T_n &:=T- \frac1k \sum_{i=1}^h T\res B_{n^{-1}}(p_i)\, ,\\
    b_n &:=\partial T_n\,.
\end{split}
\end{equation}
Observe that by Proposition \ref{pp:gooddec}, the multiplicity of $T$ is bounded from above by $2^{-1}\Mass(b)$ and moreover, for $n$ sufficiently large, the closed balls $\overline B_{n^{-1}}(p_i)$ are disjoint and do not intersect $\supp(b)\cup\BR(T)$, so that we have
\begin{equation}\label{e:massa_bienne}
    \Mass(b_n)=\Mass(b)+k^{-1}\sum_{i=1}^h\Mass(\partial(T\res B_{n^{-1}}(p_i)))\leq\Mass(b)+hk^{-1}\Mass(b)
\end{equation}
and \begin{equation}\label{e:convergence_bienne}
    \Flat_K(b_n-b)\leq k^{-1}\sum_{i=1}^h\Mass(T\res B_{n^{-1}}(p_i))\leq h(nk)^{-1}\Mass(b).
\end{equation}
For every $n$, we choose $S_n\in\OTP(b_n)$ and we apply Lemma \ref{l:integral_minimizers} to the boundaries $kb_n$ to deduce that $kS_n\in\Po_1(K)\cap\I_1(K)$. By \eqref{e:bienne} we have 
\begin{equation}\label{e:alphamass_tienne}
\MM(S_n)\leq\MM(T_n)<\MM(T). \end{equation}
The aim is to prove the following proposition.
\begin{proposition}\label{p:unique_bienne}
There exists $k_0=k_0(\alpha)$ such that for $(b_n)_{n}$ as in \eqref{e:bienne} with $k\geq k_0$ and for $n$ sufficiently large, $\OTP(b_n)=\{T_n\}$. 
\end{proposition}

In the next lemma, for any set $A$ and $\rho>0$ we denote $B_\rho(A):=\bigcup_{a\in A}B_\rho(a)$.
\begin{lemma}\label{l:conv_hauss}
    For $n \in \N \setminus \{0\}$ let $b_n$ be as in \eqref{e:bienne} and $S_n\in\OTP(b_n)$. For every subsequence $(S_{n_j})_{j\in\N}$ and current $S$ such that $\Flat_K(S_{n_j}-S)\to 0$ as $j\to\infty$ we have $S\in\OTP(b)$ and moreover for every $\rho>0$ we have $\supp(S_{n_j})\subset B_\rho(\supp(S)\cup\{p_1,\dots,p_h\})$, for $j$ sufficiently large.
\end{lemma}
\begin{proof}
The first part of the proposition is a direct consequence of Theorem \ref{t:stability_new}. 
Towards a proof by contradiction of the second part, assume that there exists $r>0$ and, for every $j$, a point 
\begin{equation}\label{e:quenne}
    q_j\in\supp(S_{n_j})\setminus B_{2r}(\supp(S)\cup\{p_1,\dots,p_h\}).
\end{equation} By [\ref{CDRMS}, Proposition 2.6]\footnote{This proposition is a polyhedral approximation theorem (similar to Theorem \ref{t:polyapprox}) for more general notions of the $\alpha$-mass.} we have 
\begin{equation}\label{e:mass_in}
    \liminf_j\MM(S_{n_j}\res B_{r}(\supp(S)\cup\{p_1,\dots,p_h\}))\geq\MM(S).
\end{equation}
On the other hand, by \eqref{e:bienne} and \eqref{e:quenne} the current $S_{n_j}$ has no boundary in $B_r(q_j)$, for $j$ sufficiently large. Moreover, by Proposition \ref{p:singlepath} the restriction $R_j$ of $S_{n_j}\res \,B_r(q_j)$ to the connected component of its support containing $q_j$ has nontrivial boundary, and more precisely applying Proposition \ref{p:simonslice} with $f(x)=|x-q_j|$ we deduce that $\emptyset\neq\supp(\partial R_j)\subset \partial B_r(q_j)$. We conclude that $\supp (R_j)$ contains a path connecting $q_j$ to a point of $\partial B_r(q_j)$. By Lemma \ref{l:integral_minimizers} such path has multiplicity bounded from below by $k^{-1}$. This allows to conclude that 
\begin{equation}\label{e:mass_out}
\MM(S_{n_j}\res B_r(q_j))\geq rk^{-\alpha}.
\end{equation}
Combining \eqref{e:quenne},\eqref{e:mass_in}, and \eqref{e:mass_out}, we conclude
$$ \liminf_j\MM(S_{n_j})\geq\MM(S)+rk^{-\alpha}=\MM(T)+rk^{-\alpha},$$
which contradicts \eqref{e:alphamass_tienne}.
\end{proof}

We now prove the following main lemma.

\begin{lemma}\label{l:main}
    There exists $k_0=k_0(\alpha)$ with the following property. Let $T$ and $(T_n)_n$ be as in \eqref{e:bienne} with $k\geq k_0$ and let $S$ and $(S_{n_j})_j$ be as in Lemma \ref{l:conv_hauss}.  Then $S_{n_j}=T_{n_j}$, for $j$  sufficiently large and in particular  $S=T$.
\end{lemma}

\begin{proof}
We divide the proof in three steps. First we prove that locally in a box around each point $p\in\{p_1,\dots,p_h\}\cap \supp(S)$ for $j$ sufficiently large the current $S_{n_j}$ is a minimizer of the $\alpha$-mass for a certain boundary whose support is a set of four \emph{almost collinear} points. Then, we analyze all the possible topologies for the minimizers with such boundary and we are able to exclude all of them except for two. Finally, we combine the local analysis with a global energy estimate to conclude.

\subsubsection{Local structure of $S_{n_j}$.}\label{Step1} Let $\rho$ be sufficiently small, to be chosen later (see (2a), (3a) and (3b) in Section \ref{Step2}). For every $i=1,\dots, h$ and for $p_i\in\supp(S)$, by Lemma \ref{l:magic_points} we can choose orthonormal coordinates $(x,y)\in\R\times\R^{d-1}$ such that, up to a dilation with homothety ratio $c$ with $$c>\frac{8}{\text{dist}(p_i,\supp(b)\cup\BR(S))},$$ denoting $Q:=[-8,8]\times B_{\rho}^{d-1}(0)$ and $B_j:=(-cn_j^{-1},0)$, $C_j:=(cn_j^{-1},0)$, for $j=1,2,\dots$, the following holds:  
\begin{itemize}
    \item [(i)] $p_i=(0,0)$;
    \item [(ii)]$S\res Q=\theta\llbracket\sigma\rrbracket$, where $\theta\in\Z$ and $\sigma:=[-8,8]\times\{0\}^{d-1}$ is positively oriented;
    \item [(iii)] $\{p_1,\dots,p_h\}\cap \sigma= \{p_i\}$;
    \item [(iv)] for $j$ sufficiently large we have $b_{n_j}\res Q =k^{-1}\theta(\delta_{B_j}-\delta_{C_j})$.
\end{itemize} 
By Lemma \ref{l:conv_hauss}, for this $\rho$, we may choose $j$ large enough such that
\begin{equation}\label{e:where_is_Sn}
    \supp(S_{n_j})\cap \bar{Q} \subset B_\rho(\sigma).
\end{equation}
For $x \in \R$ we recall the notation $S_{n_j}^x$ for the slice of $S_{n_j}\res Q$ at the point $x$
with respect to the projection $\pi:\R\times\R^{d-1}\to\R$. We infer from the flat convergence of $S_{n_j}$ to $S$ that for $\Haus^1$-a.e. $x\in[-8,8]$ we have
\begin{equation}\label{e:flat_slices}
    \Flat_K(S_{n_j}^x-\theta\delta_{(x,0)})\to 0\quad \mbox{as $j\to \infty$}
\end{equation}
and moreover by Lemma \ref{l:integral_minimizers} the multiplicities of $S_{n_j}^x$ are integer multiples of $k^{-1}$. 

We aim to prove that for $j$ sufficiently large there are points $y_j^{\pm}\in B^{d-1}_\rho(0)$ such that \begin{equation}\label{e:good_slices}
S_{n_j}^{\pm 4}=\theta\delta_{(\pm 4,y_j^{\pm})},    
\end{equation}
To this aim we seek points $x_1(j)\in[-6,-5]$, $x_2(j)\in[-2,-1]$,  $x_3(j)\in[1,2]$ and
$x_4(j)\in[5,6]$ such that for $i=1, 2, 3, 4$ we have
\begin{equation}\label{e:almost_good_slices}
    S_{n_j}^{x_i(j)}=\theta\delta_{(x_i(j),y_i(j))},
\end{equation}
for some points $y_i(j)\in B^{d-1}_\rho(0)$. If so, by Proposition \ref{p:simonslice}, 
\eqref{e:almost_good_slices} and \eqref{e:where_is_Sn} imply, denoting $$Q_1^j:= (x_1(j),x_2(j)) \times B_{\rho}^{d-1}(0)\quad\mbox{and}\quad Q_2^j:=(x_3(j),x_4(j)) \times B_{\rho}^{d-1}(0),$$
that $$\partial(S_{n_j}\res Q_1^j)=S_{n_j}^{x_2(j)}-S_{n_j}^{x_1(j)}\quad \mbox{and}\quad\partial(S_{n_j}\res Q_2^j)=S_{n_j}^{x_4(j)}-S_{n_j}^{x_3(j)}.$$
In turn, by Proposition \ref{p:singlepath}, the latter implies \eqref{e:good_slices}.

In order to prove \eqref{e:almost_good_slices}, we focus on the interval $I:=[1,2]$ as the argument for the remaining intervals is identical. Firstly, we observe that by \eqref{e:flat_slices} we have
\begin{equation}\label{e:mass_slices_larger_theta}
\liminf_j(\Mass(S_{n_j}^x))\geq \theta \quad\mbox{for $\Haus^1$-a.e. $x\in I$}.
\end{equation}
Next, denoting $\Omega:=I\times B_{\rho}^{d-1}(0)$, we claim that for $j$ sufficiently large and for every $C>0$ it holds that
\begin{equation}\label{e:slices_with_small_energy}
    \Haus^1(\{x\in I:\MM(S_{n_j}^x)\leq \theta^\alpha + C\})>0,
\end{equation}
where for a 0-current $Z:=\sum_{\ell\in\N}\theta_\ell\delta_{z_\ell}$ we denoted $\MM(Z):=\sum_{\ell\in\N}|\theta_\ell|^\alpha$.

Assume by contradiction that \eqref{e:slices_with_small_energy} is false for infinitely many indices $j$. By Corollary \ref{c:simon1}, for those indices we have
\begin{equation}\label{e:towards_contr}
    \MM(S_{n_j}\res \bar \Omega)\geq\theta^\alpha+C=\MM(S\res \bar \Omega)+C.
\end{equation}
The latter, combined with \eqref{e:alphamass_tienne}, implies that for the same indices we have 
$$\MM(S_{n_j}\res (\R^d\setminus \bar \Omega))<\MM(S\res (\R^d\setminus \bar \Omega))-C,$$
which contradicts [\ref{CDRMS}, Proposition 2.6]. 
From \eqref{e:mass_slices_larger_theta} and \eqref{e:slices_with_small_energy} we deduce that for $j$ sufficiently large there exists $x_1(j)\in I$ such that
\begin{equation}\label{e:final_good_slice}
\Mass(S_{n_j}^{x_1(j)})\geq\theta\quad\mbox{and}\quad \MM(S_{n_j}^{x_1(j)})\leq\theta^\alpha+C.
\end{equation}
Lastly we prove that if $C$ is sufficiently small, then \eqref{e:final_good_slice} implies 
\begin{equation}\label{e:GS}
    S_{n_j}^{x_1(j)}=\theta\delta_{(x_1(j),y_1(j))},
\end{equation}
for some point $y_1(j)\in B^{d-1}_\rho(0)$, thus completing the proof of \eqref{e:almost_good_slices}.

Towards a proof by contradiction of \eqref{e:GS}, observe that for every 0-current $Z=\sum_{\ell=1}^M\theta_\ell\delta_{z_\ell}$, with $M\geq 2$, $|\theta_\ell|\,\, \geq k^{-1}$ and $z_\ell$ distinct, satisfying $\Mass(Z)=\sum_\ell|\theta_\ell|\,\, \geq \theta$, the strict subadditivity of the function $t\mapsto t^\alpha$ (for $t>0$) yields the existence of a $\bar C=\bar C(\alpha, \theta, k)> 0$ such that
$$\MM(Z)=|\theta_1|^\alpha+\sum_{\ell=2}^{M}|\theta_\ell|^\alpha\geq \min\{(mk^{-1})^\alpha + (\theta-mk^{-1})^\alpha: m=1,\dots,k\theta-1\}>\theta^\alpha+\bar C.$$
This contradicts \eqref{e:final_good_slice}, by the arbitrariness of $C$.

It follows from \eqref{e:good_slices} and Proposition \ref{p:simonslice} that, denoting $$Q':=(-4,4) \times B_{\rho}^{d-1}(0), \quad A_j:= ( -4,y_j^-)\ \text{ and } \ D_j:=(4,y_j^+),$$
we have
\begin{equation}\label{e:four_points}
\partial (S_{n_j}\res Q') =\theta\left(\delta_{D_j}-\delta_{A_j}+k^{-1}\big(\delta_{B_j}-\delta_{C_j}\big)\right),   
\end{equation}
for $j$ sufficiently large (see Figure \ref{f:wurst}).

\begin{figure}[ht]
	\centering
	\captionsetup{justification=centering}
	\includegraphics[width=0.95\textwidth]{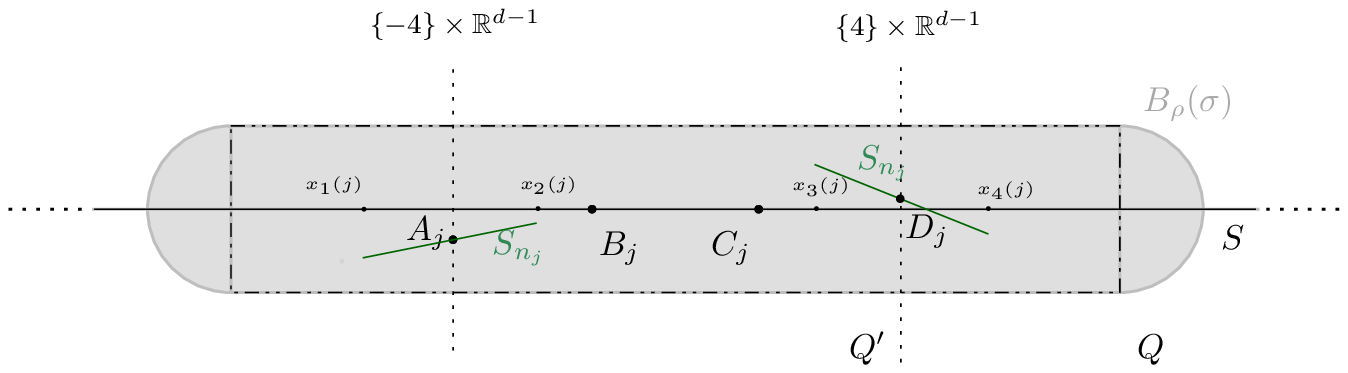}
	\caption{Representation of parts of $S_{n_j}\res Q$.}\label{f:wurst}
\end{figure}
\vspace{0.5cm}

\subsection{Analysis of the possible topologies}\label{Step2}
Now we study the possible topologies of $S_{n_j}\res Q'$. Since $S_{n_j}\in\OTP(b_{n_j})$ we must have $S_{n_j}\res Q'\in\OTP(\partial (S_{n_j}\res Q'))$. In general, we will denote by $\sigma_{PR}$ the oriented segment from the point $P$ to the point $R$. We aim to prove that for $k\geq k_0(\alpha)$ and for $\rho \leq \rho(k)$ sufficiently small it holds that $S_{n_j}\res Q' \in \{W_j, Z_j\}$, for $j$ large enough, where 
\begin{equation}\label{e:THE_ONE}
    W_j := \theta \left( \llbracket \sigma_{A_jB_j} \rrbracket + \llbracket \sigma_{C_jD_j} \rrbracket + \frac{k-1}{k} \llbracket \sigma_{B_jC_j} \rrbracket \right)\quad \mbox{and}\quad Z_j := \theta \left( \llbracket \sigma_{A_jD_j} \rrbracket + \frac{1}{k} \llbracket \sigma_{C_jB_j} \rrbracket \right) \,, 
\end{equation}
see Table \ref{FigONES}.
We will do this by excluding every other topology comparing angle conditions which are given by the multiplicities of the segments (which depend on $k$) and contradict the choice of $\rho$. Thus, when we say for $\rho$ small enough, we mean implicitly to choose $j$ large enough such that, by Lemma \ref{l:conv_hauss}, for the desired $\rho$ we have $$\supp(S_{n_j})\subset B_\rho(\supp(S)\cup\{p_1,\dots,p_h\}).$$

Write $S_{n_j}\res Q'=\sum_{i<j}a_{ij}\llbracket\sigma_{ij}\rrbracket$ as in \eqref{e:def_topo} and observe that by Lemma \ref{l:numberBranch}, since $\Haus^0(\partial (S_{n_j}\res Q'))=4$, then $\Haus^0(\BR (S_{n_j}\res Q'))\in \{0,1,2\}$. We thus analyze the three cases separately and we recall that, by Lemma \ref{l:supports_determine_current}, in order to prove \eqref{e:THE_ONE} it suffices to prove that $$\supp(S_{n_j}\res Q')=\sigma_{A_jB_j}\cup\sigma_{B_jC_j}\cup\sigma_{C_jD_j}\quad\mbox{or}\quad \supp(S_{n_j}\res Q')=\sigma_{A_jD_j}\cup\sigma_{B_jC_j}.$$

\begin{longtable}{  p{1cm} c } \label{FigONES} 

\\ $W_j$ &
    \begin{tikzpicture}
         \filldraw[black] (0,0) circle (2pt) node[anchor=south]{$A$};
        \filldraw[black] (3.6, 2) circle (2pt) node[anchor=north]{$B$};
        \filldraw[black] (5.4, 2) circle (2pt) node[anchor=north]{$C$};
        \filldraw[black] (8.8, 2.8) circle (2pt) node[anchor=north]{$D$};
        \path[draw = gray,% 
             decoration={%
            markings,%
            mark=at position 0.5   with {\arrow[scale=1.5]{>}},%
            },%
            postaction=decorate] (0,0) -- (3.6, 2)
            node[midway, above, color= gray, font=\small]{$\theta$};
        \path[draw = gray,% 
             decoration={%
            markings,%
            mark=at position 0.5   with {\arrow[scale=1.5]{>}},%
            },%
            postaction=decorate] (5.4, 2) -- (8.8, 2.8)
            node[midway, below, color= gray, font=\small]{$\theta$};
        \path[draw = gray,% 
             decoration={%
            markings,%
            mark=at position 0.5   with {\arrow[scale=1.5]{>}},%
            },%
            postaction=decorate] (3.6, 2) -- (5.4, 2)
            node[midway, above, color= gray, font=\small]{$\theta-\delta$};
    \end{tikzpicture}
    \\ \hline
\\ $Z_j$ & 
\begin{tikzpicture}
     \filldraw[black] (0,0) circle (2pt) node[anchor=south]{$A$};
    \filldraw[black] (3.6, 2) circle (2pt) node[anchor=south]{$B$};
    \filldraw[black] (5.4, 2) circle (2pt) node[anchor=south]{$C$};
    \filldraw[black] (8.8, 2.8) circle (2pt) node[anchor=north]{$D$};
    \path[draw = gray,% 
             decoration={%
            markings,%
            mark=at position 0.5   with {\arrow[scale=1.5]{>}},%
            },%
            postaction=decorate] (0,0) -- (8.8, 2.8)
            node[midway, below, color= gray, font=\small]{$\theta$};
    \path[draw = gray,% 
             decoration={%
            markings,%
            mark=at position 0.5   with {\arrow[scale=1.5]{<}},%
            },%
            postaction=decorate] (3.6, 2) -- (5.4, 2)
            node[midway, above, color= gray, font=\small]{$\delta$};
\end{tikzpicture}
\\ \hline
\caption{Representation of $W_j$ and $Z_j$. From now on $\delta := \theta/k$ and we remove the subscript $j$ from the points.}
\end{longtable}

\emph{Case 1}: $\BR(S_{n_j}\res Q') = \emptyset$. Recalling Proposition \ref{p:singlepath},  $\supp(S_{n_j}\res Q')$ must be one of the following sets, sorted alphabetically: 
\begin{itemize}
    \item [(1a)] $\sigma_{A_jB_j}\cup\sigma_{A_jC_j}\cup\sigma_{A_jD_j}$,
    \item [(1b)] $\sigma_{A_jB_j}\cup\sigma_{A_jC_j}\cup\sigma_{B_jD_j}$,
    \item [(1c)] $\sigma_{A_jB_j}\cup\sigma_{A_jC_j}\cup\sigma_{C_jD_j}$,
    \item [(1d)] $\sigma_{A_jB_j}\cup\sigma_{A_jD_j}\cup\sigma_{B_jC_j}$,
     \item [(1e)] $\sigma_{A_jB_j}\cup\sigma_{A_jD_j}\cup\sigma_{C_jD_j}$,
    \item [(1f)] $\sigma_{A_jB_j}\cup\sigma_{B_jC_j}\cup\sigma_{B_jD_j}$,
    \item [(1g)] $\sigma_{A_jB_j}\cup\sigma_{B_jC_j}\cup\sigma_{C_jD_j}$,
    \item [(1h)] $\sigma_{A_jB_j}\cup\sigma_{B_jD_j}\cup\sigma_{C_jD_j}$,
    \item [(1i)]
    $\sigma_{A_jB_j}\cup\sigma_{C_jD_j}$,
    \item [(1j)]
    $\sigma_{A_jC_j}\cup\sigma_{A_jD_j}\cup\sigma_{B_jC_j}$,
    \item [(1k)]
    $\sigma_{A_jC_j}\cup\sigma_{A_jD_j}\cup\sigma_{B_jD_j}$,    
    \item [(1l)] $\sigma_{A_jC_j}\cup\sigma_{B_jC_j}\cup\sigma_{B_jD_j}$,
    \item [(1m)] $\sigma_{A_jC_j}\cup\sigma_{B_jC_j}\cup\sigma_{C_jD_j}$,    
    \item [(1n)] $\sigma_{A_jC_j}\cup\sigma_{B_jD_j}$,
    \item [(1o)] $\sigma_{A_jC_j}\cup\sigma_{B_jD_j}\cup\sigma_{C_jD_j}$,    
    \item [(1p)] $\sigma_{A_jD_j}\cup\sigma_{B_jC_j}$,    
    \item [(1q)] $\sigma_{A_jD_j}\cup\sigma_{B_jC_j}\cup\sigma_{B_jD_j}$,
    \item [(1r)] $\sigma_{A_jD_j}\cup\sigma_{B_jC_j}\cup\sigma_{C_jD_j}$,
    \item [(1s)] $\sigma_{A_jD_j}\cup\sigma_{B_jD_j}\cup\sigma_{C_jD_j}$.
\end{itemize}
Observe that we omitted the cases
\begin{itemize}
    \item [(i)] $\sigma_{A_jB_j}\cup\sigma_{A_jC_j}\cup\sigma_{B_jC_j}$,
    \item [(ii)] $\sigma_{A_jB_j}\cup\sigma_{A_jD_j}\cup\sigma_{B_jD_j}$,
    \item [(iii)] $\sigma_{A_jC_j}\cup\sigma_{A_jD_j}\cup\sigma_{C_jD_j}$ 
    \item [(iv)] $\sigma_{B_jC_j}\cup\sigma_{B_jD_j}\cup\sigma_{C_jD_j}$
\end{itemize}
because, independently of the position of the points, the support either contains a loop or does not contain one of the four points in the support of the boundary. The only exceptions to this behavior are (ii) and (iii) only when the four points are collinear, which is not relevant, as we discuss in Sub-case 1-1 below.\\

\emph{Sub-case 1-1}. Firstly we observe that when the points $A_j,B_j,C_j$ and $D_j$ are collinear the only admissible competitor is $Z_j$.\\ 

\emph{Sub-case 1-2}. Next, we analyze the case in which no triples among the points $A_j,B_j,C_j$ and $D_j$ are contained in a line.

We immediately exclude those cases for which the corresponding set is not the support of any current with boundary $\partial(S_{n_j}\res Q')$. Hence we can exclude (1i) and (1n), because the endpoints of the two segments in the support have different multiplicities. Moreover we exclude (1d), (1j), (1q) and (1r) as well, because the segment $\sigma_{A_j,D_j}$ should have multiplicity $\theta$, being either for $A_j$ or $D_j$ the only segment in the support containing it. On the other hand, the remaining point (respectively $D_j$ or $A_j$) is an endpoint also for a different segment of the support, from which we deduce that the multiplicity of the latter segment should be 0 (see Table \ref{Fig1}).

\begin{longtable}{  p{1cm} c } \label{Fig1} 
\\ 1d &
\begin{tikzpicture}
     \filldraw[black] (0,0) circle (2pt) node[anchor=south]{$A$};
    \filldraw[black] (1.8, 1) circle (2pt) node[anchor=south]{$B$};
    \filldraw[black] (2.6, 1) circle (2pt) node[anchor=south]{$C$};
    \filldraw[black] (4.4, 1.4) circle (2pt) node[anchor=north]{$D$};
    \draw[gray, thick] (0,0) -- (1.8, 1);
    \draw[gray, thick] (0,0) -- (4.4, 1.4);
    \draw[gray, thick] (1.8, 1) -- (2.6, 1);
\end{tikzpicture}
  \qquad \qquad
\begin{tikzpicture}
     \filldraw[black] (0,0) circle (2pt) node[anchor=south]{$A$};
    \filldraw[black] (1.8, 1) circle (2pt) node[anchor=south]{$B$};
    \filldraw[black] (2.6, 1) circle (2pt) node[anchor=south]{$C$};
    \filldraw[black] (4.4, 1.4) circle (2pt) node[anchor=north]{$D$};
    \draw[gray, thick] (0,0) -- (1.8, 1);
    \draw[gray, thick] (2.6, 1) -- (4.4, 1.4);
\end{tikzpicture}
\quad  1i 
\\ \hline 
\\ 1j &
\begin{tikzpicture}
     \filldraw[black] (0,0) circle (2pt) node[anchor=south]{$A$};
    \filldraw[black] (1.8, 1) circle (2pt) node[anchor=south]{$B$};
    \filldraw[black] (2.6, 1) circle (2pt) node[anchor=south]{$C$};
    \filldraw[black] (4.4, 1.4) circle (2pt) node[anchor=north]{$D$};
    \draw[gray, thick] (0,0) -- (2.6, 1);
    \draw[gray, thick] (1.8, 1) -- (2.6, 1);
    \draw[gray, thick] (0,0) -- (4.4, 1.4);
\end{tikzpicture}
 \qquad \qquad
\begin{tikzpicture}
     \filldraw[black] (0,0) circle (2pt) node[anchor=south]{$A$};
    \filldraw[black] (1.8, 1) circle (2pt) node[anchor=south]{$B$};
    \filldraw[black] (2.6, 1) circle (2pt) node[anchor=north]{$C$};
    \filldraw[black] (4.4, 1.4) circle (2pt) node[anchor=north]{$D$};
    \draw[gray, thick] (0,0) -- (2.6, 1);
    \draw[gray, thick] (1.8, 1) -- (4.4, 1.4);
\end{tikzpicture}
\quad 1n
\\ \hline 
\\ 1q &
\begin{tikzpicture}
     \filldraw[black] (0,0) circle (2pt) node[anchor=south]{$A$};
    \filldraw[black] (1.8, 1) circle (2pt) node[anchor=south]{$B$};
    \filldraw[black] (2.6, 1) circle (2pt) node[anchor=south]{$C$};
    \filldraw[black] (4.4, 1.4) circle (2pt) node[anchor=north]{$D$};
    \draw[gray, thick] (0,0) -- (4.4, 1.4);
    \draw[gray, thick] (1.8, 1) -- (2.6, 1);
    \draw[gray, thick] (1.8, 1) -- (4.4, 1.4);
\end{tikzpicture}
 \qquad \qquad
\begin{tikzpicture}
     \filldraw[black] (0,0) circle (2pt) node[anchor=south]{$A$};
    \filldraw[black] (1.8, 1) circle (2pt) node[anchor=south]{$B$};
    \filldraw[black] (2.6, 1) circle (2pt) node[anchor=south]{$C$};
    \filldraw[black] (4.4, 1.4) circle (2pt) node[anchor=north]{$D$};
    \draw[gray, thick] (0,0) -- (4.4, 1.4);
    \draw[gray, thick] (1.8, 1) -- (2.6, 1);
    \draw[gray, thick] (2.6, 1) -- (4.4, 1.4);
\end{tikzpicture}
\quad 1r
\\ \hline
\caption{Representation of (1d), (1i), (1j), (1n), (1q), (1r).}
\end{longtable}

We exclude the following cases by direct comparison with the $\alpha$-mass of $Z_j$, for $j$ sufficiently large (see Table \ref{Fig2}):

\begin{itemize}
    \item (1a), whose corresponding $\alpha$-mass is
    $$\theta^\alpha(\Haus^1(\sigma_{A_jD_j})+k^{-\alpha}(\Haus^1(\sigma_{A_jB_j})+\Haus^1(\sigma_{A_jC_j}))>\MM(Z_j).$$
    \item (1b), whose corresponding $\alpha$-mass is
    $$\theta^\alpha((1+k^{-1})^{\alpha}\Haus^1(\sigma_{A_jB_j})+\Haus^1(\sigma_{B_jD_j})+k^{-\alpha}\Haus^1(\sigma_{A_jC_j}))>\MM(Z_j).$$
    \item (1f), whose corresponding $\alpha$-mass is
    $$\theta^\alpha(\Haus^1(\sigma_{A_jB_j})+\Haus^1(\sigma_{B_jD_j})+k^{-\alpha}\Haus^1(\sigma_{B_jC_j}))>\MM(Z_j).$$
     \item (1k), whose corresponding $\alpha$-mass is
    $$\theta^\alpha((1+k^{-1})^{\alpha}\Haus^1(\sigma_{A_jD_j})+k^{-\alpha}(\Haus^1(\sigma_{A_jC_j})+\Haus^1(\sigma_{B_jD_j})))>\MM(Z_j).$$
     \item (1l), whose corresponding $\alpha$-mass is
    $$\theta^\alpha((1+k^{-1})^{\alpha}\Haus^1(\sigma_{B_jC_j})+\Haus^1(\sigma_{A_jC_j})+\Haus^1(\sigma_{B_jD_j}))>\MM(Z_j).$$
     \item (1m), whose corresponding $\alpha$-mass is
    $$\theta^\alpha(\Haus^1(\sigma_{A_jC_j})+\Haus^1(\sigma_{C_jD_j})+k^{-\alpha}\Haus^1(\sigma_{B_jC_j}))>\MM(Z_j).$$
    \item (1o), whose corresponding $\alpha$-mass is
    $$\theta^\alpha((1+k^{-1})^{\alpha}\Haus^1(\sigma_{C_jD_j})+\Haus^1(\sigma_{A_jC_j})+k^{-\alpha}\Haus^1(\sigma_{B_jD_j}))>\MM(Z_j).$$
    \item (1s), whose corresponding $\alpha$-mass is
    $$\theta^\alpha(\Haus^1(\sigma_{A_jD_j})+k^{-\alpha}(\Haus^1(\sigma_{B_jD_j})+\Haus^1(\sigma_{C_jD_j}))>\MM(Z_j).$$
\end{itemize}

\begin{longtable}{  p{1cm} c } \label{Fig2} 
        1a &
        \begin{tikzpicture}
         \filldraw[black] (0,0) circle (2pt) node[anchor=south]{$A$};
        \filldraw[black] (3.6, 2) circle (2pt) node[anchor=south]{$B$};
        \filldraw[black] (5.4, 2) circle (2pt) node[anchor=south]{$C$};
        \filldraw[black] (8.8, 2.8) circle (2pt) node[anchor=north]{$D$};
        \path[draw = gray,% 
             decoration={%
            markings,%
            mark=at position 0.5   with {\arrow[scale=1.5]{>}},%
            },%
            postaction=decorate] (0,0) -- (3.6, 2)
            node[midway, above left, color= gray, font=\small]{$\delta$};
        \path[draw = gray,% 
             decoration={%
            markings,%
            mark=at position 0.5   with {\arrow[scale=1.5,]{<}},%
            },%
            postaction=decorate] (0,0) -- (5.4, 2)
            node[midway, above, color= gray, font=\small]{$\delta$};
        \path[draw = gray,% 
             decoration={%
            markings,%
            mark=at position 0.5   with {\arrow[scale=1.5]{>}},%
            },%
            postaction=decorate] (0,0) -- (8.8, 2.8)
            node[midway, below right, color= gray, font=\small]{$\theta$};
        \end{tikzpicture}
        
      \\ \hline \\1b &
          \begin{tikzpicture}
         \filldraw[black] (0,0) circle (2pt) node[anchor=south]{$A$};
        \filldraw[black] (3.6, 2) circle (2pt) node[anchor=south]{$B$};
        \filldraw[black] (5.4, 2) circle (2pt) node[anchor=north]{$C$};
        \filldraw[black] (8.8, 2.8) circle (2pt) node[anchor=north]{$D$};
        \path[draw = gray,% 
             decoration={%
            markings,%
            mark=at position 0.5   with {\arrow[scale=1.5]{>}},%
            },%
            postaction=decorate] (0,0) -- (3.6, 2)
            node[midway, above left, color= gray, font=\small]{$\theta+\delta$};
        \path[draw = gray,% 
             decoration={%
            markings,%
            mark=at position 0.5   with {\arrow[scale=1.5]{<}},%
            },%
            postaction=decorate] (0,0) -- (5.4, 2)
            node[midway, below, color= gray, font=\small]{$\delta$};
        \path[draw = gray,% 
             decoration={%
            markings,%
            mark=at position 0.5   with {\arrow[scale=1.5]{>}},%
            },%
            postaction=decorate] (3.6, 2) -- (8.8, 2.8)
            node[midway, above, color= gray, font=\small]{$\theta$};
        \end{tikzpicture}
\\ \hline \\ 1f &
    \begin{tikzpicture}
         \filldraw[black] (0,0) circle (2pt) node[anchor=south]{$A$};
        \filldraw[black] (3.6, 2) circle (2pt) node[anchor=south]{$B$};
        \filldraw[black] (5.4, 2) circle (2pt) node[anchor=north]{$C$};
        \filldraw[black] (8.8, 2.8) circle (2pt) node[anchor=north]{$D$};
        \path[draw = gray,% 
             decoration={%
            markings,%
            mark=at position 0.5   with {\arrow[scale=1.5]{>}},%
            },%
            postaction=decorate] (0,0) -- (3.6, 2)
            node[midway, above, color= gray, font=\small]{$\theta$};
        \path[draw = gray,% 
             decoration={%
            markings,%
            mark=at position 0.5   with {\arrow[scale=1.5]{>}},%
            },%
            postaction=decorate] (3.6, 2) -- (8.8, 2.8)
            node[midway, above, color= gray, font=\small]{$\theta$};
        \path[draw = gray,% 
             decoration={%
            markings,%
            mark=at position 0.5   with {\arrow[scale=1.5]{<}},%
            },%
            postaction=decorate] (3.6, 2) -- (5.4, 2)
            node[midway, below, color= gray, font=\small]{$\delta$};
    \end{tikzpicture}
\\ \hline
     \\ 1k &
\begin{tikzpicture}
     \filldraw[black] (0,0) circle (2pt) node[anchor=south]{$A$};
    \filldraw[black] (3.6, 2) circle (2pt) node[anchor=south]{$B$};
    \filldraw[black] (5.4, 2) circle (2pt) node[anchor=south]{$C$};
    \filldraw[black] (8.8, 2.8) circle (2pt) node[anchor=north]{$D$};
    \path[draw = gray,% 
             decoration={%
            markings,%
            mark=at position 0.5   with {\arrow[scale=1.5]{<}},%
            },%
            postaction=decorate] (0,0) -- (5.4, 2)
            node[midway, above, color= gray, font=\small]{$\delta$};
    \path[draw = gray,% 
             decoration={%
            markings,%
            mark=at position 0.5   with {\arrow[scale=1.5]{<}},%
            },%
            postaction=decorate] (3.6, 2) -- (8.8, 2.8)
            node[midway, above, color= gray, font=\small]{$\delta$};
    \path[draw = gray,% 
             decoration={%
            markings,%
            mark=at position 0.5   with {\arrow[scale=1.5]{>}},%
            },%
            postaction=decorate] (0,0) -- (8.8, 2.8)
            node[midway, below, color= gray, font=\small]{$\theta+\delta$};
\end{tikzpicture}
\\ \hline
\\ 1l &
\begin{tikzpicture}
     \filldraw[black] (0,0) circle (2pt) node[anchor=south]{$A$};
    \filldraw[black] (3.6, 2) circle (2pt) node[anchor=south]{$B$};
    \filldraw[black] (5.4, 2) circle (2pt) node[anchor=north]{$C$};
    \filldraw[black] (8.8, 2.8) circle (2pt) node[anchor=north]{$D$};
    \path[draw = gray,% 
             decoration={%
            markings,%
            mark=at position 0.5   with {\arrow[scale=1.5]{>}},%
            },%
            postaction=decorate] (0,0) -- (5.4, 2)
            node[midway, above, color= gray, font=\small]{$\theta$};
    \path[draw = gray,% 
             decoration={%
            markings,%
            mark=at position 0.5   with {\arrow[scale=1.5]{<}},%
            },%
            postaction=decorate] (3.6, 2) -- (5.4, 2)
            node[midway, below, color= gray, font=\small]{$\theta+\delta$};
    \path[draw = gray,% 
             decoration={%
            markings,%
            mark=at position 0.5   with {\arrow[scale=1.5]{>}},%
            },%
            postaction=decorate] (3.6, 2) -- (8.8, 2.8)
            node[midway, above, color= gray, font=\small]{$\theta$};
\end{tikzpicture}
\\ \hline \\ 1m &
\begin{tikzpicture}
     \filldraw[black] (0,0) circle (2pt) node[anchor=south]{$A$};
    \filldraw[black] (3.6, 2) circle (2pt) node[anchor=south]{$B$};
    \filldraw[black] (5.4, 2) circle (2pt) node[anchor=south]{$C$};
    \filldraw[black] (8.8, 2.8) circle (2pt) node[anchor=north]{$D$};
    \path[draw = gray,% 
             decoration={%
            markings,%
            mark=at position 0.5   with {\arrow[scale=1.5]{>}},%
            },%
            postaction=decorate] (0,0) -- (5.4, 2)
            node[midway, above, color= gray, font=\small]{$\theta$};
    \path[draw = gray,% 
             decoration={%
            markings,%
            mark=at position 0.5   with {\arrow[scale=1.5]{<}},%
            },%
            postaction=decorate] (3.6, 2) -- (5.4, 2)
            node[midway, above, color= gray, font=\small]{$\delta$};
    \path[draw = gray,% 
             decoration={%
            markings,%
            mark=at position 0.5   with {\arrow[scale=1.5]{>}},%
            },%
            postaction=decorate] (5.4, 2) -- (8.8, 2.8)
            node[midway, above, color= gray, font=\small]{$\theta$};
\end{tikzpicture}
\\ \hline
\\ 1o &
\begin{tikzpicture}
     \filldraw[black] (0,0) circle (2pt) node[anchor=south]{$A$};
    \filldraw[black] (3.6, 2) circle (2pt) node[anchor=south]{$B$};
    \filldraw[black] (5.4, 2) circle (2pt) node[anchor=north]{$C$};
    \filldraw[black] (8.8, 2.8) circle (2pt) node[anchor=north]{$D$};
       \path[draw = gray,% 
             decoration={%
            markings,%
            mark=at position 0.5   with {\arrow[scale=1.5]{>}},%
            },%
            postaction=decorate] (0,0) -- (5.4, 2)
            node[midway, above, color= gray, font=\small]{$\theta$};
       \path[draw = gray,% 
             decoration={%
            markings,%
            mark=at position 0.5   with {\arrow[scale=1.5]{>}},%
            },%
            postaction=decorate] (5.4, 2) -- (8.8, 2.8)
            node[midway, below, color= gray, font=\small]{$\theta+\delta$};
       \path[draw = gray,% 
             decoration={%
            markings,%
            mark=at position 0.5   with {\arrow[scale=1.5]{<}},%
            },%
            postaction=decorate] (3.6, 2) -- (8.8, 2.8)
            node[midway, above, color= gray, font=\small]{$\delta$};
\end{tikzpicture}\\ \hline  \\ 1s &
\begin{tikzpicture}
     \filldraw[black] (0,0) circle (2pt) node[anchor=south]{$A$};
    \filldraw[black] (3.6, 2) circle (2pt) node[anchor=south]{$B$};
    \filldraw[black] (5.4, 2) circle (2pt) node[anchor=south]{$C$};
    \filldraw[black] (8.8, 2.8) circle (2pt) node[anchor=north]{$D$};
    \path[draw = gray,% 
             decoration={%
            markings,%
            mark=at position 0.5   with {\arrow[scale=1.5]{>}},%
            },%
            postaction=decorate] (0,0) -- (8.8, 2.8)
            node[midway, below, color= gray, font=\small]{$\theta$};
    \path[draw = gray,% 
             decoration={%
            markings,%
            mark=at position 0.5   with {\arrow[scale=1.5]{<}},%
            },%
            postaction=decorate] (3.6, 2) -- (8.8, 2.8)
            node[midway, above, color= gray, font=\small]{$\delta$};
    \path[draw = gray,% 
             decoration={%
            markings,%
            mark=at position 0.5   with {\arrow[scale=1.5]{>}},%
            },%
            postaction=decorate] (5.4, 2) -- (8.8, 2.8)
            node[midway, below, color= gray, font=\small]{$\delta$};
\end{tikzpicture}\\ \hline
\caption{Representation of (1a), (1b), (1f), (1k), (1l), (1m), (1o), (1s).}
\end{longtable}

For $j$ sufficiently large and for $k\geq k_0(\alpha)$, the $\alpha$-mass corresponding to (1c) is (see Table \ref{Figch})
\begin{equation}\label{e:choicek1}
\begin{split}
&\theta^\alpha(\Haus^1(\sigma_{C_jD_j})+(1-k^{-1})^{\alpha}\Haus^1(\sigma_{A_jC_j})+k^{-\alpha}\Haus^1(\sigma_{A_jB_j}))\\
&> \theta^\alpha(\Haus^1(\sigma_{C_jD_j})+((1-k^{-1})^{\alpha}+k^{-\alpha})\Haus^1(\sigma_{A_jB_j}))\\
&> \theta^\alpha(\Haus^1(\sigma_{C_jD_j})+\Haus^1(\sigma_{A_jB_j})+\frac{k^{-\alpha}}{2}\Haus^1(\sigma_{A_jB_j}))\\
&> \theta^\alpha(\Haus^1(\sigma_{C_jD_j})+\Haus^1(\sigma_{A_jB_j})+k^{-\alpha}\Haus^1(\sigma_{B_jC_j}))=\MM(W_j).
\end{split}
\end{equation}
Also, for $j$ sufficiently large and for $k\geq k_0(\alpha)$, the $\alpha$-mass corresponding to (1h) is (see Table \ref{Figch})
\begin{equation}\label{e:choicek2}
\begin{split}
&\theta^\alpha(\Haus^1(\sigma_{A_jB_j})+(1-k^{-1})^{\alpha}\Haus^1(\sigma_{B_jD_j})+k^{-\alpha}\Haus^1(\sigma_{C_jD_j}))\\
&> \theta^\alpha(\Haus^1(\sigma_{A_jB_j})+((1-k^{-1})^{\alpha}+k^{-\alpha})\Haus^1(\sigma_{C_jD_j}))\\
&> \theta^\alpha(\Haus^1(\sigma_{A_jB_j})+\Haus^1(\sigma_{C_jD_j})+\frac{k^{-\alpha}}{2}\Haus^1(\sigma_{C_jD_j}))\\
&> \theta^\alpha(\Haus^1(\sigma_{A_jB_j})+\Haus^1(\sigma_{C_jD_j})+k^{-\alpha}\Haus^1(\sigma_{B_jC_j}))=\MM(W_j).
\end{split}
\end{equation}

\begin{longtable}{  p{1cm} c } \label{Figch} 
        \\1c &
        \begin{tikzpicture}
     \filldraw[black] (0,0) circle (2pt) node[anchor=south]{$A$};
    \filldraw[black] (3.6, 2) circle (2pt) node[anchor=south]{$B$};
    \filldraw[black] (5.4, 2) circle (2pt) node[anchor=south]{$C$};
    \filldraw[black] (8.8, 2.8) circle (2pt) node[anchor=north]{$D$};
    \path[draw = gray,% 
             decoration={%
            markings,%
            mark=at position 0.5   with {\arrow[scale=1.5]{>}},%
            },%
            postaction=decorate] (0,0) -- (3.6, 2)
            node[midway, above, color= gray, font=\small]{$\delta$};
    \path[draw = gray,% 
             decoration={%
            markings,%
            mark=at position 0.5   with {\arrow[scale=1.5]{>}},%
            },%
            postaction=decorate] (0,0) -- (5.4, 2)
            node[midway, below, color= gray, font=\small]{$\theta-\delta$};
    \path[draw = gray,% 
             decoration={%
            markings,%
            mark=at position 0.5   with {\arrow[scale=1.5]{>}},%
            },%
            postaction=decorate] (5.4, 2) -- (8.8, 2.8)
            node[midway, below, color= gray, font=\small]{$\theta$};
\end{tikzpicture}
\\ \hline
\\ 1h &
\begin{tikzpicture}
     \filldraw[black] (0,0) circle (2pt) node[anchor=south]{$A$};
    \filldraw[black] (3.6, 2) circle (2pt) node[anchor=north]{$B$};
    \filldraw[black] (5.4, 2) circle (2pt) node[anchor=north]{$C$};
    \filldraw[black] (8.8, 2.8) circle (2pt) node[anchor=north]{$D$};
    \path[draw = gray,% 
             decoration={%
            markings,%
            mark=at position 0.5   with {\arrow[scale=1.5]{>}},%
            },%
            postaction=decorate] (0,0) -- (3.6, 2)
            node[midway, above, color= gray, font=\small]{$\theta$};
    \path[draw = gray,% 
             decoration={%
            markings,%
            mark=at position 0.5   with {\arrow[scale=1.5]{>}},%
            },%
            postaction=decorate] (5.4, 2) -- (8.8, 2.8)
            node[midway, below, color= gray, font=\small]{$\delta$};
    \path[draw = gray,% 
             decoration={%
            markings,%
            mark=at position 0.5   with {\arrow[scale=1.5]{>}},%
            },%
            postaction=decorate] (3.6, 2) -- (8.8, 2.8)
            node[midway, above, color= gray, font=\small]{$\theta-\delta$};
\end{tikzpicture} \\ \hline 
  \caption{Representation of (1c), (1h).}
\end{longtable}

Lastly, we exclude case (1e) by direct comparison with the $\alpha$-mass of $Z_j$. For $j$ sufficiently large and for $k\geq k_0(\alpha)$, the $\alpha$-mass corresponding to (1e) is (see Table \ref{Fige})
\begin{equation}\label{e:choicek3}
\begin{split}
&\theta^\alpha((1-k^{-1})^{\alpha}\Haus^1(\sigma_{A_jD_j})+k^{-\alpha}(\Haus^1(\sigma_{A_jB_j})+\Haus^1(\sigma_{C_jD_j})))\\
&> \theta^\alpha((1-k^{-1})^{\alpha}\Haus^1(\sigma_{A_jD_j})+k^{-\alpha}\frac{1}{2}\Haus^1(\sigma_{A_jD_j}))\\
&> \theta^\alpha(\Haus^1(\sigma_{A_jD_j})+\frac{1}{4}k^{-\alpha}\Haus^1(\sigma_{A_jD_j}))\\
&> \theta^\alpha(\Haus^1(\sigma_{A_jD_j})+k^{-\alpha}\Haus^1(\sigma_{B_jC_j}))=\MM(Z_j).
\end{split}
\end{equation}

\begin{longtable}{  p{1cm} c } \label{Fige}
   \\ 1e &
\begin{tikzpicture}
     \filldraw[black] (0,0) circle (2pt) node[anchor=south]{$A$};
    \filldraw[black] (3.6, 2) circle (2pt) node[anchor=south]{$B$};
    \filldraw[black] (5.4, 2) circle (2pt) node[anchor=south]{$C$};
    \filldraw[black] (8.8, 2.8) circle (2pt) node[anchor=north]{$D$};
    \path[draw = gray,% 
             decoration={%
            markings,%
            mark=at position 0.5   with {\arrow[scale=1.5]{>}},%
            },%
            postaction=decorate] (0,0) -- (3.6, 2)
            node[midway, above, color= gray, font=\small]{$\delta$};
    \path[draw = gray,% 
             decoration={%
            markings,%
            mark=at position 0.5   with {\arrow[scale=1.5]{>}},%
            },%
            postaction=decorate] (0,0) -- (8.8, 2.8)
            node[midway, below, color= gray, font=\small]{$\theta-\delta$};
    \path[draw = gray,% 
             decoration={%
            markings,%
            mark=at position 0.5   with {\arrow[scale=1.5]{>}},%
            },%
            postaction=decorate] (5.4, 2) -- (8.8, 2.8)
            node[midway, above, color= gray, font=\small]{$\delta$};
    \end{tikzpicture}
\\ \hline
\caption{Representation of (1e).}
\end{longtable}

\emph{Sub-case 1-3}. The last situation we need to take into account is when exactly three points are collinear. We will discuss the case in which the collinear points are $A_j$, $B_j$ and $C_j$ or $A_j, C_j$ and $D_j$. The remaining cases in which the collinear points are $B_j$, $C_j$ and $D_j$ or $A_j, B_j$ and $D_j$ are symmetric and can be treated analogously, therefore we leave the analysis to the interested reader.

\emph{Sub-case 1-3-1: $A_j$, $B_j$ and $C_j$ are collinear}.

The cases (1d), (1i), (1j), (1q), (1r) can be excluded for the same reason as in the Sub-case 1-2 (see Table \ref{Fig_wrong}). We exclude cases (1b), (1f), (1l), (1n), which are coincident (see Table \ref{Fignall}), by direct comparison with the $\alpha$-mass of $Z_j$. For $j$ sufficiently large, the $\alpha$-mass corresponding to the above cases is
    $$\theta^\alpha(\Haus^1(\sigma_{A_jB_j})+\Haus^1(\sigma_{B_jD_j})+k^{-\alpha}\Haus^1(\sigma_{B_jC_j}))>\MM(Z_j).$$
\begin{longtable}{  p{1cm} c } \label{Fig_wrong} 
\\ 1d &
\begin{tikzpicture}
     \filldraw[black] (0,0) circle (2pt) node[anchor=south]{$A$};
    \filldraw[black] (1.8, 0) circle (2pt) node[anchor=south]{$B$};
    \filldraw[black] (2.6, 0) circle (2pt) node[anchor=south]{$C$};
    \filldraw[black] (4.4, 1.4) circle (2pt) node[anchor=north]{$D$};
    \draw[gray, thick] (0,0) -- (1.8, 0);
    \draw[gray, thick] (0,0) -- (4.4, 1.4);
    \draw[gray, thick] (1.8, 0) -- (2.6, 0);
\end{tikzpicture}
  \qquad \qquad
\begin{tikzpicture}
     \filldraw[black] (0,0) circle (2pt) node[anchor=south]{$A$};
    \filldraw[black] (1.8, 0) circle (2pt) node[anchor=south]{$B$};
    \filldraw[black] (2.6, 0) circle (2pt) node[anchor=south]{$C$};
    \filldraw[black] (4.4, 1.4) circle (2pt) node[anchor=north]{$D$};
    \draw[gray, thick] (0,0) -- (1.8, 0);
    \draw[gray, thick] (2.6, 0) -- (4.4, 1.4);
\end{tikzpicture}
\quad  1i 
\\ \hline 
\\ 1j &
\begin{tikzpicture}
     \filldraw[black] (0,0) circle (2pt) node[anchor=south]{$A$};
    \filldraw[black] (1.8, 0) circle (2pt) node[anchor=south]{$B$};
    \filldraw[black] (2.6, 0) circle (2pt) node[anchor=south]{$C$};
    \filldraw[black] (4.4, 1.4) circle (2pt) node[anchor=north]{$D$};
    \draw[gray, thick] (0,0) -- (2.6, 0);
    \draw[gray, thick] (1.8, 0) -- (2.6, 0);
    \draw[gray, thick] (0,0) -- (4.4, 1.4);
\end{tikzpicture}
 \qquad \qquad
\begin{tikzpicture}
\filldraw[black] (0,0) circle (2pt) node[anchor=south]{$A$};
    \filldraw[black] (1.8, 0) circle (2pt) node[anchor=south]{$B$};
    \filldraw[black] (2.6, 0) circle (2pt) node[anchor=south]{$C$};
    \filldraw[black] (4.4, 1.4) circle (2pt) node[anchor=north]{$D$};
    \draw[gray, thick] (0,0) -- (4.4, 1.4);
    \draw[gray, thick] (1.8, 0) -- (2.6, 0);
    \draw[gray, thick] (1.8, 0) -- (4.4, 1.4);
\end{tikzpicture}
\quad 1q
\\ \hline 
\\ 1r &
\begin{tikzpicture}
   \filldraw[black] (0,0) circle (2pt) node[anchor=south]{$A$};
    \filldraw[black] (1.8, 0) circle (2pt) node[anchor=south]{$B$};
    \filldraw[black] (2.6, 0) circle (2pt) node[anchor=south]{$C$};
    \filldraw[black] (4.4, 1.4) circle (2pt) node[anchor=north]{$D$};
    \draw[gray, thick] (0,0) -- (4.4, 1.4);
    \draw[gray, thick] (1.8, 0) -- (2.6, 0);
    \draw[gray, thick] (2.6, 0) -- (4.4, 1.4);  
\end{tikzpicture}
\\ \hline \caption{Representation of (1d), (1i), (1j), (1q), (1r) in the collinear case.} 
\end{longtable} 

\begin{longtable}{  p{1cm} c } \label{Fignall}
   \\ &
\begin{tikzpicture}
     \filldraw[black] (0,0) circle (2pt) node[anchor=south]{$A$};
    \filldraw[black] (3.6, 0) circle (2pt) node[anchor=south]{$B$};
    \filldraw[black] (5.4, 0) circle (2pt) node[anchor=south]{$C$};
    \filldraw[black] (8.8, 2.8) circle (2pt) node[anchor=north]{$D$};
    \path[draw = gray,% 
             decoration={%
            markings,%
            mark=at position 0.5   with {\arrow[scale=1.5]{>}},%
            },%
            postaction=decorate] (0,0) -- (3.6, 0)
            node[midway, above, color= gray, font=\small]{$\theta$};
    \path[draw = gray,% 
             decoration={%
            markings,%
            mark=at position 0.5   with {\arrow[scale=1.5]{>}},%
            },%
            postaction=decorate] (3.6,0) -- (8.8, 2.8)
            node[midway, above, color= gray, font=\small]{$\theta$};
    \path[draw = gray,% 
             decoration={%
            markings,%
            mark=at position 0.5   with {\arrow[scale=1.5]{>}},%
            },%
            postaction=decorate] (5.4, 0) -- (3.6, 0)
            node[midway, above, color= gray, font=\small]{$\delta$};
    \end{tikzpicture}
\\ \hline
\caption{Representation of (1b), (1f), (1l), (1n) in the collinear case.}
\end{longtable}

We exclude case (1a), since it coincides with case (1j), which we have already excluded and we exclude cases (1k) and (1o) because they contain a loop (see Table \ref{Figloop}).
\begin{longtable}{  p{1cm} c } \label{Figloop} 
\\ 1k &
\begin{tikzpicture}
     \filldraw[black] (0,0) circle (2pt) node[anchor=south]{$A$};
    \filldraw[black] (1.8, 0) circle (2pt) node[anchor=south]{$B$};
    \filldraw[black] (2.6, 0) circle (2pt) node[anchor=south]{$C$};
    \filldraw[black] (4.4, 1.4) circle (2pt) node[anchor=north]{$D$};
    \draw[gray, thick] (0,0) -- (2.6, 0);
    \draw[gray, thick] (0,0) -- (4.4, 1.4);
    \draw[gray, thick] (1.8, 0) -- (4.4, 1.4);
\end{tikzpicture}
  \qquad \qquad
\begin{tikzpicture}
     \filldraw[black] (0,0) circle (2pt) node[anchor=south]{$A$};
    \filldraw[black] (1.8, 0) circle (2pt) node[anchor=south]{$B$};
    \filldraw[black] (2.6, 0) circle (2pt) node[anchor=south]{$C$};
    \filldraw[black] (4.4, 1.4) circle (2pt) node[anchor=north]{$D$};
    \draw[gray, thick] (0,0) -- (2.6, 0);
    \draw[gray, thick] (2.6, 0) -- (4.4, 1.4);
    \draw[gray, thick] (1.8, 0) -- (4.4, 1.4);
\end{tikzpicture}
\quad  1o 
\\ \hline 
\caption{Representation of (1k) and (1o) in the collinear case.}
\end{longtable}

We do not need to exclude cases (1c) and (1m), since the current coincides with $Z_j$ (see Table \ref{Figcm}).

\begin{longtable}{  p{1cm} c } \label{Figcm}
   \\ &
\begin{tikzpicture}
     \filldraw[black] (0,0) circle (2pt) node[anchor=south]{$A$};
    \filldraw[black] (3.6, 0) circle (2pt) node[anchor=south]{$B$};
    \filldraw[black] (5.4, 0) circle (2pt) node[anchor=south]{$C$};
    \filldraw[black] (8.8, 2.8) circle (2pt) node[anchor=north]{$D$};
    \path[draw = gray,% 
             decoration={%
            markings,%
            mark=at position 0.5   with {\arrow[scale=1.5]{>}},%
            },%
            postaction=decorate] (0,0) -- (3.6, 0)
            node[midway, above, color= gray, font=\small]{$\theta$};
    \path[draw = gray,% 
             decoration={%
            markings,%
            mark=at position 0.5   with {\arrow[scale=1.5]{>}},%
            },%
            postaction=decorate] (5.4,0) -- (8.8, 2.8)
            node[midway, above, color= gray, font=\small]{$\theta$};
    \path[draw = gray,% 
             decoration={%
            markings,%
            mark=at position 0.5   with {\arrow[scale=1.5]{>}},%
            },%
            postaction=decorate] (3.6, 0) -- (5.4, 0)
            node[midway, above, color= gray, font=\small]{$\theta-\delta$};
    \end{tikzpicture}
\\ \hline
\caption{Representation of (1c), (1m) in the collinear case.} \endlastfoot
\end{longtable}

Lastly, cases (1e), (1h), (1s) can be excluded with the same argument used in Sub-case 1-2, since the segments in the corresponding support are in general position also when $A_j, B_j$, and $C_j$ are collinear (see Table \ref{Fehs}).

\begin{longtable}{  p{1cm} c } \label{Fehs} 
\\ 1e &
\begin{tikzpicture}
     \filldraw[black] (0,0) circle (2pt) node[anchor=south]{$A$};
    \filldraw[black] (1.8, 0) circle (2pt) node[anchor=south]{$B$};
    \filldraw[black] (2.6, 0) circle (2pt) node[anchor=south]{$C$};
    \filldraw[black] (4.4, 1.4) circle (2pt) node[anchor=north]{$D$};
    \draw[gray, thick] (0,0) -- (1.8, 0);
    \draw[gray, thick] (4.4, 1.4) -- (2.6, 0);
    \draw[gray, thick] (0,0) -- (4.4, 1.4);
\end{tikzpicture}
 \qquad \qquad
\begin{tikzpicture}
\filldraw[black] (0,0) circle (2pt) node[anchor=south]{$A$};
    \filldraw[black] (1.8, 0) circle (2pt) node[anchor=south]{$B$};
    \filldraw[black] (2.6, 0) circle (2pt) node[anchor=south]{$C$};
    \filldraw[black] (4.4, 1.4) circle (2pt) node[anchor=north]{$D$};
    \draw[gray, thick] (2.6,0) -- (4.4, 1.4);
    \draw[gray, thick] (1.8, 0) -- (0, 0);
    \draw[gray, thick] (1.8, 0) -- (4.4, 1.4);
\end{tikzpicture}
\quad 1h
\\ \hline 
\\ 1s &
\begin{tikzpicture}
   \filldraw[black] (0,0) circle (2pt) node[anchor=south]{$A$};
    \filldraw[black] (1.8, 0) circle (2pt) node[anchor=south]{$B$};
    \filldraw[black] (2.6, 0) circle (2pt) node[anchor=south]{$C$};
    \filldraw[black] (4.4, 1.4) circle (2pt) node[anchor=north]{$D$};
    \draw[gray, thick] (0,0) -- (4.4, 1.4);
    \draw[gray, thick] (4.4, 1.4) -- (1.8, 0);
    \draw[gray, thick] (2.6, 0) -- (4.4, 1.4);  
\end{tikzpicture}
\\ \hline
\caption{Representation of (1e), (1h), (1s) in the collinear case.}
\end{longtable}

\emph{Sub-case 1-3-2: $A_j$, $C_j$ and $D_j$ are collinear}.

The cases (1d), (1i), (1n), (1q) can be excluded for the same reason as in the Sub-case 1-2 (see Table \ref{Fig_wrong2}).

\begin{longtable}{  p{1cm} c } \label{Fig_wrong2} 
\\ 1d &
\begin{tikzpicture}
     \filldraw[black] (0,0) circle (2pt) node[anchor=south]{$A$};
    \filldraw[black] (1.6, 0.6) circle (2pt) node[anchor=south]{$B$};
    \filldraw[black] (2.4, 0.6) circle (2pt) node[anchor=south]{$C$};
    \filldraw[black] (4, 1) circle (2pt) node[anchor=north]{$D$};
    \draw[gray, thick] (0,0) -- (1.6, 0.6);
    \draw[gray, thick] (0,0) -- (4, 1);
    \draw[gray, thick] (1.6, 0.6) -- (2.4, 0.6);
\end{tikzpicture}
  \qquad \qquad
\begin{tikzpicture}
     \filldraw[black] (0,0) circle (2pt) node[anchor=south]{$A$};
    \filldraw[black] (1.6, 0.6) circle (2pt) node[anchor=south]{$B$};
    \filldraw[black] (2.4, 0.6) circle (2pt) node[anchor=south]{$C$};
    \filldraw[black] (4, 1) circle (2pt) node[anchor=north]{$D$};
    \draw[gray, thick] (0,0) -- (1.6, 0.6);
    \draw[gray, thick] (2.4, 0.6) -- (4, 1);
\end{tikzpicture}
\quad  1i 
\\ \hline 
\\ 1n &
\begin{tikzpicture}
     \filldraw[black] (0,0) circle (2pt) node[anchor=south]{$A$};
    \filldraw[black] (1.6, 0.6) circle (2pt) node[anchor=south]{$B$};
    \filldraw[black] (2.4, 0.6) circle (2pt) node[anchor=north]{$C$};
    \filldraw[black] (4, 1) circle (2pt) node[anchor=north]{$D$};
    \draw[gray, thick] (0,0) -- (2.4, 0.6);
    \draw[gray, thick] (1.6, 0.6) -- (4, 1);
\end{tikzpicture}
 \qquad \qquad
\begin{tikzpicture}
\filldraw[black] (0,0) circle (2pt) node[anchor=south]{$A$};
    \filldraw[black] (1.6, 0.6) circle (2pt) node[anchor=south]{$B$};
    \filldraw[black] (2.4, 0.6) circle (2pt) node[anchor=north]{$C$};
    \filldraw[black] (4, 1) circle (2pt) node[anchor=north]{$D$};
    \draw[gray, thick] (0,0) -- (4, 1);
    \draw[gray, thick] (1.6, 0.6) -- (2.4, 0.6);
    \draw[gray, thick] (1.6, 0.6) -- (4, 1);
\end{tikzpicture}
\quad 1q
\\ \hline 
\caption{Representation of (1d), (1i), (1n), (1q) in the collinear case.}
\end{longtable}

We exclude cases (1k), (1o), (1s), which are coincident (see Table \ref{Fignall2first}), by direct comparison with the $\alpha$-mass of $Z_j$. For $j$ sufficiently large, the $\alpha$-mass corresponding to the above cases is
    $$\theta^\alpha((1+k^{-1})^\alpha\Haus^1(\sigma_{C_jD_j})+\Haus^1(\sigma_{A_jC_j})+k^{-\alpha}\Haus^1(\sigma_{B_jD_j}))>\MM(Z_j).$$

\begin{longtable}{  p{1cm} c } \label{Fignall2first}
   \\ &
\begin{tikzpicture}
     \filldraw[black] (0,0) circle (2pt) node[anchor=south]{$A$};
    \filldraw[black] (3.2, 1.2) circle (2pt) node[anchor=south]{$B$};
    \filldraw[black] (4.8, 1.2) circle (2pt) node[anchor=north]{$C$};
    \filldraw[black] (8, 2) circle (2pt) node[anchor=north]{$D$};
    \path[draw = gray,% 
             decoration={%
            markings,%
            mark=at position 0.5   with {\arrow[scale=1.5]{>}},%
            },%
            postaction=decorate] (0,0) -- (4.8, 1.2)
            node[midway, above, color= gray, font=\small]{$\theta$};
    \path[draw = gray,% 
             decoration={%
            markings,%
            mark=at position 0.5   with {\arrow[scale=1.5]{>}},%
            },%
            postaction=decorate] (4.8, 1.2) -- (8, 2)
            node[midway, below, color= gray, font=\small]{$\theta+\delta$};
    \path[draw = gray,% 
             decoration={%
            markings,%
            mark=at position 0.5   with {\arrow[scale=1.5]{>}},%
            },%
            postaction=decorate] (3.2, 1.2) -- (8, 2)
            node[midway, above, color= gray, font=\small]{$\delta$};
    \end{tikzpicture}
\\ \hline
\caption{Representation of (1k), (1o), (1s) in the collinear case.}
\end{longtable}

We do not exclude cases (1j), (1m) and (1r), since the current coincides with $Z_j$ (see Table \ref{Figcm2}).
\begin{longtable}{  p{1cm} c } \label{Figcm2}
   \\ &
\begin{tikzpicture}
     \filldraw[black] (0,0) circle (2pt) node[anchor=south]{$A$};
    \filldraw[black] (3.2, 1.2) circle (2pt) node[anchor=south]{$B$};
    \filldraw[black] (4.8, 1.2) circle (2pt) node[anchor=south]{$C$};
    \filldraw[black] (8, 2) circle (2pt) node[anchor=north]{$D$};
    \path[draw = gray,% 
             decoration={%
            markings,%
            mark=at position 0.5   with {\arrow[scale=1.5]{>}},%
            },%
            postaction=decorate] (0,0) -- (4.8, 1.2)
            node[midway, above, color= gray, font=\small]{$\theta$};
    \path[draw = gray,% 
             decoration={%
            markings,%
            mark=at position 0.5   with {\arrow[scale=1.5]{>}},%
            },%
            postaction=decorate] (4.8,1.2) -- (8, 2)
            node[midway, above, color= gray, font=\small]{$\theta$};
    \path[draw = gray,% 
             decoration={%
            markings,%
            mark=at position 0.5   with {\arrow[scale=1.5]{>}},%
            },%
            postaction=decorate] (4.8, 1.2) -- (3.2, 1.2)
            node[midway, above, color= gray, font=\small]{$\delta$};
    \end{tikzpicture}
\\ \hline
\caption{Representation of (1j), (1m), (1r) in the collinear case.} \endlastfoot
\end{longtable}

We exclude cases (1a), (1c), (1e), which are coincident (see Table \ref{Fignall2}), by direct comparison with the $\alpha$-mass of $W_j$. For $j$ sufficiently large and for $k\geq k_0(\alpha)$, the $\alpha$-mass corresponding to the above cases is
\begin{equation}\label{e:choicek4}
\begin{split}
&\theta^\alpha(\Haus^1(\sigma_{C_jD_j})+(1-k^{-1})^{\alpha}\Haus^1(\sigma_{A_jC_j})+k^{-\alpha}\Haus^1(\sigma_{A_jB_j}))\\
&> \theta^\alpha(\Haus^1(\sigma_{C_jD_j})+((1-k^{-1})^{\alpha}+k^{-\alpha})\Haus^1(\sigma_{A_jB_j}))\\
&> \theta^\alpha(\Haus^1(\sigma_{C_jD_j})+(1+\frac{1}{2}k^{-\alpha})\Haus^1(\sigma_{A_jB_j}))\\
&> \theta^\alpha(\Haus^1(\sigma_{C_jD_j})+\Haus^1(\sigma_{A_jB_j}) +k^{-\alpha}\Haus^1(\sigma_{B_jC_j}))=\MM(W_j).
\end{split}
\end{equation}

\begin{longtable}{  p{1cm} c } \label{Fignall2}
      \\ &
\begin{tikzpicture}
     \filldraw[black] (0,0) circle (2pt) node[anchor=south]{$A$};
    \filldraw[black] (3.2, 1.2) circle (2pt) node[anchor=south]{$B$};
    \filldraw[black] (4.8, 1.2) circle (2pt) node[anchor=south]{$C$};
    \filldraw[black] (8, 2) circle (2pt) node[anchor=north]{$D$};
    \path[draw = gray,% 
             decoration={%
            markings,%
            mark=at position 0.5   with {\arrow[scale=1.5]{>}},%
            },%
            postaction=decorate] (0,0) -- (4.8, 1.2)
            node[midway, below, color= gray, font=\small]{$\theta-\delta$};
    \path[draw = gray,% 
             decoration={%
            markings,%
            mark=at position 0.5   with {\arrow[scale=1.5]{>}},%
            },%
            postaction=decorate] (4.8,1.2) -- (8, 2)
            node[midway, above, color= gray, font=\small]{$\theta$};
    \path[draw = gray,% 
             decoration={%
            markings,%
            mark=at position 0.5   with {\arrow[scale=1.5]{>}},%
            },%
            postaction=decorate] (0, 0) -- (3.2, 1.2)
            node[midway, above, color= gray, font=\small]{$\delta$};
    \end{tikzpicture}
\\ \hline
\caption{Representation of (1a), (1c), (1e) in the collinear case.} \endlastfoot
\end{longtable}

Lastly, cases (1b), (1f), (1h), (1l) can be excluded with the same argument used in Sub-case 1-2, since the segments in the corresponding support are in general position also when $A_j, C_j$, and $D_j$ are collinear (see Table \ref{Fehs2}).

\begin{longtable}{  p{1cm} c } \label{Fehs2} 
\\ 1b &
\begin{tikzpicture}
     \filldraw[black] (0,0) circle (2pt) node[anchor=south]{$A$};
    \filldraw[black] (1.6, 0.6) circle (2pt) node[anchor=south]{$B$};
    \filldraw[black] (2.4, 0.6) circle (2pt) node[anchor=north]{$C$};
    \filldraw[black] (4, 1) circle (2pt) node[anchor=north]{$D$};
    \draw[gray, thick] (0,0) -- (1.6, 0.6);
    \draw[gray, thick] (4, 1) -- (1.6, 0.6);
    \draw[gray, thick] (0,0) -- (2.4, 0.6);
\end{tikzpicture}
 \qquad \qquad
\begin{tikzpicture}
\filldraw[black] (0,0) circle (2pt) node[anchor=south]{$A$};
    \filldraw[black] (1.6, 0.6) circle (2pt) node[anchor=south]{$B$};
    \filldraw[black] (2.4, 0.6) circle (2pt) node[anchor=north]{$C$};
    \filldraw[black] (4, 1) circle (2pt) node[anchor=north]{$D$};
    \draw[gray, thick] (2.4,0.6) -- (1.6, 0.6);
    \draw[gray, thick] (1.6, 0.6) -- (0, 0);
    \draw[gray, thick] (1.6, 0.6) -- (4, 1);
\end{tikzpicture}
\quad 1f
\\ \hline 
\\ 1h &
\begin{tikzpicture}
     \filldraw[black] (0,0) circle (2pt) node[anchor=south]{$A$};
    \filldraw[black] (1.6, 0.6) circle (2pt) node[anchor=south]{$B$};
    \filldraw[black] (2.4, 0.6) circle (2pt) node[anchor=north]{$C$};
    \filldraw[black] (4, 1) circle (2pt) node[anchor=north]{$D$};
    \draw[gray, thick] (0,0) -- (1.6, 0.6);
    \draw[gray, thick] (4, 1) -- (1.6, 0.6);
    \draw[gray, thick] (2.4,0.6) -- (4, 1);
\end{tikzpicture}
 \qquad \qquad
\begin{tikzpicture}
\filldraw[black] (0,0) circle (2pt) node[anchor=south]{$A$};
    \filldraw[black] (1.6, 0.6) circle (2pt) node[anchor=south]{$B$};
    \filldraw[black] (2.4, 0.6) circle (2pt) node[anchor=north]{$C$};
    \filldraw[black] (4, 1) circle (2pt) node[anchor=north]{$D$};
    \draw[gray, thick] (2.4,0.6) -- (0, 0);
    \draw[gray, thick] (1.6, 0.6) -- (2.4, 0.6);
    \draw[gray, thick] (1.6, 0.6) -- (4, 1);
\end{tikzpicture}
\quad 1l
\\ \hline
\caption{Representation of (1b), (1f), (1h), (1l) in the collinear case.}
\end{longtable}

\emph{Case 2: $\BR(S_{n_j}\res Q') = \{E_j\}$.} Recalling that $E_j$ is the endpoint of at least three segments in the support of $S_{n_j}\res Q'$, see the proof of Lemma \ref{l:numberBranch}, the only possibilities are that
$\supp(S_{n_j}\res Q')$ is one of the following sets (see Table \ref{Fig16}): 
\begin{itemize}
    \item [(2a)] $\sigma_{A_jE_j}\cup\sigma_{B_jE_j}\cup\sigma_{C_jE_j}\cup\sigma$, with $\sigma\neq\sigma_{D_jE_j}$,
    \item [(2b)] $\sigma_{A_jE_j}\cup\sigma_{B_jE_j}\cup\sigma_{D_jE_j}\cup\sigma$, with $\sigma\neq\sigma_{C_jE_j}$,
    \item [(2c)] $\sigma_{A_jE_j}\cup\sigma_{C_jE_j}\cup\sigma_{D_jE_j}\cup\sigma$, with $\sigma\neq\sigma_{B_jE_j}$,
    \item [(2d)] $\sigma_{B_jE_j}\cup\sigma_{C_jE_j}\cup\sigma_{D_jE_j}\cup\sigma$, with $\sigma\neq\sigma_{A_jE_j}$,
    \item [(2e)] $\sigma_{A_jE_j}\cup\sigma_{B_jE_j}\cup\sigma_{C_jE_j}\cup\sigma_{D_jE_j}$.
\end{itemize}  

\afterpage{
 \begin{longtable}{p{2cm} c  }\label{Fig16}
     2a &
        \begin{tikzpicture}
         \filldraw[black] (0,0) circle (2pt) node[anchor=south]{$A$};
        \filldraw[black] (3.6,2) circle (2pt) node[anchor=south]{$B$};
        \filldraw[black] (5.4, 2) circle (2pt) node[anchor=south]{$C$};
        \filldraw[black] (8.8, 1.2) circle (2pt) node[anchor=north]{$D$};
        \filldraw[black] (2.4, 1.1) circle (2pt) node[anchor=north]{$E$};
        \draw[gray, thick] (0,0) -- (2.4, 1.1);
        \draw[gray, thick] (3.6, 2) -- (2.4, 1.1);
        \draw[gray, thick] (5.4, 2) -- (2.4, 1.1);
        \end{tikzpicture}\\ \hline
        2b &
        \begin{tikzpicture}
         \filldraw[black] (0,0) circle (2pt) node[anchor=south]{$A$};
        \filldraw[black] (3.6,2) circle (2pt) node[anchor=south]{$B$};
        \filldraw[black] (5.4, 2) circle (2pt) node[anchor=south]{$C$};
        \filldraw[black] (8.8, 1.2) circle (2pt) node[anchor=north]{$D$};
        \filldraw[black] (3.5, 0.9) circle (2pt) node[anchor=north]{$E$};
        \draw[gray, thick] (0,0) -- (3.5, 0.9);
        \draw[gray, thick] (3.6, 2) -- (3.5, 0.9);
        \draw[gray, thick] (8.8, 1.2) -- (3.5, 0.9);
        \end{tikzpicture}\\ \hline
        2c &
        \begin{tikzpicture}
         \filldraw[black] (0,0) circle (2pt) node[anchor=south]{$A$};
        \filldraw[black] (3.6,2) circle (2pt) node[anchor=south]{$B$};
        \filldraw[black] (5.4, 2) circle (2pt) node[anchor=south]{$C$};
        \filldraw[black] (8.8, 1.2) circle (2pt) node[anchor=north]{$D$};
        \filldraw[black] (5.6, 1.2) circle (2pt) node[anchor=north]{$E$};
        \draw[gray, thick] (0,0) -- (5.6, 1.2);
        \draw[gray, thick] (5.4, 2) -- (5.6, 1.2);
        \draw[gray, thick] (8.8, 1.2) -- (5.6, 1.2);
        \end{tikzpicture}\\ \hline
            2d &
        \begin{tikzpicture}
         \filldraw[black] (0,0) circle (2pt) node[anchor=south]{$A$};
        \filldraw[black] (3.6,2) circle (2pt) node[anchor=south]{$B$};
        \filldraw[black] (5.4, 2) circle (2pt) node[anchor=south]{$C$};
        \filldraw[black] (8.8, 1.2) circle (2pt) node[anchor=north]{$D$};
        \filldraw[black] (6.5, 1.6) circle (2pt) node[anchor=north]{$E$};
        \draw[gray, thick] (3.6,2) -- (6.5, 1.6);
        \draw[gray, thick] (5.4, 2) -- (6.5, 1.6);
        \draw[gray, thick] (8.8, 1.2) -- (6.5, 1.6);
        \end{tikzpicture}\\ \hline
            2e &
        \begin{tikzpicture}
         \filldraw[black] (0,0) circle (2pt) node[anchor=south]{$A$};
        \filldraw[black] (3.6,2) circle (2pt) node[anchor=south]{$B$};
        \filldraw[black] (5.4, 2) circle (2pt) node[anchor=south]{$C$};
        \filldraw[black] (8.8, 1.2) circle (2pt) node[anchor=north]{$D$};
        \filldraw[black] (3.7, 1.1) circle (2pt) node[anchor=north]{$E$};
        \draw[gray, thick] (0,0) -- (3.7, 1.1);
        \draw[gray, thick] (3.6,2) -- (3.7, 1.1);
        \draw[gray, thick] (5.4, 2) -- (3.7, 1.1);
        \draw[gray, thick] (8.8, 1.2) -- (3.7, 1.1);
        \end{tikzpicture}
\\
\hline
\caption{Representation of (2a), (2b), (2c), (2d), (2e). In (2a), (2b), (2c), (2d) we do not represent the segment $\sigma$.}
\end{longtable}}

We exclude case (2a), indeed by Propositions \ref{p:12punto1} and \ref{p:BCMangoli}, $E_j\in{\rm{conv}}(\{A_j, B_j, C_j\})$, hence we have for $\rho$ small and $j$ sufficiently large
\begin{equation}\label{e:angolo1}
    \pi-O(\rho)=\angle{A_jB_jC_j}\leq\angle{A_jE_jC_j}\leq\pi \,.
\end{equation} 
This contradicts Proposition \ref{p:BCMangoli} for $\rho \leq \rho(k)$, since the modulus of the multiplicity of $\sigma_{A_jE_j},\sigma_{E_jB_j}$ and $\sigma_{E_jC_j}$ belongs to $[k^{-1},\Mass(b_{n_j})]$, which by \eqref{e:massa_bienne} is contained in $[k^{-1},(1+hk^{-1})\Mass(b)]$.\\

Cases (2b), (2c) and (2d) are excluded with a similar argument as in case (2a), where the angle $\angle{A_jE_jC_j}$ in \eqref{e:angolo1} is replaced respectively by $\angle{A_jE_jD_j}$, $\angle{A_jE_jD_j}$ and $\angle{B_jE_jD_j}$.\\

We exclude case (2e) by direct comparison with the $\alpha$-mass of $Z_j$. The $\alpha$-mass corresponding to (2e) is
\begin{equation}\label{e:case2e}
\begin{split}
&\theta^\alpha(\Haus^1(\sigma_{A_jE_j})+\Haus^1(\sigma_{E_jD_j})+k^{-\alpha}(\Haus^1(\sigma_{B_jE_j})+\Haus^1(\sigma_{E_jC_j})))\\
&\geq \theta^\alpha(\Haus^1(\sigma_{A_jD_j})+k^{-\alpha}\Haus^1(\sigma_{B_jC_j}))=\MM(Z_j),
\end{split}
\end{equation}
where the inequality is strict unless $\{E_j\}=\sigma_{A_jD_j}\cap \sigma_{B_jC_j}$, namely unless the current is $Z_j$, which of course we do not need to exclude.\\


\emph{Case 3: $\BR(S_{n_j}\res Q') = \{E_j,F_j\}$.} Recalling Proposition \ref{p:singlepath} and the fact that both $E_j$ and $F_j$ are the endpoints of at least three segments in the support of $S_{n_j}\res Q'$, see the proof of Lemma \ref{l:numberBranch}, up to switching between $E_j$ and $F_j$, the only possibilities are that $\supp(S_{n_j}\res Q')$ is one of the following sets (see Table \ref{Fig17}):
\begin{itemize}
    \item [(3a)] $\sigma_{E_jF_j}\cup\sigma_{A_jE_j}\cup\sigma_{B_jE_j}\cup\sigma_{C_jF_j}\cup\sigma_{D_jF_j}$,
    \item [(3b)] $\sigma_{E_jF_j}\cup\sigma_{A_jE_j}\cup\sigma_{C_jE_j}\cup\sigma_{B_jF_j}\cup\sigma_{D_jF_j}$,
    \item [(3c)] $\sigma_{E_jF_j}\cup\sigma_{A_jE_j}\cup\sigma_{D_jE_j}\cup\sigma_{B_jF_j}\cup\sigma_{C_jF_j}$.
\end{itemize}
\vspace{1.5cm}
(3a) 
Denote by $\pi_0$ the affine 2-plane passing through $A_j$, $B_j$ and $F_j$ (and therefore containing $E_j$ as well). By Proposition \ref{p:BCMangoli} the line $\ell$ containing $\sigma_{E_jF_j}$ divides $\pi_0\setminus \ell$ into two open half-planes $\pi_0^-$ and $\pi_0^+$ containing respectively $A_j$ and $B_j$. Let $C'_j$ and $D'_j$ denote the orthogonal projections onto $\pi_0$ of $C_j$ and $D_j$ respectively and observe that $C'_j\in \pi_0^+$. This follows from the fact that by Proposition \ref{p:BCMangoli} there exists a positive constant $\kappa$ (depending on $k$) such that $\angle A_jE_jF_j\leq \pi-\kappa$ and assuming $C'_j\notin \pi_0^+$ would lead to 
$$\angle A_jB_jC'_j \leq\angle A_jE_jF_j\leq \pi-\kappa,$$
which is a contradiction, for $\rho$ sufficiently small with respect to $k$, since, due to the fact that $\angle A_jB_jC_j\geq \frac\pi 2$, 
$$\angle A_jB_jC'_j\geq\angle A_jB_jC_j\geq \pi-O(\rho).$$
On the other hand, the fact that $C'_j\in\pi_0^+$ implies that $D'_j\in\pi_0^-$, hence
$$\angle A_jE_jD'_j\leq\angle A_jE_jF_j\leq \pi-\kappa,$$
which is a contradiction for $\rho$ sufficiently small with respect to $k$, since, as above, 
$$\angle A_jE_jD'_j\geq\angle A_jE_jD_j\geq \pi-O(\rho).$$

\afterpage{
\begin{longtable}{p{2cm} c  }\label{Fig17}
     3a &
\begin{tikzpicture}
         \filldraw[black] (0,0) circle (2pt) node[anchor=south]{$A$};
        \filldraw[black] (4,1.8) circle (2pt) node[anchor=south]{$B$};
        \filldraw[black] (6, 1.8) circle (2pt) node[anchor=south]{$C$};
        \filldraw[black] (10, 0) circle (2pt) node[anchor=north]{$D$};
        \filldraw[black] (3.5, 1) circle (2pt) node[anchor=north]{$E$};
        \filldraw[black] (6.5, 1.1) circle (2pt) node[anchor=north]{$F$};
        \path[draw = gray,% 
             decoration={%
            markings,%
            mark=at position 0.5   with {\arrow[scale=1.5]{>}},%
            },%
            postaction=decorate] (0,0) -- (3.5, 1)
            node[midway, above, color= gray, font=\small]{$\theta$};
        \path[draw = gray,% 
             decoration={%
            markings,%
            mark=at position 0.5   with {\arrow[scale=1.5]{>}},%
            },%
            postaction=decorate] (6.5, 1.1)  -- (10, 0)
            node[midway, above, color= gray, font=\small]{$\theta$};
        \path[draw = gray,% 
             decoration={%
            markings,%
            mark=at position 0.5   with {\arrow[scale=1.5]{<}},%
            },%
            postaction=decorate] (6.5, 1.1)  -- (6, 1.8)
            node[midway, above right, color= gray, font=\small]{$\delta$};
        \path[draw = gray,% 
             decoration={%
            markings,%
            mark=at position 0.5   with {\arrow[scale=1.5]{>}},%
            },%
            postaction=decorate] (3.5, 1) -- (4, 1.8)
            node[midway, above left, color= gray, font=\small]{$\delta$};
        \path[draw = gray,% 
             decoration={%
            markings,%
            mark=at position 0.5   with {\arrow[scale=1.5]{>}},%
            },%
            postaction=decorate] (3.5, 1) -- (6.5, 1.1)
            node[midway, above, color= gray, font=\small]{$\theta-\delta$};
        \end{tikzpicture}\\ \hline
        3b &
\begin{tikzpicture}
         \filldraw[black] (0,0) circle (2pt) node[anchor=south]{$A$};
        \filldraw[black] (5.5, 2.6) circle (2pt) node[anchor= south east]{$B$};
        \filldraw[black] (4.5, 0.6) circle (2pt) node[anchor=north west]{$C$};
        \filldraw[black] (10, 3.0) circle (2pt) node[anchor=north]{$D$};
        \filldraw[black] (5.7, 2.0) circle (2pt) node[anchor=north west]{$F$};
        \filldraw[black] (4.3, 1.2) circle (2pt) node[anchor=south east]{$E$};
         \path[draw = gray,% 
             decoration={%
            markings,%
            mark=at position 0.5   with {\arrow[scale=1.5]{>}},%
            },%
            postaction=decorate] (0,0) -- (4.3, 1.2)
            node[midway, above , color= gray, font=\small]{$\theta$};
         \path[draw = gray,% 
             decoration={%
            markings,%
            mark=at position 0.5   with {\arrow[scale=1.5]{>}},%
            },%
            postaction=decorate] (5.7, 2.0)  -- (10, 3.0)
            node[midway, below, color= gray, font=\small]{$\theta$};
         \path[draw = gray,% 
             decoration={%
            markings,%
            mark=at position 0.6   with {\arrow[scale=1.5]{>}},%
            },%
            postaction=decorate] (5.7, 2.0)  -- (5.5, 2.6)
            node[midway, right , color= gray, font=\small]{$\delta$};
         \path[draw = gray,% 
             decoration={%
            markings,%
            mark=at position 0.5   with {\arrow[scale=1.5]{<}},%
            },%
            postaction=decorate] (5.7, 2.0) -- (4.3, 1.2)
            node[midway, below right, color= gray, font=\small]{$\theta+\delta$};
         \path[draw = gray,% 
             decoration={%
            markings,%
            mark=at position 0.5   with {\arrow[scale=1.5]{<}},%
            },%
            postaction=decorate] (4.3, 1.2) -- (4.5, 0.6)
            node[midway, left, color= gray, font=\small]{$\delta$};
        \end{tikzpicture} \\ \hline
        3c &
\begin{tikzpicture}
         \filldraw[black] (0,0) circle (2pt) node[anchor=south]{$A$};
        \filldraw[black] (4,1.8) circle (2pt) node[anchor=south]{$B$};
        \filldraw[black] (6, 1.8) circle (2pt) node[anchor=south]{$C$};
        \filldraw[black] (10, 0) circle (2pt) node[anchor=north]{$D$};
        \filldraw[black] (4.8, 0.6) circle (2pt) node[anchor=north]{$E$};
        \filldraw[black] (5.2, 1.3) circle (2pt) node[anchor=south]{$F$};
        \draw[gray, thick] (0,0) -- (4.8, 0.6);
        \draw[gray, thick] (4.8, 0.6)  -- (10, 0);
        \draw[gray, thick] (4.8, 0.6)  -- (5.2, 1.3);
        \draw[gray, thick] (5.2, 1.3) -- (4,1.8);
        \draw[gray, thick] (5.2, 1.3) -- (6, 1.8);
        \end{tikzpicture}\\
\hline
\caption{Representation of (3a), (3b), (3c).} 
\end{longtable}}

(3b) By Proposition \ref{p:BCMangoli} applied at the branch point $E_j$ we deduce that the angle between the oriented segments $\sigma_{A_j E_j}$ and $\sigma_{E_j F_j}$ tends to 0 as $k\to\infty$. By the same argument applied at the branch point $F_j$ we deduce the same property for the angle between the oriented segments $\sigma_{E_j F_j}$ and $\sigma_{F_j D_j}$. As a consequence, the angle between the oriented segments $\sigma_{A_j D_j}$ and $\sigma_{E_j F_j}$ tends to 0 as $k\to\infty$. Again, by Proposition \ref{p:BCMangoli}, the angles $\angle C_jE_jF_j$ and $\angle E_jF_jB_j$ are equal to $\frac{\pi}{2}+C(k)$ where $C(k)$ tends to 0 as $k\to\infty$. 

Next, using that the angle $\angle E_jF_jD_j$ differs from $\pi$ by a positive constant which depends only on $k$, we observe that the plane containing $A_j,C_j,F_j$ (and therefore also $E_j$) is obtained from the plane containing $D_j,B_j,E_j$ (and therefore also $F_j$) by a rotation $O$ around the line containing $\sigma_{E_jF_j}$ such that, for any fixed $k$,  $\|O-Id\|<f(\rho)$, where $f(\rho)$ tends to 0 as $\rho\to 0$. This implies that the angle between the oriented segments $\sigma_{B_jC_j}$ and $\sigma_{A_jE_j}$ is larger than $\frac{\pi}{2}-c(k)$, where $c(k)$ tends to 0 as $k\to\infty$. This is a contradiction for $\rho$ sufficiently small and $k$ sufficiently large, since for any $k$ the angle between the oriented segments $\sigma_{B_jC_j}$ and $\sigma_{A_jD_j}$ tends to 0 as $\rho\to 0$ and for any $\rho$ the angle between the oriented segment $\sigma_{A_jE_j}$ and the oriented segment $\sigma_{A_jD_j}$ tends to 0 (independently of $\rho$) as $k\to \infty$.

(3c) We exclude this case as the corresponding set is not the support of any
current with boundary $ \partial (S_{n_j} \res Q') $, because both the segments $\sigma_{A_j E_j} $ and $\sigma_{E_j D_j}$ should have multiplicity $\theta$, thus the multiplicity of $\sigma_{E_j F_j}$ would be zero.
\vspace{1.5cm}

\subsubsection{Conclusion.}\label{Step3} 
Let $p\in\{p_1,\dots,p_h\}\setminus\supp(S)$ and take $r>0$ such that 
\begin{equation}\label{e:p_isolated}
p\notin B_{3r} \big(\supp(S)\cup (\{p_1, \dots, p_h\}\setminus \{p\}) \big).    
\end{equation}
Applying Lemma \ref{l:conv_hauss}, for $j$ sufficiently large we have
\begin{equation}\label{e:sn_dentro}
    \begin{split}
        \supp(S_{n_j}) &\subset B_r(\supp(S)\cup\{p_1,\dots,p_h\})\\
        &=B_r(\supp(S)\cup\{p_1,\dots,p_h\}\setminus \{p\})\cup B_r(p)\,.
    \end{split}
\end{equation}
By \eqref{e:p_isolated} we have that $B_r(\supp(S)\cup\{p_1,\dots,p_h\}\setminus \{p\})$ and $B_{2r}(p)$ are disjoint. Define $\tilde S_{n_j} := S_{n_j} \res B_{2r}(\{p\})$. 
By Proposition \ref{p:simonslice} with $f(x)=\dist(x,\{p\})$ and \eqref{e:sn_dentro},
we have that for $j$ sufficiently large
\begin{equation}
\partial \tilde S_{n_j} = b_{n_j} \res B_{2r}(\{p\}) = -\frac{1}{k}\partial(T\res B_{{n_j}^{-1}}(p)),
\end{equation}
which is supported in exactly two points.
Since necessarily $\tilde S_{n_j}\in\OTP(\partial \tilde S_{n_j})$, we deduce that 
\begin{equation}\label{e:snoutsidesupport}
\tilde S_{n_j} = -\frac{1}{k} T\res B_{{n_j}^{-1}}(p),
\end{equation}
for $j$ sufficiently large.\\

Combining \eqref{e:snoutsidesupport}, \eqref{e:THE_ONE} and \eqref{e:alphamass_tienne}, we obtain that, for $j$ sufficiently large and for $k\geq k_0(\alpha)$
\begin{equation}
\begin{split}
        \MM &\left(S_{n_j}+\frac1k \sum_{i=1}^h T\res B_{{n_j}^{-1}}(p_i)\right)\\
        &=\MM(S_{n_j})+\sum_{i:S_{n_j}\res Q'=W_j}(1-(1-k^{-1})^\alpha)\MM(T\res B_{{n_j}^{-1}}(p_i))\\
        &\quad -\sum_{i:S_{n_j}\res Q'=Z_j}k^{-\alpha}\MM(T\res B_{{n_j}^{-1}}(p_i))-\sum_{i:p_i\notin\supp(S)}k^{-\alpha}\MM(T\res B_{{n_j}^{-1}}(p_i))\\
        &\leq \MM(T_{n_j})+\sum_{i:S_{n_j}\res Q'=W_j}(1-(1-k^{-1})^\alpha)\MM(T\res B_{{n_j}^{-1}}(p_i))\\
        &\leq \MM(T_{n_j})+\sum_{i=1}^h (1-(1-k^{-1})^\alpha)\MM(T\res B_{{n_j}^{-1}}(p_i))=\MM(T).
        \end{split}
\end{equation}
Observe that the first inequality is strict unless 
\begin{equation}\label{e:finalinclusion}
\{i:p_i\notin\supp(S)\}=\emptyset=\{i:S_{n_j}\res Q'=Z_j\}.    
\end{equation}
On the other hand, the inequality cannot be strict, since $\partial (S_{n_j}+\frac1k \sum_{i=1}^h T\res B_{{n_j}^{-1}}(p_i))=b$ by \eqref{e:bienne} and $T\in\OTP(b)$. We deduce that \eqref{e:finalinclusion} holds, which by Lemma \ref{l:magic_points} implies that $S=T$ and $S_{n_j}+\frac1k \sum_{i=1}^h T\res B_{{n_j}^{-1}}(p_i)=T$ and therefore, recalling \eqref{e:bienne}, $S_{n_j}=T_{n_j}$ for $j$ sufficiently large. 
\end{proof}

\begin{proof}[Proof of Proposition \ref{p:unique_bienne}]
Since the conclusion of Lemma \ref{l:main} holds for every converging subsequence $S_{n_j}$, we deduce that $S_n=T_n$ and therefore $\OTP(b_n)=\{T_n\}$, for $n$ sufficiently large.
\end{proof}

\subsection{Proof of Theorem \ref{t:main}}
Recall that by Lemma \ref{l:residual} it suffices to prove that the set $A_C\setminus NU_C$ is $\Flat_K$-dense in $A_C$. Fix $b\in A_C$ and $\varepsilon>0$. Let $\delta>0$ and $b''\in A_{C-\delta}$ be obtained by Lemma \ref{l:integral_bdry}. In particular let $b_I\in\I_0(K)$ be such that $b''=\eta b_I$ for some $\eta>0$.

Fix $T\in\OTP(b_I)$ and let $p_1,\dots,p_h$ be obtained applying Lemma \ref{l:magic_points} to the current $T$. Observe that $h$ depends on $T$. Let $k\in\N \setminus \{0\}$ be such that $k\geq k_0(\alpha)$ given by Proposition \ref{p:unique_bienne} and moreover $hk^{-1}C\leq\eta^{-1}\delta$.
For $n=1,2,\dots$, let $b_n$ be obtained as in \eqref{e:bienne}, where $b$ is replaced with $b_I$. By \eqref{e:massa_bienne}, for every $n$ we have 
$$\Mass(\eta b_n) = \eta\Mass(b_n)\leq\eta(\Mass(b_I)+hk^{-1}C)\leq\eta(\eta^{-1}(C-\delta)+\eta^{-1}\delta)=C.$$
Moreover, letting $S_n\in\OTP(b_n)$, by \eqref{e:alphamass_tienne} we have $\MM(\eta S_n)\leq\MM(\eta T)\leq C-\delta$, which allows to conclude that $\eta b_n\in A_C$ for every $n\in\N \setminus \{0\}$. By Proposition \ref{p:unique_bienne} we deduce that $\eta b_n\in A_C\setminus NU_C$ for $n$ sufficiently large, and by \eqref{e:convergence_bienne}, we have $$\Flat_K(\eta b_n-b)\leq\Flat_K(\eta b_n-\eta b_I)+\Flat_K(b''- b) <2\varepsilon,$$ 
for $n$ sufficiently large. By the arbitrariness of $\varepsilon$ we finish the proof of the density of $A_C\setminus NU_C$ and hence the proof of Theorem \ref{t:main}, concluding generic uniqueness of optimal transport paths. \null\nobreak\hfill\ensuremath{\square}

\end{document}